\documentclass{acta_2011}
\usepackage{harvard_acta}
\usepackage{amssymb}
\usepackage{amsmath}
\usepackage{graphicx}
\usepackage{color}

\title[Computing quantum dynamics in the semiclassical regime]{Computing quantum dynamics in the semiclassical regime}
\author[Caroline Lasser and Christian Lubich]{%
Caroline Lasser\\
{\tt classer@ma.tum.de}\\
\and
Christian Lubich\\
{\tt lubich@na.uni-tuebingen.de}
}

\pagerange{\pageref{firstpage}--\pageref{lastpage}}
\doi{XXXXXX}

\newcommand{\eg}{{\it e.g.}}
\newcommand{\ie}{{\it i.e.}}

\newcommand{\cn}{\citeasnoun}
\newcommand{\ignore}[1]{}

\newcommand{\dist}{\mathrm{dist}}
\newcommand{\tr}{\mathrm{tr}}
\newcommand{\eps}{\varepsilon}
\newcommand{\e}{{\mathrm e}}
\newcommand{\Id}{\mathrm{Id}}
\newcommand{\N}{{\mathbb N}}
\newcommand{\Rb}{{\mathbb R}}
\newcommand{\T}{{\mathbb T}}
\newcommand{\C}{{\mathbb C}}

\newcommand{\calF}{{\mathcal F}}
\newcommand{\calK}{{\mathcal K}}
\newcommand{\calM}{{\mathcal M}}

\newcommand{\calS}{{\mathcal S}}

\newcommand{\calV}{{\mathcal V}}
\newcommand{\calW}{{\mathcal W}}

\newtheorem{remark}{Remark}[section]
\newtheorem{theorem}{Theorem}[section]
\newtheorem{lemma}[theorem]{Lemma}
\newtheorem{corollary}[theorem]{Corollary}
\newtheorem{proposition}[theorem]{Proposition}
\newtheorem{definition}{Definition}[section]

\def\bch{\color{black}}
\def\ech{\color{black}\,}

\begin{document}

\label{firstpage}
\maketitle

\begin{abstract}
The semiclassically scaled time-dependent multi-particle  Schr\"odinger equation describes, {\it inter alia}, quantum dynamics of nuclei in a molecule. It poses the combined computational challenges of high oscillations and high dimensions. This paper reviews and studies numerical approaches that are robust to the small semiclassical parameter. We present and analyse variationally evolving Gaussian wave packets, Hagedorn's semiclassical wave packets, continuous superpositions of both thawed and frozen Gaussians, and Wigner function approaches to the direct computation of expectation values of observables.
Making good use of classical mechanics is essential for all these approaches. The arising aspects of time integration and high-dimensional quadrature are also discussed.
%
%Numerical approaches to the semiclassically scaled time-dependent multi-particle  Schr\"odinger equation that are robust to the small semiclassical parameter, are presented and studied in this work. Making good use of classical mechanics turns out essential to address the combined challenge of high oscillations and high dimensions. We present and analyse variationally evolving Gaussian wave packets, Hagedorn's semiclassical wave packets, continuous superpositions of both thawed and frozen Gaussians, and Wigner function based approaches to the direct computation of expectation values of observables. The arising aspects of time integration and high-dimensional quadrature are also discussed.
\end{abstract}

\tableofcontents 

%\section{Introduction}
\section{Introduction}

``{\it Semiclassical}\/'' is a notion that arises in a variety of situations in physics and chemistry and comes with an even larger variety of different viewpoints and analytical and computational techniques in mathematics. We just refer to the monographs by \cn{Hel18} and \cn{Zwo12} to illustrate the wide span of what is referred to as ``semiclassical", from a range of quantum-mechanical problems and their approximate models to a branch of microlocal analysis. In a vague but useful conception,  a quantum-mechanical problem is considered to be semiclassical when its approximation can make good use of classical mechanics (see \eg\ \cn{Mil74}). This aspect will be emphasised here. 

The basic equation in this field is 
the {\it time-dependent Schr\"odinger equation in semiclassical scaling},
\begin{equation}
\label{tdse-sc}
\I\eps \,\partial_t \psi(x,t) = -\frac{\eps^2}2 \Delta_x \psi(x,t) + V(x)\psi(x,t), \qquad 0 < \eps \ll 1,
\end{equation}
for the complex-valued wave function $\psi$ that depends on spatial variables $x\in\Rb^d$ and time $t\in\Rb$.  Here, $\Delta_x$ is the Laplacian on $\Rb^d$, and $V:\Rb^d\to\Rb$ is a smooth potential that acts as a multiplication operator in \eqref{tdse-sc}. The small scaling parameter $\eps\ll 1$ can be viewed as an effective Planck constant, \ie\ as the ratio between Planck's constant $\hbar$ and a characteristic action (length times momentum or energy divided by frequency) of the physical system under consideration. The semiclassically scaled Schr\"odinger equation
arises, in particular, in the description of quantum dynamics of nuclei in a molecule, where the square of the small parameter $\eps$ equals the mass ratio of electrons and nuclei, and where the dimension $d$ equals three times the number of nuclei.

Equation \eqref{tdse-sc} with an arbitrarily small parameter $\eps>0$ is the partial differential equation for which numerical approximation by various approaches will be studied in this paper.  Computing approximate solutions to the initial value problem for \eqref{tdse-sc} is challenging because
\begin{itemize}
\item solutions are highly oscillatory in space and time;
\item the problem is high-dimensional.
\end{itemize}
This combination of high oscillations and high dimensions makes standard discretisations by grid-based numerical methods infeasible, whether finite difference, finite element or pseudospectral methods. Hence different, problem-adapted and asymptotic-preserving (as $\eps\to 0$) numerical approaches need to be developed, and this is the subject of the present paper.

There already exists an {\em Acta Numerica\,} review of mathematical and computational methods for semiclassical Schr\"odinger equations by \cn{JinMS11}, which presents their perspective on the vast subject just as this paper presents ours.  The two reviews complement each other. Moreover, we include here many new results and developments from the past decade.

In the motivating Section~\ref{sec:bo} we describe how the adiabatic or Born--Oppenheimer approximation to full molecular quantum dynamics (which includes electrons and nuclei) leads to a  Schr\"odinger equation~\eqref{tdse-sc}  for the nuclei only, which is semiclassically scaled because of the small mass ratio $\eps^2=m/M$ of electrons and nuclei. It is again the small mass ratio that justifies the Born--Oppenheimer approximation as an effective model of the quantum dynamics of nuclei, with an $O(\eps)$ accuracy.

In Section~\ref{sec:gwp}, the high-dimensional and highly oscillatory wave function is approximated by a single, variationally evolving complex Gaussian. This simple yet basic approximation yields an $O(\eps^{1/2})$ error in the $L^2$-norm and an $O(\eps)$ error in averages of observables. When written in Hagedorn's parametrisation of the complex Gaussian, the equations of motion show a remarkable correspondence with classical mechanics.

In Section~\ref{sec:hagwp} we present Hagedorn's semiclassical wave packets. These are ingeniously constructed polynomials times complex Gaussians, which make it possible to approximate the wave function to higher orders in $\eps$ by a combination of classical mechanics for the Gaussian parameters and a Galerkin method for the polynomial part that is robust as $\eps\to 0$.

In Section~\ref{sec:csg} we study continuous superpositions of Gaussians for the approximation of the wave function. The width of the evolving Gaussians can be either thawed or frozen, corresponding to Gaussian beams or the Herman--Kluk approximation, respectively. Both cases yield an approximation of $O(\eps)$ accuracy in the 
$L^2$-norm.  They both lead to a computational approach in which quadrature points in phase space (``particles") are transported by the classical flow and the linearised classical equations of motion are solved along the trajectories.

In Section~\ref{sec:wigner} we aim at directly approximating averages of observables rather than the wave function. We show how the combination of Egorov's theorem, which relates quantum observables and classically propagated observables, and of Wigner or Husimi functions, which represent averages of quantum observables as integrals over classical phase space, leads to a computational approach of $O(\eps^2)$ accuracy in which quadrature points in phase space are transported by the classical flow.

We note that Sections~\ref{sec:hagwp}, \ref{sec:csg} and~\ref{sec:wigner} can be read independently of each other, but they depend on material in Section~\ref{sec:gwp} to a varying degree.

The approximations described in Sections~\ref{sec:gwp}--\ref{sec:wigner} require appropriate time discretisation and the computation of high-dimensional integrals. These important computational aspects are considered in Sections~\ref{sec:time} and~\ref{sec:quad}, respectively.

In Section~\ref{sec:further} we briefly address some important topics that are not treated in detail in this article: %time-dependent potentials,
systems of semiclassical Schr\"odinger equations,
WKB-type approximations and nonlinear semiclassical  Schr\"odinger equations.

Aside from the introductory Section~\ref{sec:bo} and the final Section~\ref{sec:further}, we have  aimed to provide full proofs of all key results, emphasising basic ideas and techniques. In doing so, we have arrived at various new results and proofs, which are not published elsewhere. In the core Sections~\ref{sec:gwp}--\ref{sec:wigner} we have put references to the literature into notes at the end of the sections, whereas in the other, more diverse sections, references are integrated into the running text.

%\newpage
%\section{Born--Oppenheimer approximation}
\def\for{\quad\ \hbox{ for}\quad }
\def\with{\quad\ \hbox{ with}\quad }
\newcommand{\sfrac}[2]{\mbox{\footnotesize$\displaystyle{#1\over#2}$}}

%\section{Born--Oppenheimer approximation}
\section{Quantum dynamics of nuclei in molecules}
\label{sec:bo}

We describe how semiclassically scaled Schr\"odinger equations arise in the formulation of quantum dynamics of nuclei in molecules, with the  electron/nucleus mass ratio appearing as the squared small scaling parameter~$\eps$. The semiclassical Schr\"odinger equation for the nuclei comes about by restricting the electronic degrees of freedom in the molecular wave function to a potential energy surface. This approximation is known as the {\it adiabatic} or {\it Born--Oppenheimer approximation}. From the numerical analyst's viewpoint, this is a particular Galerkin approximation to the full molecular quantum dynamics. It is motivated and rigorously justified by the multi-scale nature of the problem, \ie\ for small parameters $\eps\ll 1$. 

\subsection{Molecular Hamiltonian}
We describe a molecule in terms of nuclei and electrons.  We assume that there are $N$ nuclei, with masses $M_n$ and charge numbers~$Z_n$. The position coordinates are collected in a vector 
\[
x = (x_1,\ldots,x_N)\quad\text{with each}\quad x_n\in\Rb^3.
\] 
Similarly, we collect the position coordinates of $L$ electrons in a vector 
\[
y = (y_1,\ldots,y_L)\quad\text{with each}\quad y_l\in\Rb^3.
\] 
Electronic mass and elementary charge are denoted by $m$ and $e$, respectively.  The kinetic energy operator
\[
T = T_N + T_e
\]
is the sum of the kinetic energy operators for the nuclei and the electrons, 
\[
T_N = -\sum_{n=1}^N \frac{\hbar^2}{2M_n} \Delta_{x_n},\qquad T_e = -\sum_{l=1}^L \frac{\hbar^2}{2m} \Delta_{y_l},
\] 
where \bch $\hbar$ is the reduced Planck's constant,\ech which has the physical dimension of an action (length times momentum or energy divided by frequency).
The potential energy operator
\[
V = V_{NN} + V_{Ne} + V_{ee}
\]
comprises Coulomb interactions between nuclei and nuclei, nuclei and electrons, as well as electrons and electrons. They are
\begin{align*}
V_{NN}(x) = \sum_{1\le m<n\le N} \frac{Z_m Z_n e^2}{|x_m-x_n|},\qquad V_{Ne}(x,y) = -\sum_{n=1}^N \sum_{l=1}^L \frac{Z_ne^2}{|x_n-y_l|}
\end{align*}
and
\[
V_{ee}(y) = \sum_{1\le k<l\le L} \frac{e^2}{|y_k-y_l|}.
\]
Adding the kinetic and the potential energy operator, we obtain the molecular Schr\"odinger operator (or Hamiltonian)
\[
H_{\rm mol} = T + V.
\]
Before turning to the dynamics induced by the molecular Schr\"odinger operator, we briefly 
comment on its basic functional analytic properties.

%\begin{remark}
In the following we distinguish between functions depending on both nuclear and electronic coordinates $x$ and $y$, functions only depending on the nuclear coordinates $x$ and functions only depending on the electronic coordinates $y$. We use the shorthand notation
\[
L^2_{x,y} = L^2(\Rb^{3N+3L}),\qquad
L^2_x = L^2(\Rb^{3N}),\qquad L^2_y = L^2(\Rb^{3L})
\]
for the corresponding Lebesgue spaces of square-integrable complex-valued functions and analogously for Sobolev spaces. \bch The inner products are conjugate linear in the first argument.\ech 
%\end{remark}

%\begin{remark}
We view the molecular Hamiltonian $H_{\rm mol}$ as an operator on the whole molecular function space $L^2_{x,y}$.
%without any anti-symmetrisation in the electronic degrees of freedom. The Pauli exclusion principle is not relevant in the present context.
%\end{remark}

%\begin{remark}
\cn{Kat51} proved that the molecular Hamiltonian $H_{\rm mol}$ is a self-adjoint linear operator with domain $D(H_{\rm mol}) = H^2_{x,y}$, see also \cite[Theorem X.16]{ReeS2}. 
%\end{remark}

\subsection{Molecular Schr\"odinger equation}
The quantum motion of a molecule is described by the time-dependent Schr\"odinger equation,
\[
\I\hbar\ \frac{\partial\Psi}{\partial t} = H_{\rm mol} \Psi,\qquad \Psi |_{t=0} = \Psi_0,
\]
which has a unique solution
\[
\Psi = \Psi(x_1,\ldots,x_N,y_1,\ldots,y_L,t)
\]
for all times $t\in\Rb$ for all square-integrable initial data 
\[
\Psi_0\in L^2_{x,y}.
\]
Using the unitary group
\[
\e^{-\I H_{\rm mol}\, t/\hbar},\qquad t\in\Rb,
\]
associated with the molecular Hamiltonian, the solution of the Schr\"o\-dinger equation may be written as
\[
\Psi(\cdot,t) = \e^{-\I H_{\rm mol} \, t/\hbar} \, \Psi_0.
\]
For general square-integrable initial data, it depends continuously on time. For initial data in the domain 
$D(H_{\rm mol}) = H^2_{x,y}$, the solution is continuously differentiable with respect to time.

\begin{remark}
The spectral representation of self-adjoint linear operators was developed by 
\cn[Section IX and Appendix II]{vNe30}, \cn{Sto29}, and \cn{Rie30}; see also 
\cite[Chapter VIII.3]{ReeS1}. Given the existence of the spectral representation of $H_{\rm mol}$, 
we obtain the existence of its unitary group without further ado.
\end{remark}

While it is reassuring that existence and uniqueness of the solution are guaranteed, the direct numerical simulation of the molecular Schr\"odinger equation is 
notoriously difficult, or rather intractable, for at least two reasons: 
\begin{description}
\item
{\em High dimensionality:} A molecular configuration space $\Rb^{3N+3L}$ is finite-dimensional, but typically high-dimensional. For example, a single CO$_2$ molecule has $N=3$ atoms and $L=22$ electrons. Hence, its configuration space has dimension $3N+3L = 75$. Conventional discretizations of partial differential equations by finite differences or finite element methods scale exponentially with respect to the dimension of the configuration space. 
Thus, they are inappropriate for the molecular Schr\"odinger equation, and one must search for alternatives. 
\item\item
{\em Multiple scales:} Nuclei are much heavier than electrons, which implies that nuclei move on considerably different time and length scales than electrons. For example, electronic motion is typically associated with the order of several tens or hundreds of attoseconds, while chemical reactions 
take several tens or hundreds of femtoseconds. We recall that 
\[
1\,\text{attosecond} = 10^{-18}\,{\rm s} \;\ll 1\; \text{femtosecond} = 10^{-15}\,{\rm s}.
\]
\end{description}
Our interest here is in numerical methods that resolve the femtosecond time scale, the time scale of chemical reactions.

\subsection{Electronic Schr\"odinger equation}
The mass discrepancy of nuclei and electrons motivates to clamp the nuclei, \ie, to fix a nuclear configuration $x=(x_1,\ldots,x_N)$, and to consider the associated electronic Hamiltonian
\[
H_e(x) = T_e + V_{Ne}(x,\cdot) + V_{ee}.
\]
The electronic Hamiltonian depends parametrically on the nuclear configuration and acts on functions of the electronic coordinates 
$y=(y_1,\ldots,y_L)$. 
%As for the molecular Hamiltonian, anti-symmetrization of the electronic coordinates is neglected here. 
The electronic Hamiltonian is a self-adjoint operator with domain $ H^2_y$ 
that is bounded from below  \cite[Theorem 1]{Kat51}. In view of the fermionic nature of electrons, the domain $D(H_e(x))$ is usually restricted to
the closed subspace of anti-symmetric functions in $H^2_y$.  The electronic Hamiltonian has a real eigenvalue $E(x)$ as the lowest point in its spectrum $\sigma(H_e(x))$,
\[
E(x) = \min\sigma(H_e(x)),
\]
which is called the ground state energy. Together with an eigenfunction 
$\Phi(x,\cdot)$ it satisfies the stationary electronic Schr\"odinger equation
\[
H_e(x)\Phi(x, \cdot) = E(x)\Phi(x,\cdot).
\]
The problem of computing a solution to this high-dimensional eigenvalue problem is often referred to as the electronic structure problem. 

\begin{remark}
Computing eigenvalues and eigenfunctions of the 
electronic Schr\"o\-dinger equation is the primary concern of computational quantum chemistry; see, \eg \ 
\cn{SzaO96} or \cn{Jen16} 
and from a more mathematical viewpoint \cn{CanDKLM03}, \cn{CanLM06}, \cn{LeB05}, and \cn{LinLY19}. Here we simply assume that this problem is solved in 
some satisfactory way.
\end{remark}

\subsection{Potential energy surfaces}

The electronic ground state energy %need not be obtained via the stationary electronic Schr\"odinger equation, but can also 
can be characterized 
using the Courant--Fischer min-max principle \cite[Theorem~4.2.6]{HorJ13}, as
\[
E(x) = \inf_{\dim(U) = 1} \ \sup_{\phi\in U,\,\|\phi\|_{L^2_y} = 1} \left\langle \phi, H_e(x)\phi\right\rangle_{L^2_y},
\]
where $U$ ranges over the one-dimensional subspaces of $D(H_e(x))$. 
For higher electronic eigenvalues a variational characterisation is also convenient, and 
the $m$th eigenvalue is given by
\[
E_m(x) = \inf_{\dim(U) = m} \ \sup_{\phi\in U,\,\|\phi\|_{L^2_y} = 1} \left\langle \phi, H_e(x)\phi\right\rangle_{L^2_y}\ \for m\ge 1.
\]
Now, $U$ ranges over the $m$-dimensional subspaces of $D(H_e(x))$. In this way, we also obtain an enumeration 
of the discrete electronic spectrum that is continuous with respect to the nuclear configuration, regardless of the 
multiplicities of the eigenvalues that may change when varying $x$. 

The functions that map a nuclear configuration to an electronic eigenvalue,
\[
\Rb^{3N}\to\Rb,\quad x\mapsto E_m(x)
\]
or
\[
\Rb^{3N}\to\Rb,\quad x\mapsto E_m(x) + V_{NN}(x),
\]
are called the potential energy surfaces of the molecule. They are important for 
our basic understanding of chemical reactions. 

\cn[Section~6]{HunG80} proved that there is a Lipschitz constant $\beta>0$ such that
\[
\left| E_m(x) - E_m(\tilde x) \right| \le \beta \left| x-\tilde x\right|
\]
for all $m\ge 1$ and all nuclear configurations $x,\tilde x\in\Rb^{3N}$. The Lipschitz constant depends on the number of nuclei and 
their charge numbers as well as the number of electrons. 

The general Lipschitz estimate 
for the functions $x\mapsto E_m(x)$ can be refined further.
% as follows: 
\cn[Corollary of Theorem~1]{Hun86} proved that
%\begin{theorem}[Analyticity]
non\-degenerate 
discrete eigenvalues of the electronic Hamiltonian $H_e(x)$ are analytic in \bch $x=(x_1,\ldots,x_N)\in\C^{3N}$ 
in a neighbourhood of any nuclear configuration $x\in\Rb^{3N}$ satisfying 
$x_k\neq x_\ell$ if $k\neq \ell$.\ech  
%\end{theorem}

%This result is proved in \cite[Corollary of Theorem~1]{Hun86}. 
The two 
conditions for analyticity are not redundant. As for non-degeneracy, the standard example 
is provided by the $2\times 2$ matrix
\[
\begin{pmatrix}\xi_1 & \xi_2\\ \xi_2 & -\xi_1\end{pmatrix}
\]
which is an entire function of $\xi\in\C^2$, but has the eigenvalues
$\pm(\xi_1^2 + \xi_2^2)^{1/2}$ that are not even differentiable at $\xi=0$. 
Such degeneracies occur in many polyatomic molecules;  
see \eg\ \cite{DomYK11}. They are referred to as conical intersections, since the graphs of 
the two eigenvalues form two intersecting, rotationally symmetric cones.

\medskip
For diatomic molecules, one can adapt the coordinate system such 
that 
\[
x_1 = (-\tfrac12 r,0,0)\quad\text{and}\quad x_2 = (\tfrac12 r,0,0).
\] 
Then, the electronic eigenvalues just depend on one single variable, the internuclear distance 
$r = |x_1-x_2|$.  The simplest molecular ion, the dihydrogen cation H$_2^+$ ($N=2$ nuclei and $L=1$ electron) 
shows that analyticity may break down at $r=0$ in a rather subtle way, since its electronic ground state 
energy has the asymptotic expansion
\[
E(r) = -2 + \tfrac23 (2r)^2 -\tfrac23 (2r)^3 + \tfrac{22}{135} (2r)^4 - \tfrac{2}{9} (2r)^5 \log(r) + O(r^5) 
\]
as $r\to0$\bch ;\ech see \cite{Kla83}.

\begin{remark}
There are two choices for the definition of an electronic Hamiltonian. Setting
\[
H_e(x) = T_e + V_{Ne}(x,\cdot) + V_{ee},
\]
only the nucleus-electron and electron-electron interactions 
are taken into account, so the electronic eigenvalues satisfy the above Lipschitz estimate. Alternatively, one may consider
\[
H_e(x) + V_{NN}(x)
\] 
as an electronic Hamiltonian acting on functions of the electronic coordinates. This operator has the same spectral properties as $H_e(x)$, and even the same eigenfunctions. However, since its eigenvalues are 
\[
E_m(x) +V_{NN}(x),
\] 
they are not globally continuous with respect to $x$.   
\end{remark}

\subsection{Schr\"odinger equation for the nuclei via adiabatic approximation}
%on an Electronic Energy Surface}

We fix a potential energy surface $E(x)$ and let $\Phi(x,\cdot)\in L^2_y$ be a corresponding eigenfunction of unit norm that depends continuously on $x$. For fixed nuclear coordinates $x$, 
the solution of the time-dependent electronic Schr\"odinger equation
$$%\begin{equation}\label{II:elec-tdse}
\I\hbar\, \partial_t \Psi_e  = H_e(x)\Psi_e
$$%\end{equation}
with initial data $\psi_0(x)\Phi(x,\cdot)$
is given by
$$%\begin{equation}\label{II:elec-tdse-sol}
\Psi_e(x,y,t) =  e^{-iE(x)t/\hbar} \psi_0(x)\cdot \Phi(x,y),
$$%\end{equation}
and so the solution  stays for all times in the subspace of $L^2_{x,y}$ given by
\begin{equation}\label{II:M-ad}
\calV = \{ v\in L^2_{x,y}: v(x,y)=\psi(x)\,\Phi(x,y), 
\ \psi\in L^2_x \} .
\end{equation}
This motivates the {\it adiabatic approximation} to the molecular 
Schr\"o\-dinger equation,
which is the Galerkin approximation  on the subspace $\calV$: Given the initial value $u_0\in \calV$,
find $u(t)=u(\cdot,\cdot,t)\in\calV$ such that
\begin{equation} \label{II:adi-galerkin}
\bigl\langle \I \hbar\, \partial_t u - H_{\rm mol}u \mid v \bigr\rangle = 0 \quad\text{ for all } v \in \calV, \qquad u(0)=u_0.
\end{equation}
This leads
to a {\it Schr\"odinger equation for the nuclei} on
the electronic energy surface~$E$\,: a calculation shows that $u(x,y,t)=\psi(x,t) \,\Phi(x,y)$, where the nuclear wave function $\psi$ satisfies
\begin{eqnarray*}%\label{II:nuc}
&& i\hbar\,\partial_t \psi  = H_N\psi
\quad\text{ with }\quad
H_N = T_N + V_{NN}+E + B_1 + B_2\,,
\\
\nonumber
&& 
B_1 =  \sum_{n=1}^N {\hbar\over M_n} 
\,\Im \langle \nabla_{x_n}\Phi \,|\, \Phi \rangle_{L^2_y}\cdot
(-\I\hbar\nabla_{x_n})\,,\quad
B_2 = \sum_{n=1}^N {\hbar^2\over 2M_n} \,
\| \nabla_{x_n} \Phi \|_{L^2_y}^2 .
\end{eqnarray*}
The Hamiltonian $H_N$ acts on functions of only the nuclear coordinates $x$, with the electronic eigenvalue
$E$ as a potential. The last two terms $B_1$ and $B_2$ contain derivatives of the
electronic wave function $\Phi$ with respect to the nuclear coordinates $x$.
They are usually neglected in computations, first on practical grounds because they are expensive to compute 
or simply not available and second by the 
formally compelling yet factually dubious argument that they can be neglected because they have the large nuclear masses $M_n$ in the denominator and are of lower differentiation order than the kinetic energy term. The resulting simplified
approximation with the Hamiltonian 
\begin{equation}\label{II:HBO}
H_{\rm BO}=T_N+V_{NN}+E
\end{equation}
is known as the {\it time-dependent Born--Oppenheimer approximation}.
It describes the motion of the nuclei as driven by the potential energy surface $E$ of the electrons. The 
vast majority of computations in molecular dynamics are based on this approximation.

The term $B_2$ can indeed be safely neglected: it can be shown that its omission introduces an error that is of the same magnitude as the approximation error in the adiabatic approximation.

The term $B_1$, known as the 
Berry connection, vanishes for real eigenfunctions $\Phi$ and, more generally, 
it can be made to vanish by a transformation
$\Phi(x,y) \to e^{i\theta(x)}\Phi(x,y)$ with $\theta$ satisfying 
$\nabla_{x_n}\theta(x) =
-\Im \langle \nabla_{x_n}\Phi \,|\, \Phi \rangle_{L^2_y}$. 
This transformation of $\Phi$ changes 
$\psi(x,t)\to e^{-i\theta(x)}\psi(x,t)$. 
The function $\theta$ is uniquely determined up to a constant if 
$\Phi$ is a smooth function  on all of $\Rb^{3N}$ or on a simply connected domain, but else $\theta$ is only locally uniquely determined. In the latter case, $B_1$ can cause physical effects that are not captured otherwise; see 
\cn{Ber84} and \cn{Sim83}.

\subsection{Semiclassical scaling}\index{semiclassical scaling}
\label{II:subsec:semiclassical}
The success of the adiabatic approximation relies on
the smallness of the mass ratio of electrons and nuclei,
$$%\begin{equation}\label{II:eps}
\eps^2 = \frac m M \ll 1
$$%\end{equation}
with $M=\min_n M_n$.
For ease of presentation, we assume in the following that the masses of the nuclei are all equal: $M_n=M$ for all $n$.
In atomic units ($\hbar=1$, $m=1$, $r=1$, $e=1$) and with the small parameter
$\eps$, the molecular Hamiltonian 
then takes the form
$$%\begin{equation}\label{II:H-eps}
H_{\rm mol}^\eps = -\frac{\eps^2}2 \Delta_x + H_e(x) \quad\text{ with }\quad
H_e(x)= -\tfrac12 \Delta_y - V(x,\cdot).
$$%\end{equation}
We are interested in solutions to the Schr\"odinger equation of bounded energy, and in particular of bounded kinetic energy
$$
\langle \Psi \,|\, -\tfrac{\eps^2}2 \Delta_x  \Psi \rangle =
\tfrac12\,\| \eps\nabla_x \Psi \|^2 = O(1)\,.
$$
For a wave packet $e^{\I p\cdot x} a(x)$, with a smooth and fast decaying amplitude function $a$, this condition corresponds to
a momentum $p \sim \eps^{-1}$ and hence to a velocity $v=p/M\sim \eps$.
Motion of the nuclei over a distance $\sim 1$ can thus be expected 
on a time scale $\eps^{-1}$. We therefore rescale time
$$
t \to t/\eps\,,
$$
so that, with respect to this new time, nuclear motion over distances $\sim 1$
can be expected to occur at time $\sim 1$. The molecular Schr\"odinger equation in the rescaled time then takes the form
\begin{equation}\label{II:schroed-eps}
\I\eps \partial_t\Psi = H_{\rm mol}^\eps\Psi\,.
\end{equation}
The Schr\"odinger equation for the nuclei becomes
\begin{eqnarray}\label{II:nuc-eps}
&&\I\eps \partial_t\psi = H_N^\eps\psi\ \text{ with }\
\bch H_N^\eps=-\frac{\eps^2}2\Delta_x +V_{NN} + E +\eps B_1+\eps^2 B_2\,,\ech
\\
\nonumber
&&\qquad
B_1= \Im \langle \nabla_{x}\Phi \,|\, \Phi \rangle_{L^2_y}\cdot
\widehat p\,,\quad\
B_2 = \tfrac{1}2\,\| \nabla_{x} \Phi \|_{L^2_y}^2 \,,
\end{eqnarray}
with the momentum operator $\widehat p=-\I\eps\nabla_x$.
We are interested in solutions over times $t\sim 1$.

\subsection{Error of the adiabatic approximation}
We present an error bound of the adiabatic approximation that was obtained independently by \cn{SpoT01} and by \cn{MarS02}.
In addition to the smallness of  the scaling parameter $\eps$, a {\it spectral gap} condition is required: the 
eigenvalue $E(x)$ is separated from the remainder of the spectrum $\sigma(H_e(x))$ of  the electronic
Hamiltonian $H_e(x)$,
$$%\begin{equation}\label{II:delta}
{\rm dist}\bigl( E(x),\, \sigma(H_e(x))\setminus\{ E(x)\} \bigr) \ge \delta >0 \for x\in\Rb^{3N}.
$$%\end{equation}
The Coulomb interactions of the nuclei are mollified to
smooth bounded potentials: It is assumed that the potential $V(x,y)$ has a bounded partial derivative with respect to $x$. 
%\begin{equation}\label{II:V-bound}
%\| \nabla_x V(x,y) \| \le C_V \for x\in\Rb^{3N},\, y\in \Rb^{3L}
%\end{equation}
Initial data are considered to lie in the 
approximation space $\calV$ of (\ref{II:M-ad}),
$$%\begin{equation}\label{II:init-bound}
\Psi_0(x,y) =\psi_0(x)\Phi(x,y) \with 
\| H_N^\eps\psi_0 \| \le C_0\,, \  \|\psi_0\|=1~.
$$%\end{equation}
We consider the adiabatic approximation 
%$u(t) = u(\cdot,\cdot,t)$, with initial data $\Psi_0$,  determined by
%the time-dependent variational principle:
%\begin{equation}\label{II:qvar-adi}
%\sfrac{\partial u}{\partial t}\in \calM \quad\hbox{ such that }\quad
%\Bigl\langle \vv\,\Big|\, \sfrac{\partial u}{\partial t} - 
%\sfrac 1{i\eps} H_{\rm mol}^\eps u \Bigr\rangle=0
%\quad \forall\ \vv \in \calM\,.
%\end{equation}
%We know already that 
$$%\begin{equation}\label{II:u-psi-Phi}
u(x,y,t)=\psi(x,t)\Phi(x,y)\,,
$$%\end{equation}
where $\psi(x,t)$ is 
the solution of the 
nuclear Schr\"odinger equation~(\ref{II:nuc-eps})
with initial data $\psi_0(x)$. This is
compared with the exact solution $\Psi(t)=\Psi(\cdot,\cdot,t)$ of the molecular Schr\"odinger equation (\ref{II:schroed-eps}) with the above initial data 
$\Psi_0$.

\begin{theorem}[space-adiabatic theorem, \cn{Teu03}]
\label{II:thm:ad}
Under the conditions stated above, the error of the adiabatic approximation  is bounded in the $L^2$-norm by
$$
\| u(t) - \Psi(t) \| \le c\,(1+t)\,\eps \for t\ge 0\,,
$$
where $c$ is independent of $\eps$ and $t$ and initial data as specified above,
but depends on the spectral gap $\delta$.
%of (\ref{II:gap}) (uniform for $x\in\Rb^{3N}$), 
%on bounds of partial derivatives with respect to $x$ 
%up to third order of
%the eigenfunctions $\Phi$, and on the bounds 
%$C_V$ of (\ref{II:V-bound}) and $C_0$ of (\ref{II:init-bound}).
\end{theorem}

\noindent
\cn{Teu03} actually proves a more general result and provides a wealth of related theory. The twisted pseudo-differential calculus developed by \cn{MarS09} also includes Coulomb-type interactions. A short proof of the theorem as stated is given in \cite[Section II.2.5]{Lub08}. 

In the global version stated above, Theorem~\ref{II:thm:ad} remains valid for the time-dependent Born--Oppenheimer approximation when the terms $B_1$ and $B_2$ in \eqref{II:nuc-eps} are dropped. This is no longer true for local versions, where the eigenfunction 
$\Phi$ is defined only on a domain that is not simply connected.

The result is related in spirit (though not proof technique) to the time-adiabatic theorem of
\cn{BorF28} and \cn{Kat50}, which states that in a quantum system with a slowly time-varying Hamiltonian,  a wave function that is initially an eigenfunction, approximately remains an eigenfunction of the Hamiltonian over 
long times.

It is known that the
adiabatic approximation generally breaks down near crossings of eigenvalues. One remedy, also covered by the space-adiabatic theory of~\cn{Teu03},  is to enlarge the approximation space
by including several energy bands that are well separated from the remaining 
ones in the region of physical interest, using
$$%\begin{equation}\label{II:M-ad-2}
\calV = \{ u\in L^2_{x,y}: u(x,y)=\sum_{j=1}^m\psi_j(x)\Phi_j(x,y)
, \ \psi_j\in L^2_x \}~,
$$%\end{equation}
where $\Phi_j(x,\cdot)$ $(j=1,\dots,m)$ span an invariant subspace 
of the electronic Hamiltonian $H_e(x)$. The Galerkin approximation
on $\calV$ then leads to a system of coupled Schr\"odinger equations
with an $m\times m$ matrix-valued potential.

%
%\begin{equation}\label{II:schrod-system}
%i\hbar\, \frac{\partial\psi}{\partial t} = T_N\psi + B_1\psi + B_2\psi + V\psi \for 
%\psi=\begin{pmatrix} \psi_1 \\ \psi_2 \end{pmatrix}
%\end{equation}
%with the matrix-valued potential
%\begin{equation}\label{II:matrix-potential}
%V = \begin{pmatrix} V_{11} & V_{12} \\ 
%V_{21} & V_{22} \end{pmatrix} \with
%V_{ij}(x) = \langle \Phi_i(x,\cdot)\,|\, H_e(x) \,|\, 
%\Phi_j(x,\cdot) \rangle_{L^2_y} 
%\end{equation}
%and with the diagonal operators $B_j=\begin{pmatrix} B_j^1 & 0 \\ 
%0 & B_j^2 \end{pmatrix}$, where $B_j^k$ are defined as $B_j$ in
%(\ref{II:nuc}) with $\Phi_k$ instead of $\Phi$.
% 
%

%\newpage
%\section{Variational Gaussian wave packets}
\section{Variational Gaussian wave packets}
\label{sec:gwp}

In this section we approximate solutions to the semiclassically scaled Schr\"o\-din\-ger equation by variationally evolving complex Gaussians. For a para\-metrisation due to Hagedorn, the equations of motion of the parameters of the Gaussian wave packet are remarkably close to the equations of motion of classical mechanics and their linearisation. It is shown that the Gaussian wave packets approximate the exact wave function with Gaussian initial data with an $O(\eps^{1/2})$ error in the $L^2$-norm and with an $O(\eps)$ error for expectation values of observables over times $t\sim 1$.

\subsection{Variational approximation by complex Gaussians}
We consider the Schr\"odinger equation in semiclassical scaling, 
\begin{equation}\label{tdse}
\I\eps\partial_t \psi = H\psi,\qquad H = -\tfrac{\eps^2}{2}\Delta_x + V,
\end{equation}
as it emerges from the time-dependent Born--Oppenheimer approximation to the motion of nuclei. We let $d\ge 1$ (often $d\gg 1$) denote the dimension of the 
nuclear configuration space, and assume that the potential
$V:\Rb^d\to\Rb$ is a smooth real-valued function. \bch As before, the $L^2$ inner product is conjugate linear 
in the first argument.\ech

We seek an approximation $u(\cdot,t)$ (later written as $u(t)$ for short) to the solution $\psi(t)$ in the manifold of complex-valued Gaussian functions
\begin{align*}
& {\mathcal M} = \left\{ u\in L^2(\Rb^d)\ \Big|\  u(x) = \exp\!\left( \tfrac\I\eps \left(\tfrac12 (x-q)^T C(x-q) + p^T(x-q) + \zeta\right) \right), \right. \\
& \hspace*{3em} \left. q\in\Rb^{d},\ p\in\Rb^{d},\  C=C^T\in\C^{d\times d}, \ \Im C \text{ is positive definite}, \ \zeta\in\C\right\},
\end{align*}
where we find by a direct computation that for such an $u\in{\mathcal M}$ of unit $L^2$-norm, the position and momentum averages equal the parameters $q$ and $p$, respectively:
$$
\langle u\,|\, xu \rangle =q, \quad\ \langle u \,|\, - \I\eps \nabla_x u \rangle =p.
$$

We invoke the Dirac--Frenkel time-dependent variational approximation principle, which determines the approximate solution 
\[
u(t)\in{\mathcal M}
\] 
by requiring that the residual of the Schr\"odinger equation be orthogonal to the tangent space ${\mathcal T}_{u(t)}{\mathcal M}$
of the manifold ${\mathcal M}$ at $u(t)$:  for all times $t\in\Rb$,
\begin{align}\nonumber
&\partial_t u(t)\in{\mathcal T}_{u(t)}{\mathcal M}\ \text{ is such that}\\*[0.25em] \label{eq:var}
&\left\langle v\mid -\I\eps\partial_t u(t) + Hu(t)\right\rangle = 0\ \text{ for all}\  v\in{\mathcal T}_{u(t)}{\mathcal M}.
\end{align}
With  the orthogonal projection $P_u:L^2(\Rb^d)\to {\mathcal T}_{u}{\mathcal M}$ onto the tangent space ${\mathcal T}_u{\mathcal M}$, we can rewrite this equivalently as
\begin{equation} \label{PuHu}
\I\eps \partial_t u = P_u H u.
\end{equation}
We start by having a closer look at the tangent space.

\medskip
\begin{lemma}[tangent space]\label{lem:tangent space}
At every Gaussian function $u\in{\mathcal M}$, the tangent space equals
$$
{\mathcal T}_u{\mathcal M} = \left\{ \varphi u\ \big|\   \text{$\varphi$ is a complex $d$-variate polynomial of degree at most $2$} \right\}.
$$
\end{lemma}
In particular, ${\mathcal T}_u{\mathcal M}$ is a complex-linear subspace of $L^2(\Rb^d)$, in the sense that 
$v\in{\mathcal T}_u{\mathcal M}$ implies $\I v\in{\mathcal T}_u{\mathcal M}$.
Moreover, for  all differential operators $A$ of order at most $2$ with constant coefficients, we have $Au\in{\mathcal T}_u{\mathcal M}$.

\begin{proof}
Since the tangent space consists of derivatives of paths on ${\mathcal M}$ passing through $u$, we find that every $v\in{\mathcal T}_u{\mathcal M}$ is of the form
\begin{align*}
v(x) = 
& \frac\I\eps \Bigl(  -\dot q^T C(x-q) +\tfrac12 (x-q)^T \dot C (x-q) + \dot p^T(x-q) -p^T\dot q+ \dot\zeta
\Bigr)u(x)
\end{align*}
with arbitrary  $(\dot q,\dot p)\in \Rb^{2d}$, $\dot C={\dot C}^T\in\C^{d\times d}$, $\dot\zeta\in\C$. 
Since every complex polynomial of degree $\le 2$ can be written in the form of the prefactor of the above tangent vector $v$, we 
obtain the stated characterisation of ${\mathcal T}_u{\mathcal M}$. 
\end{proof}

As a direct corollary we obtain the following property, which is a major motivation for considering localised Gaussians to approximate the wave function $\psi$.

\begin{proposition} [exactness for quadratic potentials] 
\label{prop:exact-gaussian}
If the potential $V$ is quadratic, then the variational approximation is exact: $u(t)=\psi(t)$,  provided that the initial data are Gaussian, \ie\
$u(0)=\psi(0)\in{\mathcal M}$.
\end{proposition}

\begin{proof} If $V$ is quadratic, then $Hu=-\tfrac{\eps^2}{2}\Delta_x u + Vu$ is a quadratic polynomial times $u$ and hence, by Lemma~\ref{lem:tangent space},  is in the tangent space ${\mathcal T}_u{\mathcal M}$ for $u\in{\mathcal M}$. It follows that $P_uHu=Hu$, and hence $u(t)$ and $\psi(t)$ satisfy the same differential equation. 
\end{proof}

We remark that this result and its proof extend directly to the case of a time-dependent quadratic potential $V(\cdot,t)$.

In the case of a non-quadratic potential, Lemma~\ref{lem:tangent space} simplifies \eqref{PuHu} to
\begin{equation} \label{PuVu}
\I\eps \partial_t u = -\tfrac{\eps^2}{2}\Delta_x u + P_u V u.
\end{equation}
We will study the term $P_uVu$ in detail below, in Proposition~\ref{prop:project}.

\subsection{Conservation properties}

The variational Gaussian wave packets enjoy norm and energy conservation, as the exact Schr\"odinger solution does.

\begin{proposition}[norm and energy conservation]\label{prop:cons}
The time-depen\-dent variational approximation \R{eq:var} is norm- and energy-conserving.
\end{proposition}

\begin{proof} We use the defining relation \R{eq:var} for $\partial_t u(t)\in{\mathcal T}_{u(t)}{\mathcal M}$.
Since the Hamiltonian $H$ is self-adjoint, we have
\begin{align*}
\frac{\D}{\D t} \langle u(t)\mid H u(t)\rangle 
&= \langle \partial_t u(t)\mid Hu(t)\rangle + \langle u(t) \,|\, H \partial_t u(t)\rangle \\
&= 2\,\Re\langle \partial_t u(t) \mid Hu(t)\rangle \\
&= 2\,\Re\langle \partial_t u(t) \mid \I\eps\partial_t u(t)\rangle \\
&= 0,
\end{align*}
which proves energy conservation.
Since $u(t)\in{\mathcal T}_{u(t)}{\mathcal M}$, we also have
\begin{align*}
\frac{\D}{\D t} \|u(t)\|^2 
&= 2 \,\Re\!\left\langle u(t)\mid \partial_t u(t)\right\rangle \\
&= 2\,\Re\!\left\langle u(t)\mid (\I\eps)^{-1}Hu(t)\right\rangle \\
& = 0,
\end{align*}
which proves norm conservation.
\end{proof}

The quantum-mechanical  total linear momentum and
 total angular momentum operators, written for the position variables $x=(x_1,\dots,x_N)$ with $x_k\in\Rb^3$, are
\begin{equation}\label{P}
\begin{array}{l}
P = \sum_{k=1}^N  (-i\varepsilon\nabla_{x_k})
\quad\hbox{and}\quad\\[2mm]
L = \text{Hermitian part of }\ \sum_{k=1}^N x_k \times (-i\varepsilon\nabla_{x_k}),
\end{array}
\end{equation}
respectively. They both commute with the kinetic energy operator $T=-\tfrac12 \eps^2\Delta_x$.
%\bch Already said above: 
%(Here we consider\, $x=(x_1,\dots,x_N)$ with $x_k\in\Rb^3$.)
%\ech
For potentials that are invariant under translations, that is,
$$
V(x_1,\dots, x_N) = V(x_1+ r,\dots, x_N+ r) \quad\ \hbox{for all $r\in\Rb^3$},
$$
$P$ commutes also with multiplication with $V$. For potentials invariant under rotations, that is,
$$
V(x_1,\dots, x_N) = V(Rx_1,\dots, Rx_N) \quad\ \hbox{for all $R\in SO(3)$},
$$
$L$ commutes with multiplication with $V$, and hence also with the
Hamiltonian $H$. Their averages $\langle \psi \,|\, P\psi\rangle$ and $\langle \psi \,|\, L\psi\rangle$ are therefore conserved along solutions $\psi(t)$ to the Schr\"odinger equation with Hamiltonian $H=T+V$. As we show next, this conservation property is retained under variational Gaussian wave packet dynamics. We will use the shorthand notation
\[
\langle A\rangle_u =  \langle u \,|\, A u\rangle
\] 
for self-adjoint operators $A$ that have a well-defined average with respect to a given Gaussian wave packet $u\in{\mathcal M}$.
For a vector-valued operator (such as $P$ and $L$ above), $\langle A\rangle_u$ is the vector of the averages of the components of $A$.
%(For example,
%potentials with only distance-dependent pair interactions
%$V_{jk}\left( |x_j - x_k| \right)$
%are invariant
%under both translations and rotations.)
In the following result, $q(t)=(q_1(t),\dots,q_N(t))$ and $p(t)=(p_1(t),\dots,p_N(t))$ with $q_k(t),p_k(t)\in\Rb^3$ denote the time-dependent position and momentum parameters, respectively, of a variational Gaussian wave packet $u(t)$.

\begin{proposition} [linear and angular momentum conservation] 
\label{prop:ang-mom} 
In a  translation-invariant potential, variational Gaussian wave packet dynamics  conserves the total linear momentum
$\langle P \rangle_{u(t)}$, which equals the classical linear momentum $\sum_{k=1}^N p_k(t)$.

In a rotation-invariant potential, variational Gaussian wave packet dynamics conserves the total angular momentum
$\langle L \rangle_{u(t)}$, which equals the classical angular momentum $\sum_{k=1}^N q_k(t) \times p_k(t)$.
\end{proposition}

\begin{proof} We formulate the proof for the angular momentum $L$. The self-adjoint operator $L$  commutes with the Hamiltonian $H$, that is,
${[H,L]=0}$. Using Lemma~\ref{lem:tangent space}, we find that $Lu\in {\mathcal T}_{u}{\mathcal M}$ for $u\in{\mathcal M}$. It then follows that
$$
\frac{d}{dt}\langle u \, | \, L  u \rangle =
2\, {\rm Re}\, \langle Lu \,|\, \partial_t u \rangle =
2\, {\rm Re}\, \langle Lu \,|\, \tfrac1{\I\varepsilon}\, Hu \rangle =
\langle u \, | \, \tfrac1{\I\varepsilon}\,[H,L]  u \rangle = 0,
$$
 where we use \eqref{eq:var} in the
second equality. The proof for linear momentum is the same, since $P$ has the same two properties that were stated for $L$ at the beginning of this proof. Finally, a direct calculation shows that for a Gaussian wave packet, the quantum-mechanical linear and angular momentum equal their classical counterparts for the position parameters~$q$ and momenta~$p$.
\end{proof}

\subsection{Error in the $L^2$-norm and in expectation values of observables}

We now turn to error bounds which show that, for small $\eps$, the variational Gaussian approximation stays close to the solution of the Schr\"odinger equation over time intervals independent of $\eps$, as long as the Gaussian wave packet remains localised of width $O(\sqrt\eps)$. Over such time intervals, the error in the $L^2$-norm is $O(t\sqrt\eps)$, and the error in observables is $O(t\eps)$.

\begin{theorem}[Error bounds] \label{thm:error-gauss}
We consider the Gaussian wave packet approximation $u(t)$ determined by the Dirac--Frenkel variational principle \R{eq:var}. 
We assume the following:
\begin{enumerate}
\item[1.]
The eigenvalues of the positive definite width matrix $\Im C(t)$ are bounded from below by a constant $\rho>0$, for all $t\in[0,\overline t]$.
\item[2.]
The potential function $V$ is three times continuously differentiable with a polynomially bounded third derivative. 
\end{enumerate}
Then, the error between the Gaussian wave packet $u(t)$ and the  solution $\psi(t)$ of the Schr\"odinger equation \eqref{tdse} with 
Gaussian initial data $\psi(0) = u(0)$ of unit norm satisfies the following bounds:

(a) The error in the $L^2$-norm is bounded by
\[
\|u(t) - \psi(t)\| \le c\, t \sqrt\eps, \qquad 0\le t \le \overline t.
\] 

(b) Let $A$ be an observable that has a polynomially bounded Weyl symbol (see~Section~\ref{sec:wigner}; \eg~a symmetrised polynomial in the position operator $\widehat q$ and the momenta operator $\widehat p$, with $(\widehat q\varphi)(x)=x\varphi(x)$ and $(\widehat p\varphi)(x) = -\I\eps \nabla_x \varphi(x)$). Then, the error in the expectation value of the observable $A$ over $u(t)$,  \ie\
$\langle A \rangle_{u(t)}= \langle u(t) \,|\, A  u(t) \rangle$,
is bounded by
\[
|\langle A \rangle_{u(t)} - \langle A \rangle_{\psi(t)}| \le c\, t \,\eps, \qquad 0\le t \le \overline t.
\]
In both (a) and (b),  $c<\infty$ is independent of $\eps$ and $t$ but depends on $\rho$.
\end{theorem}

We remark that this is an {\it a posteriori} error bound since Assumption~1. about $\rho$ involves the Gaussian approximation. 
However, we will see in Section~\ref{subsec:ehrenfest} that the decay of $\rho$ with time is essentially  determined by the Lyapunov exponent of the {\it classical} equations of motion $\dot q=p$, $\dot p=-\nabla V(q)$.

For the proof of the theorem we first give some simple lemmas that will also be useful later on.

\begin{lemma}[stability] \label{lem:stability}
Suppose that $e(t)$ satisfies the Schr\"odinger equation up to a defect $d(t)$,
$$
\partial_t e = \frac1{\I\eps}He + d, \qquad e(0)=e_0.
$$
Then,
$$
\| e(t) \| \le \| e_0\| + \int_0^t  \| d(s) \|\, \D s.
$$
\end{lemma}

\begin{proof}
Since $H$ is self-adjoint, $\langle e | He\rangle$ is real, and then
the Cauchy-Schwarz inequality implies
\begin{align*}
\|e\| \tfrac{\D}{\D t}\|e\| 
&= \tfrac12 \tfrac{\D}{\D t}\|e\|^2 %\\*[0.5em]
= \Re\langle e\mid \partial_te\rangle \\
&= \Re\langle e\mid \tfrac{1}{\I\eps}He\rangle 
+ \Re\langle e\,|\, d\rangle
=  \Re\langle e\,|\, d\rangle%\\
\le \|e\| \ \| d\|.
\end{align*}
Integration from $0$ to $t$ yields the result.
\end{proof}

It will be useful to have an explicit formula for the norm of a Gaussian wave packet.

\begin{lemma}[norm]\label{lem:norm}
For $u\in{\mathcal M}$, we have
\[
\|u\| = (\pi\eps)^{d/4}\ \det(\Im C)^{-1/4}\ \exp(-\tfrac1\eps\, \Im\zeta).
\]
Therefore, if $\|u\| = 1$, then
\[
\exp(-\tfrac1\eps\,\Im\zeta) = (\pi\eps)^{-d/4} \ \det(\Im C)^{1/4}.
\]
\end{lemma}

\begin{proof}
We write
\[
|u(x)|^2  = \exp\!\left( -\tfrac1\eps \left( (x-q)^T \Im C(x-q) + 2\,\Im\zeta\right) \right)
\]
and obtain
\[
\int_{\Rb^d} |u(x)|^2 \D x = (\pi\eps)^{d/2} \ \det(\Im C)^{-1/2} \ \exp(-\tfrac2\eps\, \Im\zeta). 
\]
\end{proof}

The following observation will also be used repeatedly.

\begin{lemma}[moments] 
\label{lem:int-bound} For arbitrary $m,n\ge 0$, there exists $c_{m,n}<\infty$ such that the following holds true:
For any symmetric positive definite width matrix $\Im C$ with smallest eigenvalue bounded from below by $\rho>0$ and  for all positive $\eps>0$,
\begin{align*}
(\pi\eps)^{-d/4}\ \det(\Im C)^{1/4} &
\left(\int_{\Rb^d}  |x|^{2m} \Bigl(1+\frac\rho\eps |x|^2\Bigr)^n  \exp(-\tfrac1\eps x^T \Im C x) \D x\right)^{1/2}
\\
& \hspace{5.3cm} \le c_{m,n} \Bigl( \frac {\eps} {\rho} \Bigr)^{m/2}.
\end{align*}
\end{lemma}

\begin{proof} 
We diagonalise the width matrix
\[
\Im C = S^T \, {\rm diag}(\lambda_1,\ldots,\lambda_d) \ S,
\]
with eigenvalues $\lambda_1,\ldots,\lambda_d>\rho$ and $S\in\Rb^{d\times d}$ orthogonal. 
Substituting $y=Sx$, the Gaussian integral can be rewritten as
\[
\frac{(\lambda_1\cdots\lambda_d)^{1/4}}{(\pi\eps)^{d/4}} 
\left(\int_{\Rb^d}  |y|^{2m} \Bigl(1+\frac\rho\eps |y|^2\Bigr)^n  
\exp(-\tfrac1\eps (\lambda_1y_1^2 + \cdots +\lambda_d y_d^2)) \D y\right)^{1/2}.
\]
Then, we substitute $z_j=\sqrt{\lambda_j/\eps}y_j$ and obtain the upper bound
\[
\Bigl( \frac {\eps} {\rho} \Bigr)^{m/2} \pi^{-d/4} 
\left(\int_{\Rb^d}  |z|^{2m} \Bigl(1+|z|^2\Bigr)^n  
\exp(-|z|^2) \D z\right)^{1/2},
\]
which gives the result, since the integral depends only on $m$ and $n$.
\end{proof}

\begin{proof} (of Theorem~\ref{thm:error-gauss} (a): $L^2$-error bound)
Using \eqref{PuHu}, we have
\begin{align*}
\partial_t (u-\psi) 
&= \tfrac{1}{\I\eps} H(u-\psi) + \tfrac{1}{\I\eps}(P_uH - H)u \\
&= \tfrac{1}{\I\eps}H(u-\psi) - \tfrac{1}{\I\eps} P^\perp_uH u,
\end{align*}
where $P^\perp_u = \Id - P_u$ is the projection onto the orthogonal complement, and 
\[
 \|\bch \tfrac{1}{\I\eps} P^\perp_{u(s)}Hu(s)\ech\| =
 \dist(\tfrac{1}{\I\eps}H u(s),{\mathcal T}_{u(s)}{\mathcal M}).
\]
By Lemma~\ref{lem:stability} we therefore have
\begin{align*}
\|u(t)-\psi(t)\| \le \int_0^t \dist\bigl(\tfrac{1}{\I\eps}H u(s),{\mathcal T}_{u(s)}{\mathcal M}\bigr) \D s.
\end{align*}
Now we evaluate this general estimate for the Gaussian manifold ${\mathcal M}$. 
Let $q\in\Rb^d$ be the position of $u\in{\mathcal M}$. We let $U_q:\Rb^d\to\Rb$ denote the second order Taylor polynomial of $V$ at $q$ and let $W_q:\Rb^d\to\Rb$ be the corresponding remainder, 
\[
V = U_q + W_q.
\]
By Lemma~\ref{lem:tangent space}, we have $\Delta_x u, U_q u\in{\mathcal T}_u{\mathcal M}$, and therefore
\begin{align*}
\dist(\tfrac{1}{\I\eps} Hu,{\mathcal T}_u{\mathcal M}) &= 
\dist( \tfrac{1}{\I\eps}W_q u,{\mathcal T}_u{\mathcal M}) \le \tfrac{1}{\eps} \|W_q u\|.
\end{align*}
Using the norm conservation of Proposition~\ref{prop:cons} together with Lemma~\ref{lem:norm}, we write
\begin{align*}
\|W_q u\| &= (\pi\eps)^{-d/4}\ \det(\Im C)^{1/4}\\
 &\left(\int_{\Rb^d} |W_q(x)|^2 \exp(-\tfrac1\eps(x-q)^T \Im C (x-q)) \D x\right)^{1/2}.
\end{align*}
%Since for all $\alpha >0$,
%\[
%\int_\Rb s^6 \e^{-\alpha s^2} \D s = \tfrac{15}8  \sqrt\pi \,\alpha^{-7/2},
%\]
%and since $|W_q(x)|^2 \le \frac1{3!} \gamma |x-q|^6 $, where $\gamma$ is a bound of the third derivative of $V$,
%a diagonalisation of the width matrix $\Im C$ implies the estimate 
%\[
%\| W_q u\|^2 \le c \ \eps^{3/2},
%\]
%where the constant $c$ depends on $\rho$ and $\gamma$. An estimate of the same type is still obtained if the third derivative of $V$ is not bounded by a constant, but by a polynomial.
Since 
$W_q(x)$  is the non-quadratic remainder at $q$, and by the assumption on the polynomial boundedness of $V$ (of degree $n$, say), we have  $|W_q(x)|^2 \le c_3 |x-q|^6 (1+|x-q|^{2n})$, and hence Lemma~\ref{lem:int-bound} implies the estimate 
\[
\| W_q u\| \le c \ \eps^{3/2},
\]
where the constant $c$ depends on $\rho$. In summary, we obtain
\[
\|u(t)-\psi(t)\| \le \int_0^t \tfrac1\eps \|W_{q(s)}u(s)\| \D s \le  c\,  t \sqrt\eps,
\]
which is the stated $L^2$-error bound.
\end{proof}

We now turn to the error in observables. In view of Lemma~\ref{lem:tangent space}, for any Gaussian $u\in{\mathcal M}$, there is a quadratic potential $U_u$ such that
$$
P_u Hu = -\tfrac{\eps^2}{2}\Delta_x u+ U_u u, \quad\ie\quad   P_u Vu = U_u u \in {\mathcal T}_u{\mathcal M}.
$$
We introduce the remainder potential $W_u=V-U_u$, so that
\begin{equation} \label{Wu}
Hu = P_uHu + W_u u = P_uHP_u u + W_u u,
\end{equation}
where we used that $u\in {\mathcal T}_u{\mathcal M}$ for all $u\in {\mathcal M}$ in the last equality.
Theorem~\ref{thm:error-gauss} (b) (error in observables) then follows directly from the following two lemmas.

\begin{lemma}[error in observables] \label{lem:obs-err-eq} 
With the unitary group denoted by $U(t)= \exp(-\I tH/\eps)$ and the notation 
$A(t)=U(t)^*A U(t)$,
we have
$$
\langle A \rangle_{u(t)} - \langle A \rangle_{\psi(t)} = \int_0^t \bigl\langle \tfrac 1{\I\eps} [ W_{u(t-s)}, A(s) ] \bigr\rangle_{u(t-s)}\D s.
$$
\end{lemma}

\begin{lemma}[commutator bound] \label{lem:obs-err-ineq}
Under the conditions of Theorem~\ref{thm:error-gauss}, we have
$$
\bigl| \bigl\langle \tfrac 1{\I\eps} [ W_{u(t-s)}, A(s) ] \bigr\rangle_{u(t-s)} \bigr| \le c\,\eps, \qquad 0\le s \le t \le \bar t.
$$
\end{lemma}

\medskip
The proof of Lemma~\ref{lem:obs-err-ineq} requires techniques that are different from those of this section and  is therefore deferred to Section~\ref{sec:proofobs}. Lemma~\ref{lem:obs-err-eq} is proved here.

\begin{proof} (of Lemma~\ref{lem:obs-err-eq}) 
We compare the outcome of a variational expectation value with the actual one. Writing the true solution as
$\psi(t) = U(t)\psi_0$ in terms of the unitary group, we can reformulate the difference of the expectation 
values in view of $\psi_0=u_0$ as
\[
\langle \psi(t)\!\mid\! A\psi(t)\rangle - \langle u(t)\!\mid\! A u(t)\rangle =
\int_0^t \tfrac{\D}{\D s} \langle u(t-s)\mid U(s)^* A U(s) u(t-s)\rangle \D s.
\]
For $A(s) = U(s)^* A U(s)$ we note that
\[
\tfrac{\D}{\D s} A(s) = \tfrac{1}{\I\eps} \left( A(s)H - H A(s)\right).
\]
Therefore, with $P_u^\perp =\Id -P_u$,
\begin{align*}
&\tfrac{\D}{\D s} \langle u(t-s)\mid A(s) u(t-s)\rangle 
=
\tfrac{1}{\I\eps} \left( 
\langle P_{u(t-s)}H u(t-s)\mid A(s) u(t-s)\rangle \right.\\
&\left.
+ \langle u(t-s)\mid( A(s)H - H A(s)) u(t-s)\rangle
-  \langle u(t-s)\mid A(s) P_{u(t-s)}H u(t-s)\rangle
\right)\\
&=
\tfrac{1}{\I\eps} \left( 
\langle u(t-s)\mid A(s) P^\perp_{u(t-s)}H u(t-s)\rangle
-\langle P^\perp_{u(t-s)}H u(t-s)\mid A(s) u(t-s)\rangle 
\right).
\end{align*}
Hence, 
\begin{align*}
\langle \psi(t)\!\mid\! A\psi(t)\rangle - \langle u(t)\!\mid\! A u(t)\rangle 
&= 
\tfrac{1}{\I\eps}\int_0^t \left( 
\langle u(t-s)\mid A(s) P^\perp_{u(t-s)}H u(t-s)\rangle \right.
\\
& \qquad \quad \left.
\!\! -\langle P^\perp_{u(t-s)}H u(t-s)\mid A(s) u(t-s)\rangle 
\right) \!\D s.
\end{align*}
Since $u(t)\in{\mathcal T}_{u(t)}{\mathcal M}$, we may write
\[
P_{u(t)}H u(t) =  P_{u(t)}H P_{u(t)}u(t) = H_{u(t)}u(t),
\]
with $H_{u}=P_uHP_u$. The above formula for the expectation values then simplifies to 
\[
\langle \psi(t)\!\mid\! A\psi(t)\rangle - \langle u(t)\!\mid\! A u(t)\rangle =
\int_0^t \langle u(t-s)\mid\tfrac{1}{\I\eps}[A(s), H-H_{u(t-s)}] u(t-s)\rangle \D s.
\]
Since we have $(H-H_u)u=W_u u$ by \eqref{Wu}, the result follows.
\end{proof}

\subsection{Differential equations for the parameters}

In this subsection we derive the following equations of motion for the parameters of variational Gaussian wave packet dynamics.

\begin{theorem} [equations of motion for parameters] 
\label{thm:eom-gauss}
For  Gaussian initial data  $u_0\in\mathcal M$ of unit norm, the variational Gaussian approximation \eqref{eq:var} is obtained with
parameters $(q(t),p(t),C(t),\zeta(t))$ that satisfy the following ordinary differential equations:
\begin{align*}
\dot{q} &= p,\\
\dot{p} &= -\langle \nabla_x V\rangle_u,\\
\dot{C} &= -C^2 - \langle \nabla_x^2 V\rangle_u,\\
\dot{\zeta} &= \tfrac12|p|^2 -\langle V\rangle_u + \tfrac{\I\eps}{2}\,\tr(C) + 
\tfrac{\eps}{4} \,\tr( (\Im C)^{-1}\langle\nabla^2_x V\rangle_u).
\end{align*}
\end{theorem}

The first observation is that the parameters change slowly, on the time scale $O(1)$, as opposed to the highly oscillatory Gaussian wave packet defined by them, which has $O(1)$ changes on time intervals of length $O(\eps)$.
We note further that the first two equations look very similar to the classical equations of motion 
$$
\dot q =p, \qquad \dot p= - \nabla V(q),
$$ 
which are obtained from $\dot q =p, \ \dot{p} = -\langle \nabla_x V\rangle_u$ in the limit $\eps\to 0$ if  the potential is continuous and the wave packet is localised as stated in assumption 1.~of Theorem~\ref{thm:error-gauss}.
Moreover, the last equation determines $\zeta$, up to small perturbations, as the action integral that corresponds to the classical Lagrange function $\tfrac12|p|^2 - V(q)$. Only the differential equation for the matrix $C$ seems not to be related to classical mechanics at first sight, but this apparent exemption will be resolved in the next subsection.

For the proof of Theorem~\ref{thm:eom-gauss} we first derive auxiliary results that allow us to give an explicit expression for the projected potential term $P_uVu$ in \eqref{PuVu}.
We start by setting up an orthonormal basis of the tangent space.

\begin{lemma}[orthonormal basis of the tangent space]\label{lem:onb}
For a  Gaussian $u\in{\mathcal M}$ of unit norm, let the invertible real $d\times d$ matrix $Q$ be a Choleski factor of the inverse of the width matrix, that is,
\[
\Im C = (QQ^T)^{-1}.
\]
For $x\in\Rb^d$, we denote 
\[
y = \tfrac{1}{\sqrt\eps} \,Q^{-1}(x-q),
\]
and define the following multivariate scaled Hermite functions $\varphi_k:\Rb^d\to\C$ for multi-indices $k\in\N^d$ with $|k|:=\sum_i k_i\le 2$:
\[
\varphi_k(x) = \left\{\begin{array}{ll}
u(x) & \text{if}\ |k|=0,\\*[1ex]
\sqrt2 y_m u(x) & \text{if}\ |k|=1\ \text{and}\  k_m=1,\\*[1ex]
\tfrac{1}{\sqrt2}(-1+2y_m^2)u(x) & \text{if} \ |k|=2 \ \text{and}\ k_m=2,\\*[1ex]
2 y_m y_n u(x) & \text{if}\ |k|=2\ \text{and}\  k_m = k_n = 1.
\end{array}\right.
\]
Then, $\{\varphi_k\}_{|k|\le 2}$ is an orthonormal basis of ${\mathcal T}_u{\mathcal M}$.
\end{lemma}

\begin{proof}
By Lemma~\ref{lem:tangent space}, the  tangent space consists of all quadratic  polynomials times the Gaussian wave packet $u$. 
To prove the orthonormality of the functions $\varphi_k$, we write the Gaussian density as
\begin{align*}
|u(x)|^2 &= \exp(-\tfrac2\eps\,\Im(\zeta)) \exp(-\tfrac1\eps(x-q)^T \Im(C)(x-q))\\
&=\exp(-\tfrac2\eps\,\Im(\zeta)) \exp( -y^T y)
\end{align*}
and calculate some Gaussian integrals.
\end{proof}

The orthonormal basis $\{\varphi_k\}_{|k|\le 2}$ allows us to explicitly determine projections of functions to the tangent space.
We start with functions whose gradient and Hessian vanish on average.

\begin{lemma}\label{lem:project}
We consider a normalised Gaussian $u\in{\mathcal M}$ and a smooth function $W:\Rb^d\to\Rb$ with at most polynomial growth satisfying
\[
\langle \nabla_x W\rangle_u = 0\quad\text{and}\quad \langle\nabla_x^2 W\rangle_u = 0.
\]
Then,
\[
P_uWu = \langle W\rangle_u u.
\]
\end{lemma}

\begin{proof} We compute the projection
\[
P_uWu = \sum_{|k|\le 2} \langle \varphi_k|Wu\rangle \varphi_k
\]
by examining its summands order by order. The zeroth order provides 
\[
\langle \varphi_0|\, Wu\rangle \varphi_0 = \langle u|W u\rangle u = \langle W\rangle_u u.
\]
For the first-order terms, a partial integration shows the vector identity
\begin{align*}
\langle \sqrt2 yu| Wu\rangle & = \int_{\Rb^d} \sqrt 2 y W(x) |u(x)|^2 \D x\\
&=\sqrt{\tfrac\eps2}Q^T \langle \nabla_x W\rangle_u,
\end{align*}
since the gradient of the Gaussian density satisfies
\[
\nabla_x |u(x)|^2 = -\tfrac{2}{\sqrt\eps}Q^{-T} y |u(x)|^2.
\]
Therefore, 
\[
\sum_{|k|=1} \langle\varphi_k|Wu\rangle\varphi_k = 0.
\]
For the second-order terms, we work with the complex symmetric matrix
\[
\Phi_2(x) = \tfrac{1}{\sqrt2}\left(- I_d+2yy^T\right) u(x)\in\C^{d\times d}.
\]
It contains all the basis functions $\varphi_k(x)$ with $|k|=2$ up to multiplicative factors. 
The $d(d+1)/2$ basis functions are redundantly placed on the $d^2$~entries of the matrix $\Phi_2(x)$ as follows.
The main diagonal carries the $\varphi_k(x)$ with $k_m=2$ for $m=1,\ldots,d$. 
The off-diagonal entries contain the functions $\varphi_k(x)/\sqrt2$ for
$k_m=k_n=1$ with $m\neq n$. The off-diagonal factors $1/\sqrt2$ compensate for listing 
these basis functions twice, and we have
\[
\sum_{|k|=2} \langle\varphi_k|Wu\rangle\varphi_k = 
\tr\left( \langle \Phi_2|Wu\rangle \Phi_2\right).
\] 
Two partial integrations provide
\begin{align*}
\langle \Phi_2 | Wu\rangle 
&= \int_{\Rb^d} \tfrac{1}{\sqrt2}(- I_d+2yy^T) W(x)|u(x)|^2 \D x\\
&= \tfrac{\eps}{2\sqrt2} Q^T \langle \nabla_x^2 W\rangle_u\,Q,
\end{align*} 
since the Hessian of the Gaussian density satisfies
\[
\nabla_x^2 |u(x)|^2 = \tfrac{1}{\eps} Q^{-T}(-2 I_d + 4yy^T)Q^{-1}|u(x)|^2.
\]
Therefore, 
\[
\sum_{|k|=2} \langle\varphi_k|Wu\rangle \varphi_k = 0. 
\]
\end{proof}

For a general potential function, the tangent space projection produces the following quadratic polynomial times the Gaussian.

\begin{proposition}[potential term projected to the tangent space] \label{prop:project}
For a normalised Gaussian $u\in{\mathcal M}$ centred at $q$ with width matrix $C$ and for a smooth potential $V:\Rb^d\to\Rb$ with at most polynomial growth, we have
\[
P_uVu = \left( \alpha+ a^T (x-q) + \tfrac12(x-q)^TA(x-q)\right) u,
\]
where 
\begin{align*}
\alpha &= \langle V \rangle _u-\tfrac\eps4\tr((\Im C)^{-1}\langle\nabla_x^2 V\rangle_u),\\
a &= \langle \nabla_x V\rangle_u,\quad\text{and}\quad
A = \langle \nabla_x^2 V\rangle_u.
\end{align*}
\end{proposition}
 
\begin{proof}
We write the potential as
\[
V = \langle V\rangle_u + \langle \nabla_x V\rangle_u^T (x-q) + \tfrac12(x-q)^T\langle\nabla_x^2V\rangle_u (x-q) + W
\]
and observe that 
\begin{align*}
P_uVu &= \left(\langle V\rangle_u + \langle \nabla_x V\rangle_u^T (x-q) + \tfrac12(x-q)^T\langle\nabla_x^2V\rangle_u (x-q) \right)u \\*[1ex]
&\quad+ P_uWu.
\end{align*}
The function $W$ has the average
\begin{align*}
\langle W\rangle_u 
&= -\tfrac12\langle (x-q)^T \langle\nabla_x^2 V\rangle_u (x-q)\rangle_u\\
&=-\tfrac\eps4 \, \tr\!\left((\Im C)^{-1}\langle \nabla_x^2 V\rangle_u\right).
\end{align*}
For the gradient and Hessian, we have    
\begin{align*}
\nabla_x V &= \langle \nabla_x V\rangle_u + \langle\nabla_x^2 V\rangle_u (x-q) + \nabla_x W,\\
\nabla_x^2 V &= \langle \nabla_x^2 V\rangle_u + \nabla_x^2W,
\end{align*}
and therefore
\[
\langle \nabla_x W\rangle_u = 0\quad\text{and}\quad \langle\nabla_x^2 W\rangle_u = 0.
\]
By Lemma~\ref{lem:project}, $P_uWu = \langle W\rangle_u u$,  
so that the claimed identity for $P_uVu$ is proved.
\end{proof}
 
The projection of the potential term is the crucial ingredient for determining the  equations of motion for the parameters. 
The time derivative and the Laplacian only require differentiation of the Gaussian wave packet. 
  
\begin{proof} (of Theorem~\ref{thm:eom-gauss})
We start by calculating the first derivatives of the approximate solution $u=u(x,t)$. We obtain
\begin{align*}
&-\I\eps\partial_t u \\
&= \left( -\dot{q }^T C (x-q ) + \tfrac12(x-q )^T \dot{C } (x-q ) + \dot{p }^T(x-q ) - p ^T \dot{q } + \dot{\zeta }\right) u
\end{align*}
and
\begin{align*}
-\I\eps\nabla_x u = \left( C (x-q ) + p \right) u.
\end{align*}
Hence, 
\begin{align*}
-\tfrac{\eps^2}{2}\Delta_x u = \left( \tfrac12(x-q )^T C ^2(x-q ) + p ^TC (x-q ) + \tfrac12|p |^2 \bch- \tfrac{\I\eps}{2}\tr(C )\ech\right) u. 
\end{align*}
By Proposition~\ref{prop:cons}, the norm of the variational approximation is conserved, and Proposition~\ref{prop:project} then provides
\begin{align*}
P_u(Vu) &= \left( \langle V \rangle _u-\tfrac\eps4\tr((\Im C )^{-1}\langle\nabla_x^2 V\rangle_u \right)u\\
&+ \langle\nabla_x V\rangle_u^T (x-q )u + \tfrac12(x-q )^T\langle\nabla_x^2V\rangle_u (x-q )u.
\end{align*}
Hence, we solve the variational equation by matching corresponding powers in the polynomials. We set
$\dot{q} = p$.
Then, the linear, quadratic and constant terms result in the expressions for $\dot p$, $\dot C$ and $\dot \zeta$ as stated.
\end{proof} 

We further note the following bound of the difference between averages and point evaluations.

\begin{lemma}[averages] \label{lem:av}
For smooth $V:\Rb^d\to\Rb$ with bounded second order derivatives and 
a normalised complex Gaussian $u\in\calM$ with position centre $q\in\Rb^d$, we have
\[
\left| \langle V\rangle_u - V(q) \right| \le c\,\eps, 
\]
where the constant $c>0$ depends on the second order derivatives of $V$ and the eigenvalues of the matrix $\Im C$.
\end{lemma}

\begin{proof}
We write
\begin{align*}
\langle V\rangle_u &=
 (\pi\eps)^{-d/2} (\det\Im C)^{1/2} \int_{\Rb^d} V(x) \exp(-\tfrac1\eps (x-q)^T \Im C (x-q)) \D x\\
 &=  \pi^{-d/2} (\det\Im C)^{1/2} \int_{\Rb^d} V(q+\sqrt\eps y) \exp(-y^T \Im C y) \D y.
\end{align*}
We perform a first-order Taylor approximation of $V(q+\sqrt\eps y)$ at $q$.
% with remainder term in integral form.
%\[
%V(q+\sqrt\eps y) = V(q) + \sqrt\eps\, \nabla V(q)^T y + \eps\, V_2(q,y).
%\]
%with 
%\[
%V_2(q,y) =  
%\sum_{|m|=2} \frac{1}{m!} \,y^m \int_0^1 (1-\vartheta)^2\, (\partial^m V)(q+\sqrt\eps\vartheta y)  \D\vartheta.
%\]
%Integrating the Gaussian against these three summands, we 
We then obtain $V(q)$ due to the unit integral and 
no contribution from $\nabla V(q)$ due to the vanishing first order moments. The remainder is 
of order $\eps$ and depends on both $V$ and $\Im C$.
\end{proof}

\subsection{Hagedorn's parametrisation of Gaussian wave packets}
\label{subsec:hag-gauss}
We now turn to a different parametrisation of complex Gaussians, in which the complex symmetric width matrix $C$  with positive definite imaginary part is factorised into two matrices with remarkable properties.

A real $2d\times 2d$ matrix $Y$ is called
{\it symplectic} if it satisfies the quadratic relation
$$%\begin{equation}\label{V:symp}
Y^TJY = J \quad\text{ with }\quad J=
\bch\begin{pmatrix} 0 & -\Id \\ \Id & 0 
\end{pmatrix}.\ech
$$%\end{equation}
With this notion, we have the following characterisation of complex symmetric matrices with positive definite imaginary part.

\begin{lemma}[symplecticity] \label{lem:hag-rel}
Let $Q$ and $P$ be complex $d\times d$ matrices such that the real $2d\times 2d$ matrix
\begin{equation}\label{PQ-symp}
Y =
\begin{pmatrix} \Re Q & \Im Q \\ \Re P & \Im P 
\end{pmatrix} \quad\text{is symplectic}.
\end{equation}
Then, $Q$ and $P$ are invertible, and 
\begin{equation} \label{CPQ}
C=PQ^{-1}
\end{equation}
is complex symmetric (\ie\ $C=C^T$) with positive definite imaginary part
\begin{equation}\label{QQ}
\Im C = (QQ^*)^{-1}.
\end{equation}
Conversely, every complex symmetric matrix $C$ with positive definite imaginary part can be written as $C=PQ^{-1}$ with matrices $Q$ and $P$ satisfying (\ref{PQ-symp}) and (\ref{QQ}).
\end{lemma}
We note that the symplecticity condition \eqref{PQ-symp} is equivalent to the relations
\begin{equation}\label{hag-rel}
\begin{array}{l}
Q^T P - P^T Q = 0 \\[1mm]
Q^* P - P^* Q = \bch 2\I\, \Id\,.\ech
\end{array}
\end{equation}
The factorisation \eqref{CPQ} is not unique, since multiplying $Q$ and $P$  from the \bch right \ech with a unitary matrix preserves the relations (\ref{hag-rel}).

\begin{proof} The second equation of 
(\ref{hag-rel}), multiplied from the left and the right with a vector 
$v\in\C^d$, yields
$$
(Qv)^*(Pv) - (Pv)^*(Qv) = 2\I\, \| v \|^2\,,
$$
which shows that $v=0$ is the only vector in the null-space 
of $Q$ or $P$. Hence, these matrices are invertible. The first equation of (\ref{hag-rel}), multiplied from the left with
$(Q^{-1})^T$ and from the right with $Q^{-1}$, gives
$$
PQ^{-1} - (Q^{-1})^T P^T = 0,
$$
which shows that $C=PQ^{-1}$ is complex symmetric.
Moreover, we have
\begin{align*}
(\Im C)  (QQ^*) &= \tfrac1{2\I}(PQ^{-1} - (Q^{-1})^*P^*)QQ^* 
\\
&=
\tfrac1{2\I}(PQ^* - (Q^*)^{-1}(P^*Q)Q^*) =\bch \Id \ech,
\end{align*}
where we use the second equation of (\ref{hag-rel}) for substituting $P^*Q=Q^*P - \bch 2\I \,\Id\ech $ to obtain the last equality. This gives us \eqref{QQ}.

Conversely, for a complex symmetric matrix $C$ with positive definite imaginary part, we set $Q=(\Im C)^{-1/2}$ and $P=CQ$. A direct calculation shows that these matrices satisfy the relations
(\ref{hag-rel}).
\end{proof}

We now parametrise the variational Gaussian wave packet approximation $u(\cdot,t)\in{\mathcal M}$ as
\begin{equation}\label{hag-par}
u(x,t) = \exp\!\left( \tfrac\I\eps \left(\tfrac12 (x-q(t))^T P(t)Q(t)^{-1}(x-q(t)) + p(t)^T(x-q(t)) + \zeta(t)\right) \right),
\end{equation}
where $P(t),Q(t)$ are required to satisfy the symplecticity condition \eqref{PQ-symp}. We  have the following differential equations for the parameters.

\begin{theorem} [equations of motion for Hagedorn's parameters] 
\label{thm:eom-gauss-PQ}
For  Gaussian initial data  $u_0\in\mathcal M$ of unit norm, the variational Gaussian approximation \eqref{eq:var} is obtained with
parameters $(q(t),p(t),Q(t),P(t),\zeta(t))$ that satisfy the following ordinary differential equations:
\begin{align*}
\dot{q} &= p, \quad\ \
\dot{p} = -\langle \nabla_x V\rangle_u,\\
\dot{Q} &= P, \quad\
\dot{P} =  - \langle \nabla_x^2 V\rangle_u Q,
\end{align*}
and $\zeta$ satisfies the same differential equation as in Theorem~\ref{thm:eom-gauss}, with $C=PQ^{-1}$ 
\bch and $(\Im C)^{-1}=QQ^*$. \ech
Moreover, $P(t)$ and $Q(t)$ then have the symplecticity property \eqref{PQ-symp} for all $t$, provided that this is satisfied for the initial values.
\end{theorem}

The equations for $q,p,\zeta$ are clearly the same as in Theorem~\ref{thm:eom-gauss}, but the quadratic differential equation for $C$ is replaced by equations for $P$ and $Q$ that are close to the linearisation of the classical equations of motion $\dot q=p$, $\dot p = -\nabla  V(q)$ around the solution $(q,p)$, that is, of the matrix differential equations $\dot Q =P$, $\dot P = - \nabla^2 V(q)Q$ that are linear in $(Q,P)$.

\begin{proof} If $\dot{Q} = P$ and
$\dot{P} =  - \langle \nabla_x^2 V\rangle_u Q$, then $C=PQ^{-1}$ satisfies
\begin{align*}
\dot C &= - PQ^{-1} \dot Q Q^{-1} + \dot P Q^{-1} \\
&= - PQ^{-1} PQ^{-1}  - \langle \nabla_x^2 V\rangle_u Q Q^{-1} = -C^2 - \langle \nabla_x^2 V\rangle_u ,
\end{align*}
which is the differential equation for $C$ in Theorem~\ref{thm:eom-gauss}. It remains to show that the symplecticity property of $Q$ and $P$ is conserved. We have
\begin{align*}
\frac{\D}{\D t}(Q^T P - P^T Q) &= \dot Q^T P + Q^T \dot P - \dot P^T Q - P^T \dot Q \\
&= P^T P -Q^T  \langle \nabla_x^2 V\rangle_u Q +
Q^T  \langle \nabla_x^2 V\rangle_u Q -P^TP =0,
\end{align*}
and similarly $\frac{\D}{\D t}(Q^* P - P^* Q) =0$. This shows that \eqref{hag-rel}, and hence \eqref{PQ-symp}, remains satisfied for all times.
\end{proof}

We refer to the matrices $Q$ and $P$ as position and momentum matrices, respectively, in view of their equations of motion. We will see in the following subsection and in Section~\ref{sec:hagwp} that their role goes well beyond merely providing more pleasing differential equations.

\subsection{Ehrenfest time}
\label{subsec:ehrenfest}
The time scale on which quantum dynamics can be approximated by classical dynamics is known as the {\it Ehrenfest time}. 
%By Theorem~\ref{thm:error-gauss} we know that a variational Gaussian wave packet approximates the full wave function for small $\eps$ on a time interval $0\le t \le \bar t$ provided that the Gaussian remains well localised with a width bounded by $\sqrt{\eps}/\rho$, where $\rho>0$ is an $\eps$-independent lower bound of the smallest eigenvalue of the width matrices $\Im C(t)$ for $0\le t \le \bar t$. 
The proof of Theorem~\ref{thm:error-gauss} shows that the $L^2$-norm error of the Gaussian approximation is small on a time interval $0\le t \le \bar t$ provided that
$\eps^{1/2} \ll \rho^{3/2}$, \ie\ $\rho^{-1} \ll \eps^{-1/3}$, where $\rho>0$ is an $\eps$-independent lower bound of the smallest eigenvalue of the width matrices $\Im C(t)$ for $0\le t \le \bar t$. 
By \eqref{QQ} we have $\Im C(t) = (Q(t)Q(t)^*)^{-1}$, and hence the smallest eigenvalue of $\Im C(t)$ is equal to 
$1/\| Q(t) \|_2^2$, the inverse square of the matrix $2$-norm of $Q(t)$.
On the other hand, by Theorem~\ref{thm:eom-gauss-PQ}, $Q(t)$ is up to $O(\eps)$ terms a solution to the linearisation of the classical equations of motion $\dot q =p$, $\dot p =-\nabla V(q)$.  
The growth of $\| Q(t) \|_2$ in time is therefore described by the  Lyapunov exponent $\lambda_*$ of the {\it classical} dynamics: for all $\lambda = \lambda_*+\delta$ with $\delta>0$, we have $\| Q(t) \|_2 \le c_\delta \,\e^{\lambda t}$. The Lyapunov exponent thus determines the Ehrenfest time: for all $\delta>0$,
$$
\log \rho^{-1} \le \log  \max_{0\le t \le \bar t} \| Q( t) \|_2^2 \le 2(\lambda_*+\delta) \bar t + 2 \log c_\delta.
$$
The Lyapunov exponent $\lambda_*$ is always non-negative for classical mechanics. For $\lambda_*>0$ (where we can choose $\delta=\lambda_*$), the condition $\rho^{-1} \ll \eps^{-1/3}$ thus yields the condition that $\lambda_* \bar t$ be logarithmic in $\eps^{-1}$.

The Lyapunov exponent  is zero for integrable systems, for which $\| Q(t) \|_2$ grows linearly with $t$, \ie\  $\| Q(t) \|_2\le c t$, so that for integrable systems
$$
 \rho^{-1} \le   \max_{0\le t \le \bar t} \| Q( t) \|_2^2 \le (c \bar t)^2.
$$
The condition $\rho^{-1} \ll \eps^{-1/3}$ then yields the asymptotic condition $\bar t \ll \eps^{-1/6}$.

\subsection{Gaussians evolving on a quadratic potential}
\label{subsec:gaussian-quadratic}
We already know from Proposition~\ref{prop:exact-gaussian} that in the case of a quadratic potential $V$, variationally evolving complex Gaussians are exact solutions to the Schr\"odinger equation \eqref{tdse}. We are now in a position to give the precise form and the equations of motion of the parameters of such a Gaussian solution. 

\begin{proposition}[Gaussian solutions for a quadratic potential]
\label{V:thm:gwp-quad}
Let $V$ be a quadratic potential, let 
$(q(t),p(t),Q(t),P(t))$  be a solution of the classical equations of motion and their linearisation,
\begin{equation}\label{eom-classical}
\dot q=p, \ \ \dot p=-\nabla V(q) \quad\text{and}\quad \dot Q=P,\ \ \dot P=-\nabla^2V(q)Q,
\end{equation}
and let $S(t)= \int_0^t \bigl(\frac12|p(s)|^2-V(q(s))\bigr) {\mathrm d}s$ be the corresponding classical action.
Let $Q(0)$ and $P(0)$ satisfy 
the symplecticity relations (\ref{hag-rel}).
Then, the complex Gaussian
\begin{align}\label{hag-gauss}
e^{\I S(t)/\eps} \,&(\pi\eps)^{-d/4} (\det Q(t))^{-1/2}\ \times 
\\
&\exp\Bigl( \frac \I{2\eps} (x-q(t))^T P(t)Q(t)^{-1} (x-q(t)) + 
\frac \I\eps p(t)^T(x-q(t)) \Bigr)
 \nonumber
\end{align}
is a solution to the time-dependent Schr\"odinger equation (\ref{tdse}) of unit $L^2$-norm. Moreover, the matrices $Q(t)$ and $P(t)$ satisfy 
the symplecticity relations \eqref{hag-rel} for all $t$.
\end{proposition}

The branch of the square root of $\det Q(t)$ is to be chosen such that this becomes a continuous function of $t$.

\begin{proof} Since a variationally evolving Gaussian is an exact solution in the case of a quadratic potential, it suffices to show that a variational Gaussian takes the stated form. We have the quadratic polynomial
$$
V(x) = V(q) + \nabla V(q)^T (x-q) + \tfrac12 (x-q)^T \nabla^2V(q)(x-q).
$$
Since  for a normalised Gaussian $u\in{\mathcal M}$ we have $P_uVu=Vu$, Proposition~\ref{prop:project} shows that
\begin{align*}
&V(q)=  \langle V \rangle _u-\tfrac\eps4\tr((\Im C)^{-1}\langle\nabla_x^2 V\rangle_u),
\\
&\nabla V(q) = \langle \nabla V \rangle_u, \quad 
\nabla^2 V(q) = \langle \nabla^2 V \rangle_u.
\end{align*}
Theorem~\ref{thm:eom-gauss-PQ} therefore shows that the differential equations for the parameters become \eqref{eom-classical} together with
$$
\dot\zeta = \tfrac12 |p|^2 - V(q) + \tfrac12 \I\eps\, \tr(C ),\quad\text{ with } C=PQ^{-1},
$$
and that $Q(t),P(t)$ satisfy \eqref{hag-rel} for all $t$. The Gaussian \eqref{hag-gauss} is of unit norm by Lemma~\ref{lem:norm}, \bch while the relation $\Im C=(QQ^*)^{-1}$ follows from Lemma~\ref{lem:hag-rel}.\ech So it remains to show that
$$
e^{\I \zeta(t)/\eps} =
e^{\I S(t)/\eps} \, (\pi\eps)^{-d/4} (\det Q(t))^{-1/2} e^{\I\phi}
$$
for some real phase $\phi$ that is independent of $t$. As we already know that this holds at $t=0$, taking the logarithm and differentiating with respect to $t$ shows that this is satisfied if, for all $t$,
$$
\frac\I\eps \dot\zeta(t)=\frac\I\eps \dot S(t) - \tfrac12 \frac{d}{dt} \log \det Q(t) .
$$
This holds true because of the differential equation for $\zeta$ and because
$$
\frac{d}{dt} \log \det Q = \tr(\dot Q Q^{-1}) = \tr(PQ^{-1}) = \tr(C).
$$
This yields the stated result.
\end{proof}
We remark that the result and its proof extend in a straightforward way to time-dependent quadratic potentials $V(\cdot,t)$. Replacing the averages by point evaluations in solving the equations of motion of Theorem~\ref{thm:eom-gauss-PQ}, \ie\ solving instead the differential equations \eqref{eom-classical} for a general smooth potential $V$,
%\[
%\dot q = p,\quad \dot p = -\nabla V(q), \quad \dot Q = P,\quad \dot P = -\nabla^2 V(q)Q,
%\]
therefore corresponds to solving the Schr\"odinger equation with a locally quadratic approximation to the potential at the current position $q(t)$.

\subsection{Notes}

The time-dependent variational principle was first used by \cn{Dir30} in the context of what is now called the time-dependent Hartree--Fock method. This was then taken up in the book by \cn{Fre34}.  We refer to \cn{KraS81} and \cn{Lub08} for dynamic, geometric and approximation aspects of the Dirac--Frenkel time-dependent variational principle.

Variational Gaussian wave packets became prominent through the work of \cn{Hel76}.
In the chemical literature, simplified variants of variational Gaussian wave packets are often used, in particular with a diagonal width matrix (Hartree Gaussians; see \cn{Hel76}) or with a fixed width matrix (frozen Gaussians; see \cn{Hel81}). Variational linear combinations of Gaussian wave packets are also used successfully; see \cn{RicPSWBL15} and references therein. The book chapter by \cn{VanB19} reviews the 
application of various non-variational thawed Gaussian approximations, in particular of the 
single-Hessian thawed Gaussian approximation that is energy-conserving despite being non-variational; see also \cite{BegCV19}.

The conservation of energy of a variational Gaussian wave packet generally follows from the time-dependent variational principle, which yields a Hamiltonian system on the approximation manifold.   
The conservation of norm and linear and angular momentum are instances of the general result that the average of a self-adjoint operator is conserved in the variational approximation if the operator commutes with the Hamiltonian and is such that it maps every point on the approximation manifold into its tangent space. These aspects are discussed in \cn[Section II.1]{Lub08}.

Error bounds for  Gaussian wave packets  were first given by \cn{Hag80} in a non-variational setting (with point evaluations instead of averages of the potential and its gradient and Hessian). The $O(\sqrt\eps)$ $L^2$-error bound of Theorem~\ref{thm:error-gauss} for variational Gaussian wave packets was given by \cn[Section II.4]{Lub08}, whereas the $O(\eps)$ error bound for the averages of observables appears to be new. We note that these error bounds are not valid for Gaussians with width matrices that are restricted to be diagonal or fixed.

The equations of motion for the parameters of a multi-variate variational Gaussian wave packet, as given in Theorem~\ref{thm:eom-gauss}, were first derived by \cn{CoaK90}.
The  differential equations for the parameters have a non-canonical Hamiltonian structure that was studied by \cn{FaoL06} and, with different geometric interpretations, by \cn{OhsL13}; see also \cn{Ohs15b} for further geometrical aspects of variational Gaussian wave packets and their use for obtaining conservation laws from symmetries. 

\bch The relationship between the Ehrenfest time and the classical Lyapunov exponent was studied in a more general setting by \cn{ComR97}, using different arguments from those given here; see e.g. \cn{BamGP99} and \cn{HagJ00} for further developments.\ech

The factorisation \eqref{CPQ}   and its use with evolving Gaussian wave packets were introduced by \cn{Hag80}, where also Proposition~\ref{V:thm:gwp-quad} was shown.  In our text we use the notation of \cn[Chapter~V]{Lub08}, working
with the matrices $(Q,P)$ which correspond to $(A,\I B)$ in Hagedorn's work, because of their relation to symplectic matrices and because their equations of motion are close to linearised classical equations of motion. The set of complex symmetric matrices with positive definite imaginary part is known as the Siegel half-space, after \cn{Sie43} who studies  generalised M\"obius transformations of this set based on symplectic matrices. The geometry behind the factorisation \eqref{CPQ} and its relation to Siegel's results and with Marsden--Weinstein reduction was explored by \cn{Ohs15a}.

%\newpage
%\section{Hagedorn's semiclassical wave packets}
\def\jvec{{\langle j \rangle}}

\section{Hagedorn's semiclassical wave packets}
\label{sec:hagwp}
%The time-dependent Schr\"odinger equation in semiclassical scaling,
%\index{semiclassical scaling}
%\begin{equation}\label{V:tdse}
%\I\eps\frac{\partial\psi}{\partial t} = -\frac{\eps^2}{2m}
%\Delta\psi + V\psi
%\end{equation}
%with a small parameter $\eps$, has typical solutions that are
%wavepackets of width $\sim \sqrt{\eps}$, highly oscillatory with wavelength $\sim\eps$, 
%and with the envelope moving at velocity $\sim 1$. 
%Grid-based numerical methods therefore need very fine resolution for small $\eps$ and thus become computationally infeasible or at least very expensive; cf.~Markowich, Pietra \& Pohl (\cite{MaPP99}). The highly oscillatory solution behaviour also excludes the approximation of the wave function on sparse grids in higher dimensions, because the necessary smoothness requirements for this technique are not met for small $\eps$; see Gradinaru~(\cite{Gra08}).

In this section, the wave function is approximated to higher order in the semiclassical parameter $\eps$ by a moving and deforming complex Gaussian times a polynomial. Such a wave packet is written in a parameter-dependent orthonormal basis constructed by  \cn{Hag98}. This basis has very favourable propagation properties in the time-dependent Schr\"odinger equation (\ref{tdse}).

\subsection{Construction of Hagedorn functions via ladder operators}

We start from a normalised $d$-dimen\-sional complex Gaussian in Hagedorn's parametrisation (see Sections~\ref{subsec:hag-gauss} and~\ref{subsec:gaussian-quadratic}):
\begin{align}\label{V:phi-0}
&\varphi_0^\eps[q,p,Q,P](x) = 
\\
&\qquad (\pi\eps)^{-d/4} (\det Q)^{-1/2}
\exp\Bigl( \frac \I{2\eps} (x-q)^T PQ^{-1} (x-q) + 
\frac \I\eps p^T(x-q) \Bigr)
\nonumber
\end{align}
with $q,p\in\Rb^d$ and matrices $Q,P\in\C^{d\times d}$ that satisfy the symplecticity relations (\ref{hag-rel}). 
This Gaussian is of unit $L^2$-norm by Lemmas~\ref{lem:norm} and~\ref{lem:hag-rel}. The 
branch of the square root is chosen appropriately, such that
$(\det Q(t))^{-1/2}$ is a continuous function of $t$ for a continuous family of invertible matrices $Q(t)$. We write
$\varphi_0$ or $\varphi_0^\eps$ instead of $\varphi_0^\eps[q,p,Q,P]$ when the parameters are clear from the context.

We denote the position and momentum {\it operators} by 
$\widehat q=(\widehat q_j)_{j=1}^d$ and
$\widehat p=(\widehat p_j)_{j=1}^d$, respectively: 
$$
\widehat q \psi = x \psi,\quad\
\widehat p \psi = -\I\eps \nabla_x\psi.
$$
The hats are added to avoid confusion with the Gaussian parameters $q$ and~$p$. The commutator relation 
\begin{equation}\label{V:qp-comm}
\frac1{\I\eps}\, [\widehat q_j, \widehat  p_k] 
= \delta_{jk}
\end{equation}
is fundamental in quantum mechanics and, in particular, in the construction of the harmonic oscillator eigenfunctions (the Hermite functions) using Dirac's ladder.
Extending that construction, we consider Hagedorn's {\it parameter-dependent ladder operators}
$A=(A_j)_{j=1}^d$ and $A^\dagger=(A_j^\dagger)_{j=1}^d$ defined as
\begin{equation}\label{V:ladder}
\begin{array}{l}
A\phantom{^\dagger} = A[q,p,Q,P] 
= \displaystyle 
-\frac \I{\sqrt{2\eps}}\Bigl( P^T(\widehat q - q) -
Q^T(\widehat p - p) \Bigr) \\[3mm]
A^\dagger = A^\dagger[q,p,Q,P] = \displaystyle 
\frac \I{\sqrt{2\eps}}\Bigl( P^*(\widehat q - q) -
Q^*(\widehat p - p) \Bigr)\,.
\end{array}
\end{equation}
We note that for $d=1$, $\eps=1$, $q=0$, $p=0$, $Q=1$, $P=\I$, these operators reduce to  Dirac's ladder operators $\frac1{\sqrt{2}}(\widehat q+\I \widehat p)$ and $\frac1{\sqrt{2}}(\widehat q-\I \widehat p)$, respectively.
The key properties
of those operators extend as follows.

\begin{lemma}[commutator relations] \label{V:lem:ladder}
If $Q$ and $P$ satisfy the symplecticity relations (\ref{hag-rel}), then we have the commutator identities
%(for $j,k=1,\ldots,d$)
\begin{equation}\label{V:ladder-comm}
[A_j, A_k^\dagger] = \delta_{jk} \,.
\end{equation}
Moreover, $A_j^\dagger$ is adjoint to $A_j$ on the 
Schwartz space $\calS$ of functions on $\Rb^d$ that are infinitely differentiable and decay faster than any negative power:
\begin{equation}\label{V:ladder-adj}
\langle A_j^\dagger \varphi \,|\,\psi \rangle =
\langle  \varphi \,|\,A_j \psi \rangle
\qquad \text{for all } \varphi,\psi\in \calS\,.
\end{equation}
\end{lemma}

\begin{proof} With $Q=(Q_{jk})$ and $P=(P_{jk})$, we have
(with $q=p=0$ for simplicity)
$$
[ A_j, A_k^\dagger] = \frac1{2\eps} \Bigl[
\sum_{\ell=1}^d \bigl(
P_{\ell j} \widehat q_\ell -
Q_{\ell j} \widehat p_\ell \bigr)
\, , \,
\sum_{m=1}^d \bigl(
\overline P_{m k} \widehat q_m  -
\overline Q_{m k} \widehat p_m \bigr)
\Bigr]\,.
$$
By the canonical commutator identities (\ref{V:qp-comm}), this simplifies to
$$
[ A_j, A_k^\dagger] = \frac \I2 
\sum_{\ell=1}^d \bigl( -P_{\ell j}\overline Q_{\ell k} +
Q_{\ell j} \overline P_{\ell k} \bigr) =
\frac \I2 \bigl( -Q^*P + P^*Q \bigr)_{k,j},
$$
and by the symplecticity relation (\ref{hag-rel}), this equals $\delta_{jk}$.

The relation (\ref{V:ladder-adj}) follows directly from the fact that 
$\widehat q_\ell$ and $\widehat p_\ell$ are 
symmetric operators on $\calS$.
\end{proof}

\begin{lemma}[null-space]\label{lem:ladder-null}
If $Q$ and $P$ satisfy the symplecticity relations 
(\ref{hag-rel}), then 
the complex Gaussian $\varphi_0=\varphi_0^\eps[q,p,Q,P]$ of (\ref{V:phi-0}) 
spans the null-space of $A$.
\end{lemma}

\begin{proof} If $\varphi\in\calS$ is in the null-space of $A$, then it must satisfy the linear system of partial differential equations
$$
-\I\eps\nabla\varphi(x) - p\varphi(x) = C(x-q)\varphi(x)
$$
with the complex symmetric matrix $C=PQ^{-1}$. Multiples of
$\varphi_0$ are the only non-trivial solutions of this equation, since substituting $\varphi=\rho \varphi_0$ yields the equation $\nabla\rho=0$, so that $\rho$ is constant.
\end{proof}

With the properties established in these two lemmas, we can now construct eigenfunctions of the operators 
$A_jA_j^\dagger$ to eigenvalues $1,2,3,\ldots$ by the same arguments that are familiar from Dirac's construction of the eigenfunctions of the harmonic oscillator $\frac12(\widehat q^2 + \widehat p^2)= \frac1{\sqrt{2}}(\widehat q+\I \widehat p)\,\frac1{\sqrt{2}}(\widehat q-\I \widehat p)-\frac12$.

Let $k=(k_1,\ldots,k_d)$ be a multi-index with non-negative integers $k_j$, and let
\bch$\langle j \rangle = (0,\ldots, 1, \ldots, 0)$\ech denote the $j$th $d$-dimensional unit vector.

With Lemmas~\ref{V:lem:ladder} and~\ref{lem:ladder-null} it follows that
$$
A_jA_j^\dagger\varphi_0 = A_j^\dagger A_j\varphi_0 + \varphi_0=\varphi_0\,,
$$
and hence $\varphi_0$ is an eigenfunction of $A_jA_j^\dagger$ to the 
eigenvalue $1$. Applying the operator $A_j^\dagger$ to both sides of
this equation, we see that $\phi_{\langle j \rangle} :=A_j^\dagger\varphi_0$ is an
eigenfunction of $A_j^\dagger A_j$ to the 
eigenvalue $1$.

On the other hand, we further have by Lemmas~\ref{V:lem:ladder} and~\ref{lem:ladder-null} that
$$
A_lA_j^\dagger\varphi_0 = A_j^\dagger A_l\varphi_0 =0 \quad\ \text{ for }l\ne j,
$$
and applying the operator $A_l^\dagger$ to both sides of
this equation, we see that $\phi_{\langle j \rangle} =A_j^\dagger\varphi_0$ is an
eigenfunction of $A_l^\dagger A_l$ to the 
eigenvalue $0$.

We continue in this way to construct recursively
\begin{equation}\label{raise}
\phi_{k+\langle j \rangle}=A_j^\dagger\phi_k
\end{equation}
 for every multi-index $k=(k_1,\ldots,k_d)$ with non-negative integers $k_l$. We thus 
obtain eigenfunctions $\phi_k$
to $A_l^\dagger A_l$ with eigenvalue $k_l$ and  to $A_lA_l^\dagger$ with eigenvalue $k_l+1$.
These eigenfunctions are not yet normalised.
To achieve this, we note that by Lemma~\ref{V:lem:ladder},
$$
\| A_j^\dagger\phi_k \|^2 = \langle A_j^\dagger\phi_k \,|\, A_j^\dagger\phi_k \rangle =
\langle \phi_k\,|\, 
A_jA_j^\dagger\phi_k \rangle = (k_j+1)\,\|\phi_k\|^2\,.
$$
We therefore obtain normalised eigenfunctions to $A_lA_l^\dagger$ and $A_l^\dagger A_l$
by defining  the functions
$\varphi_k=\varphi_k^\eps[q,p,Q,P]$ recursively by
\begin{equation}\label{V:raising}
\varphi_{k+\langle j \rangle} = \frac1{\sqrt{k_j+1}}\, A_j^\dagger\varphi_k.
\end{equation}
Since $A_j\varphi_{k+\langle j \rangle}= \frac1{\sqrt{k_j+1}}\, A_jA_j^\dagger\varphi_k =
\sqrt{k_j+1}\, \varphi_k$, we also have (replacing $k_j+1$ by $k_j$)
\begin{equation}\label{III:lowering}
\varphi_{k-\langle j \rangle} = \frac1{\sqrt{k_j}}\, A_j\varphi_{k},
\end{equation}
so that $A_j^\dagger$ and $A_j$ can be viewed as raising and lowering operators, respectively, in the $j$th component of the multi-index.
%These relations explain the names of {\it raising operator} and
%{\it lowering operator} for $A^\dagger$ and $A$, respectively,
%and of {\it ladder operators} for both of them.
Multiplying (\ref{V:raising}) by $\sqrt{k_j+1}$ and (\ref{III:lowering}) by $\sqrt{k_j}$, summing the resulting formulas
and using the
definitions of $A_j$ and $A_j^\dagger$ and the fact that
$QQ^*$ is a real matrix by Lemma~\ref{lem:hag-rel}, we obtain the three-term
recurrence relation
\begin{equation}\label{V:rec}
Q\Bigl(\sqrt{k_j+1}\,\varphi_{k+\jvec}(x)\Bigr)_{j=1}^d = 
\sqrt{\frac2\eps}\, (x-q)\varphi_{k}(x) 
- \overline Q \Bigl(\sqrt{k_j}\,\varphi_{k-\jvec}(x)\Bigr)_{j=1}^d \,.
\end{equation}
This permits us to evaluate $\varphi_k(x)$ at any given point $x$.
It also shows that $\varphi_k$ is the product of a polynomial of degree
$|k|=k_1+\ldots+k_d$ with the Gaussian $\varphi_0$. Writing this property as
\begin{equation}\label{phik-pol}
\varphi_k(x) = \frac{1}{\sqrt{2^{|k|}k!}} \ p_k(\tfrac{1}{\sqrt\eps}Q^{-1}(x-q))\ \varphi_0(x),
\end{equation}
we obtain the polynomial three-term recurrence relation
\begin{equation}\label{phik-pol-rec}
\Bigl(p_{k+\jvec}(x)\Bigr)_{j=1}^d = 2x \, p_{k}(x) 
- 2Q^{-1}\,\overline Q \Bigl(k_j\,p_{k-\jvec}(x)\Bigr)_{j=1}^d. 
\end{equation}
This shows that if the matrix $Q^{-1}\,\overline Q$ is diagonal, then each $p_k$ is the tensor product of appropriately scaled univariate Hermite polynomials. Indeed, if $Q^{-1}\,\overline Q = {\rm diag}(\lambda_1,\ldots,\lambda_d)$, 
then
\[
p_k(x) = \prod_{j=1}^d \lambda_j^{k_j/2} \ H_{k_j}\!\Bigl(\frac{x_j}{\sqrt{\lambda_j}}\Bigr)
\]
where $(H_n)_{n\ge0}$ are the Hermite polynomials defined by the three-term recurrence relation 
\bch $H_{n+1}(y) = 2yH_n(y) - 2n H_{n-1}(y)$, $H_0(y) = 1$.\ech If the matrix $Q^{-1}\,\overline Q$ is not diagonal, then 
the polynomials $p_k$ do not have tensor product form. 

\medskip
The Hagedorn functions $\varphi_k$ are eigenfunctions 
of the symmetric operators $A_j A_j^\dagger$ for each $j$. Therefore, they are orthogonal:
$$
(k_j+1) \langle  \varphi_k,\varphi_\ell \rangle = \langle A_j A_j^\dagger \varphi_k,\varphi_\ell \rangle=
\langle  \varphi_k,A_j A_j^\dagger\varphi_\ell \rangle= (\ell_j+1) \langle \varphi_k,\varphi_\ell \rangle,
$$
so that $ \langle \varphi_k,\varphi_\ell \rangle=0$ for $k\ne \ell$. The functions $\varphi_k$ even form a complete orthonormal system, as
can be shown by extending the arguments in the proof of completeness of the Hermite functions. We do not present this proof here, as our interest in the following will be in the evolution of localised wave packets with a moderate polynomial degree rather than in the expansion of general functions in $L^2(\Rb^d)$ in the orthonormal basis of Hagedorn functions.

We summarise the above construction as follows.

 \begin{theorem}[Hagedorn functions]
\label{V:thm:hag-fun}
The functions $\varphi_k=\varphi_k^\eps[q,p,Q,P]$ 
defined by (\ref{V:phi-0}) and 
(\ref{V:raising}) form a complete $L^2$-orthonormal set of functions.
\end{theorem}

\subsection{Evolution of Hagedorn wave packets for quadratic Hamiltonians}
We now consider the Hamiltonian $H=-\frac{\eps^2}{2}\Delta + V= \frac12 \widehat p\cdot\widehat p +V$ with a quadratic potential $V$ and take a  Hagedorn function as initial data for the time-dependent Schr\"odinger equation \eqref{tdse}. There is the following remarkable result.

\begin{theorem}[Hagedorn functions on a quadratic potential]
\label{V:thm:hagwp-quad}
Let $V$ be a quadratic potential, let 
$(q(t),p(t),Q(t),P(t))$  be a solution of the classical equations of motion and their linearisation,
\begin{equation}\label{V:classical-eom}
\dot q=p, \ \ \dot p=-\nabla V(q) \quad\text{and}\quad \dot Q=P,\ \ \dot P=-\nabla^2V(q)Q,
\end{equation}
and let $S(t)= \int_0^t \bigl(\frac12|p(s)|^2-V(q(s))\bigr) {\mathrm d}s$ be the corresponding classical action.
Assume that $Q(0)$ and $P(0)$ satisfy 
the symplecticity relations \eqref{hag-rel}.
Then, for every multi-index $k$,
$$%\begin{equation}\label{V:psi-gwp}
e^{\I S(t)/\eps} \,\varphi_k^\eps[q(t),p(t),Q(t),P(t)](x)
$$%\end{equation}
is a solution of the time-dependent Schr\"odinger equation
(\ref{tdse}). Moreover, the symplecticity relations  \eqref{hag-rel} are satisfied by $Q(t)$ and $P(t)$ for all $t$.
\index{theorem!Hagedorn}
\end{theorem}
 
 As a preparation for the proof, we first consider the evolution of the parameter-dependent ladder operators under the quadratic Hamiltonian. Along a solution $(q(t),p(t),Q(t),P(t))$ of the classical equations \eqref{V:classical-eom} we consider the time-depen\-dent operators
$$
A_j(t) = A_j[q(t),p(t),Q(t),P(t)],\qquad
A_j^\dagger(t) = A_j^\dagger[q(t),p(t),Q(t),P(t)]\,.
$$
These operators evolve according to the following equations.

\begin{lemma}[ladder evolution for quadratic potentials]\label{V:lem:A-evol} Let $H=-\frac{\eps^2}{2}\Delta + V$ be the Hamiltonian of \eqref{tdse} with 
a quadratic potential $V$. Then we have
$$
\dot A_j = \frac1{\I\eps}[A_j,H]\,,\qquad
\dot A_j^\dagger = -\frac1{\I\eps}[A_j^\dagger ,H]\,.
$$
\end{lemma}

\begin{proof} With \eqref{V:classical-eom} we obtain for $A(t)=(A_j(t))$ 
$$
\dot A = \frac \I{\sqrt{2\eps}} \bigl(
Q^T \nabla V(\widehat q) +  P^T \widehat p \bigr).
$$
The same expression is obtained for $\frac1{\I\eps}[A,H]$ with the commutator relations (\ref{V:qp-comm}), which further yield
$\frac1{\I\eps}[\widehat q_j, \widehat  p_k^2] = 
\delta_{jk}\cdot 2 \widehat  p_k$  and
$\frac1{\I\eps}[\widehat q_j^2, \widehat  p_k] = 
\delta_{jk}\cdot 2 \widehat  q_j$. The result for $A^\dagger$ is obtained by taking complex conjugates.
\end{proof}

\begin{proof} (of Theorem~\ref{V:thm:hagwp-quad})
By Theorem~\ref{V:thm:gwp-quad},
the result holds true for ${k=0}$. In view of the construction of the functions $\varphi_k$ by (\ref{V:raising}), the result for general $k$ follows by induction if we can show that with a solution
$\psi(\cdot,t)$, also $A_j^\dagger(t)\psi(\cdot,t)$ is a solution to the Schr\"odinger equation \eqref{tdse}. This follows from 
$$
\I\eps \frac\partial{\partial t}(A_j^\dagger\psi)= 
\I\eps {\dot A}_j^\dagger\psi + A_j^\dagger H\psi =
\Bigl( \I\eps {\dot A}_j^\dagger\psi + [A_j^\dagger,H]\psi\Bigr) + H A_j^\dagger \psi ,
$$
because the expression in big brackets vanishes by 
Lemma~\ref{V:lem:A-evol}. The preservation of the symplecticity relations \eqref{hag-rel} is already known from Proposition~\ref{V:thm:gwp-quad}.
\end{proof}

Like Proposition~\ref{V:thm:gwp-quad}, Theorem~\ref{V:thm:hagwp-quad} also remains valid, with the same proof, for a time-dependent quadratic potential $V(\cdot,t)$.

\subsection{A Galerkin method for non-quadratic potentials}
For the time-dependent Schr\"odinger equation (\ref{tdse}) with a non-quadratic potential $V$, the wave function will now be approximated by a finite linear combination of Hagedorn functions whose parameters evolve according to the classical equations of motion. The coefficients of the linear combination are determined by a Galerkin condition on the time-varying approximation space that is spanned by a basis of finitely many evolving Hagedorn functions.

We search for an approximation to the wave function of the form
\begin{equation}\label{V:hag-series}
\psi_\calK(x,t)=e^{\I S(t)/\eps} 
\sum_{k\in\calK} c_k(t)\, \varphi_k(x,t)
\end{equation}
where we  abbreviate $\varphi_k(x,t)=\varphi_k^\eps[q(t),p(t),Q(t),P(t)](x)$ for a solution
$(q(t),p(t),Q(t),P(t))$  to the classical equations
\eqref{V:classical-eom} and where $S(t)= \int_0^t \bigl(\frac12|p(s)|^2-V(q(s))\bigr) {\mathrm d}s$ is the corresponding classical action.  
The finite multi-index set $\calK$ may, \eg\ be a cube 
$\{k\in \N^d\,:\,|k_j|\le K\}$ or a  simplex $\{ k \in \N^d\,:\, |k|= \sum_{j=1}^d k_j \le K\}$ or a hyperbolically reduced set $\{ {k \in \N^d}\,:\,  \prod_{j=1}^d (1+k_j) \le K\}$. The initial data will be assumed to be a linear combination over a smaller set $\calK_0\subset\calK$.

The coefficients $c_k(t)$ are determined by the Galerkin condition
\begin{equation}\label{V:hag-galerkin}
\Big\langle \varphi_\ell(t) \,\Big|\, 
\I\eps\frac{\partial \psi_\calK}{\partial t}(t)
-H\psi_\calK(t) \Big\rangle = 0
\quad\ \text{for each $\ell\in\calK$ and for all $t$.}
\end{equation}
If we write the potential as
$$
V=U_{q(t)} + W_{q(t)}
$$
with the quadratic Taylor polynomial of $V$ at the position $q$,
$$
U_q(x) =V(q) + \nabla V(q) ^T(x-q) +\tfrac12 (x-q)^T \nabla^2 V(q) (x-q),
$$  
and with the non-quadratic remainder $W_q$, then we have
\begin{eqnarray}
\label{V:defect}
&&\I\eps\frac\partial{\partial t}(c_k e^{\I S/\eps}\varphi_k) - 
H(c_ke^{\I S/\eps}\varphi_k) =
\I\eps \dot c_k e^{\I S/\eps}\varphi_k 
\\
\nonumber
&&\quad
+ \ c_k \Bigl( \I\eps\frac\partial{\partial t}(e^{\I S/\eps}\varphi_k) +
\frac{\eps^2}{2}\Delta (e^{\I S/\eps}\varphi_k) -
U_q e^{\I S/\eps}\varphi_k \Bigr) - c_k e^{\I S/\eps} W_q \varphi_k.
\end{eqnarray}
The term in big brackets vanishes by the version of
Theorem~\ref{V:thm:hagwp-quad} 
for time-dependent quadratic potentials, since just the quadratic part $U_q$ of the potential $V$ determines
$q(t),p(t),Q(t),P(t),S(t)$.
The Galerkin condition \eqref{V:hag-galerkin} then gives us the following differential equations for the coefficients 
$c(t)=(c_k(t))_{k\in\calK}\,$:
\begin{equation}\label{V:hag-coeff-ode}
\I\eps \dot c(t) = G(t)c(t) \quad\text{ with }\quad
G(t)=\bigl( \langle \varphi_\ell(\cdot,t)\,|\,W_{q(t)}\varphi_k(\cdot,t)
\rangle\bigr)_{\ell,k\in\calK}\,.
\end{equation}
We now easily observe norm conservation of the approximation: 

\begin{lemma}[norm conservation]
The Galerkin approximation \eqref{V:hag-series} is norm-conserving.
\end{lemma}

\begin{proof}
Since the matrices $G(t)$ are Hermitian, the Euclidean norm of the coefficient vector $c(t)$ is preserved and hence, because of the orthonormality of the Hagedorn functions $\varphi_k(t)$, the $L^2$-norm of the approximate wave function $\psi_\calK(t)$ is also preserved.
\end{proof}

With this approximation method, the differential equations \eqref{V:classical-eom} for the parameters $(q(t),p(t),Q(t),P(t))$ of the Hagedorn functions are obviously independent of the coefficients $c_k(t)$, but the differential equations \eqref{V:hag-coeff-ode} depend on the parameters through the basis functions $\varphi_k(\cdot,t)=\varphi_k^\eps[q(t),p(t),Q(t),P(t)]$. We remark that with a fully variational approach for parameters and coefficients such a separation would not be attained. However, we do not have exact energy conservation with the semi-variational approach considered here.

An important feature is that the coefficients are slowly varying: 
$$
\|\dot c(t)\| \le C\eps^{1/2},
$$ 
which holds true because  the norm of the Galerkin matrix $G(t)$ is bounded by $C\eps^{3/2}$ with a constant $C$ that depends on the multi-index set $\calK$ and on the width of the underlying Gaussian $\varphi_0(\cdot,t)$. This follows directly from Lemma~\ref{lem:int-bound}. We even have the following stronger bounds for the entries of the matrix $G(t)$.

\begin{lemma}[bounds for the entries of the Galerkin matrix]
\label{V:lem:F-bound}
Let\\ ${\varphi_k^\eps=\varphi_k^\eps[q,p,Q,P]}$ for $\eps>0$ and $k\in\N^d$ denote the Hagedorn functions with parameters $(q,p,Q,P)$, and let $W$ be a real-valued function on $\Rb^d$ with arbitrarily many polynomially bounded derivatives such that $W$ and its first two derivatives vanish at $q$. For $k,\ell\in\N^d$ we then have the bound
$$
\bigl| \langle \varphi_\ell^\eps \,|\, W \varphi_k^\eps \rangle \bigr| \le c \, \eps^{\mu/2} \quad\text{ with }\quad \mu= \max\bigl(|k-\ell|,3),
$$
where $c$ is independent of $\eps$ but depends on $k,\ell$, on the matrix 2-norm of~$Q$ and on bounds of derivatives of~$W$.
\end{lemma}

\begin{proof} We start from the Taylor expansion of $W$ at $q$ with integral remainder term:
\begin{align*}
W(x) = &\sum_{3 \le |m| \le N} \frac1{|m|!} \,\partial^m W(q) (x-q)^m 
\\
&+\ 
\int_0^1 \frac{(1-\theta)^N}{N!} \sum_{|m|=N+1} \partial^m W(q+\theta(x-q)) (x-q)^m\,{\mathrm d}\theta,
\end{align*}
where for a multi-index $m=(m_1,\dots,m_d)$, we denote the partial derivatives 
$\partial^m W = \partial_1^{m_1}\dots \partial_d^{m_d} W$ and the product $(x-q)^m=(x_1-q_1)^{m_1}\dots (x_d-q_d)^{m_d}$.

So we need to study the inner product $\langle \varphi_\ell^\eps \,|\, (x-q)^m \varphi_k^\eps \rangle$
for $k,\ell,m$. By the normalisation of the wave packets, we have
\[
|\langle \varphi_\ell^\eps \,|\, (x-q)^m \varphi_k^\eps \rangle| \le \|(x-q)^m \varphi_k\|.
\]
Next we write $\varphi_k^\eps$ in the form \eqref{phik-pol}, viz.
\bch\[
\varphi_k^\eps(x) = \frac{1}{\sqrt{2^{|k|}k!}} \, p_k\!\left(\tfrac{1}{\sqrt\eps}Q^{-1}(x-q)\right) \varphi_0(x).
\]\ech
Hence, there exists a constant $\gamma>0$ that depends on $k$ and $Q^{-1}\,\overline Q$, such that
\begin{align*}
&\| (x-q)^m \varphi_k^\eps \| \\
&\le \gamma \left( \int_{\Rb^d} |x|^{2m} \ 
\left(1 + \tfrac{1}{\eps}|Q^{-1}x|^2\right)^{k} \ \exp(-\tfrac1\eps x^T (QQ^*)^{-1} x) \D x \right)^{1/2}. 
\end{align*}
By Lemma~\ref{lem:int-bound},
\[
\langle \varphi_\ell^\eps \,|\, (x-q)^m \varphi_k^\eps\rangle = O(\eps^{|m|/2}).
\] 
Here we note further that the smallest eigenvalue of the width matrix $\Im C=(QQ^*)^{-1}$ equals the square of the largest singular value of $Q$, which is the matrix 2-norm of $Q$. 

Moreover, in view of the orthogonality of the Hagedorn functions and the three-term recurrence relation \eqref{V:rec}, the inner product 
$\langle \varphi_\ell^\eps \,|\, (x-q)^m \varphi_k^\eps \rangle$ can be non-zero only if
there exist standard unit vectors $\langle j_1 \rangle,\dots,\langle j_{|m|} \rangle$ of $\Rb^d$ such that with appropriate signs,
$$
k\pm \langle j_1 \rangle \pm \dots \pm \langle j_{|m|} \rangle = \ell,
$$
which requires
$
|k-\ell| \le |m|.
$
So we obtain
\begin{equation}\label{V:klm}
\bigl| \langle \varphi_\ell^\eps \,|\, (x-q)^m \varphi_k^\eps \rangle \bigr| = 
\begin{cases} 0 &\mbox{if } |m|<|k-\ell| \\
O(\eps^{|m|/2}) & \mbox{if } |m| \ge |k-\ell|. \end{cases}
\end{equation}
Using these estimates after inserting the Taylor expansion of $W$ at $q$ in the expression $ \langle \varphi_\ell^\eps \,|\, W \varphi_k^\eps \rangle$ gives the result. 
\end{proof}

\subsection{Approximation  of higher order in $\eps$ by the Hagedorn--Galerkin method}
Theorem~\ref{thm:error-gauss} provided an $O(\eps^{1/2})$ error bound over time $O(1)$ for variational Gaussians. For any chosen multi-index set $\calK$, a bound of the same approximation order $1/2$ in $\eps$ can also be obtained for the Hagedorn--Galerkin method of the previous subsection. However, as we will show next, an error bound of any prescribed order in $\eps$ is obtained if the initial wave packet has non-vanishing coefficients $c_k(0)$ only for $k$ in some fixed initial multi-index set $\calK_0$ and the multi-index set $\calK$ of the Galerkin method is substantially larger than $\calK_0$. We begin by studying the particular case where $\calK_0=\{ k_0 \}$ has a single element $k_0\in \N^d$. The case of a general finite initial set $\calK_0$ then follows by linear superposition.

\begin{theorem} [higher-order error bound]
\label{V:thm:hag-err}
Let $N\ge 2$ and suppose that for some $k_0\in \N^d$, the multi-index set $\calK$ of the Galerkin approximation \eqref{V:hag-series} is such that
\begin{equation}\label{V:K}
\{ k\in \N^d\,:\, |k-k_0| < 3N \} \subset \calK.
\end{equation}
Let $\psi(\cdot,t)$ denote the solution of the Schr\"odinger equation \eqref{tdse} with a smooth, polynomially bounded potential $V$, for initial data given by a single Hagedorn function with multi-index $k_0$:
$$
\psi(\cdot,0) = \varphi_{k_0}^\eps[q(0),p(0),Q(0),P(0)],
$$
where $Q(0),P(0)$ satisfy the symplecticity relations \eqref{hag-rel}. Let $\psi_\calK(\cdot,t)$ denote the Galerkin approximation \eqref{V:hag-series} with parameters $(q(t),p(t),Q(t),P(t))$  satisfying the classical equations
\eqref{V:classical-eom} and with coefficients $c_k(t)$ for $k\in\calK$ determined by \eqref{V:hag-coeff-ode}, with initial data 
$\psi_\calK(\cdot,0)=\psi(\cdot,0)$, i.e., $c_k(0)=\delta_{k,k_0}$. Then, the error is bounded by
$$
\| \psi_\calK(\cdot,t) - \psi(\cdot,t) \| \le C \eps^{N/2} \quad\text{ for }\quad 0\le t\le T,
$$
where $C$ is independent of $\eps$, but depends on $k_0$, $\calK$, on $\max_{0\le t \le T}\|Q(t)\|_2$ and $T$.
\end{theorem}

\begin{proof} The proof is in three parts: in the first part we derive bounds for the coefficients $c_k(t)$ that decay geometrically with powers of $\eps$ for $k$ moving away from $k_0$. In the second part we use this to bound the defect that results from inserting the Galerkin approximation into the Schr\"odinger equation. Finally, we conclude from a small defect to a small error.

(a) We begin by showing that for $0\le t \le T$,
\begin{equation}\label{V:coeff-decay}
c_k(t) =O( \eps^{|k-k_0|/6}),
\end{equation}
where the constant symbolised by the $O$-notation depends on $k_0$, $\calK$ and $T$ but is independent of $\eps$. 
Let us define, for some $\alpha\ge 0$,
$$
b_k(t) = c_k(t) / \eps^{\alpha |k-k_0|}.
$$
Then $b_k(0)=\delta_{k,k_0}$, and since $c(t)$ solves \eqref{V:hag-coeff-ode}, $b(t)=(b_k(t))_{k\in\calK}$ is a solution to the differential equation
\bch
$$
\dot b(t) =  F(t) b(t) \quad\text{with }\ F(t)=(f_{\ell k}(t)) , \ \ 
f_{\ell k}(t) = \frac 1{\I \eps} \eps^{-\alpha |\ell-k_0|} g_{\ell k}(t) \eps^{\alpha|k-k_0|}.
$$
We choose $\alpha$ maximal such that $f_{\ell k}(t)=O(1)$ for all $\ell,k$.
By Lemma~\ref{V:lem:F-bound}, $g_{\ell k}(t) = O(\eps^{\max(|k-\ell|,3)/2})$, 
\ech
and by the triangle inequality we have
$- |\ell-k_0| + |k-k_0| \ge - |k-\ell|$. So we require, for arbitrary $n=|k-\ell|$,
$$
-1 - \alpha n + \tfrac12\max(n,3) \ge 0.% \quad\text{for all $n\in \N$}.
$$
For $n=3$ this yields the condition $\alpha\le 1/6$, and it is then seen that the inequality is satisfied for all $n\in\N$ with $\alpha=1/6$. For this $\alpha$, we thus have $b(t)=O(1)$, which proves \eqref{V:coeff-decay}.

(b) We show that the defect obtained on inserting the approximation $\psi_\calK$ into the Schr\"odinger equation,
$$
d(t) := \frac{\partial \psi_\calK}{\partial t}(t)  -  \frac 1 {\I\eps}  H \psi_\calK(t) ,
$$
is bounded, uniformly for $0\le t \le T$, by 
\begin{equation} \label{V:defect-bound}
\| d(t) \| = O(\eps^{N/2}).
\end{equation}
By \eqref{V:defect} and \eqref{V:hag-coeff-ode}, the defect equals
$$
d(t) =   \frac 1 {\I\eps} P_\calK(t)^\perp W_{q(t)} \psi_\calK(t),
$$
where $P_\calK(t)^\perp={\mathrm{Id}} - P_\calK(t)$ and $P_\calK(t)$ is the orthogonal projection from $L^2(\Rb^d)$ onto the space spanned by $\varphi_k(t)=\varphi_k^\eps[q(t),p(t),Q(t),P(t)]$ for $k\in\calK$. 

We split $W_{q(t)}$ into its $(N+1)$st degree Taylor polynomial $W_{N+1}(\cdot,t)$ at $q(t)$ and the remainder $R_{N+1}(\cdot,t)$. 
With the integral formula for the remainder term, Lemma~\ref{lem:int-bound} shows that
\begin{equation} \label{V:RN}
\| \tfrac1\eps R_{N+1}(t)  \psi_\calK(t)\| = O(\eps^{(N+2)/2-1}) = O(\eps^{N/2}).
\end{equation}
On the other hand, $W_{N+1}(t)\psi_\calK(t)$ is a linear combination of Hagedorn functions $\varphi_\ell(t)$ with $|\ell| \le K +N+1$, where
$K=\max_{k\in\calK}|k|$. We thus have the finite sum
$$
 \frac 1 {\I\eps} P_\calK(t)^\perp W_{N+1}(t) \psi_\calK(t) =  \frac 1 {\I\eps}\, e^{\I S(t)/\eps}
 \sum_{\ell \notin \calK} \sum_{k\in\calK} c_k(t) \langle \varphi_\ell(t), W_{N+1}(t) \varphi_k(t) \rangle,
 $$
 where the first sum is over all $\ell\in\N^d$ with $|\ell| \le K +N+1$ and $\ell\notin\calK$.
 By Lemma~\ref{V:lem:F-bound}, 
 $
  \langle \varphi_\ell(t), W_{N+1}(t) \varphi_k (t)\rangle = O(\eps^{\max(|k-\ell|,3)/2}),
 $
and together with \eqref{V:coeff-decay} this gives us
$$
 \frac 1 {\I\eps} \,c_k(t) \langle \varphi_\ell(t), W_{N+1}(t) \varphi_k(t) \rangle = O(\eps^{-1+|k-k_0|/6 + \max(|k-\ell|,3)/2}).
$$
By the triangle inequality, we have $|k-k_0|\ge |\ell-k_0|-|k-\ell|$, and by condition \eqref{V:K} we have $|\ell-k_0|\ge 3N$ for $\ell\notin\calK$. Therefore,
$$
-1+\frac{|k-k_0|}6 + \frac{\max(|k-\ell|,3)}2 \ge -1 + \frac{N}2 -\frac{|k-\ell|}6+ \frac{\max(|k-\ell|,3)}2 \ge \frac N2,
$$
so that
$$
 \frac 1 {\I\eps} \,c_k(t) \langle \varphi_\ell(t), W_N(t) \varphi_k(t) \rangle = O(\eps^{N/2}),
$$
which together with \eqref{V:RN} yields  \eqref{V:defect-bound}.

(c) Finally, with Lemma~\ref{lem:stability} we conclude from a small defect to a small error: for $0\le t \le T$,
$$
\| \psi_\calK(t) - \psi(t) \| \le \int_0^t \| d(s)\|  \D s = O(\eps^{N/2}).
$$
This gives the stated result.
\end{proof}

\begin{remark}
For $N=2$, condition \eqref{V:K} can be replaced by the weaker condition
$$
\{ k\in \N^d\,:\, |k-k_0| \le 3 \} \subset \calK.
$$
We first obtain  $c_k(t)=O(\eps^{1/2})$ for $k\ne k_0$. As in (b) above, using this bound, the above condition on $\calK$ and Lemma~\ref{V:lem:F-bound} for $k=k_0$ and $|\ell-k_0|\ge 4$, we obtain  the defect bound $\|d(t)\|=O(\eps)$, which yields the $O(\eps)$ error bound. 
%In particular, if we start from a Gaussian ($k_0=0$), then including cubic polynomials times a Gaussian yields an $O(\eps)$ approximation to the exact wave function, as opposed to the $O(\eps^{1/2})$ approximation by just a Gaussian (see Theorem~\ref{thm:error-gauss}).
\end{remark}

By the above proof and by linearity, we immediately get the following more general result where the coefficients of the initial data decay sufficiently fast away from some finite multi-index set $\calK_0$.

\begin{theorem} [higher-order error bound]
\label{V:cor:hag-err}
Let $N\ge 2$ and suppose that for some finite multi-index set $\calK_0\subset \N^d$, the multi-index set $\calK$ of the Galerkin approximation \eqref{V:hag-series} is such that, with the distance $\mathrm{dist}(k,\calK_0) := \min_{k_0\in\calK_0} {|k-k_0|}$,
$$%\begin{equation}\label{V:K}
\{ k\in \N^d\,:\, \mathrm{dist}(k,\calK_0) < 3N  \} \subset \calK.
$$%\end{equation}
Let $\psi(\cdot,t)$ denote the solution of the Schr\"odinger equation \eqref{tdse} with a smooth, polynomially bounded potential $V$, with initial data given by a linear combination of Hagedorn function with multi-indices $k\in\calK$:
$$
\psi(\cdot,0) = \sum_{k\in\calK} c_k(0) \varphi_{k_0}^\eps[q(0),p(0),Q(0),P(0)],
$$
where $Q(0),P(0)$ satisfy the symplecticity relations \eqref{hag-rel} and the coefficients are bounded by
$$
|c_k(0)| \le C_0 \,\eps^{\mathrm{dist}(k,\calK_0)/6},  \qquad k\in\calK.
$$
Let $\psi_\calK(\cdot,t)$ denote the Galerkin approximation \eqref{V:hag-series} with the parameters $(q(t),p(t),Q(t),P(t))$  satisfying the classical equations
\eqref{V:classical-eom} and with coefficients $c_k(t)$ for $k\in\calK$ determined by \eqref{V:hag-coeff-ode}, with initial data 
$\psi_\calK(\cdot,0)=\psi(\cdot,0)$. Then, the error is bounded by
$$
\| \psi_\calK(\cdot,t) - \psi(\cdot,t) \| \le C \eps^{N/2} \quad\text{ for }\quad 0\le t\le T,
$$
where $C$ is independent of $\eps$, but depends on $\calK_0$, $\calK$, $C_0$, on $\max_{0\le t \le T}\|Q(t)\|_2$ and $T$.
\end{theorem}

The above proof shows that as long as the position matrix satisfies a bound $\| Q(t) \|_2 \le M$ and the positions remain in a fixed compact set, the error is bounded by
$$
\| \psi_\calK(\cdot,t) - \psi(\cdot,t) \| \le C_M \,t\,\eps^{N/2},
$$
with $C_M$ independent of $t$, 
so that potentially the approximation remains accurate over very long time scales. The growth of $\| Q(t) \|_2$ is determined by the linearised classical equations \eqref{V:classical-eom}. In the case of a positive Lyapunov exponent, the time scale cannot exceed a constant times $\log \eps^{-1}$, but in more favourable cases, such as integrable and near-integrable systems, where $\| Q(t) \|_2$ has only linear growth for all (or extremely long) times, 
the time scale of approximation can be longer; see the discussion of the Ehrenfest time in Section~\ref{subsec:ehrenfest}.

\subsection{Notes}
The construction and results in this section are due to \cn{Hag98} and \cn{Hag81}.   
Beyond the estimates of finite order in $\eps$ shown here, 
\cn{HagJ99} and \cn{HagJ00} derived exponentially small estimates in the case of an analytic potential.

\cn{Hag98} also showed that his functions behave well under the scaled Fourier transform
$
\calF_\eps \varphi(\xi) \!=\! (2\pi\eps)^{-d/2} \!\int_{\Rb^d}
\varphi(x)\, \e^{-\I\xi\cdot x/\eps}\D x
$, \bch$\xi\in\Rb^d$\ech\;: for every~$k$, 
$$%\begin{equation}\label{V:fourier-0}
\calF_\eps \varphi_k^\eps [q,p,Q,P](\xi) =
e^{-ip\cdot q/\eps} \varphi_k^\eps [p,-q,P,-Q](\xi).
$$%\end{equation}
\cn{LasT14} gave transformation properties of the Hagedorn functions under Wigner and Fourier--Bros--Iagolnitzer (FBI) transforms and derive remarkable properties of the polynomial factor of the Hagedorn functions. The generating function for these polynomials was given by \cn{Hag15} and \cite{DieKT17}.

Hagedorn's semiclassical wave packets are closely related to, or coincide with, generalised coherent states or generalised squeezed states as studied by \cn{Com92}, \cn{Rob07} and
\cn{ComR12}. The precise relationship was expounded by \cn{LasT14} and more recently by \cn{Ohs19}.

\cn{Ohs18} derived differential equations for the variational approximation by a {\it single} Hagedorn function $\varphi_k^\eps$ of arbitrary index $k$, for an approximate Hamiltonian.
A fully variational approximation by semiclassical wave packets  \eqref{V:hag-series} appears not to have been considered in the literature, not least because a separation of the motion of the position and momentum parameters $[q,p,Q,P]$ from that of the coefficients $c_k$ is not feasible for a fully variational approximation, as opposed to the semi-variational approximation considered here.

Hagedorn wave packets were proposed as a computational tool for semiclassical quantum dynamics by
\cn{FaoGL09}; see  also  \cn{GraH14}.

%\newpage
%\section{Continuous Gaussian approximations}
%\newpage
\def\calB{{\mathcal B}}
\def\calC{{\mathcal C}}
\def\calD{{\mathcal D}}
\def\calH{{\mathcal H}}
\def\calI{{\mathcal I}}
\def\calJ{{\mathcal J}}
\def\calL{{\mathcal L}}

\section{Continuous superpositions of Gaussians}
\label{sec:csg}
In this section we prove that wave functions of the semiclassical Schr\"odinger equation \eqref{tdse-sc} with general $L^2$ initial data can be approximated by continuous superpositions of Gaussian wave packets with an $O(\eps)$ error in the $L^2$-norm. We will explore 
both major types of such approximations, based on either frozen or thawed evolving Gaussians. In both cases, the approximations can be numerically realised with particle methods that use the classical equations of motion and their linearisation.

\subsection{Continuous superpositions of thawed and frozen Gaussians}

We consider semiclassical approximation for the Schr\"odinger equation \eqref{tdse-sc}
with a smooth, subquadratic potential function $V:\Rb^d\to\Rb$ and general initial data $\psi_0$. For convenience we take the initial wave function in the class of complex-valued Schwartz functions $\calS(\Rb^d)$, but we note that our error bounds extend directly to general $L^2$ initial data by density. We recall that a Schwartz function on $\Rb^d$ is an infinitely differentiable function that together with its derivatives decays faster than the inverse of any polynomial. 

Our main tool for the representation of wave functions will be the wave packet 
transform.
% and related bounded operators that be can be derived from them. 
We again let $\langle\cdot\mid\cdot\rangle$ denote the $L^2$ inner product on $\Rb^d$ and let $\|\cdot\|$ denote the $L^2$-norm.

\begin{proposition}[wave packet transform]\label{prop:wp_trafo} 
For any Schwartz function $g:\Rb^d\to\C$ of unit norm we define the corresponding wave packet\footnote{We will not indicate the obvious dependence on $\eps$ in the notation: $g_z=g_z^\eps$.}
\[
g_z(x) = \eps^{-d/4}\, g\!\left(\frac{x-q}{\sqrt\eps}\right)\, \e^{\I p\cdot (x-q)/\eps},\qquad x\in\Rb^d
\]
for $z=(q,p)\in\Rb^{2d}$.
Then, for every Schwartz function  $\psi\in \calS(\Rb^d)$, 
\begin{align*}
\psi (x) &= (2\pi\eps)^{-d} \int_{\Rb^{2d}} \langle g_z|\psi\rangle\, g_z (x) \D z,\qquad x\in\Rb^d,\\
\|\psi\|^2  &= (2\pi\eps)^{-d} \int_{\Rb^{2d}} |\langle g_z|\psi\rangle|^2 \D z.
\end{align*}
\end{proposition}

\begin{proof}
%We use the Fourier inversion formula to obtain
We use the inversion formula for the Fourier transform: a Schwartz function $f$  is reconstructed from its scaled Fourier transform
$$
 \calF_\eps f(\xi) =  (2\pi\eps)^{-d/2} \int_{\Rb^d}f(x) \,\e^{-\I\xi\cdot x/\eps}\D x
$$
by the inversion formula
$$
f(x) = (2\pi\eps)^{-d/2} \int_{\Rb^d} \calF_\eps f(\xi) \,\e^{\I\xi\cdot x/\eps}\D \xi .
$$
For a Schwartz function $\psi$ and for $x\in\Rb^d$, this yields
\begin{align*}
&(2\pi\eps)^{-d} \int_{\Rb^{2d}} \langle g_z|\psi\rangle\, g_z(x) \D z \\
&= (2\pi\eps)^{-d} \eps^{-d/2} \int_{\Rb^{3d}} \overline g(\tfrac{y-q}{\sqrt\eps}) \, g(\tfrac{x-q}{\sqrt\eps})\, 
\e^{\I p\cdot(x-y)/\eps}\,\psi(y)\D (y,q,p)\\
&= \eps^{-d/2} \int_{\Rb^d} (2\pi\eps)^{-d/2} \int_{\Rb^d} 
\calF_\eps\bigl( \overline g(\tfrac{\cdot-q}{\sqrt\eps})  \, \psi \bigr)(p)\, 
\e^{\I p\cdot x/\eps} \D p \
g(\tfrac{x-q}{\sqrt\eps})\,  \D q
\\
&= \eps^{-d/2} \int_{\Rb^{d}} |g(\tfrac{x-q}{\sqrt\eps})|^2\,\psi(x)\D q \, =\,  \psi(x),%\qquad x\in\Rb^d,
\end{align*}
since the normalisation of the function $g$ implies 
\[
\eps^{-d/2} \int_{\Rb^{d}} |g(\tfrac{x-q}{\sqrt\eps})|^2\,\D q = \int_{\Rb^d} |g(y)|^2 \D y \ =\  1.
\]
Moreover, by the inversion formula proved above, 
\begin{align*}
\|\psi\|^2 &= \big\langle \psi\mid (2\pi\eps)^{-d}\int_{\Rb^{2d}}\langle g_z|\psi\rangle g_z \D z\big\rangle \\
&= (2\pi\eps)^{-d} \int_{\Rb^{2d}} 
\langle g_z|\psi\rangle \langle \psi|g_z\rangle \D z \\
&= (2\pi\eps)^{-d} \int_{\Rb^{2d}} 
|\langle g_z|\psi\rangle|^2 \D z.
\end{align*}
\end{proof}

\bch 
We note that if $\psi$ is a Schwartz function, then also $z\mapsto\langle g_z\mid\psi\rangle$ is a Schwartz function. This property ensures the existence of many integrals appearing below.
\ech

We use the inversion formula of Proposition~\ref{prop:wp_trafo} for the standard normalised Gaussian profile
\[
g(x) = \pi^{-d/4} \exp(-\tfrac{1}{2} |x|^2),\qquad x\in\Rb^d
\]
and represent the initial data for the Schr\"odinger evolution as
\[
\psi_0 = (2\pi\eps)^{-d} \int_{\Rb^{2d}} \langle g_z|\psi_0\rangle\, g_z \D z.
\]
We will now explore two related possibilities of transforming the identity 
\[
\e^{-\I t H/\eps}\psi_0 = (2\pi\eps)^{-d} \int_{\Rb^{2d}} \langle g_z|\psi_0\rangle\, 
\e^{-\I t H/\eps} g_z \D z
\]
into semiclassical approximations. Both approaches use the same basic 
ingredients from classical mechanics.
\begin{itemize}
\item[-] The flow $\Phi^t:\Rb^{2d}\to\Rb^{2d}$ 
of the classical equations of motion
\begin{equation}\label{csg:class-eom}
\dot q = p,\qquad \dot p = -\nabla V(q),
\end{equation}
which to  $z\in\Rb^{2d}$ associates the solution $\Phi^t(z)=(q(t,z),p(t,z))$ with initial datum $z$.
\item[-]  The linearised equations of motion, linearised at $q(t)$,
\begin{equation}\label{eq:linflow}
\dot Q = P,\qquad \dot P = -\nabla^2 V(q) Q.
\end{equation}
\item[-] The action integral 
\begin{equation} \label{csg:action}
S(t,z) = \int_0^t \left(\tfrac12|p(s,z)|^2 - V(q(s,z))\right) \D s.
\end{equation}
for the trajectory $\Phi^t(z) = (q(t,z),p(t,z))$
with initial datum $z\in\Rb^{2d}$.
\end{itemize}

\subsubsection{Thawed Gaussian approximation} 
We approximate an individual time-evolved Gaussian wave packet by the solution
of the 
time-dependent  Schr\"odinger equation with locally quadratic approximation to the potential:
\[
\I\eps\partial_t\varphi_z^{\rm th} = -\tfrac{\eps^2}{2}\Delta\varphi_z^{\rm th}+ U_{q(t,z)}\varphi_z^{\rm th},
\qquad \varphi_z^{\rm th}(0) = g_z,
\]
where $U_q$ denotes the second order Taylor polynomial of $V$ expanded around 
the point $q$. This motivates us to consider the thawed Gaussian superposition
\[
\calI_{\rm th}(t)\psi_0 = (2\pi\eps)^{-d} \int_{\Rb^{2d}} \langle g_z|\psi_0\rangle\, 
\varphi_z^{\rm th}(t) \D z
\]
as an approximation to the Schr\"odinger solution $\psi(t)$. 
From Proposition~\ref{V:thm:gwp-quad} (in the version for a time-dependent quadratic potential) we know 
that $\varphi_z^{\rm th}(t)$ is a Gaussian wave packet of the form 
\[
\varphi_z^{\rm th}(t) = \e^{\I S(t,z)/\eps} \,\bigl(g[C(t,z)]\bigr)_{\Phi^t(z)}
\]
with a thawed Gaussian profile  
\[
g[C](x) =  \pi^{-d/4} \det(\Im C)^{1/4} 
\exp\!\left(\tfrac{\I}{2} x^T C x\right),
\qquad x\in\Rb^d.
\]
The width matrix $C(t,z)$ is determined by the matrix Riccati equation
\begin{equation}  \label{csg:riccati}
\begin{aligned} 
\dot C(t,z) &= - C(t,z)^2 - \nabla^2 V(q(t,z)), \\
C(0,z) &= \I\,\Id_d,
\end{aligned}
\end{equation}
or equivalently in Hagedorn's parametrisation (see Section~\ref{subsec:hag-gauss}) as 
$$
C(t,z)=P(t,z)Q(t,z)^{-1},
$$ 
where
$Q(t,z)$ and $P(t,z)$ are the solution to the linearised classical equations of motion \eqref{eq:linflow} corresponding to the linearisation at $q(t,z)$ and with initial data 
$$
Q(0,z)=\Id, \quad\ P(0,z)=\I\, \Id.
$$

We may thus rewrite the thawed approximation as
\[
\calI_{\rm th}(t)\psi_0 = (2\pi\eps)^{-d} \int_{\Rb^{2d}} \langle g_z|\psi_0\rangle\, 
\e^{\I S(t,z)/\eps} \,\bigl(g[C(t,z)]\bigr)_{\Phi^t(z)} \D z.
\]

Let us compare the classically thawed wave packet $\varphi_z^{\rm th}(t)$ with its variational cousin 
$u(t)$: 
\begin{itemize}
\item[-] Both functions provide approximations to the Schr\"odinger solution $\psi(t)$ 
of order $\sqrt\eps$ in the $L^2$-norm and of order $\eps$ for expectation values. 
(Analogous arguments to those for the proof  
of Theorem~\ref{thm:error-gauss} apply to $\varphi_z^{\rm th}(t)$ as well.)\\*[-2ex]
\item[-] Both functions conserve norm, while only the variational Gaussian~$u(t)$ is energy-conserving.\\*[-2ex]
\item[-] The classical equations of motion require point evaluations of the potential and its 
derivatives and are thus computationally less demanding than the variational equations of motion, which are 
built on averages. 
\end{itemize}

We will find that the oscillations of the different Gaussians in the continuous superposition 
$\calI_{\rm th}(t)\psi_0$ improve accuracy and we will prove an error estimate of order 
$\eps$ in the $L^2$-norm (Theorem~\ref{theo:thawed}).

\subsubsection{Frozen Gaussian approximation}

The Herman--Kluk propagator is based on a different approach for the dynamics of the individual building blocks. 
They are still Gaussian wave packets with classically moving phase space centres, but their width matrices are kept frozen. 
The low accuracy of such a frozen Gaussian 
\[
\varphi_z^{\rm fr}(t) = \e^{\I S(t,z)/\eps} g_{\Phi^t(z)}
\]
is compensated by an amazing reweighting in phase space: using the linearised classical motion we construct a smooth function $a_\natural:\Rb\times\Rb^{2d}\to\C$, known as the Herman--Kluk prefactor, such 
that 
\[
\calI_\natural(t)\psi_0 = (2\pi\eps)^{-d} \int_{\Rb^{2d}} \langle g_z|\psi_0\rangle\, 
a_\natural(t,z)\,\e^{\I S(t,z)/\eps}\, g_{\Phi^t(z)}  \D z
\]
approximates the Schr\"odinger solution $\psi(t)$ with the same asymptotic accuracy as the 
superposition of thawed Gaussians, that is, with an error of order~$\eps$ with respect to the $L^2$-norm (Theorem~\ref{theo:hk}). 

\subsection{Accuracy of the Gaussian superposition and numerical algorithm}

\subsubsection{Thawed Gaussians}
We consider an origin-centred Gaussian $g[C]_0$ with complex symmetric width matrix~$C\in\C^{d\times d}$ 
with positive definite imaginary part. The semiclassical Fourier transform of such a function is  
a Gaussian with width matrix $-C^{-1}$, that is, 
\[
\calF_\eps g[C]_0 = g[-C^{-1}]_0.
\]
A short calculation analogous to the one in Lemma~\ref{lem:hag-rel} shows that Hagedorn's 
parametrisation of $C=PQ^{-1}$ allows us to write the imaginary part of the matrix $-C^{-1}$ as  
\[
\Im\!(-C^{-1}) = (PP^*)^{-1}.
\]
We thus have for the minimal eigenvalues that 
\begin{align*}
&\lambda_{\rm min}(\Im C) = \|Q\|^{-2},\\
&\lambda_{\rm min}(\Im(-C^{-1})) = \|P\|^{-2}.
\end{align*}
Both the width of a thawed Gaussian in position and in momentum space play a crucial role 
when analysing the accuracy of the continuous thawed superposition. The following spectral 
parameter will be important for estimating the approximation error. 

\begin{definition}[spectral parameter]\label{def:spec_par} Let $C\in\C^{d\times d}$ be a complex symmetric matrix with positive definite imaginary part. 
We set $\rho_Q = \lambda_{\rm min}(\Im C)$ and $\rho_P = \lambda_{\rm min}(\Im\!(-C^{-1}))$ 
and call the positive number
\[
\rho_* = \frac{\rho_Q\rho_P}{2\rho_Q+2\rho_P} = \tfrac12 (\|Q\|^2+\|P\|^2)^{-1}
\]
the {\em spectral parameter} of the matrix $C$. 
\end{definition}

A uniform lower bound on the spectral parameter is enough to prove that the continuous 
thawed superposition is accurate of order $\eps$ for arbitrary initial data that are 
Schwartz functions. 

\begin{theorem}[thawed Gaussian approximation]\label{theo:thawed}
We now consider the thawed Gaussian approximation
\[
\calI_{\rm th}(t)\psi_0 = (2\pi\eps)^{-d} \int_{\Rb^{2d}} \langle g_z| \psi_0\rangle\,
\e^{\I S(t,z)/\eps} \,(g[C(t,z)])_{\Phi^t(z)} \D z
\]
for an arbitrary Schwartz function $\psi_0:\Rb^d\to\C$.  
We assume that the potential function $V:\Rb^d\to\Rb$ is smooth and its derivatives of order $\ge 2$ 
are all bounded. We also assume that for the solutions $C(t,z)$ of the Riccati equation \eqref{csg:riccati},
%\begin{align*}
%\dot C(t,z) &= - C(t,z)^2 - \nabla^2 V(q(t,z)),\\
%C(0,z) &= \I\,\Id_d,
%\end{align*}
there exists a positive lower bound $\rho>0$ such that the spectral parameters satisfy
\[
 \rho_*(t,z) \ge \rho \quad\ \text{for all }\ (t,z)\in\left[\,0,\bar t\,\right]\times\Rb^{2d}.
\]
Then, the solution $\psi(t)$ of the time-dependent Schr\"odinger equation with potential $V$ and initial data $\psi_0$ satisfies
\[
\|\psi(t) - \calI_{\rm th}(t)\psi_0\| \le c \,t\,\eps\,\|\psi_0\|,\qquad 0\le t\le \bar t,
\]
where the constant $c<\infty$ depends on the spectral bound $\rho$, on derivative bounds of $V$ and on $\bar t$  but is independent of $\eps$, $\psi_0$ and $t\le\bar t$. 
\end{theorem}

The proof proceeds in three steps and is given later on. We will establish the following: 

\paragraph{First step: basic norm bounds.}
We consider oscillatory integral operators of the form 
\[
(\calI\psi)(x) =   (2\pi\eps)^{-d} \int_{\Rb^{2d}} \langle g_z| \psi\rangle\, a(z)\, (x-\Phi_q(z))^m \,
(g[C(z)])_{\Phi(z)} \D z,
\]
that have the following building blocks:
\begin{enumerate}\setlength{\itemsep}{1ex}
\item[-] $a:\Rb^{2d}\to\C$ is a bounded function, 
\item[-] $\Phi=(\Phi_q,\Phi_p):\Rb^{2d}\to\Rb^{2d}$ is a volume-preserving 
map, 
\item[-] $\{C(z)\,|\, z\in\Rb^{2d}\}$ is a family of complex symmetric matrices with positive definite 
imaginary part having a positive lower bound for its spectral parameters. 
\end{enumerate}
For  a multi-index $m=(m_1,\dots,m_d)\in\N_0^d$ and for $x=(x_1,\dots,x_d)\in\Rb^d$ we write $x^m=x_1^{m_1}\dots x_d^{m_d}$ and we denote $|m|=m_1+\dots+m_d$.

We derive norm bounds that are of order $\eps^{|m|/2}$ and 
depend on the norm of $\psi$, the supremum of $|a(z)|$, and spectral bounds for the 
matrices $C(z)$. 

\paragraph{Second step: norm bounds accounting for the collective oscillation.}
We analyse the phase of the oscillatory integral operator $\calI_{\rm th}(t)$ taking into 
account the interplay between the action integral $S(t,z)$ and the Hamiltonian flow~$\Phi^t(z) =(q(t,z),p(t,z))$. By an 
elegant integration by parts, we obtain norm bounds for integral operators of the form
\[
(\calI(t)\psi)(x) =  (2\pi\eps)^{-d} \int_{\Rb^{2d}} \langle g_z| \psi\rangle\, (x-q(t,z) )^m \, 
\e^{\I S(t,z)/\eps} (g[C(t,z)])_{\Phi^t(z)} \D z,
\]
that depend on the parity of $|m|$ and are of order $\eps^{\lceil |m|/2\rceil}$, where $\lceil |m|/2\rceil$ 
denotes the smallest integer $\ge |m|/2$. 

\paragraph{Third step: defect calculation.}
We calculate the defect and apply the stability Lemma~\ref{lem:stability}. 
Together with the norm bounds we will have proved Theorem~\ref{theo:thawed}. 

%\medskip
%One can decompose the solution of the Riccati equation that governs the evolution of the thawed width matrices 
%according to
%\[
%C(t,z) = P(t,z)Q(t,z)^{-1}.
%\] 
%Then one arrives at the linearised equations of motion,
%\begin{equation}\label{eq:linflow}
%\dot Q(t,z) = P(t,z),\qquad \dot P(t,z) = -\nabla^2 V(q(t,z)) Q(t,z)
%\end{equation}
%subject to the initial conditions 
%\[
%Q(0,z) = \Id_d\quad\text{and}\quad P(0,z) = \I\Id_d.
%\]
\subsubsection{Frozen Gaussians}
The frozen Gaussian approximation relies on the same information from the classical dynamics as 
the thawed one. It uses the flow, the linearised flow, and the action integral. However, 
the linearisation now defines a reweighting factor that allows to keep the individual Gaussians of frozen 
unit width. Let us define this reweighting factor.

\begin{definition}[Herman--Kluk prefactor]\label{def:hk} 
Let $Q_\natural(t,z)$ and $P_\natural(t,z)$ denote the solution to the linearised equations of motion \eqref{eq:linflow} subject to the initial conditions 
\[
Q_\natural(0,z) = \Id\quad\text{and}\quad P_\natural(0,z) = -\I \,\Id.
\]
Define the Herman--Kluk matrix as the complex $d\times d$ matrix
\[
M_\natural(t,z) = Q_\natural(t,z) + \I P_\natural(t,z),
\]
whose invertibility will be proved in Lemma~\ref{lem:phase}. Then, 
the smooth complex-valued function 
$a_\natural:\Rb\times\Rb^{2d}\to\C$, 
\[
a_\natural(t,z) = \sqrt{2^{-d}\det(M_\natural(t,z))},
\]
is called the {\em Herman--Kluk prefactor}. The branch of the square root is determined by 
continuity with respect to time.
\end{definition}

We note that, with $D\Phi^t(z)$ denoting the Jacobian matrix of the flow map,
\begin{align*}
M_\natural(t,z) &= 
(\Id, \I \, \Id) \, D\Phi^t(z) \begin{pmatrix} \Id \\ - \I \, \Id \end{pmatrix}
\\
&=\partial_q q(t,z)-\I\partial_p q(t,z) +\partial_p p(t,z) + \I\partial_q p(t,z) .
\end{align*}

%
%Writing the Herman--Kluk matrix as 
%\[
%M_\natural(t,z) = Q_\natural(t,z) + \I P_\natural(t,z)
%\]
%with
%\[
%Q_\natural(t,z) = (\partial_q - \I\partial_p)q(t,z)\quad\text{and}\quad
%P_\natural(t,z) = (\partial_q - \I\partial_p)p(t,z),
%\]
%the two summands $Q_\natural(t,z)$ and $P_\natural(t,z)$ can be determined 
%by solving the linearised equations of motion \eqref{eq:linflow} subject to the initial conditions 
%\[
%Q_\natural(0,z) = \Id\quad\text{and}\quad P_\natural(0,z) = -\I \Id.
%\]
In contrast to the thawed evolution, which uses the linearised flow for multiplying 
its components to build the thawed width matrices, the frozen dynamics sums
the corresponding matrices and takes a determinant. Either way, 
we obtain continuous superpositions that are first order accurate with respect 
to the semiclassical parameter $\eps$. 

\begin{theorem}[Herman--Kluk propagator]\label{theo:hk}
We consider the Herman--Kluk propagator 
\[
(\calI_\natural(t)\psi_0)(x) = (2\pi\eps)^{-d} \int_{\Rb^{2d}} \langle g_z| \psi_0\rangle\,a_\natural(t,z)\,
\e^{\I S(t,z)/\eps} \,g_{\Phi^t(z)}(x) \D z, \quad\ x\in \Rb^d,
\]
for arbitrary initial data $\psi_0:\Rb^d\to\C$ that are Schwartz functions.  
We assume that the potential function $V:\Rb^d\to\Rb$ is smooth and its derivatives of order $\ge 2$ 
are all bounded. Then, the solution $\psi(t)$ of the time-dependent Schr\"odinger equation with initial data $\psi_0$ satisfies
\[
\|\psi(t) - \calI_\natural(t)\psi_0\| \le\bch c \ech  \,t\,\eps\,\|\psi_0\|,\qquad 0\le t\le \bar t,
\]
where the constant $c <\infty$ depends on derivative bounds of $V$ and on $\bar t$ but   
is independent of $\eps$, $\psi_0$ and $t\le\bar t$. 
\end{theorem}

The proof strategy for the frozen approximation is analogous to the thawed one. Let us briefly comment on 
its three main steps. 

\paragraph{First step: basic norm bounds.} We establish basic norm bounds  
for frozen integral operators that contain polynomial terms of the form $(x-\Phi_q(z))^m$. Again these 
bounds are of order $\eps^{|m|/2}$. However, their proof is considerably simpler than 
in the thawed case, since the frozen Gaussian profiles allow for a direct Fourier 
inversion argument that is unfortunately not applicable for the thawed operators.

\paragraph{Second step: norm bounds accounting for the collective oscillations. } 
The second step of the proof consists again in analysing the phase of the oscillatory 
integral operator and provides the improved norm bound that is of order  
$\lceil |m|/2\rceil$. The calculations are slimmer than in the thawed case, 
since there are no quadratic polynomials generated by derivatives of the width matrices.

\paragraph{Third step: defect calculation.} 
The third step calculates the defect of the frozen approximation and performs an 
integration by parts for determining the Herman--Kluk prefactor. In contrast to the first 
two steps of the proof, here the analysis of the frozen approximation is a bit more 
demanding than of the thawed one. 

\begin{remark}[mass and energy conservation]
Both the thawed and the frozen Gaussian approximation fail to be mass- or energy-conserving unless the potential is a polynomial of 
degree $\le 2$. 
\end{remark}

\subsubsection{Numerical algorithm}\label{alg:gauss_super}
As well as answering an interesting question about the approximation power of oscillatory Gaussian integrals, 
both Theorem~\ref{theo:thawed} and Theorem~\ref{theo:hk} also motivate simple particle methods for the approximation of the Schr\"odinger solution $\psi(t,x)$ that are accurate to first order with respect to 
the semiclassical parameter~$\eps$, provided that the initial wave packet transform 
$z\mapsto\langle g_z|\psi_0\rangle$ is accessible and the numerical quadrature used is sufficiently accurate.
The algorithm reads as follows.
\begin{enumerate}
\setlength{\itemsep}{1ex}
\item[1.] Choose a set of numerical quadrature points $z_i\in\Rb^{2d}$
and evaluate the initial transform $\langle g_z|\psi_0\rangle$ at the points $z_i$.
\item[2.] Transport the points $z_i$ by the classical flow $\Phi^t$, solve the linearised equations of motion and compute the action integrals. The linearised equations of motion are solved for the initial 
conditions 
\begin{align*}
Q(0,z_i) &= \Id_d, \ P(0,z_i) = \I\,\Id_d\qquad \text{(thawed Gaussians)}\\
Q_\natural(0,z_i) &= \Id_d,\ P_\natural(0,z_i) = -\I\,\Id_d \quad \text{(frozen Gaussians)}.
\end{align*} 
\item[3.] Extract from the linearised flow computation either the width matrices or the Herman--Kluk 
prefactor, that is, 
\begin{align*}
C(t,z_i) &= P(t,z_i) Q(t,z_i)^{-1}\hspace*{7em}\text{(thawed Gaussians)}\\
a_{\natural}(t,z_i) &= 2^{-d/2}\det\!\left(Q_\natural(t,z_i) + \I P_\natural(t,z_i)\right)^{1/2}\quad\text{(frozen Gaussians)}.
\end{align*}
\item[4.] Form the (possibly weighted) sums 
\[
(2\pi\eps)^{-d}\ \sum_{i}\, \langle g_{z_i}|\psi_0\rangle \, 
\e^{\I S(t,z_i)/\eps} \, (g[C(t,z_i)])_{\Phi^t(z_i)}(x) \, w_i
\]
or 
\[
(2\pi\eps)^{-d}\ \sum_{i}\, \langle g_{z_i}|\psi_0\rangle \, a_\natural(t,z_i)\,
\e^{\I S(t,z_i)/\eps} \, g_{\Phi^t(z_i)}(x) \, w_i
\]
according to the chosen quadrature rule.
\end{enumerate}

\subsection{Norm bounds for continuous superpositions of frozen Gaussians}

Since the first step for proving the error estimate for the frozen Gaussian approximation 
is considerably less demanding than for the thawed one, we start by 
deriving basic norm bounds for oscillatory integral operators that generalise 
the Herman--Kluk propagator. This first basic estimate requires neither  
that the profile function is Gaussian nor that the time evolution stems from a 
classical Hamiltonian flow. 

%\medskip
%For $x=(x_1,\dots,x_d)\in\Rb^d$ and a multi-index $m=(m_1,\dots,m_d)\in\N_0^d$ we write $x^m=x_1^{m_1}\dots x_d^{m_d}$ and we denote $|m|=m_1+\dots+m_d$. 

\begin{proposition}[norm bound, general profile]\label{prop:basic_bound}
We assume the following.
\begin{itemize}
\item[1.]
Let $g\in \calS(\Rb^d)$ be of unit norm.\\*[-2ex]
\item[2.] 
Let $a:\Rb^{2d}\to\C$ be a measurable and bounded function.\\*[-2ex]
\item[3.] 
Let $\Phi=(\Phi_q,\Phi_p):\Rb^{2d}\to\Rb^{2d}$ be a volume-preserving diffeomorphism. 
\end{itemize}
For $m\in \N_0^d$ we define, for $\psi\in\calS(\Rb^d)$,
\[
(\calI \psi)(x) = (2\pi\eps)^{-d} \int_{\Rb^{2d}} \langle g_z|\psi\rangle\, 
a(z)\,(x-\Phi_q(z))^m\, g_{\Phi(z)}(x) \D z.
\]
Then, $\calI\psi$ is square-integrable and satisfies
\[
\|\calI\psi\| \ \le\  c_m \,\eps^{|m|/2} \sup_{z\in\Rb^{2d}}|a(z)| \ \|\psi\|,
\]
where $c_m>0$ depends on the $m$th moment of $g$. In particular, $c_0 = 1$. 
\end{proposition}

\begin{proof} We use the inner products and the associated norms in the Hilbert spaces $L^2(\Rb^d)$ and $L^2(\Rb^{2d})$. 
We distinguish them by subscripts $x$ and $z$, respectively. \bch Note that both inner products are 
conjugate linear in the first argument.\ech
Let $\varphi\in L^2(\Rb^d)$. Then,  
\begin{align*}
&\langle \varphi\,|\, \calI\psi\rangle_x \\
&=   
\left\langle (2\pi\eps)^{-d/2} 
\left\langle (x-q)^m g_{z}|\varphi \right\rangle_x\circ\Phi \ \Big|\  
a(z)\, (2\pi\eps)^{-d/2}\langle g_z|\psi\rangle_x\right\rangle_z.
\end{align*}
Hence, by the Cauchy--Schwarz inequality, 
\begin{align*}
&\left|\langle \varphi\,|\, \calI\psi\rangle_x\right| \\*[1ex]
&\ \le \|a\|_\infty\ 
(2\pi\eps)^{-d/2} \left\|\left\langle (x-q)^m\,g_{z}|\varphi\right\rangle_x\circ\Phi \right\|_z
(2\pi\eps)^{-d/2}\left\| \langle g_z|\psi\rangle_x\right\|_z.
\end{align*}
Since the profile function $g$ has unit norm, the norm formula of Proposition~\ref{prop:wp_trafo} implies 
\[
(2\pi\eps)^{-d}\left\|\langle g_z|\psi\rangle_x\right\|_z^2 \ =\  (2\pi\eps)^{-d} \int_{\Rb^{2d}} \left|\langle g_z|\psi\rangle_x\right|^2\, \D z =  \|\psi\|_x^2.
\]
Since $\Phi$ is volume-preserving, we also have
\[
\left\|\langle (x-q)^m \, g_{z}|\varphi\rangle_x\circ\Phi\right\|_z^2
\ =\  \left\|\langle (x-q)^m \, g_z|\varphi\rangle_x\right\|_z^2.
\]
We observe that 
\bch\[
(x-q)^m g_z(x) \ =\  \eps^{|m|/2} \left(x^m g\right)_z\!(x)
\]\ech
and apply Proposition~\ref{prop:wp_trafo} to the normalised profile function $x^m g/\|x^m g\|$. 
We obtain
\[
(2\pi\eps)^{-d} \left\|\langle (x-q)^m \, g_{z}|\varphi\rangle_x\circ\Phi\right\|_z^2 
\ =\  \eps^{|m|}\, \|x^m g\|_x^2\, \|\varphi\|_x^2.
\]
In summary, setting 
\[
c_m = \|x^m g\|_x,
\] 
we have proved 
\[
\left|\langle \varphi\,|\, \calI\psi\rangle_x\right| \ \le\  c_m\, \eps^{|m|/2} \sup_{z\in\Rb^{2d}}|a(z)| \ 
\|\varphi\|_x\ \|\psi\|_x.
\]
\end{proof}

Proposition~\ref{prop:basic_bound} ensures the well-definedness of the oscillatory 
Herman--Kluk integral. With 
\begin{align*}
g(x) &= \pi^{-d/4} \exp(-\tfrac12|x|^2),\quad m=0,\\*[1ex]
a(z) &= \e^{\I S(t,z)/\eps} a_\natural(t,z), \quad\text{and}\quad 
\Phi(z) = \Phi^t(z),
\end{align*}
it implies that $\calI_\natural(t)\psi_0$ is a square-integrable function with 
\[
\|\calI_\natural(t)\psi_0\| \le \sup_{z\in\Rb^{2d}} |a_\natural(t,z)| \, \|\psi_0\|. 
\]
With a slight adjustment, the proof of Proposition~\ref{prop:basic_bound} extends to a larger 
class of integral operators, that incorporates not just powers $(x-\Phi_q(z))^m$ but functions 
of the more general form  
\[
(x-\Phi_q(z))^m \, b(\Phi_q(z),x),
\]
where $b:\Rb^{2d}\to\Rb$ is a measurable and bounded function. 
We will need this extension later on, when estimating the non-quadratic contributions of the 
potential to the approximation error.

\begin{corollary}[norm bound, general profile]\label{cor:norm_bound}
Let $b:\Rb^{2d}\to\C$ be a measurable and bounded function. Under the assumptions 
of Proposition~\ref{prop:basic_bound} define for any square-integrable $\psi:\Rb^d\to\C$
\[
(\calI \psi)(x) = (2\pi\eps)^{-d} \int_{\Rb^{2d}} \langle g_z|\psi\rangle\, 
a(z)\,(x-\Phi_q(z))^m\,b(\Phi_q(z),x)\, g_{\Phi(z)}(x) \D z.
\]
Then, $\calI\psi$ is square-integrable and satisfies
\[
\|\calI\psi\| \ \le\  c_m \ \eps^{|m|/2} \sup_{z\in\Rb^{2d}}|a(z)|\ 
\sup_{(x,z)\in\Rb^{3d}}|b(q,x)|\  \|\psi\|,
\]
where $c_m>0$ depends on the moments of the profile function~$g$. In particular, $c_0 = 1$.
\end{corollary}

\begin{proof}
We literally repeat the proof of Proposition~\ref{prop:basic_bound}, adding  
the following observation. Denote $f(x) = x^m g(x)$. We have \bch by Fourier inversion, that\ech
\begin{align*}
&\left\|\langle (x-q)^m \,b(q,x)\, g_z|\varphi\rangle_x\right\|_z^2\\*[1ex]
&= \eps^{|m|-d/2} \int_{\Rb^{4d}} \overline b(q,y) b(q,x)
\overline f(\tfrac{y-q}{\sqrt\eps}) f(\tfrac{x-q}{\sqrt\eps})\overline\varphi(y) \varphi(x) 
\e^{\I p\cdot(x-y)/\eps}\D (x,y,z)\\*[1ex]
&= (2\pi\eps)^{d}\, \eps^{|m|-d/2} \int_{\Rb^{2d}} |b(q,x)|^2 \,
|f(\tfrac{x-q}{\sqrt\eps})|^2\, |\varphi(x)|^2 \D (x,q)\\*[1ex]
&\le (2\pi\eps)^{d}\, \eps^{|m|}\, \sup_{(q,x)\in\Rb^{2d}}|b(q,x)|^2 \ \|f\|^2 \ \|\varphi\|^2. 
\end{align*}
\end{proof}

So far we have not used the fact that the profile function of the Herman--Kluk propagator is Gaussian and we have not used the properties of the Hamiltonian flow, nor have we incorporated the phase contributions 
from the action integral. However, a more detailed analysis will allow us to recognise 
that the contributions from monomials $(x-q(t,z))^m$ can be smaller than 
expected, depending on the parity of $|m|$. 
For this analysis, the second step for the proof of Theorem~\ref{theo:hk}, we open the inner product integral in $\calI_\natural(t)\psi_0$ and write
\[
\langle g_z| \psi_0\rangle\,\e^{\I S(t,z)/\eps} \,g_{\Phi^t(z)}(x)
= (\pi\eps)^{-d/2} \int_{\Rb^{d}} \e^{\I \Psi(t,x,y,z)/\eps}\, \psi_0(y) \D y,
\]
using a complex-valued phase function $\Psi(t,x,y,z)$, that is quadratic 
with respect to $y-q$ and $x-q(t,z)$. This function has 
remarkable properties.  

\begin{lemma}[phase function, frozen Gaussians]\label{lem:phase}
The phase function 
\begin{align*}
&\Psi: \Rb\times\Rb^d\times\Rb^d\times\Rb^{2d}\to\C,\\
&\Psi(t,x,y,z) = \tfrac{\I}{2}\left(|y-q|^2+|x-q(t,z)|^2\right) \\
&\qquad\qquad\qquad- p\cdot(y-q) + p(t,z)\cdot(x-q(t,z)) + S(t,z)
\end{align*}
satisfies for all $(t,x,y,z)$
\[
(\I\partial_q + \partial_p)\Psi(t,x,y,z) = M(t,z)^T (x-q(t,z)),
\]
where $M(t,z)$ denotes the complex $d\times d$ matrix 
\[
M(t,z) = \partial_q q(t,z) -\I\partial_p q(t,z) + \partial_p p(t,z) + \I\partial_q p(t,z).
\]
The matrix $M(t,z)$ has the following properties: 
\begin{enumerate}
\item[1.] $M(t,z)$ is invertible with 
\[
|\det M(t,z)| = \sqrt{\det\!\left(\Id_{2d} + D\Phi^t(z)^T D\Phi^t(z)\right)}.
\]
\item[2.] Its time derivative satisfies
\begin{align*}
\partial_t M(t,z) 
&= \partial_q p(t,z) -\I\partial_p p(t,z) \\*[1ex]
&\quad - \nabla^2 V(q(t,z))(\partial_p q(t,z)+\I\partial_q q(t,z)) .
\end{align*}
\item[3.] If $V$ is a polynomial of degree $\le 2$, then $M(t)$ is independent of $z$. 
\end{enumerate}
\end{lemma}

\begin{proof}
We start by differentiating the action integral. By the classical equations of motion, we have 
\begin{align*}
\partial_q S(t,z) 
&= \int_0^t \left(\partial_q p(s,z)^T p(s,z) - \partial_q q(s,z)^T\nabla V(q(s,z))\right) \D s\\
&= \int_0^t \left(\partial_q \dot q(s,z)^T p(s,z) + \partial_q q(s,z)^T\dot p(s,z)\right) \D s\\
&= \int_0^t \tfrac{d}{\D s}\left(\partial_q q(s,z)^T p(s,z)\right) \D s
\ =\ \partial_q q(t,z)^T p(t,z)-p,
\end{align*}
\bch since $\partial_q q(0,z) = \Id_d$.\ech
Analogously,
\[
\partial_p S(t,z) \ =\  \int_0^t \tfrac{d}{\D s}\left(\partial_p q(s,z)^T p(s,z)\right) \D s
\ =\ \partial_p q(t,z)^T p(t,z).
\]
Now we compute the derivatives of the phase function, 
\begin{align*}
&\partial_q\Psi(t,x,y,z) = -\I(y-q) -\I\,\partial_q q(t,z)^T(x-q(t,z)) + p \\*[1ex]
&\qquad + \partial_q p(t,z)^T(x-q(t,z))
- \partial_q q(t,z)^Tp(t,z) + \left(\partial_q q(t,z)^T p(t,z)-p\right)\\*[1ex]
&\qquad= -\I(y-q) +\left(\partial_q p(t,z)-\I\,\partial_q q(t,z)\right)^T(x-q(t,z))
\end{align*}
and
\begin{align*}
&\partial_p\Psi(t,x,y,z) = -\I\,\partial_p q(t,z)^T(x-q(t,z)) -(y-q) \\*[1ex]
&\qquad + \partial_p p(t,z)^T(x-q(t,z))
- \partial_p q(t,z)^Tp(t,z) + \partial_p q(t,z)^Tp(t,z)\\*[1ex]
&\qquad= \left(\partial_p p(t,z)-\I\,\partial_p q(t,z)\right)^T(x-q(t,z)) -(y-q).
\end{align*}
This implies
\[
(\I\partial_q + \partial_p)\Psi(t,x,y,z) = M(t,z)^T(x-q(t,z)).
\] 
To prove invertibility, we work with the Jacobian matrix of the flow, 
\[
D\Phi^t(z) = \begin{pmatrix} \partial_q q(t,z) & \partial_p q(t,z)\\*[1ex] 
\partial_q p(t,z) & \partial_p p(t,z)\end{pmatrix}.
\]
and the matrix \bch$F = F(t,z)$ defined by\ech
\begin{align*}
F &:= \Id_{2d} + D\Phi + \I J(\Id_{2d}-D\Phi) \\*[1ex]
&= 
\begin{pmatrix}\Id_d + \partial_q(q+\I p) & \partial_p(q+\I p) -\I\Id_d\\
\partial_q(p-\I q) + \I\Id_d & \Id_d + \partial_p(p-\I q)\end{pmatrix}.
\end{align*}
Since the two left blocks of $F$ commute, its determinant satisfies
\begin{align*}
\det(F) &= \det(F_{11}F_{22}-F_{21}F_{12})\\*[1ex]
&=\det(\partial_q(q+\I p) + \partial_p(p-\I q) + \I \partial_q(p-\I q) -\I\partial_p(q+\I p))\\*[1ex]
& = \det(2M) = 2^d \det(M).
\end{align*}
We calculate 
\begin{align*}
F^*F &= (\Id_{2d} + D\Phi)^T(\Id_{2d} + D\Phi) + (\Id_{2d}-D\Phi)^T(\Id_{2d}-D\Phi)\\*[1ex]
&= 2\Id_{2d} + 2D\Phi^T D\Phi, 
\end{align*}
where we have used that the cross-terms involving $J$ vanish, 
\begin{align*}
&\I(\Id_{2d} + D\Phi)^T J (\Id_{2d}-D\Phi) - \I (\Id_{2d}-D\Phi)^TJ^T(\Id_{2d} + D\Phi)\\*[1ex]
&= \I (-J D\Phi+ D\Phi^TJ) + \I(JD\Phi-D\Phi^T J) = 0, 
\end{align*}
due to the symplecticity of $D\Phi^t(z)$. 
Hence, 
\begin{align*}
|\det M|^2 &= 4^{-d} \det(F^* F)\\
 &= \det(\Id_{2d} + D\Phi^T D\Phi).
\end{align*}
For the time derivative we simply calculate
\[
\partial_t M(t,z) = (\partial_q -\I\partial_p) p(t,z) - \nabla^2 V(q(t,z)) (\partial_p +\I\partial_q) q(t,z).
\]
If the potential $V$ is a polynomial of degree $\le 2$, 
then the Jacobian $D\Phi^t$ and thus $M(t)$ are independent of $z$. 
\end{proof}

The key observation of Lemma~\ref{lem:phase} is that the Wirtinger derivative 
\[
(\I\partial_q+\partial_p)\Psi(t,x,y,z)
\] 
is independent of the variable $y$ and linear in $x-q(t,z)$. This allows for an integration by parts that turns monomial terms $(x-q(t,z))^m$ into powers of the form $\eps^{\lceil |m|/2\rceil}$, where $\lceil |m|/2\rceil$ 
denotes the smallest integer $\ge |m|/2$. This explains the high accuracy of appropriate continuous superpositions of frozen Gaussians, which is of order $\eps$ instead of $\sqrt\eps$. 

\begin{proposition}[norm bound, frozen Gaussians]\label{prop:improved_bound} 
Let the functions $a,b:\Rb^{2d}\to\C$ be smooth and bounded. 
Let
\[
g(x) = \pi^{-d/4}\exp(-\tfrac12|x|^2), \quad x\in\Rb^d,
\] 
be the standard normalised Gaussian and let $m\in\N_0^d$. 
For any square-integrable function $\psi:\Rb^d\to\C$ define
\begin{align*}
&\calI(t)\psi \\
&= (2\pi\eps)^{-d} \int_{\Rb^{2d}} \langle g_z|\psi\rangle \,a(t,z)\, (x-q(t,z))^m \,b(q(t,z),x) \,
\e^{\I S(t,z)/\eps} g_{\Phi^t(z)} \D z
\end{align*}
Then, $\calI(t)\psi$ is square-integrable and satisfies
\[
\|\calI(t)\psi\| \le\  \gamma_{m}\, c_{|m|}(a,b,\Phi^t)\  \eps^{\lceil |m|/2\rceil} \|\psi\|
\]
where the constant $\gamma_{m}>0$ depends on the polynomial degree $m$ and  
\begin{align*}
& c_{|m|}(a,b,\Phi^t) =\\
& \sup_{|\alpha|\le |m|,(x,z)\in\Rb^{3d}} \|\partial^\alpha_z (a(z) b(q,x) M(t,z)^{-T})\| \ 
\sup_{|\alpha|\le |m|, z\in\Rb^{2d}} |\partial^\alpha_z\Phi^t(z)|.
\end{align*}
\end{proposition}

\begin{proof}
We perform an inductive proof over $|m|$. The case $|m| = 0$ is already covered by the norm bound of 
Corollary~\ref{cor:norm_bound}. We therefore start with $|m|=1$, where $m=e_j$ for some $j=1,\ldots,d$.  
We use the derivative formula of Lemma~\ref{lem:phase}, 
\[
(x-q(t,z))\, \e^{\I\Psi(t,x,y,z)/\eps} = 
\tfrac{\eps}{\I} \,M(t,z)^{-T}(\I\partial_q + \partial_p) \,\e^{\I\Psi(t,x,y,z)/\eps}.
\]
Integration by parts yields
\begin{align*}
&\calI(t)\psi(x) \\
&= \frac{1}{2^{d} (\pi\eps)^{3d/2}} \int_{\Rb^{3d}} a(t,z) (x-q(t,z))_j\, b(q(t,z),x)\,  
\e^{\I\Psi(t,x,y,z)/\eps} \psi(y) \D(y,z)\\
&= \frac{\eps\I}{2^{d} (\pi\eps)^{3d/2}} \sum_{k=1}^d  \int_{\Rb^{3d}} 
(\I\partial_q+\partial_p)_k \left(a(t,z)\,M(t,z)^{-T}_{jk}\, b(q(t,z),x)\right) \\*[1ex]
&\hspace*{14em}\times\e^{\I\Psi(t,x,y,z)/\eps} \psi(y) \D(y,z).
\end{align*}
Using Corollary~\ref{cor:norm_bound}, we obtain some constant $\gamma_{e_j}>0$ such that
\[
\|\calI(t)\psi\| \ \le\  \gamma_{e_j}\, c_1(a,b,\Phi^t)\,\eps\, \|\psi\|. 
\]
For the inductive step, we consider 
\begin{align*}
&\calI(t)\psi(x) =\\
&\frac{1}{2^{d} (\pi\eps)^{3d/2}} \int_{\Rb^{3d}} a(t,z) (x-q(t,z))^{m+e_j}\, b(q(t,z),x)\,  
\e^{\I\Psi(t,x,y,z)/\eps} \psi(y) \D(y,z)\\
&=\frac{\eps\I}{2^{d} (\pi\eps)^{3d/2}} \sum_{k=1}^d \int_{\Rb^{3d}} a_{jk}(t,x,z)
\e^{\I\Psi(t,x,y,z)/\eps} \psi(y) \D(y,z)
\end{align*} 
with
\begin{align*}
a_{jk}(t,x,z) 
&= (\I\partial_q+\partial_p)_k\, \left( a(t,z)\,M(t,z)^{-T}_{jk}\, b(q(t,z),x)\right)(x-q(t,z))^m\\*[1ex]
&\quad+ a(t,z)\,M(t,z)^{-T}_{jk}\, b(q(t,z),x)\ (\I\partial_q+\partial_p)_k (x-q(t,z))^m. 
\end{align*}
We therefore write
\[
\calI(t) = \eps\left(\calI_0(t) + \calI_1(t)\right),
\]
where the integral operator $\calI_0(t)$ contains monomials in $x-q(t,z)$ of order $|m|$, while 
those in $\calI_1(t)$ are of order $|m|-1$. By the inductive hypothesis, there exist some constants 
$\widetilde\gamma_0,\widetilde\gamma_1>0$ such that
\[
\|\calI_0(t)\psi\| \ \le\   \widetilde \gamma_0\, c_{|m|+1}(a,b,\Phi^t) \,\eps^{\lceil |m|/2\rceil} \|\psi\|
\]
and
\[
\|\calI_1(t)\psi\| \ \le\  \widetilde \gamma_1\, c_{|m|}(a,b,\Phi^t) \,\eps^{\lceil (|m|-1)/2\rceil} \|\psi\|.
\]
We have
\[
\lceil(|m|-1)/2\rceil+1 = \left\{ 
\begin{array}{ll}
|m|/2+1 = \lceil (|m|+1)/2\rceil, & \text{for $m$ even,}\\*[1ex]
(|m|+1)/2 = \lceil (|m|+1)/2\rceil, & \text{for $m$ odd,}
\end{array}
\right. 
\]
and therefore
\[
\|\calI(t)\psi\| \ \le\  \gamma_{m+e_j}\,c_{|m|}(a,b,\Phi^t) \,\eps^{\lceil (|m|+1)/2\rceil} \, \|\psi\|
\]
for some constant $\gamma_{m+e_j}>0$.
\end{proof}

\subsection{The Herman--Kluk propagator}\label{sec:hk_theorem}
We next calculate the defect of a frozen Gaussian superposition that is built on an arbitrary, 
not yet determined weight factor $a(t,z)$.

\begin{lemma}[defect calculation]\label{lem:res}
For a smooth and bounded function $a:\Rb\times\Rb^{2d}\to\C$ and square-integrable $\psi_0:\Rb^d\to\C$ we consider
\[
\calI(t)\psi_0 = (2\pi\eps)^{-d} \int_{\Rb^{2d}} \langle g_z| \psi_0\rangle\,a(t,z)\,
\e^{\I S(t,z)/\eps} \,g_{\Phi^t(z)}(x) \D z.
\]
Then, the defect 
\[
d(t) = \left(\tfrac{1}{\I\eps} H-\partial_t\right)\calI(t)\psi_0
\]
satisfies
\[%\label{eq:res}
d(t) = (2\pi\eps)^{-d} \int_{\Rb^{2d}} \langle g_z| \psi_0\rangle\,\delta(t)\,
\e^{\I S(t,z)/\eps} \,g_{\Phi^t(z)} \D z,
\]
where 
\begin{align*}
&\delta(t,x,z) = -\partial_t a(t,z) + \tfrac{1}{\I\eps} a(t,z)\left(-\tfrac12|x-q(t,z)|^2 +\tfrac{\eps d}{2}\right) \\*[1ex]
&+ \tfrac{1}{\I\eps} a(t,z)\left( \tfrac12 (x-q(t,z))^T\, \nabla^2 V(q(t,z)) (x-q(t,z)) + W_{q(t,z)}(x)\right),
\end{align*}
and 
\[
W_q = V - U_q
\] 
denotes the non-quadratic remainder of the potential $V$ expanded around a point $q\in\Rb^d$.
\end{lemma}

\begin{proof}
We start by calculating the two time derivatives, namely
\[
\partial_t \e^{\I S(t)/\eps} = -\tfrac{1}{\I\eps}\left( \tfrac12|p(t)|^2-V(q(t))\right) \e^{\I S(t)/\eps}
\]
and 
\begin{align*}
\partial_t g_{\Phi^t} &= 
\partial_t\left(-\tfrac{1}{2\eps}|x-q(t)|^2 + \tfrac\I\eps p(t)^T (x-q(t)) \right) g_{\Phi^t}\\
&= \tfrac{1}{\I\eps}\left( (\I p(t)-\dot p(t))^T(x-q(t)) +|p(t)|^2\right) g_{\Phi^t}.
\end{align*}
This implies that
\begin{align*}
&\partial_t\left( \e^{\I S(t)/\eps}g_{\Phi^t}\right) \\
&= 
\tfrac{1}{\I\eps}\left( \tfrac12|p(t)|^2 + V(q(t)) + (\I p(t)-\dot p(t))^T(x-q(t)) \right) \e^{\I S(t)/\eps}g_{\Phi^t}
\end{align*}
and consequently
\begin{align*}
&\partial_t \,\calI(t)\psi_0 \\
&= (2\pi\eps)^{-d} \int_{\Rb^{2d}} \langle g_z| \psi_0\rangle\, 
\tfrac{1}{\I\eps}\left( \I\eps\partial_t a(t) + \left( \tfrac12|p(t)|^2 + V(q(t)) \right) a(t) \right)
\e^{\I S(t)/\eps} g_{\Phi^t} \D z\\
&\quad + (2\pi\eps)^{-d} \int_{\Rb^{2d}} \langle g_z| \psi_0\rangle\, \tfrac{1}{\I\eps}\, a(t) \,
(\I p(t)-\dot p(t))^T(x-q(t)) \,\e^{\I S(t)/\eps}g_{\Phi^t} \D z.
\end{align*}
We also calculate the space derivatives,
\[
\partial_j g_z(x) = \left( -\tfrac1\eps(x-q)_j + \tfrac\I\eps p_j\right) g_z(x),\qquad j=1,\ldots,d,
\]
and
\[
-\tfrac{\eps^2}{2} \Delta_x g_z(x) = \left( -\tfrac12|x-q|^2 + \I p^T(x-q) + \tfrac12|p|^2 +\tfrac{\eps d}{2}\right)g_z(x).
\]
Adding the time and space derivatives together with a Taylor expansion of the potential, 
we obtain the claimed form of the defect.
\end{proof}

In view of the frozen norm bounds determined in Proposition~\ref{prop:improved_bound}, the calculated form of the defect is encouraging. 
The leading-order terms are 
\begin{itemize}
\item[-] 
the time derivative of the reweighting function $a(t,z)$, 
\item[-] the contributions 
that are quadratic in $x-q(t,z)$, one stemming from the Laplacian, the other from the quadratic 
expansion of the potential. 
\end{itemize}
The non-quadratic remainder of the potential will contribute to the overall error, 
and we have a first indication of an integral representation 
of the dynamics that is exact for quadratic potentials. 

\medskip
Next, we explicitly calculate the terms generated by the integration by parts 
used for the improved norm bound in Proposition~\ref{prop:improved_bound}.  
We will obtain an explicit formula for the reweighting function, which solely depends 
on the classical flow map and provides an error of order $\eps$.  

\begin{proof}(of Theorem~\ref{theo:hk})\quad 
We let $A(t,z)$ denote the real symmetric $d\times d$ matrix
\[
A(t,z) = -\Id+\nabla^2 V(q(t,z))
\] 
and write the defect determined in Lemma~\ref{lem:res} as
\[
\delta(t,x,z) = -\partial_t a(t,z) + \delta_2(t,x,z) + 
\tfrac{1}{\I\eps}a(t,z) \left( \tfrac{\eps d}{2} + W_{q(t,z)}(x)\right), 
\]
where the quadratic part is given by 
\[
\delta_2(t,x,z) = \tfrac{1}{2\I\eps}\,  a(t,z)\,(x-q(t,z))^T\,A(t,z)  (x-q(t,z))
\]
and the non-quadratic remainder by
\[
W_q(x) = \tfrac12 \sum_{|m|=3} (x-q)^m \int_0^1 (1-\theta)^2\ \partial^m V((1-\theta)q+\theta x) \D\theta.
\]
We perform an integration by parts for the quadratic contributions. By the derivative formula of 
Lemma~\ref{lem:phase}, 
\[
(x-q(t,z))\, \e^{\I\Psi(t,x,y,z)/\eps} = 
\tfrac{\eps}{\I} \,M(t,z)^{-T}(\I\partial_q + \partial_p) \,\e^{\I\Psi(t,x,y,z)/\eps},
\]
we have
\begin{align*}
&\int_{\Rb^{2d}} \delta_2(t,x,z) \e^{\I\Psi(t,x,y,z)/\eps} \D z \\*[1ex]
&= 
-\int_{\Rb^{2d}}  a(t,z)\,\tfrac12(x-q(t,z))^T\,A(t,z) M(t,z)^{-T} \left(\I\partial_q + \partial_p\right)\e^{\I\Psi(t,x,y,z)/\eps}\D z\\*[1ex]
&=
\int_{\Rb^{2d}} \widetilde \delta_2(t,x,z)\,\e^{\I\Psi(t,x,y,z)/\eps} \D z
\end{align*}
with
\begin{align*}
&\widetilde \delta_2(t,x,z) \\
&\ = \tfrac{1}{2} \sum_{k,\ell=1}^d 
\left(\I\partial_q + \partial_p\right)_\ell \left( a(t,z)\,(x-q(t,z))_k \,(A(t,z) M(t,z)^{-T})_{k,\ell}\right).
\end{align*}
We calculate the derivative
\[
\left(\I\partial_q + \partial_p\right)_\ell (x-q(t,z))_k = -\left(\I\partial_q q(t,z)^T + \partial_p q(t,z)^T\right)_{k,\ell}, 
\]
and obtain
\begin{align*}
&\tfrac{1}{2}\sum_{k,\ell=1}^d (A(t,z) M(t,z)^{-T})_{k,\ell}
\left(\I\partial_q + \partial_p\right)_\ell (x-q(t,z))_k \\
&= -\tfrac{1}{2}\ \tr\left(\left(\I\partial_q q(t,z)^T + \partial_p q(t,z)^T\right) A(t,z)\, M(t,z)^{-T}\right).
\end{align*}
By the phase function Lemma~\ref{lem:phase}, 
\begin{align*}
&-\left(\I\partial_q q^T + \partial_p q^T\right) A - \I M^T \\*[1ex]
&=  \left(\I\partial_q q^T + \partial_p q^T\right) \left(\Id-\nabla^2 V\right) - 
\left( \I\partial_qq^T + \I\partial_p p^T - \partial_q p^T + \partial_p q^T\right)\\*[1ex]
&= \partial_t M^T.
\end{align*}
We thus obtain
\begin{align*}
&\widetilde\delta_2(t,x,z) + \tfrac{d}{2\I} \ =\  
\tfrac{1}{2}\,\tr\left(\partial_t M(t,z)^T\ M(t,z)^{-T}\right) \bch a(t,z)\ech\\
&\quad+ 
\tfrac{1}{2} \sum_{k,\ell=1}^d (x-q(t,z))_k \left(\I\partial_q + \partial_p\right)_\ell \left( a(t,z)\,(A(t,z) M(t,z)^{-T})_{k,\ell}\right).
\end{align*}
This suggests choosing $a(t,z)$ as the solution of the ordinary differential equation
\[
\partial_t a(t,z) = \tfrac12\,\tr\left(\partial_t M(t,z)^T\ M(t,z)^{-T}\right) a(t,z),\qquad a(0,z) = 1.
\]
By Liouville's formula,
\[
\partial_t \det(M(t,z)) = \det(M(t,z))\ \tr(\partial_t M(t,z) M(t,z)^{-1}). 
\]
Therefore, 
\[
a(t,z) \ =\  \sqrt{2^{-d}\det(M(t,z))} \ =\  a_\natural(t,z).
\]
The corresponding defect 
\[
d_\natural(t) = (2\pi\eps)^{-d} \int_{\Rb^{2d}} \langle g_z| \psi_0\rangle\,\delta_\natural(t)\,
\e^{\I S(t,z)/\eps} \,g_{\Phi^t(z)}\D z,
\] 
is determined by the function
\begin{align*}
\delta_\natural(t,x,z) &=  \tfrac{1}{2} \sum_{k,\ell=1}^d (x-q(t,z))_k \left(\I\partial_q + \partial_p\right)_\ell \left( a_\natural(t,z)\,(A(t,z) M(t,z)^{-T})_{k,\ell}\right) \\
&\qquad + \tfrac{1}{\I\eps}a_\natural(t,z) W_{q(t,z)}(x). 
\end{align*}
The function $\delta_\natural$ consists of two parts. The first one is linear in $x-q(t,z)$. The frozen norm bound 
of Proposition~\ref{prop:improved_bound} results in an upper bound for the corresponding oscillatory integral of 
order $\eps^{\lceil 1/2\rceil} = \eps$. 
The second part involving the non-quadratic remainder of the potential is cubic in $x-q(t,z)$ but divided by $\eps$, resulting in an upper bound that is of order 
\[
\eps^{\lceil 3/2\rceil-1} = \eps^{2-1} = \eps
\] 
as well. Altogether we obtain
\[
\|d_\natural(t)\| \ \le\  C\, \eps\, \|\psi_0\|\qquad 0\le t\le \bar t,
\]
where the constant $C>0$ is independent of $\eps$ and $\psi_0$. 
By the stability Lemma~\ref{lem:stability}, the error then satisfies
\[
\|\psi(t) - \calI_\natural(t)\psi_0\| \ \le\  \int_0^t \|d_\natural(s)\| \D s \ \le\  C\, t\, \eps\,\|\psi_0\|.
\]
\end{proof}

As for the variational Gaussian approximation and the Hagedorn wave packets, we observe 
exactness for quadratic potentials. 

\begin{corollary}[exactness for quadratic potentials] If the potential function $V$ is quadratic, then the Herman--Kluk propagator is exact, that is, $\psi(t) = \calI_\natural(t)\psi_0$ for all square-integrable initial data $\psi_0$ and all times $t$.
\end{corollary}

\begin{proof} Quadratic potentials have a constant Hessian $\nabla^2 V$. Therefore, the Jacobian of the flow, 
the matrices $A(t)$ and $M(t)$, and the Herman--Kluk prefactor $a_\natural(t)$ do not depend 
on $z$. Consequently, $d_\natural(t) = 0$.
\end{proof}

\subsection{First norm bounds for continuous thawed Gaussian superpositions} 

The basic norm bounds we have used to analyse the frozen Gaussian superposition, that is, Proposition~\ref{prop:basic_bound} and Corollary~\ref{cor:norm_bound}, do not directly apply for the thawed Gaussian superposition 
\[
\calI_{\rm th}(t)\psi_0 = (2\pi\eps)^{-d} \int_{\Rb^{2d}} \langle g_z|\psi_0\rangle\, 
\e^{\I S(t,z)/\eps} g(t)_{\Phi^t(z)} \D z,
\] 
due to the $z$-dependence of the Gaussian width matrix $C(t,z)$. It is natural to 
assume that the imaginary parts of the time-evolved family $C(t,z)$ are bounded 
from below in the sense that
\[
\Im C(t,z) \ge \Im C_0 \ \text{for all}\ (t,z)\in\Rb\times \Rb^{2d},
\]
where $C_0$ is a complex symmetric $d\times d$ matrix with $\Im C_0>0$. 
Then, it is tempting to write 
\[
g[C(z)]_{\Phi^t(z)}(x) = a_0(t,z) \,b_0(t,x,z)\, g[C_0]_{\Phi^t(z)}(x)
\] 
with a determinantal prefactor  
\[
a_0(t,z) =  \left(\frac{\det\Im C(t,z)}{\det\Im C_0}\right)^{1/4} 
\]
and an integrable Gaussian function 
\[
b_0(t,x,z) = \exp(\tfrac{\I}{2\eps}(x-q(t,z))^T (C(t,z)-C_0)(x-q(t,z))).
\]
The proof of Proposition~\ref{prop:basic_bound} easily accommodates the 
fixed width Gaussian wave packet $g[C_0]_{\Phi^t(z)}(x)$ and the renormalisation function $a_0(t,z)$. However,  
the elegant Fourier inversion argument of Corollary~\ref{cor:norm_bound} is blocked, since 
the function $b_0(t,x,z)$ depends on $x$ and $z$ simultaneously, while the $z$-dependence is 
not only caused by $q(t,z)$ but also by the thawed width matrix $C(t,z)$. Hence, 
we need to resort to alternative, more general techniques for controlling 
oscillatory integral operators. For this,  we reformulate the previously employed wave packet inversion 
and norm formulas from a more abstract point of view. 

\begin{proposition}[wave packet transform]\label{prop:bargmann} 
Let $g:\Rb^d\to\C$ denote a square-integrable function of unit norm. The adjoint operator of the isometry
\[
\calB: L^2(\Rb^d)\to L^2(\Rb^{2d}), \quad (\calB\psi)(z) = (2\pi\eps)^{-d/2} \langle g_z|\psi\rangle,
\]
is given by
\begin{align*}
&\calB^*: L^2(\Rb^{2d})\to L^2(\Rb^d), \\
&(\calB^*\Psi)(x) = (2\pi\eps)^{-d/2} \int_{\Rb^{2d}} \Psi(z) g_z(x) \D z.
\end{align*}
It satisfies
\[
\calB^*\calB\psi = \psi\quad\text{and}\quad\|\calB\psi\| = \|\psi\|
\]
for all $\psi\in L^2(\Rb^d)$.
\end{proposition}

\begin{proof}
For all \bch$\Psi\in L^2(\Rb^{2d})$\ech and $\psi\in L^2(\Rb^d)$, we have 
\bch
\begin{align*}
\langle \Psi,\calB\psi\rangle 
&= (2\pi\eps)^{-d/2} \int_{\Rb^{3d}} 
\overline{\Psi(z) g_z(x)} \psi(x) \D(x,z)\\
&= \Big\langle (2\pi\eps)^{-d/2}\int_{\Rb^{2d}} \Psi(z) \,g_z \D z\mid \psi\Big\rangle,
\end{align*}\ech
which proves the claimed form of the adjoint operator. Moreover, 
\[
\calB^*\calB\psi = (2\pi\eps)^{-d} \int_{\Rb^{2d}} \langle g_z|\psi\rangle g_z \D z = \psi.
\]
and 
\[
\|\calB\psi\|^2 \ =\ \langle \calB\psi|\calB\psi\rangle \ =\ 
\langle \psi| \calB^*\calB\psi\rangle \ =\ \|\psi\|^{2}.
\]
\end{proof}

Using the wave packet transform $\calB$, we lift the thawed operator 
$\calI_{\rm th}(t)$ to phase space and there prove its boundedness, unhindered by the $z$-dependent Gaussian width matrices.  By the isometric inversion property of the wave 
packet transform proved above, we have for all $\psi\in L^2(\Rb^d)$, 
\[
\|\calI_{\rm th}(t)\psi\| \ =\ \|\underbrace{\calB\, \calI_{\rm th}(t)\, \calB^*}_{=:\calI_{\calB}(t)} \calB\psi\| 
\ \le\  \|\calI_\calB(t)\|\ \|\psi\|,
\]
such that a bound on the operator norm of $\calI_{\calB}(t)$ will provide a bound for the 
thawed Gaussian superposition. The proof of the next result, 
Proposition~\ref{prop:bound_thawed}, follows this line of 
argumentation:

\begin{proposition}[thawed norm bound]\label{prop:bound_thawed}
We assume the following:
\begin{itemize}
\item[1.]
Let $\{C(z)|\,z\in\Rb^{2d}\}$ be a family of complex symmetric $d\times d$ matrices with positive 
definite imaginary part. We denote the corresponding normalised Gaussian profile functions by 
\[
g[C(z)](x) = (\pi\eps)^{-d/4} (\det\Im C(z))^{1/4} \exp(\tfrac\I2 x^T C(z) x)
\]
for $x\in\Rb^d$. We assume the existence of $\rho>0$ such that the spectral parameters of 
the matrices satisfy
%\begin{equation}\label{eq:rho_star}
\[
\rho_*(z) \ge \rho\quad \text{for all}\ z\in\Rb^{2d}.
\]
%\end{equation}
\item[2.] 
Let $a:\Rb^{2d}\to\C$ be a measurable and bounded function.\\*[-2ex]
\item[3.] 
Let $\Phi:\Rb^{2d}\to\Rb^{2d}$ be a volume-preserving diffeomorphism. 
\end{itemize}
For any square-integrable function $\psi:\Rb^d\to\C$ define
\[
(\calI \psi)(x) = (2\pi\eps)^{-d} \int_{\Rb^{2d}} \langle g[\I\Id]_z|\psi\rangle\, 
a(z)\,g[C(z)]_{\Phi(z)}(x) \D z.
\]
Then, $\calI\psi$ is square-integrable and satisfies
\[
\|\calI\psi\| \ \le\ (2/\rho)^d \sup_{z\in\Rb^{2d}}|a(z)| \ \|\psi\|.
\]
\end{proposition}

\begin{remark}
Before entering the proof, we note that the above estimate might be too pessimistic. 
For the case of frozen Gaussians of unit width, that is, $C(z) = \I\,\Id$ for all $z$, we  
have $\rho_*(z) = 1/4$, whereas the frozen norm 
bound in Proposition~\ref{prop:basic_bound} provides the dimension-independent estimate 
\[
\|\calI\psi\|\ \le \ \sup_{z\in\Rb^{2d}}|a(z)|\,\|\psi\|
\]
for all square-integrable functions $\psi:\Rb^d\to\C$. 
\end{remark}

\begin{proof} We first calculate the integral kernel of the wave packet transformed operator 
$\calI_{\calB} := \calB\, \calI\, \calB^*$. 
We use the abbreviation $g = g[\I\,\Id]$ for the standard Gaussian.  
Let $\Psi\in L^2(\Rb^{2d})$ and $X\in\Rb^{2d}$. Then, 
\begin{align*}
(\calI_{\calB}\Psi)(X) &= (2\pi\eps)^{-d/2}\ \calB\calI \bch\int_{\Rb^{2d}} \Psi(Y) g_Y(\cdot) \D Y\ech\\
&= (2\pi\eps)^{-3d/2}\ \calB \int_{\Rb^{4d}}\Psi(Y) \,\langle g_z\,|\, g_Y\rangle\, a(z)\, 
\bch g[C(z)]_{\Phi(z)}(\cdot)\ech \D(Y,z)\\
&= (2\pi\eps)^{-2d} \int_{\Rb^{4d}}\Psi(Y) \,\langle g_z\,|\, g_Y\rangle\, a(z)\, 
\langle g_X\,|\,g[C(z)]_{\Phi(z)}\rangle \D(Y,z).
\end{align*} 
Hence, the integral kernel of $\calI_{\calB}$ is given by
\[
k_{\cal B}(X,Y) = (2\pi\eps)^{-2d} \int_{\Rb^{2d}}\langle g_z\,|\, g_Y\rangle\, a(z)\, 
\langle g_X\,|\,g[C(z)]_{\Phi(z)}\rangle \D z.
\]
We aim at bounding
\begin{align*}
C_1 &:= \sup_{X\in\Rb^{2d}} \int_{\Rb^{2d}} |k_{\calB}(X,Y)| \D Y ,\\*[1ex]
C_2 &:= \sup_{Y\in\Rb^{2d}} \int_{\Rb^{2d}} |k_{\calB}(X,Y)| \D X,
\end{align*}
since then, by the Schur test, see \eg\, \cite[Theorem~0.3.1]{Sog17}, 
\[
\|\calI_{\calB}\| \le \sqrt{C_1 C_2}. 
\]
We now use two estimates on inner products of Gaussian wave packets. First, 
an easy calculation using the Fourier transform of a standardized Gaussian function yields 
\[
|\langle g_z\,|\, g_Y\rangle| \ \le\  \exp(-\tfrac{1}{4\eps}|Y-z|^2). 
\]
Second, a fairly involved technical calculation also provides  
\[
|\left\langle g_{X} \mid g[C(z)]_{\Phi(z)} \right\rangle | \ \le\   
\exp(-\tfrac{\rho}{2\eps}\, |X-\Phi(z)|^2),
\]
where $\rho>0$ is a lower bound for the spectral parameter of the matrix family 
$\{C(z)\mid z\in\Rb^{2d}\}$. 
These two exponential estimates yield
\begin{align*}
|k_{\cal B}(X,Y)| 
&\le \|a\|_\infty\, (2\pi\eps)^{-2d} \int_{\Rb^{2d}} |\langle g_z\,|\, g_Y\rangle|\ 
|\langle g_X\,|\,g[C(z)]_{\Phi(z)}\rangle| \D z\\
&\le \|a\|_\infty\, (2\pi\eps)^{-2d} \int_{\Rb^{2d}} 
\exp(-\tfrac{1}{4\eps}|Y-z|^2 -\tfrac{\rho}{2\eps}|X-\Phi(z)|^2)
\D z.
\end{align*}
Then, for all $X\in\Rb^{2d}$,
\begin{align*}
&\int_{\Rb^{2d}} |k_{\cal B}(X,Y)| \D Y \\
&\le \|a\|_\infty\, (2\pi\eps)^{-2d} \int_{\Rb^{4d}} 
\exp(-\tfrac{1}{4\eps}|Y-z|^2 -\tfrac{\rho}{2\eps}|X-\Phi(z)|^2)
\D (Y,z)\\*[1ex]
&= \|a\|_\infty\, (2\pi\eps)^{-2d}\, (4\pi\eps)^{d}\, (2\pi\eps/\rho)^d
\ =\  \|a\|_\infty\,   (2/\rho)^d,   
\end{align*}
where we have used that $\Phi$ is volume-preserving. 
Similarly, for all $Y\in\Rb^{2d}$ we obtain 
\[
\int_{\Rb^{2d}} |k_{\cal B}(X,Y)| \D X \ \le\  \|a\|_\infty\,   (2/\rho)^d.   
\]
In summary,
\[
\|\calI\| \ =\  \|\calI_\calB\| \ \le\  \|a\|_\infty \, (2/\rho)^d.
\]
\end{proof}

Proposition~\ref{prop:bound_thawed} ensures the well-definedness of the 
oscillatory integral defining the continuous thawed superposition $\calI_{\rm th}(t)\psi_0$. 
Indeed, with the matrix $C(z)$ being the solution of the Riccati equation
\[
\dot C(t,z) = -C(t,z)^2 - \nabla^2 V(q(t,z)),\qquad C(0,z) = \I\Id, 
\]
and $\rho_*(t)>0$ the corresponding spectral parameter, with 
\begin{align*}
a(z) &= \e^{\I S(t,z)/\eps}\quad\text{and}\quad \Phi(z) = \Phi^t(z),
\end{align*}
Proposition~\ref{prop:bound_thawed} proves that $\calI_{\rm th}(t)\psi_0$ is a square-integrable function with 
\[
\|\calI_{\rm th}(t)\psi_0\| \ \le\  (2/\rho(t))^d \, \|\psi_0\|. 
\]
Analysing the approximation error of the thawed Gaussian approximation, we will encounter 
the non-quadratic remainder of the potential $V$, when expanded around a classical trajectory. 
We therefore have to extend the previous norm bounds such that they also cover terms of 
the form   
\[
(x-\Phi_q(z))^m \, b(\Phi_q(z),x),
\]
where $b:\Rb^{2d}\to\Rb$ is a smooth and bounded function. 
For this, we follow the same strategy of proof as before, however, working with the 
cross-Wigner function of two Gaussian wave packets instead of its inner product. 

\begin{proposition}[thawed norm bound]\label{prop:bound_thawed_ext}
We consider a smooth function $b:\Rb^{2d}\to\C$, that is bounded together 
with all its derivatives. Under the assumptions 
of Proposition~\ref{prop:bound_thawed}, we define for any square-integrable 
function $\psi:\Rb^d\to\C$
\begin{align*}
&(\calI \psi)(x) = \\
&(2\pi\eps)^{-d} \int_{\Rb^{2d}} \langle g[\I\Id]_z|\psi\rangle\, 
a(z)\,(x-\Phi_q(z))^m\,b(\Phi_q(x),x)\, g[C(z)]_{\Phi(z)}(x) \D z.
\end{align*}
Then, $\calI\psi$ is square-integrable, and 
\[
\|\calI\psi\| \ \le\ \gamma_m(\rho) \,c_{|m|}(a,b)\, \eps^{|m|/2}\, \|\psi\|,
\]
where the constant $0<\gamma_m(\rho)<\infty$ depends on the polynomial degree $m$ and the 
spectral bound $\rho$ and 
\begin{align*}
c_{|m|}(a,b)  &= \sup\{|a(z)|\mid z\in\Rb^{2d}\}\\*[1ex]  
&\qquad\times\sup\{|\partial^\alpha_y b(x,y)|\mid (x,y)\in\Rb^{2d}, |\alpha|\le |m|+2d+2\}.
\end{align*}
\end{proposition}

\begin{proof}
As before in Proposition~\ref{prop:bound_thawed}, we lift the integral kernel of the operator $\calI$ to phase space via the wave packet 
transform and work with 
\begin{align*}
&k_{\cal B}(X,Y) \\
&= (2\pi\eps)^{-2d} \int_{\Rb^{2d}}\langle g_z\,|\, g_Y\rangle\, a(z)\, 
\langle g_X\,|(x-\Phi_q(z))^m\, b(\Phi_q(z),x)\,g[C(z)]_{\Phi(z)}\rangle \D z
\end{align*}
for all $X,Y\in\Rb^{2d}$. Again using the Schur test for the operator norm, 
we aim to estimate
\[
\sup_{X\in\Rb^{2d}} \int_{\Rb^{2d}}  |k_{\cal B}(X,Y)| \D Y\quad\text{and}\quad
\sup_{Y\in\Rb^{2d}} \int_{\Rb^{2d}}  |k_{\cal B}(X,Y)| \D X. 
\]
We introduce the shorthand notation
\[
b_m(w,z) = (w_q-\Phi_q(z))^m\, b(\Phi_q(z),w_q),\quad w,z\in\Rb^{2d},
\] 
for the new contribution and write the crucial  inner product in terms of the cross-Wigner function as
\begin{align*}
\bch I_*(X,z)\ech &\bch:=\ech \langle g_X\mid (x-\Phi_q(z))^m\, b(\Phi_q(z),x)\, g[C(z)]_{\Phi(z)}\rangle\\*[1ex]
&= \int_{\Rb^{2d}} b_m(w,z)\ \calW(g_X,g[C(z)]_{\Phi(z)})(w) \D w,
\end{align*}
where
\begin{align*}
& \calW(g_X,g[C(z)]_{\Phi(z)})(w) \\
&= 
(2\pi\eps)^{-d} \int_{\Rb^d} \overline{g_X(w_q+\tfrac12 y)} \ g[C(z)]_{\Phi(z)}(w_q-\tfrac12 y)\  
\e^{\I w_p\cdot y/\eps} \D y. 
\end{align*}
The cross Wigner function of Gaussian wave packets can explicitly be determined as
\begin{align*}
&\calW(g_{X},g[C(z)]_{\Phi(z)})(w) = \gamma_*(X,z)\, (\pi\eps)^{-d}\, \e^{\I (X-\Phi(z))\cdot Jw/\eps}\\
&\qquad\times
\exp(\tfrac{\I}{2\eps}\,(w-\tfrac12(X+\Phi(z)))^T\calC(z)(w-\tfrac12(X+\Phi(z))))\bch.\ech
\end{align*}
See for example \cite[Proposition~244]{Gos11}, where the calculation is carried out for the 
special case $X = \Phi(z) = 0$. The scalar prefactor $\gamma_*(X,z)\in\C$ depends on 
$C(z),\Phi(z),X$, while $\calC(z)$ is a complex symmetric $2d\times 2d$ matrix with positive 
definite imaginary part, that depends solely on the original matrix $C(z)$. It requires some linear algebra to 
show that
\[
|\gamma_*(X,z)|\le 1,
\]
and that the imaginary part of $\calC(z)$ satisfies the lower bound
\[
\Im(\calC(z)) \ge 4\rho\Id .
\]
Moreover, 
\begin{equation}\label{eq:calC}
\tr(\Im\calC(z)) = d\quad\text{and}\quad \|\Re\calC(z)\| \le \sqrt{3d}
\end{equation}
for all $z\in\Rb^{2d}$. Therefore, we may write 
\begin{align*}
& I_*(X,z) = \\
&\quad\frac{\gamma_*(X,z)}{(\pi\eps)^{d}} \int_{\Rb^{2d}} b_m(w,z) 
\e^{\I (X-\Phi(z))\cdot Jw/\eps + \I\, \calC(z)(w-(X+\Phi(z))/2)^2/(2\eps)} \D w.
\end{align*}
The change of variables 
\[
w = \tfrac12(X+\Phi(z)) + \sqrt\eps W
\] 
combined with the identity 
\[
(X-\Phi(z))\cdot J(X+\Phi(z)) = 2 X\cdot J\Phi(z)
\]
yields 
\begin{align*}
&I_*(X,z) =\\*[1ex]
&\quad\frac{\gamma_*(X,z)}{\pi^{d}}\ \e^{\I X\cdot J\Phi(z)/\eps} \int_{\Rb^{2d}} 
\widetilde b_m(W,X,z) \e^{\I (X-\Phi(z))\cdot JW/\sqrt\eps}\e^{\I W\cdot \calC(z) W/2} \D W
\end{align*}
with
\begin{align*}
& \widetilde b_m(W,X,z) = b_m(\sqrt\eps\, W+\tfrac12(X+\Phi(z)),z) \\*[1ex]
&=\  \eps^{|m|/2} \left(W_q+\tfrac{1}{2\sqrt\eps}(X-\Phi(z))_q\right)^m\, 
b(\Phi_q(z),\sqrt\eps\, W_q+\tfrac12(X+\Phi(z))_q).
\end{align*}
We now view $I_*(X,z)$ as the windowed Fourier transform of the smooth function 
$W\mapsto \widetilde b_m(W,X,z)$ evaluated in the point 
\[
\eta = J^T(X-\Phi(z))/\sqrt\eps,
\] 
where the window is the complex Gaussian function
$W\mapsto \exp(\I W\cdot\calC(z) W/2)$. Then, repeated integration by parts provides 
spectral decay of $I_*(X,z)$ with respect to $\eta$. That is, for all $n\in\N$, 
\begin{align*}
& |I_*(X,z)| \le\\
&  \quad 
\gamma_{m,n}(\rho) \,\eps^{|m|/2} \left( 1 + \left|\tfrac{X-\Phi(z)}{\sqrt\eps}\right|^2\right)^{-n/2} 
\sup_{|\alpha|\le |m|+n} \|\partial^\alpha_y b(\Phi_q(z),y)\|_\infty, 
\end{align*}
where the constant $\gamma_{m,n}(\rho)>0$ depends on the lower spectral bound $\rho$ for the imaginary part $\Im\calC(z)$.
For more general Gaussian windows, the constant for controlling 
spectral decay also depends on the trace of the imaginary part and the 
spectral norm of the real part of the matrix $\calC(z)$, which in our case are negligible 
due to the special spectral properties given in \eqref{eq:calC}. 
We now choose $n=2d+2$ and obtain for all $X\in\Rb^{2d}$,
\begin{align*}
&\int_{\Rb^{2d}} |k_{\cal B}(X,Y)| \D Y \ \le\  
\frac{\gamma_{m}(\rho)\,c_{|m|}(a,b)\, \eps^{|m|/2}}{(2\pi\eps)^{2d}} \\
&\qquad\times
\int_{\Rb^{4d}} \exp(-\tfrac{1}{4\eps}|Y-z|^2) 
\left( 1 + \left|\tfrac{X-\Phi(z)}{\sqrt\eps}\right|^2\right)^{-(d+1)}\D(Y,z)\\*[1ex]
& \qquad = \gamma_m(\rho)/d! \ c_{|m|}(a,b)\ \eps^{|m|/2} .
\end{align*}
Similarly we obtain for all $Y\in\Rb^{2d}$,
\begin{align*}
&\int_{\Rb^{2d}} |k_{\cal B}(X,Y)| \D X \ \le\   
\frac{\gamma_m(\rho)\,c_{|m|}(a,b)\,\eps^{|m|/2}}{(2\pi\eps)^{2d}} \\
&\qquad\times \int_{\Rb^{4d}} \exp(-\tfrac{1}{4\eps}|Y-z|^2) 
\left( 1 + \left|\tfrac{X-\Phi(z)}{\sqrt\eps}\right|^2\right)^{-(d+1)}\D(X,z)\\*[1ex]
& \qquad = \gamma_m(\rho)/d! \ c_{|m|}(a,b)\ \eps^{|m|/2}.
\end{align*}
\end{proof}

\subsection{Analysis of the thawed phase function}

We next analyse the dynamical properties of the thawed Gaussian approximation, following 
the same strategy developed for the Herman--Kluk propagator. We open the inner product 
integral in $\calI_{\rm th}(t)\psi_0$ and write
\[
\langle g_z| \psi_0\rangle\,\e^{\I S(t,z)/\eps} \,g(t)_{\Phi^t(z)}(x)
= (\pi\eps)^{-d/2} \int_{\Rb^{d}} \e^{\I \Psi_{\rm th}(t,x,y,z)/\eps}\, \psi_0(y) \D y,
\]
using a complex-valued phase function $\Psi_{\rm th}(t,x,y,z)$, that is quadratic 
with respect to $y-q$ and $x-q(t,z)$. We will now explicitly describe that the Wirtinger derivative of 
the thawed phase function is quadratic with respect to $x-q(t,z)$ and does not depend on $y$. 
The invertibility of the resulting matrix $M_{\rm th}(t,x)$ allows for the crucial 
integration by parts that reveals the optimal approximation power of the thawed Gaussian approximation. 

\begin{lemma}[thawed phase function]\label{lem:phase_thawed}
The phase function 
\begin{align*}
&\Psi_{\rm th}: \Rb\times\Rb^d\times\Rb^d\times\Rb^{2d}\to\C,\\
&\Psi_{\rm th}(t,x,y,z) = \tfrac{1}{2}\left(\I |y-q|^2+(x-q(t,z))^T C(t,z)(x-q(t,z))\right) \\
&\qquad\qquad\qquad- p\cdot(y-q) + p(t,z)\cdot(x-q(t,z)) + S(t,z)
\end{align*}
satisfies for all $(t,x,y,z)$
\begin{align*}
&(\I\partial_q + \partial_p)\Psi_{\rm th}(t,x,y,z) = M_{\rm th}(t,z)(x-q(t,z)) \\
&\qquad+ \left(\tfrac{1}{2}(x-q(t,z))^T \left((\I\partial_{q_j}+\partial_{p_j})C(t,z)\right) (x-q(t,z))\right)_{j=1}^d,
\end{align*}
where $M_{\rm th}(t,z)$ denotes the complex $d\times d$ matrix
\begin{align*}
& M_{\rm th}(t,z) \\
&= -\I\partial_q q(t,z)^T C(t,z) +\partial_p p(t,z)^T + \I\partial_q p(t,z)^T  -\partial_p q(t,z)^T C(t,z),
\end{align*}
that is invertible and satisfies
\[
|\det M_{\rm th}| \ \ge\  2^{d/2} \det(\Im C)^{1/2}.
\]
\end{lemma}

As before for the frozen Gaussian approximation,
we use the Wirtinger derivative of the phase function for an integration by parts 
that proceeds analogously up to additional book-keeping for  
the quadratic terms in $x-q(t,z)$ that occur, because the width matrix $C(t,z)$ 
depends on the phase space variable $z$. 

\begin{proposition}[thawed bound, revisited]\label{prop:improved_bound_thawed} 
Let $b:\Rb^{2d}\to\C$ be a smooth and bounded function. 
Let
\[
g(t,x) = \pi^{-d/4} \det(\Im C(t,z))^{1/4}  \exp(\tfrac\I2 x^T C(t,z)x), \quad x\in\Rb^d,
\] 
be the origin centred thawed Gaussian with width matrix $C(t,z)$, $z\in\Rb^{2d}$. 
We assume the existence of $\rho>0$ such that the spectral parameters of 
these matrices satisfy the lower bound
%\begin{equation}\label{eq:rho_star}
\[
\rho_*(t,z) \ge \rho\quad\text{for all}\ (t,z)\in\Rb\times\Rb^{2d}.
\]
%\end{equation}
Let $m\in\N_0^d$. 
For any square-integrable function $\psi:\Rb^d\to\C$ define
\begin{align*}
&\calI_{\rm th}(t)\psi \\
&= (2\pi\eps)^{-d} \int_{\Rb^{2d}} \langle g_z|\psi\rangle \,(x-q(t,z))^m \,b(q(t,z),x) \,
\e^{\I S(t,z)/\eps} g(t)_{\Phi^t(z)} \D z
\end{align*}
Then, $\calI_{\rm th}(t)\psi$ is square-integrable and satisfies
\[
\|\calI(t)\psi\| \le\  \gamma_{m}(b,\rho)\, d_{|m|}(M_{\rm th}^{-1},\Phi^t,C(t))\  
\eps^{\lceil |m|/2\rceil} \|\psi\|,
\]
where 
\begin{align*}
& d_{|m|}(M_{\rm th}^{-1},\Phi^t,C(t)) =\\
& \sup_{|\alpha|\le |m|,z\in\Rb^{2d}} 
\left( \|\partial^\alpha_z M(t,z)^{-1}\| \ |\partial^\alpha_z\Phi^t(z)|\ \|\partial^\alpha_z C(t)\|\right).
\end{align*}
The constant $\gamma_{m}(b,\rho)>0$ depends on the polynomial degree $m$, the 
spectral bound $\rho$ and higher-order derivative bounds for the function $b$. 
\end{proposition}

\subsection{The error of the continuous thawed Gaussian approximation}

The improved thawed norm bound of Proposition~\ref{prop:improved_bound_thawed} 
provides the aimed for error estimate for the thawed Gaussian superposition, 
once we have calculated its defect. 

\begin{lemma}[defect calculation]\label{lem:res_thawed}
For an arbitrary square-integrable initial datum $\psi_0:\Rb^d\to\C$ we consider
\[
\calI_{\rm th}(t)\psi_0 = (2\pi\eps)^{-d} \int_{\Rb^{2d}} \langle g_z| \psi_0\rangle\,
\e^{\I S(t,z)/\eps} \,g(t)_{\Phi^t(z)}(x) \D z.
\]
Then, the defect 
\[
d_{\rm th}(t) = \left(\tfrac{1}{\I\eps} H-\partial_t\right)\calI_{\rm th}(t)\psi_0
\]
satisfies
\[
d_{\rm th}(t) = (2\pi\eps)^{-d} \int_{\Rb^{2d}} \langle g_z| \psi_0\rangle\,\tfrac{1}{\I\eps} W_{q(t)}\,
\e^{\I S(t,z)/\eps} \,g(t)_{\Phi^t(z)} \D z,
\]
where 
$W_q = V - U_q$ 
denotes the non-quadratic remainder of the potential~$V$ expanded around a point $q\in\Rb^d$.
\end{lemma}

Having established the defect as being essentially a cubic polynomial divided by $\eps$, 
the stability lemma and the improved norm bound allow us to estimate the thawed 
approximation as being of order $\eps$, 
as stated in Theorem~\ref{theo:thawed} whose proof we can now conclude.

\begin{proof}(of Theorem~\ref{theo:thawed})
By Lemma~\ref{lem:res_thawed}, the defect satisfies
\[
d_{\rm th}(t) = (2\pi\eps)^{-d} \int_{\Rb^{2d}} \langle g_z| \psi_0\rangle\,\tfrac{1}{\I\eps} W_{q(t)}\,
\e^{\I S(t,z)/\eps} \,g_{\Phi^t(z)}\D z. 
\] 
Since the remainder potential is cubic in $x-q(t,z)$ but divided by $\eps$, 
the improved norm estimate of Proposition~\ref{prop:improved_bound_thawed} provides 
an upper bound for the defect that is of order 
\[
\eps^{\lceil 3/2\rceil-1} = \eps^{2-1} = \eps,
\] 
that is, 
\[
\|d_{\rm th}(t)\| \ \le\  C\, \eps\, \|\psi_0\|\qquad 0\le t\le \bar t,
\]
where the constant $C>0$ is independent of $\eps$ and $\psi_0$. 
By the stability Lemma~\ref{lem:stability}, the error then satisfies
\[
\|\psi(t) - \calI_{\rm th}(t)\psi_0\| \ \le\  \int_0^t \|d_{\rm th}(s)\| \D s \ \le\  C\, t\, \eps\,\|\psi_0\|.
\]
\end{proof}

As for the variational Gaussian approximation, the Hagedorn wave packets and the Herman--Kluk propagator, 
we observe exactness for quadratic potentials. 

\begin{corollary}[exactness for quadratic potentials] If the potential function $V$ is quadratic, then the thawed Gaussian approximation is exact, that is, $\psi(t) = \calI_{\rm th}(t)\psi_0$ for all square-integrable initial data 
$\psi_0$ and all times~$t$.
\end{corollary}

\subsection{Postponed proofs for the thawed Gaussian superposition}

Here we collect the proofs of the technical results that were left 
out in our previous discussion of the thawed Gaussian superposition. 
They refer to the following topics:
\begin{itemize}
\setlength{\itemsep}{0.25ex}
\item[-] inner products of Gaussian wave packets with different width matrix,
\item[-] calculation of the cross Wigner function for Gaussian wave packets,
\item[-] decay properties of a Gaussian windowed Fourier transform,  
\item[-] basic analysis of the thawed phase function $\Psi_{\rm th}(t,x,y,z)$,
\item[-] integration by parts for the improved norm bound of Proposition~\ref{prop:improved_bound_thawed}, 
\item[-] calculation of the defect for the continuous thawed superposition.
\end{itemize}

\subsubsection{Inner products of Gaussian wave packets for proving Proposition~\ref{prop:bound_thawed_ext}}

We now calclulate the inner product of two Gaussian wave packets with different width matrices and estimate the magnitude of this phase space function. We used this bound when proving Proposition~\ref{prop:bound_thawed_ext}.

\begin{lemma}[inner product of Gaussians]\label{lem:gauss} 
Let $C\in\C^{d\times d}$ be a complex symmetric matrix such that $\Im C>0$. 
Denote by $\rho_*>0$ the spectral parameter of $C$. 
Then, there exist
\begin{enumerate}
\item[1.] a complex symmetric matrix $\calC\in\C^{2d\times 2d}$ with 
$\Im\calC \ge \rho_*\Id$,
\item[2.] a complex number $\gamma\in\C$ with $|\gamma|\le 1$,
\end{enumerate}
such that 
\[
\left\langle g_{z_1} \mid g[C]_{z_2} \right\rangle = 
\gamma\ \e^{\I p_2\cdot (q_1-q_2)/\eps}\, \exp(\tfrac{\I}{2\eps} (z_1-z_2)^T \calC (z_1-z_2)).
\]
for all phase space centres $z_1,z_2\in\Rb^{2d}$.
\end{lemma}

\begin{remark}
For the inner product of two Gaussians with unit width, that is, for $C = \I \Id$, we have 
$\gamma=1$ and $\calC = \I/2\, \Id$. In particular, the only eigenvalue of 
$\Im\calC$ is $1/2$, while the spectral parameter $\rho_*=1/4$ offers an overly 
pessimistic lower spectral bound. 
\end{remark}

\begin{proof} To avoid writing differences of phase space centres later on, we start by observing that for 
arbitrary square-integrable functions $f,h:\Rb^d\to\C$
\begin{align*}
&\langle f_{z_1}\,|\,h_{z_2}\rangle \\
&= \eps^{-d/2} \int_{\Rb^d} 
\overline{f(\tfrac{x-q_1}{\sqrt\eps})}\ \, h(\tfrac{x-q_2}{\sqrt\eps}) \,
\e^{\I (-p_1\cdot(x-q_1) + p_2\cdot(x-q_2))/\eps} \D x\\*[1ex]
&= \eps^{-d/2} \int_{\Rb^d} 
\overline{\bch f(\tfrac{y-q_1+q_2}{\sqrt\eps})\ech} \, h(\tfrac{y}{\sqrt\eps}) \,
\e^{\I (-p_1\cdot\bch(y-q_1+q_2)\ech + p_2\cdot y)/\eps} \D y\\*[1ex]
&= \e^{\I p_2\cdot (q_1-q_2)/\eps} \ \langle f_{z_1-z_2}\,|\,h_0\rangle
\end{align*}
for all $z_1,z_2\in\Rb^{2d}$, since 
\begin{align*}
&-p_1\cdot(y-q_1-q_2) + p_2\cdot y \\
&\qquad = -(p_1-p_2)\cdot (y-q_1-q_2) + p_2\cdot (q_1-q_2).
\end{align*}
Hence, it is enough to analyse $\langle g_z\,|\,g[C]_0\rangle$ for arbitrary $z\in\Rb^{2d}$.
We have
\begin{align*}
&\langle g_z\,|\,g[C]_0\rangle \\
&=  (\pi\eps)^{-d/2}\, \det(\Im C)^{1/4} \int_{\Rb^d} 
\e^{-\tfrac{1}{2\eps}|x-q|^2 -\tfrac{\I}{\eps}p\cdot(x-q) +\tfrac{\I}{2\eps}\, x\cdot C x} \D x.
\end{align*}
By \cite[Theorem~1 in Appendix~A]{Fol89}, we have
\[
\int_{\Rb^d} \exp(-\pi x^T A x - 2\pi\I v^T x) \D x = (\det A)^{-1/2} \exp(-\pi v^T A^{-1}v),
\]
where $A$ is a complex symmetric $d\times d$ matrix with positive definite real part and $v$ a complex vector.
Using this formula with
\[
A = \tfrac{1}{2\pi\eps} (\Id-\I C),\qquad v = \tfrac{1}{2\pi\eps}(p+\I q),
\]
we obtain
\[
\langle g_z\,|\,g[C]_0\rangle \ =\  \gamma \exp(\tfrac{\I}{2\eps} z^T \calC z),
\]
where the prefactor is given by 
\[
\gamma = \frac{2^{d/4}\det \Im C^{1/4}}{\det(\Id-\I C)^{1/2}},
\]
and the quadratic form 
\[
\tfrac{\I}{2\eps} z^T\calC z \ =\   
-\tfrac{1}{2\eps}|q|^2 + \tfrac{\I}{\eps} p^T q - \tfrac{1}{2\eps} (p+\I q)^T (\Id-\I C)^{-1}(p+\I q)
\]
is induced by the complex symmetric $2d\times 2d$ matrix
\[
\calC = \begin{pmatrix}\I\,\Id - \I(\Id-\I C)^{-1} & \Id - (\Id-\I C)^{-1}\\
\Id - (\Id-\I C)^{-1} & \I(\Id-\I C)^{-1}
\end{pmatrix}.
\]
The upper bound on $\gamma$ and the positive definiteness of $\Im\calC$ are proved in Lemma~\ref{lem:det_pd}.
\end{proof}

The following lemma contains the linear algebra for bounding the determinantal prefactor 
and proving the positive definiteness of the imaginary part of the width matrix of 
the phase space Gaussian. 

\begin{lemma}[linear algebra estimates]\label{lem:det_pd} Let $C$ be a complex 
symmetric $d\times d$ matrix with positive definite imaginary part and spectral parameter $\rho_*>0$. 
Then, $\Id-\I C$ is invertible, and we have
\[
\frac{2^{2d}\det \Im C}{|\det(\Id-\I C)^2|} \le 1.
\] 
Moreover, the complex symmetric $2d\times 2d$ matrix
\[
\calC = \begin{pmatrix}\I\,\Id - \I(\Id-\I C)^{-1} & \Id - (\Id-\I C)^{-1}\\
\Id - (\Id-\I C)^{-1} & \I(\Id-\I C)^{-1}
\end{pmatrix}
\]
has an imaginary and a real part satisfying 
\[
\tr(\Im\calC) = d,\qquad\Im\calC \ge \rho_*\Id_{2d},
\qquad\|\Re\calC\|_2 \ \le\  \|\calC\|_2 \ \le\  \sqrt{3d}.
\] 
\end{lemma}

\begin{proof}
The determinantal fraction can be estimated by the Ostrowski--Taussky inequality \cite[Theorem~7.8.19]{HorJ13}. 
Indeed, the real part of the matrix $\Id-\I C$ equals $\Id+\Im C$ and is thus positive definite. Therefore,  
\[
\det(\Id+\Im C) = \det\Re\!(\Id-\I C) \le |\det(\Id -\I C)|,
\] 
which implies
\[
\frac{2^{2d}\det \Im C}{|\det(\Id-\I C)^2|} 
\ \le\  \frac{\det(4\,\Im C)}{\det(\Id+\Im C)^2} \ =\  
\prod_{\bch\lambda\in\sigma\!({\rm Im} C)\ech} \frac{4\lambda}{(1+\lambda)^2} \ \le\  1.
\]
We next prove that the imaginary part of $\calC$, 
\[
\Im\calC = \begin{pmatrix}\Id - \Re\!(\Id-\I C)^{-1} & -\Im\!(\Id-\I C)^{-1}\\ 
-\Im\!(\Id-\I C)^{-1} & \Re\!(\Id-\I C)^{-1}\end{pmatrix},
\]
is positive definite. We use Hagedorn's decomposition $C=PQ^{-1}$, where $P,Q$ are complex invertible $d\times d$ 
matrices satisfying 
\[
Q^*P-P^*Q = 2\I\,\Id\quad\text{and}\quad Q^TP-P^T Q=0.
\] We write 
\[
\Id-\I C = (Q-\I P)Q^{-1}
\] 
and observe that
\begin{align*}
\Re\!(\Id-\I C)^{-1} &= \tfrac12(Q(Q-\I P)^{-1} + (Q^*+\I P)^{-1}Q^*)\\
&= \tfrac12 (Q^*+\I P^*)^{-1} (2Q^* Q + \I P^*Q - \I Q^* P)(Q-\I P)^{-1}\\
&= (Q^*+\I P^*)^{-1} (Q^* Q + \Id)(Q-\I P)^{-1}. 
\end{align*}
Similarly we obtain
\begin{align*}
\Im\!(\Id-\I C)^{-1}
&= \tfrac{1}{2\I} (Q^*+\I P^*)^{-1} (\I P^*Q + \I Q^* P)(Q-\I P)^{-1}\\
&= (Q^*+\I P^*)^{-1} (Q^*P - \I\,\Id) (Q-\I P)^{-1}\\
&= (Q^*+\I P^*)^{-1} (P^*Q + \I\,\Id) (Q-\I P)^{-1}
\end{align*}
and
\begin{align*}
\Id &= (Q^*+\I P^*)^{-1} (Q^*Q + \I P^* Q- \I Q^* P + P^*P) (Q-\I P)^{-1}\\
&= (Q^*+\I P^*)^{-1} (Q^*Q + 2\,\Id + P^*P) (Q-\I P)^{-1}.
\end{align*}
Setting
\[
\calB = \begin{pmatrix}(Q-\I P)^{-1} & 0\\ 0 & (Q- \I P)^{-1}\end{pmatrix},
\] 
we obtain 
\begin{align*}
\Im\calC  &=  \calB^{*}
\begin{pmatrix}\Id + P^* P & -P^*Q-\I\Id \\ -Q^*P+\I\Id & \Id + Q^* Q\end{pmatrix}
\calB\\*[1ex]
&= \calB^{*}
\left(
\Id + \begin{pmatrix}P^* P & -P^*Q\\ -Q^*P& Q^* Q\end{pmatrix} + \I J
\right)
\calB.
\end{align*}
Therefore, for all $z\in\C^{2d}$,
\[
z^*\,\Im\calC z \ =\  \|\calB z\|^2 + \|P(\calB z)_q - Q(\calB z)_p\|^2 + \I (\calB z)^*J\calB z.
\]
Since $J$ is skew-symmetric, the third summand vanishes. Neglecting the second nonnegative summand, 
we thus obtain
\[
z^*\,\Im\calC z \ \ge\  \lambda_{\rm min}(\calB^*\calB)\, \|z\|^2,
\]
and it remains to analyse the matrix $\calB^*\calB$.  
The minimal eigenvalue of $\calB^*\calB$ is the inverse of the maximal eigenvalue of $\calB^{-1}\calB^{-*}$. 
For all $x\in\C^d$,
\begin{align*}
x^*(Q-\I P)(Q^*+\I P^*)x &= 
\|Q^*x\|^2 + 2\,\Re\langle Q^*x,\I P^*x\rangle + \|P^*x\|^2 \\*[1ex]
&\le 2 \left(\|Q^*x\|^2 + \|P^*x\|^2\right)\\*[1ex]
&\le 2\left(\lambda_{\rm max}(QQ^*) + \lambda_{\rm max}(PP^*)\right) \|x\|^2.
\end{align*}
Therefore
\[
\lambda_{\rm min}(\calB^*\calB) \ \ge\  \frac{\rho_Q\rho_P}{2(\rho_Q+\rho_P)} \ =\  \rho_*.
\]
To bound the spectral norm of the real part $\Re\calC$, we use that the $j$th eigenvalue of 
$\Re\calC$ is dominated by the $j$th singular value of $\calC$, $j=1,\ldots,2d$, where both eigenvalues and singular 
values are ordered descendingly. This implies 
\[
\|\Re\calC\|_2 \le \|\calC\|_2 \le \|\calC\|_F. 
\]
We write
\[
\calC = \begin{pmatrix}\I A & A\\ A & -\I A + \I\Id\end{pmatrix}
\quad\text{with}\quad 
A = \Id - (\Id-\I C)^{-1}.
\]
Then, 
\begin{align*}
\tr(\calC^*\calC) &= \tr\begin{pmatrix}-\I A^* & A^*\\ A^* & \I A^* - \I\Id\end{pmatrix}
\begin{pmatrix}\I A & A\\ A & -\I A + \I\Id\end{pmatrix}\\*[1ex]
&= \tr\!\left(3A^*A + (\I A^* - \I\Id)(-\I A + \I\Id)\right)\\*[1ex]
&= \tr(4A^*A - 2\Re A + \Id).
\end{align*}
We observe that
\[
A = (\Id-\I C-\Id)(\Id-\I C)^{-1} = (\Id+ \I C^{-1})^{-1}
\]
and calculate for the real part of $A$ that
\begin{align*}
2\,\Re A &= \left( (\Id+ \I C^{-1})^{-1} + (\Id- \I C^{-*})^{-1}\right)\\
&= (\Id- \I C^{-*})^{-1}(\Id-\I C^{-*} + \Id + \I C^{-1})(\Id+ \I C^{-1})^{-1}\\
&= (\Id- \I C^{-*})^{-1}(2\,\Id + 2\,\Im\!(-C^{-1}))(\Id+ \I C^{-1})^{-1}.
\end{align*}
Since
\[
(\Id- \I C^{-*})(\Id+ \I C^{-1}) = \Id + 2\,\Im\!(-C^{-1}) + C^{-*}C^{-1},
\]
we have
\begin{align*}
\tr(\calC^*\calC) 
&= \tr\!\left((\Id- \I C^{-*})^{-1}(3\,\Id + C^{-*}C^{-1})(\Id+ \I C^{-1})^{-1}\right)\\*[1ex]
&= \tr\!\left((\Id + 2\,\Im\!(-C^{-1}) + C^{-*}C^{-1})^{-1}(3\,\Id + C^{-*}C^{-1})\right)\\*[1ex]
&\le \tr(3\,\Id) \ =\  3d,
\end{align*}
where the last estimate relies on the fact that all involved matrices are positive definite. Hence,
\[
\|\Re\calC\|_2 \ \le\  \sqrt{\tr(\calC^*\calC)} \ \le\  \sqrt{3d}.
\]
\end{proof}

\subsubsection{Gaussian cross-Wigner functions for proving Proposition~\ref{prop:bound_thawed_ext}}
We next calculate the formula of the cross-Wigner function of two Gaussian wave packets that 
was used to prove the extended norm bound of Proposition~\ref{prop:bound_thawed_ext}. 
We separate the calculation into two parts. 
First, we analyse how the wave packet transform
\[
f_z(x) = \eps^{-d/4}\, f\!\left(\tfrac{x-q}{\sqrt\eps}\right)\, \e^{\I p\cdot x/\eps},\qquad z=(q,p)\in\Rb^{2d},
\]
influences the cross-Wigner transform
\[
\calW(f,h)(z) = 
(2\pi\eps)^{-d} \int_{\Rb^d} \overline{f(q+\tfrac12 y)} \ h(q-\tfrac12 y)\  \e^{\I p\cdot y/\eps} \D y
\]
of two arbitrary square-integrable functions $f,h:\Rb^d\to\C$. 
Then, we explicitly calculate the Gaussian cross-transform as a Gaussian wave packet in phase space. 

\begin{lemma}[cross-Wigner function]\label{lem:crossW}
For arbitrary $f,h\in L^2(\Rb^d)$ and $z_1,z_2\in\Rb^{2d}$ 
the cross-Wigner function satisfies
\[
\calW(f_{z_1},h_{z_2})(z) = \mu_{12}\, \exp(\tfrac{\I}{\eps} (z_1-z_2)^T Jz)\ 
\calW(f_0,h_0)(z-\tfrac12(z_1+z_2))
\]
with $\mu_{12} = \exp(\tfrac{\I}{2\eps} (p_1-p_2)^T(q_1+q_2))$.  
\end{lemma}

\begin{proof}
We have
\begin{align*}
&\calW(f_{z_1},h_{z_2})(z) = 
(2\pi\eps)^{-d} \int_{\Rb^d} \overline{f_{z_1}(q+\tfrac12 y)} \ h_{z_2}(q-\tfrac12 y)\  \e^{\I p\cdot y/\eps} \D y\\
&=
(2\pi\eps)^{-d} \eps^{-d/2} \int_{\Rb^d} \overline{f\!\left(\tfrac{q-q_1+y/2}{\sqrt\eps}\right)} \ 
h\!\left(\tfrac{q-q_2-y/2}{\sqrt\eps}\right)\ \e^{-\I p_1\cdot(q-q_1+y/2)/\eps}  \\ 
&\hspace{12em}\times\e^{\I p_2\cdot(q-q_2-y/2)/\eps}\ \e^{\I p\cdot y/\eps} \D y.
\end{align*}
With the translational change of variables $x = y+ q_2- q_1$, we obtain
\begin{align*}
q-q_1+\tfrac12y &= q-\tfrac12(q_1+q_2) + \tfrac12x,\\
q-q_2-\tfrac12y &= q-\tfrac12(q_1+q_2) - \tfrac12x, 
\end{align*}
and therefore
\begin{align*}
&\calW(f_{z_1},h_{z_2})(z) \\
&= (2\pi\eps)^{-d} \eps^{-d/2} \int_{\Rb^d} \overline{f\!\left(\tfrac{q-(q_1+q_2)/2+x/2}{\sqrt\eps}\right)} \ 
h\!\left(\tfrac{q-(q_1+q_2)/2-x/2}{\sqrt\eps}\right) \e^{\I p\cdot (x-q_2+q_1)/\eps}\\ 
&\times\exp(-\tfrac{\I}{\eps} p_1^T(q-\tfrac12(q_1+q_2)+\tfrac12x)) \ 
\exp(\tfrac{\I}{\eps} p_2^T(q-\tfrac12(q_1+q_2)-\tfrac12x)) \D x\\*[1.5ex]
&= 
(2\pi\eps)^{-d}\eps^{-d/2} \ \exp(\tfrac{\I}{2\eps} (p_1-p_2)^T(q_1+q_2)) \ 
\e^{\I q\cdot(p_2-p_1)/\eps}\ \e^{\I p\cdot (q_1-q_2)/\eps} \\*[1ex]
& \qquad\times \int_{\Rb^d} \overline{f\!\left(\tfrac{q-(q_1+q_2)/2+x/2}{\sqrt\eps}\right)} \ 
h\!\left(\tfrac{q-(q_1+q_2)/2-x/2}{\sqrt\eps}\right) \e^{\I (p - (p_1+p_2)/2)\cdot x/\eps}\D x\\*[1ex]
&= \mu_{12}\, \exp(\tfrac{\I}{\eps} (z_1-z_2)^T Jz)\ 
\calW(f_0,h_0)(z-\tfrac12(z_1+z_2)).
\end{align*}
\end{proof}

The above lemma allows to calculate the cross-Wigner transform for functions centred 
in the origin and then to move the resulting transform in phase space. We will apply this 
approach to Gaussian wave packets with different width matrices. 

\begin{lemma}[Gaussian cross-Wigner function]\label{lem:gauss_cross}
We consider a complex symmetric matrix $C\in\C^{d\times d}$ with $\Im C>0$. Then, there exist
\begin{enumerate}
\item[1.] a complex symmetric matrix $\calC_W\in\C^{2d\times 2d}$ with $\Im\calC_W >0$,
\item[2.]  a complex number $\gamma\in\C$ with $|\gamma|\le 1$, 
\end{enumerate}
such that for all $z,z_1,z_2\in\Rb^{2d}$
\begin{align*}
&\calW(g_{z_1},g[C]_{z_2})(z) \\
&= \gamma\, \mu_{12}\, (\pi\eps)^{-d}\, \e^{\I (z_1-z_2)\cdot Jz/\eps}\,\exp(\tfrac{\I}{2\eps}\,(z-\tfrac12(z_1+z_2))^T\calC_W(z-\tfrac12(z_1+z_2))),
\end{align*}
where 
\[
\mu_{12} = \exp(\tfrac{\I}{2\eps} (p_1-p_2)^T(q_1+q_2)).
\]  
In particular, if $C=\I\Id$, then $\gamma = 1$ and $\calC_W = 2\I\Id$.
\end{lemma}

\begin{proof}
By Lemma~\ref{lem:crossW}, 
\[
\calW(g_{z_1},g[C]_{z_2})(z) = \gamma_{12} \, \e^{\I (z_1-z_2)\cdot Jz/\eps}\,
\calW(g_0,g[C]_0)(z-\tfrac12(z_1+z_2)).
\]
We have
\[
\calW(g_{0},g[C]_{0})(z) = 
(2\pi\eps)^{-d} \int_{\Rb^d} \overline{g_{0}(q+\tfrac12 y)} \ g[C]_{0}(q-\tfrac12 y)\  
\bch \e^{\I p\cdot y/\eps} \D y.\ech
\]
The product of the two Gaussian functions reads
\begin{align*}
&\overline{g_0(q+\tfrac12 y)} \ g[C]_{0}(q-\tfrac12 y) \\
&= (\pi\eps)^{-d/2}(\det\Im C)^{1/4} \  
\e^{-\tfrac{1}{2\eps}|q+y/2|^2 + \tfrac{\I}{2\eps}(q-y/2)\cdot C(q-y/2)}\\
&= (\pi\eps)^{-d/2}(\det\Im C)^{1/4} \  
\e^{\tfrac{\I}{2\eps}q\cdot(C+\I\Id)q}\ 
\e^{-\tfrac{1}{8\eps} y\cdot(\Id-\I C)y -\tfrac{\I}{2\eps}(C-\I\Id) q\cdot y}.
\end{align*}
We now use the Gaussian integral formula 
\[
\int_{\Rb^d} \exp(-\pi y^T A y - 2\pi\I v^T y) \D y = (\det A)^{-1/2} \exp(-\pi v^T A^{-1}v)
\]
for the complex symmetric matrix 
\[
A = \tfrac{1}{8\pi\eps}(\Id-\I C),\qquad \Re A>0,
\]
and the complex vector $v = \tfrac{1}{4\pi\eps}((C-\I\Id)q -2p)$. We write 
the quadratic form resulting from the Fourier integral as
\[
-\pi v^T A^{-1}v = -\tfrac{1}{2\eps}((C-\I\Id)q -2p)^T (\Id-\I C)^{-1} ((C-\I\Id)q -2p).
\]
We observe that
\begin{align*}
(C-\I\Id)(\Id-\I C)^{-1} &= -\I (2\I C + (\Id-\I C)) (\Id-\I C)^{-1} \\
&= 2C(\Id-\I C)^{-1} - \I \Id
\end{align*}
and
\begin{align*}
&(C-\I\Id)(\Id-\I C)^{-1}(C-\I\Id) \\
&= -\I(2C(\Id-\I C)^{-1} - \I \Id)(2\I C + (\Id-\I C))\\
&= 4C^2(\Id-\I C)^{-1} - 3\I C - \Id. 
\end{align*}
Since
\begin{align*}
& C+\I\Id +\I\left(4C^2(\Id-\I C)^{-1} - 3\I C - \Id\right) = 4C + 4\I C^2(\Id-\I C)^{-1}\\
&= 4\left(C(\Id-\I C) + \I C^2\right))(\Id-\I C)^{-1} = 4C(\Id-\I C)^{-1}, 
\end{align*}
we may write 
\[
\calW_{g_0,g[C]_0}(z) = 
(\pi\eps)^{-d} \ \gamma\ 
\exp(\tfrac{\I}{2\eps}z^T\calC_W z),
\]
where the determinantal prefactor is given by
\[
\gamma = 2^{d/2} \left(\frac{\det\Im C}{\det(\Id-\I C)^2}\right)^{1/4},
\]
and $\calC_W$ denotes the complex symmetric $2d\times 2d$ matrix
\[
\calC_W =
\begin{pmatrix}
4C(\Id-\I C)^{-1} & -4\I C(\Id-\I C)^{-1} - 2 \Id\\*[0.5ex]
-4\I C(\Id-\I C)^{-1} - 2 \Id & 4\I(\Id-\I C)^{-1}
\end{pmatrix}.
\]
We compare $\calC_W$ with the complex symmetric matrix $\calC$ of 
Lemma~\ref{lem:det_pd}. Since
\begin{align*}
\Id - (\Id-\I C)^{-1} &= (\Id-\I C-\Id)(\Id-\I C)^{-1} = -\I C(\Id-\I C)^{-1},
\end{align*}
we have
\[
\calC_W = 4\,\calC + \begin{pmatrix}0 & -2\Id\\ -2\Id & 0\end{pmatrix}
\quad\text{and}\quad \Im\calC_W = 4\,\Im\calC > 0.
\]
\end{proof}

\subsubsection{Gaussian windowed Fourier transforms for proving 
Proposition~\ref{prop:bound_thawed_ext}}
The next estimate generalizes the classic integration by parts argument for 
proving the superalgebraic decay of the Fourier transform of a smooth function.

\begin{proposition}[Gaussian Fourier integral]\label{prop:fourier}
Let $C\in\C^{d\times d}$ be a complex symmetric matrix with positive definite 
imaginary part. Denote
\[
\rho = \lambda_{\rm min}(\Im C),\quad \tau = \tr(\Im C),\quad \mu = \|\Re C\|.
\]
Then, for all $m\in\N_0^{d}$ and all $n\ge 0$ there exists a constant $c_{m,n}(\rho,\tau,\mu)>0$ such that, for all
smooth functions $f:\Rb^{d}\to\C$ that are bounded together with all its derivatives, 
and for all $z\in\Rb^{d}$, the integral 
\[
I(z) = (2\pi)^{-d/2} \int_{\Rb^{d}} (w+z)^m\, f(w)\, \exp(\tfrac{\I}{2} w^T C w) \, \e^{\I z\cdot w} \D w
\]
satisfies the estimate
\[
|I(z)| \ \le\  c_{m,n}(\rho,\tau,\mu)\, \left(1+|z|^2\right)^{-n/2} \sup_{|\alpha|\le |m|+n} \|\partial^\alpha f\|.
\]
\end{proposition}

\begin{proof} 
We start by considering $z=0$. Since
\[
I(0) = (2\pi)^{-d/2} \int_{\Rb^{d}} w^m\, f(w)\, \exp(\tfrac{\I}{2} w^T C w) \D w, 
\]
the bound on Gaussian moments given in Lemma~\ref{lem:int-bound} provides a constant 
$\gamma_{0,m}(\rho)>0$ such that
\[
|I(0)| \ \le\  \gamma_{0,m}(\rho)\,\|f\|_\infty.
\]
We now consider $z\neq0$ in the following. We combine the polynomial factor $(w+z)^m$ and the function $f$ 
together with the non-decaying part of the Gaussian, 
\[
f_C(w,z) = (w+z)^m f(w) \exp(\tfrac{\I}{2}w^T \Re C w). 
\] 
We write
\[
I(z) = (2\pi)^{-d/2} \int_{\Rb^{d}} f_C(w,z) \exp(-\tfrac{1}{2} w^T \Im C w) \, \e^{\I z\cdot w} \D w.
\]
We have
\begin{align*}
&\exp(-\tfrac{1}{2} w^T \Im C w) \e^{\I z\cdot w} \\
&\quad= -\frac{(\Im Cw+\I z)^T}{|(\Im C w,z)|^2}\ \partial_w \left(\exp(-\tfrac{1}{2} w^T \Im C w) \e^{\I z\cdot w}\right).
\end{align*}
Therefore, by partial integration,
\[
I(z) = (2\pi)^{-d/2} \int_{\Rb^{2}}
 {\rm div}_w\!\left(\frac{(\Im Cw+\I z) f_C(w,z)}{|(\Im Cw,z)|^2}\right)\, 
\e^{-\tfrac{1}{2} \bch w\cdot {\rm Im}C w + \I z\cdot w\ech} \D w .
\]
We decompose
\[
{\rm div}_w\!\left(\frac{(\Im Cw+\I z)f_C(w,z)}{|(\Im Cw,z)|^2}\right) = g_1(w,z) + g_2(w,z)
\]
with
\[
g_1(w,z) = 
f_C(w,z) \ {\rm div}_w\!\left(\frac{(\Im Cw+\I z)}{|(\Im Cw,z)|^2}\right)
\]
and
\[
g_2(w,z) = 
\frac{(\Im Cw+\I z)f_C(w,z)}{|(\Im Cw,z)|^2}\cdot\nabla f_C(w,z).
\]
Correspondingly we write the integral as
\[
I(z) = I_1(z) + I_2(z).
\]
To estimate the first integral $I_1(z)$, we calculate the divergence, 
\[
{\rm div}_w\,\frac{\Im C w+\I z}{|(\Im Cw,z)|^2} = \frac{\tr(\Im C)-2(\Im Cw+\I z)\cdot \Im Cw}{|(\Im Cw,z)|^4},
\]
and obtain 
\begin{align*}
\left| {\rm div}_w\,\frac{\Im Cw+\I z}{|(\Im Cw,z)|^2} \right| & \le\  \frac{\tr(\Im C)+3|(\Im Cw,z)|^2}{|(\Im Cw,z)|^4} \\*[1ex]
& \le\  \tr(\Im C)|z|^{-4} + 3|z|^{-2}.
\end{align*}
Hence, 
\[
g_1(w,z) \le \left(\tr(\Im C)|z|^{-4} + 3|z|^{-2}\right) |w+z|^{|m|}\  \|f\|_\infty,
\]
and there exists a constant $\gamma_{1,m}(\rho,\tau)>0$ such that 
\[
|I_1(z)| \le  \gamma_{1,m}(\rho,\tau)\ |z|^{|m|-2} \ \|f\|_\infty. 
\]
To estimate the second integral $I_2(z)$, we observe that 
\[
\frac{|\Im Cw+\I z|}{|(\Im Cw,z)|^2} \ =\  |(\Im Cw,z)|^{-1} \ \le\  |z|^{-1}
\]
and
\[
|\nabla_w f_C(w,z)| \le \sup_{|\alpha|\le 1} \|\partial^\alpha f\| 
\left( |m|\,|w+z|^{|m|-1} + |w+z|^{|m|}(1 + |\Re C w|)\right). 
\]
Therefore, there exists a constant $\gamma_{2,m}(\rho,\mu)>0$ such that 
\[
|I_2(z)| \ \le\  \gamma_{2,m}(\rho,\mu) |z|^{|m|-1} \sup_{|\alpha|\le 1} \|\partial^\alpha f\|.
\]
Hence, combining the bounds for $z=0$ and $z\neq0$, we have proved the existence of a constant 
$\gamma_m(\rho,\tau,\mu)>0$ guaranteeing for all $z$ that
\[
|I(z)| \ \le\  \gamma_m(\rho,\tau,\mu) \left(1+|z|^2\right)^{(|m|-1)/2} \sup_{|\alpha|\le 1} \|\partial^\alpha f\|. 
\]
Repeating the previous integration by parts $|m|+ n$ times generates 
higher-order derivatives of the functions 
\[
w\mapsto f_C(w,z)\quad\text{and}\quad w\mapsto (\Im Cw+\I z)/|(\Im C w,z)|^{-2},
\] 
which can be bounded in terms of 
\[
\sup_{|\alpha|\le |m|+n}\|\partial^\alpha f\|_\infty\quad\text{and}\quad(1+|z|^2)^{-n/2},
\]
providing the claimed estimate. 
\end{proof}

We will use the previous polynomial estimate inside an integral over phase space. The following 
result guarantees integrability for sufficiently large polynomial degree. 

\begin{lemma}[polynomial integral]\label{lem:int_poly} 
\[
\int_{\Rb^{2d}} \left(1 + |z|^2\right)^{-(d+1)}\D z \ =\  \pi^d/d!\ .
\]
\end{lemma}

\begin{proof}
Let $n>2d$. We use two-dimensional polar coordinates to write
\begin{align*}
&\int_{\Rb^{2d}} (1 + |z|^2)^{-n/2}\D z \\
&=(2\pi)^d \int_{[0,\infty)^d} (1+ r_1^2+\cdots + r_d^2)^{-n/2} \,r_1\cdots r_d\, \D(r_1,\ldots,r_d). 
\end{align*}
Since for all $c\ge 0$,
\[
2 \int_0^\infty (1+r^2+c)^{-n/2} \,r \D r = \frac{(1+c)^{-n/2+1}}{n/2-1},
\]
we have
\[
\int_{\Rb^{2d}} \left(1 + |z|^2\right)^{-n/2}\D z = \frac{\pi^d}{(n/2-1)\cdots (n/2-d)},
\]
which gives the claimed formula for $n=2d+2$.
\end{proof}

\subsubsection{Properties of the thawed phase function, Lemma~\ref{lem:phase_thawed}}
\label{proof:phase_thawed}

Here we verify the differentiation formula for the thawed phase function $\Psi_{\rm th}$ and 
prove invertibility of the matrix $M_{\rm th}$. 

\begin{proof} 
We recall that the gradients of the action integral satisfy
\begin{align*}
\partial_q S(t,z) &= \partial_q q(t,z)^T p(t,z)-p,\\
\partial_p S(t,z) &=  \partial_p q(t,z)^T p(t,z).
\end{align*}
We therefore obtain for the derivatives of the phase function, 
\begin{align*}
&\partial_q\Psi_{\rm th}(t,x,y,z) = -\I(y-q) + \tfrac{1}{2}(x-q(t,z))^T \partial_q C(t,z)(x-q(t,z))\\*[1ex]
&\qquad-\partial_q q(t,z)^T C(t,z)(x-q(t,z)) + p + \partial_q p(t,z)^T(x-q(t,z))\\*[1ex]
&\qquad- \partial_q q(t,z)^Tp(t,z) + \left(\partial_q q(t,z)^T p(t,z)-p\right)\\*[1ex]
&\qquad= -\I(y-q) + \tfrac{1}{2}(x-q(t,z))^T \partial_q C(t,z)(x-q(t,z)) \\*[1ex]
&\qquad+\left(\partial_q p(t,z)- C(t,z)\partial_q q(t,z)\right)^T(x-q(t,z))
\end{align*}
and
\begin{align*}
&\partial_p\Psi_{\rm th}(t,x,y,z) = \tfrac{1}{2}(x-q(t,z))^T \partial_p C(t,z)(x-q(t,z))\\*[1ex]
&\qquad-\partial_p q(t,z)^T C(t,z)(x-q(t,z)) - (y-q) + \partial_p p(t,z)^T(x-q(t,z))\\*[1ex]
&\qquad - \partial_p q(t,z)^Tp(t,z) + \partial_p q(t,z)^Tp(t,z)\\*[1ex]
&\qquad= - (y-q)+ \tfrac{1}{2}(x-q(t,z))^T \partial_p C(t,z)(x-q(t,z))\\*[1ex]
&\qquad+\left(\partial_p p(t,z)- C(t,z)\partial_p q(t,z)\right)^T(x-q(t,z)).
\end{align*}
This implies
\begin{align*}
&(\I\partial_q + \partial_p)\Psi(t,x,y,z) = M_{\rm th}(t,z)(x-q(t,z)) + \\
&\qquad \left(\tfrac{1}{2}(x-q(t,z))^T \left((\I\partial_{q_j}+\partial_{p_j})C(t,z)\right) (x-q(t,z))\right)_{j=1}^d.
\end{align*}
with
\begin{align*}
& M_{\rm th}(t,z) \\
&= -\I\partial_q q(t,z)^T C(t,z) +\partial_p p(t,z)^T + \I\partial_q p(t,z)^T  -\partial_p q(t,z)^T C(t,z).
\end{align*}
To prove invertibility, we decompose the thawed matrix as
\[
M_{\rm th} = \begin{pmatrix}\I\,\Id & \Id\end{pmatrix} 
\begin{pmatrix}\partial_q q^T & \partial_q p^T\\*[1ex] \partial_p q^T & \partial_p p^T\end{pmatrix}
\begin{pmatrix}-C\\ \Id\end{pmatrix}.
\]
Accounting for the presence of the matrix $C$ in the above decomposition, we do not work directly 
with the plain product $M_{\rm th}\, M_{\rm th}^*$, but with the weighted matrix
$M_{\rm th}\, (\Im C)^{-1}\, M_{\rm th}^*$. We analyse the $C$-dependent contribution to this product, 
that is, the matrix
\[
P_{\rm th} = \begin{pmatrix}-C\\ \Id\end{pmatrix}(\Im C)^{-1} \begin{pmatrix}-C^* & \Id\end{pmatrix}. 
\]
We have
\begin{align*}
P_{\rm th} &= \begin{pmatrix}C(\Im C)^{-1}C^* & -C(\Im C)^{-1}\\ -(\Im C)^{-1}C^* & (\Im C)^{-1}\end{pmatrix}\\*[1ex] 
& = \begin{pmatrix}\Im C + \Re C(\Im C)^{-1}\Re C & -\Re C(\Im C)^{-1}-\I \Id\\ 
-(\Im C)^{-1}\Re C + \I\Id & (\Im C)^{-1}\end{pmatrix}.
\end{align*}
Hence, $P_{\rm th}$ is of the form 
\[
P_{\rm th} = \begin{pmatrix}A^2 + BA^{-2}B & -BA^{-2}\\ -A^{-2}B & A^{-2}\end{pmatrix} + \I J
\]
with $A=(\Im C)^{1/2}$ and $B=\Re C$. The first summand of $P_{\rm th}$ allows for a block 
Cholesky factorization, and we obtain 
\[
P_{\rm th} = \Lambda^T\Lambda + \I J
\]
with
\[
\Lambda = \begin{pmatrix} A & 0\\ -A^{-1}B & A^{-1}\end{pmatrix}\ .
\]
Altogether, we have established that  
\begin{align*}
M_{\rm th} (\Im C)^{-1} M_{\rm th}^* 
&= \begin{pmatrix}\I\,\Id & \Id\end{pmatrix} \left( D\Phi^T \Lambda^T \Lambda D\Phi + \I J\right)
\begin{pmatrix}-\I\,\Id\\ \Id\end{pmatrix}\\
&=\left( \Lambda D\Phi\begin{pmatrix}-\I\,\Id\\ \Id\end{pmatrix}\right)^*
\left( \Lambda D\Phi\begin{pmatrix}-\I\,\Id\\ \Id\end{pmatrix}\right) + 2\,\Id\ .
\end{align*}
This implies for the determinant 
\[
|\det M_{\rm th}| \ \ge\  2^{d/2} \det(\Im C)^{1/2},
\]
and we have proved that $M_{\rm th}$ is invertible.
\end{proof}

\subsubsection{Improved norm bound, Proposition~\ref{prop:improved_bound_thawed} }
\label{proof:improved_bound_thawed}
Using integration by parts, we prove that polynomial powers in the oscillatory integral 
operator lower the norm bound with respect to the semi-classical parameter. 

\begin{proof} As for Proposition~\ref{prop:improved_bound}, 
we perform an inductive proof over $|m|$. The case $|m| = 0$ is already covered by the norm bound of 
Proposition~\ref{prop:bound_thawed_ext}. 
We therefore start with $|m|=1$, where $m=e_j$ for some $j=1,\ldots,d$.  
We use the derivative formula of Lemma~\ref{lem:phase_thawed}, 
\begin{align*}
&(x-q(t,z))\, \e^{\I\Psi_{\rm th}(t,x,y,z)/\eps} \\
&=
\left(\tfrac{\eps}{\I} \, M_{\rm th}(t,z)^{-1}(\I\partial_q + \partial_p) - v(t,x,z) \right)\,
\e^{\I\Psi_{\rm th}(t,x,y,z)/\eps}.
\end{align*}
with
\begin{align*}
&v(t,x,z) = \\
&M_{\rm th}(t,z)^{-1} \left(\tfrac{1}{2}(x-q(t,z))^T \left((\I\partial_{q_j}+\partial_{p_j})C(t,z)\right) (x-q(t,z))\right)_{j=1}^d.
\end{align*}
Integration by parts yields
\begin{align*}
&\bch \calI_j(t)\psi(x)\ech \\
&= 2^{-d} (\pi\eps)^{-3d/2} \int_{\Rb^{3d}} (x-q(t,z))_j\, b(q(t,z),x)\,  
\e^{\I\Psi(t,x,y,z)/\eps} \psi(y) \D(y,z)\\
&= \eps\I \,\sum_{k=1}^d \calI_{jk}(t)\psi(x) - \calI_2(t)\psi(x)
\end{align*}
with
\begin{align*}
&\calI_{jk}(t)\psi(x) = 2^{-d} (\pi\eps)^{-3d/2} \\
&\qquad\int_{\Rb^{3d}} \,(\I\partial_q+\partial_p)_k\left(b(q(t,z),x)\, M_{\rm th}(t,z)^{-1}_{jk}\right)\,  
\e^{\I\Psi(t,x,y,z)/\eps} \psi(y) \D(y,z)
\end{align*}
and
\begin{align*}
&\calI_{2}(t)\psi(x) \\
&= 2^{-d} (\pi\eps)^{-3d/2} \int_{\Rb^{3d}} 
b(q(t,z),x)\, v_j(t,x,z) \, \e^{\I\Psi(t,x,y,z)/\eps} \bch \psi(y) \D(y,z).\ech
\end{align*}
All the operators $\calI_{jk}(t)$ contain terms that are of order zero with respect to $x-q(t,z)$, 
while the corresponding polynomial terms in $\calI_2(t)$ are of order two. Therefore, 
by Proposition~\ref{prop:bound_thawed_ext},
\[
\|\calI_{jk}(t)\psi\| \ \le \gamma_0(b,\rho_*)\ d_1(M_{\rm th}(t)^{-1},\Phi^t)\, 
\, \|\psi\|
\]
and 
\[
\|\calI_2(t)\psi\| \ \le\  \gamma_2(b,\rho_*)\ d_0(M_{\rm th}(t)^{-1})\ d_1(C(t))\,\eps\, \|\psi\|.
\]
This implies 
\begin{align*}
\|\bch \calI_j(t)\psi\ech\| &\ \le\   \eps\,\sum_{k=1}^d \|\calI_{jk}(t)\psi\| + 
\|\calI_2(t)\psi\| \\
& \ \le\   \gamma_{0,2}(b,\rho_*)\ d_1(M_{\rm th}(t)^{-1},\Phi^t,C(t))\ \eps\, \|\psi\|.
\end{align*}
For the inductive step, we consider 
\begin{align*}
&\bch\calI_j(t)\psi(x)\ech \\
&= 2^{-d} (\pi\eps)^{-3d/2} \int_{\Rb^{3d}} (x-q(t,z))^{m+e_j}\, b(q(t,z),x)\,  
\e^{\I\Psi(t,x,y,z)/\eps} \psi(y) \D(y,z)\\*[1ex]
&=\eps\I\left(\calI_{|m|}(t) + \calI_{|m|-1}(t)\right)\psi(x)  + \calI_{|m|+2}(t)\psi(x) 
\end{align*}
with
\begin{align*}
&\calI_{|m|}(t)\psi(x) = 2^{-d} (\pi\eps)^{-3d/2}\ \sum_{k=1}^d \int_{\Rb^{3d}} (x-q(t,z))^m\\*[1ex]  
&\quad\times(\I\partial_q+\partial_p)_k \left(b(q(t,z),x)\, M_{\rm th}(t,z)^{-1}_{jk}\right)
\e^{\I\Psi(t,x,y,z)/\eps} \psi(y) \D(y,z)
\end{align*}
and 
\begin{align*}
&\calI_{|m|-1}(t)\psi(x) = 2^{-d} (\pi\eps)^{-3d/2}\ \sum_{k=1}^d 
\int_{\Rb^{3d}}b(q(t,z),x)\, M_{\rm th}(t,z)^{-1}_{jk}\\*[1ex]
&\quad\times
(\I\partial_q+\partial_p)_k \left((x-q(t,z))^m\right)\ 
\e^{\I\Psi(t,x,y,z)/\eps} \psi(y) \D(y,z)
\end{align*}
and
\begin{align*}
&\calI_{|m|+2}(t)\psi(x) = - 2^{-d} (\pi\eps)^{-3d/2}\\*[1ex]
&\quad\int_{\Rb^{3d}} (x-q(t,z))^{m}\, b(q(t,z),x)\,v_j(t,x,z)
\e^{\I\Psi(t,x,y,z)/\eps} \psi(y) \D(y,z).
\end{align*} 
These three integral operators contain monomials in $x-q(t,z)$ of order $|m|$, $|m|-1$
and $|m|+2$, respectively. By the inductive hypothesis, 
\begin{align*}
\|\calI_{|m|}(t)\psi\| & \ \le \ \gamma_{|m|}(b,\rho_*)\ d_{|m|+1}(M_{\rm th}(t)^{-1}, \Phi^t, C(t))\ 
\bch\eps^{\lceil{|m|/2\rceil}}\ech\,\|\psi\|,\\*[1ex]
\|\calI_{|m|-1}(t)\psi\| & \ \le \ \gamma_{|m|-1}(b,\rho_*)\ d_{|m|}(M_{\rm th}(t)^{-1}, \Phi^t, C(t))\ 
\eps^{\lceil(|m|-1)/2\rceil}\,\|\psi\|.
\end{align*}
By the norm estimate of Proposition~\ref{prop:bound_thawed_ext}, 
\[
\|\calI_{|m|+2}(t)\psi\| \ \le\  \gamma_{|m|+2}(\rho_*,b)\, d_{1}(C(t))\, \eps^{(|m|+2)/2}\,\|\psi\|.
\]
Combining the three estimates, we obtain
\begin{align*}
\|\bch\calI_j(t)\psi\ech\| &\ \le\  \eps\left(\|\calI_{|m|}(t)\psi\| + \|\calI_{|m|-1}(t)\psi\|\right) + 
\|\calI_{|m|+2}(t)\psi\| \\*[1ex]
& \ \le\ 
\gamma_{|m|+2}(b,\rho_*)\,d_{|m|+1}(M_{\rm th}(t)^{-1},\Phi^t,C(t)) \eps^{\lceil(|m|+1)/2\rceil}\,\|\psi\|, 
\end{align*}
where we have used that 
\[
\lceil(|m|-1)/2\rceil+1 = \lceil (|m|+1)/2\rceil.
\]
\end{proof}

\subsubsection{Defect calculation, Lemma~\ref{lem:res_thawed}}\label{proof:res_thawed}
The calculation of the defect of the thawed Gaussian approximation is similar to that  
for the variationally determined Gaussian wave packet. 

\begin{proof}
The time derivative of the thawed Gaussian's normalization, 
\[
\partial_t (\det \Im C(t))^{1/4}=  -\tfrac12\, (\det \Im C(t))^{1/4}\  \tr\,C(t),
\]
is elegantly calculated via the relation 
$\Im C(t) = (Q(t)Q(t)^*)^{-1}$ 
in Hagedorn's parametrization of complex symmetric matrices with positive definite imaginary part. Then, one obtains 
\begin{align*}
\partial_t g(t)_{\Phi^t} &= -\tfrac12\tr\,C(t) g(t)_{\Phi^t}\\
& + \partial_t\left(\tfrac{\I}{2\eps}(x-q(t))^T C(t)(x-q(t)) + \tfrac\I\eps p(t)^T (x-q(t))\right) 
g(t)_{\Phi^t}
\end{align*}
and consequently
\[
\partial_t\left( \e^{\I S(t)/\eps}g(t)_{\Phi^t}\right) 
=\tfrac{1}{\I\eps}\, f_1(t)\,\e^{\I S(t)/\eps}g(t)_{\Phi^t}
\]
with
\begin{align*}
f_1 &= -\tfrac{\I\eps}{2} \tr(C) - \tfrac12|p|^2 + V(q) \\
&\qquad + p^T\dot q + (C\dot q-\dot p)^T(x-q) -\tfrac12(x-q)^T \dot C(x-q).
\end{align*}
The analogue of the variational calculation yields 
\[
\tfrac{1}{\I\eps}H \,g(t)_{\Phi^t}  = \tfrac{1}{\I\eps} f_2(t) g(t)_{\Phi^t}
\]
for the action of the Schr\"odinger operator, with
\begin{align*}
f_2 &= -\tfrac{\I\eps}{2}\tr(C) + \tfrac12|p|^2  + V(q) + (Cp+\nabla V(q))^T (x-q) \\*[1ex]
&\qquad+ \tfrac12(x-q)^T(C^2+\nabla^2 V(q))(x-q) + p^T C(x-q) + W_q.
\end{align*}
By the classical equations of motion, we have 
$f_2(t)-f_1(t) = W_{q(t)}$ 
and therefore the claimed representation of the defect. 
\end{proof}

\subsection{Notes}

The wave packet transform of Propositions~\ref{prop:wp_trafo} and \ref{prop:bargmann} has different names in different mathematical subcultures. In time-frequency analysis, it is known as windowed Fourier transform, short-time Fourier transform or continuous wavelet transform when alluding to \cite{GroM84}.
For the special case of a Gaussian window function, it is called a 
Gabor transform in reference to \cn{Gab46}. 
In the coherent states monograph of \cn[Chapter~1.2.3]{ComR12}, the wave packet transform is 
called the Fourier--Bargmann 
transform. The books of \cn[Chapter~3.3]{Fol89} and \cn[Chapter~3]{Mar02} use the term 
Fourier--Bros--Iagolnitzer transform, which is often abbreviated to FBI transform. 

The notion of a thawed Gaussian superposition seems to originate in chemical physics. 
Numerical analysts commonly refer to such approximations as Gaussian beams. 
They have been applied successfully to dispersive wave equations in the high frequency 
regime, but also to hyperbolic problems. \cn{JinMS11} gave an earlier account of this line of research 
for the semiclassical Schr\"odinger equation. In the Gaussian beam literature, WKB initial data
\[
\psi_0(x) = (A_0(x) + \eps A_1(x) + \cdots) \e^{\I S_{\rm in}(x)/\eps}
\]
are usually considered, either approximated by Gaussian wave packets with 
centres $z=(q,p)$ in some subset of phase space or by first choosing the position 
components $q$ in some subset of configuration space and then fixing the momenta to be 
$p=\nabla S_{\rm in}(q)$. For Gaussian beam superpositions with centres on a Lagrangian 
submanifold of phase space, the first optimal error estimate -- an order $\eps$ result in 
$L^2(\Rb^d)$ -- was proved by \cn{Zhe14} using WKB techniques with a \bch matching procedure  \ech
at caustic points. In contrast, \cn{LiuRT16} considered Gaussian beam superpositions that are 
position guided and improved their earlier results \cite{LiuRT13} to the 
optimal order $\eps$ in $L^2$-norm. Their direct analysis of the oscillatory 
integral operators also allows for an extension to Sobolev and sup-norm 
estimates. 
The thawed approach developed by \cn{BerBCN17} and \cn{CorGN17} represents an 
arbitrary square-integrable initial datum by a Gabor frame of Gaussian wave packets 
that are centred on a  phase space lattice.
The performed error analysis, however,   
only yields an order $\sqrt\eps$ result. The line of argument for proving 
the optimal order $\eps$ estimate presented in Theorem~\ref{theo:thawed}  
is indicated in \cn[Sections 3--4]{Rob10} but not carried out fully there.  
 
The class of frozen Gaussian approximations is similar in spirit to the thawed 
superpositions, but only evolves the phase space centres and corresponding action integrals.
The first method of this type 
seems to have been proposed by \cn{Hel81}, who did not consider a time-dependent 
weighting factor but argued that collective correlation might improve the accuracy despite 
the constraint that the wave packets are not allowed 
to spread. This ansatz was followed up by \cn{HerK84}, who introduced a time-dependent 
weight to the frozen approximation. Later, in \cite{KluHD86}, the original lengthy expression  
of the Herman--Kluk prefactor was simplified to the more handy form of Definition~\ref{def:hk}, 
and numerical results for a particle method with initial Monte Carlo sampling were presented. 
Fifteen years later, a couple of papers in the chemical literature critically questioned the validity of the Herman--Kluk 
approximation, and it was \cn{Kay06} who developed the elaborate integration 
by parts performed in Theorem~\ref{theo:hk} to justify a semiclassical expansion in powers 
of the semiclassical parameter, which has the Herman--Kluk propagator as its leading term. 
\bch The first correction term of this expansion was numerically explored in \cite{HocK06}.
\ech
The first mathematically rigorous error estimates were due to \cn{SwaR09}. A noticeably simplified approach 
to the required norm bounds relying on Bargmann kernel estimates was later proposed by \cn{Rob10}.
The direct use of the wave packet transform, however, as pursued in 
Proposition~\ref{prop:basic_bound} and Corollary~\ref{cor:norm_bound}, appears to be new and offers 
an even more elementary method for assessing the accuracy of the frozen approximation. 
The numerical realisation of the Herman--Kluk propagator as a particle method was considered in 
 \cite{LasS17}.

%\newpage
%\section{Wigner functions}
\def\op{\mathrm{op}}
\def\calH{\mathcal{H}}
\def\calL{\mathcal{L}}
\def\calR{\mathcal{R}}

\section{Wigner functions} 
\label{sec:wigner}
In this section we derive approximations to expectation values of observables that are obtained directly without previously computing the wave function. They are based on Egorov's theorem, which relates quantum observables and classically propagated observables, and on Wigner or Husimi functions, which represent averages of quantum observables as integrals over classical phase space.
The combination leads to a computational approach of $O(\eps^2)$ accuracy in which quadrature points in phase space (particles) are transported by the classical equations of motion.

\subsection{Weyl quantisation}

To every classical observable, {\it i.e.}~a smooth function $a:\Rb^{2d}\to\Rb$ on phase space, we would like to assign a quantum observable, {\it i.e.}~a self-adjoint linear operator $\widehat a = \op(a)$ acting on a suitable subspace ${\rm dom}(\widehat a)$ of $L^2(\Rb^d)$, such that the coordinate projections $(q,p)\mapsto q_j$ and 
$(q,p)\mapsto p_j$ for $j=1,\ldots,d$ result in the familiar position and momentum operators,
\[
\widehat{q_j}\psi(x) = x_j\psi(x)\quad\text{and}\quad \widehat{p_j}\psi(x) = -\I\eps\partial_j\psi(x).
\]
However, due to the commutator relation 
\[
\tfrac{1}{\I\eps}[\widehat q_j,\widehat p_k] = \delta_{jk},
\]
there are several ways to quantise even a simple function such as the product $(q,p)\mapsto q_j p_j$. 
Weyl quantisation takes a democratic attitude to this non-commutative challenge and 
works for bivariate polynomials as follows. 

\begin{definition}[Weyl-quantised polynomial] Let $j,k\ge 0$, and consider the polynomial
$a:\Rb^{2}\to\Rb$ with $a(q,p) = q^j p^k$ for all $(q,p)\in\Rb^2$. 
We define the Weyl-quantised operator $\op(a)$ of the polynomial function $a$ according to
\[
\op(a) = \frac{1}{(j+k)!} 
\sum_{\sigma\in S_{j+k}} 
\sigma(\underbrace{\widehat q,\ldots,\widehat q}_{j\ \text{times}},
\underbrace{\widehat p,\ldots,\widehat p}_{k\ \text{times}}),
\]
where, for a permutation $\sigma$ of $j+k$ elements and operators $A_1,\ldots,A_{j+k}$, we denote 
$
\sigma(A_1,\ldots,A_{j+k}) = A_{\sigma(1)} \cdots A_{\sigma(j+k)}.
$
\end{definition}

This democratic average of all the possible orderings of the position 
and the momentum operator has remarkable properties. Imposing linearity 
of the quantisation map $a\mapsto\op(a)$, the binomial theorem yields for all $\alpha,\beta\in\Rb$ that
\begin{align*}
\op((\alpha q+\beta p)^n) 
&= \sum_{k=0}^n \binom{n}{k} \alpha^k \beta^{n-k} \op(q^k p^{n-k})\\
&= \sum_{k=0}^n \binom{n}{k} \alpha^k \beta^{n-k} \frac{1}{n!} 
\sum_{\sigma\in S_{n}} 
\sigma(\underbrace{\widehat q,\ldots,\widehat q}_{k\ \text{times}},
\underbrace{\widehat p,\ldots,\widehat p}_{n-k\ \text{times}})\\
&=(\alpha \widehat q + \beta \widehat p)^n,
\end{align*}
and for all $N\in\N$, 
\[
\bch\op\Big(\ \sum_{n=0}^N\frac{\I^n}{n!}(\alpha q + \beta p)^n \Big) \ech = 
\sum_{n=0}^N\frac{\I^n}{n!}(\alpha \widehat q + \beta \widehat p)^n.
\]
Formally passing to the limit $N\to\infty$, this suggests defining
$$
\op(\exp(\I(\alpha q + \beta p))) := \exp(\I(\alpha \widehat q + \beta \widehat p)).
$$
We let
\[
\calF_\eps a(\alpha,\beta) = \frac{1}{2\pi\eps} \int_{\Rb^2} a(q,p) \ \e^{-\I(\alpha q+ \beta p)/\eps} d(q,p)
\]
denote \bch the semiclassically scaled, unitary\ech Fourier transform of a Schwartz function $a:\Rb^2\to\Rb$. By the Fourier inversion formula we may write 
\[
a(q,p) = \frac{1}{2\pi\eps} \int_{\Rb^2} \calF_\eps a(\alpha,\beta) \ \e^{\I(\alpha q + \beta p)/\eps} d(\alpha,\beta).
\]
This motivates to set
\[
\op(a) = \frac{1}{2\pi\eps} \int_{\Rb^2} \calF_\eps a(\alpha,\beta)\  
\e^{\I(\alpha \widehat q + \beta \widehat p)/\eps} d(\alpha,\beta),
\]
where the operator-valued integration has to be performed carefully. 
The unitary operator $\e^{\I(\alpha \widehat q + \beta \widehat p)/\eps}$ 
is often referred to as a Heisenberg--Weyl translation operator. 
The following lemma explains this notion.

\begin{lemma}[Heisenberg--Weyl translation]
For all $\alpha,\beta\in\Rb^d$ and for all $\varphi\in L^2(\Rb^d)$, 
\[
\e^{\I(\alpha\cdot \widehat q + \beta\cdot \widehat p)/\eps}\varphi(x) = 
\e^{\I\alpha\cdot\beta/(2\eps)} \e^{\I \alpha\cdot x/\eps} \varphi(x+\beta), \quad x\in\Rb^d.
\] 
\end{lemma}

\begin{proof}
We use the unitary evolution problem for the unbounded self-adjoint operator $\alpha\cdot \widehat q + \beta\cdot \widehat p$ for Schwartz class initial data,
\[
\I\eps\partial_t\psi = (\alpha\cdot \widehat q + \beta\cdot\widehat p)\psi,\qquad \psi(0)=\varphi,
\]
and verify that its solution $\psi(t,x) = \e^{-\I(\alpha\cdot \widehat q + \beta\cdot \widehat p)\,t/\eps}\varphi(x)$ satisfies  
\[
\psi(t,x) = \e^{\I\alpha\cdot\beta t^2/(2\eps)} \e^{-\I \alpha\cdot x t/\eps} \varphi(x-t\beta).
\]
Indeed, we differentiate the above formula and obtain 
\begin{align*}
\I\eps\partial_t\psi(t,x) 
&= (-\alpha\cdot\beta t + \alpha\cdot x)\psi(t,x) 
- \I\eps \beta \e^{\I\alpha\cdot\beta t^2/(2\eps)} \e^{-\I \alpha\cdot x t/\eps} \nabla\varphi(x-t\beta)\\
&= (\alpha\cdot x -\I\eps\beta\cdot\partial_x)\psi(t,x) = (\alpha\cdot\widehat q +\beta\cdot\widehat p)\psi(t,x). 
\end{align*}
Then our claim follows by setting $t=-1$. 
\end{proof}

Using the  formula of the lemma for the phase space translation in the one-dimensional case $d=1$, we rewrite the operator $\op(a)$ as an integral operator,
\begin{align*}
&\op(a)\varphi(x) \\
&=\frac{1}{2\pi\eps} \int_{\Rb^2} \calF_\eps a(\alpha,\beta)\  
\e^{\I(\alpha \widehat q + \beta \widehat p)/\eps} \varphi(x) d(\alpha,\beta)\\
&= \frac{1}{(2\pi\eps)^{2}} \int_{\Rb^4} a(w,z)\,
\e^{\I\alpha(x+\beta/2-w)/\eps -\I\beta z/\eps}\, \varphi(x+\beta) \; d(w,z,\alpha,\beta)\\
&=\frac{1}{2\pi\eps} \int_{\Rb^2} a(x+\tfrac{1}{2}\beta,z)\, \e^{-\I \beta z/\eps} \varphi(x+\beta) d(z,\beta)\\
&=\frac{1}{2\pi\eps} \int_{\Rb^2} a(\tfrac12(x+y),z) \e^{\I z(x-y)/\eps} \varphi(y) d(z,y).
\end{align*}
We naturally arrive at the following definition of Weyl-quantised Schwartz functions. We recall that a Schwartz function on $\Rb^{2d}$ is a smooth ({\it i.e.} infinitely differentiable) function that decays faster than the inverse of any polynomial.

\begin{definition}[Weyl-quantised Schwartz function] For a Schwartz function $a:\Rb^{2d}\to\Rb$, we define the hermitian operator
$\op(a)$ on $L^2(\Rb^d)$ by setting, for all 
$\varphi\in L^2(\Rb^d)$,
\begin{equation}\label{def:op}
\op(a)\varphi(x) = \int_{\Rb^d} \kappa_a(x,y) \varphi(y)\D y
\end{equation}
with
\[
\kappa_a(x,y) = (2\pi\eps)^{-d}\int_{\Rb^d} a(\tfrac12(x+y),p)\ \e^{\I p\cdot(x-y)/\eps} dp.
\]
\end{definition}

We note that the integral kernel $\kappa_a$ can also be used to define the Weyl quantisation of smooth 
polynomially bounded functions $a:\Rb^{2d}\to\Rb$, if the partial Fourier transform 
is understood in a distributional sense and the domain of the operator $\op(a)$ is chosen 
carefully. In particular, we may view the Hamiltonian operator 
\[
H=-\tfrac{\eps^2}{2}\Delta_x + V
\] 
as the semiclassically scaled Weyl quantisation $H = \op(h)$  
of the classical Hamilton function $h(q,p) = \tfrac12|p|^2 + V(q)$.  

%\subsection{Semiclassical commutator estimates}
\subsection{Quantum versus classical evolution of observables}

For expectation values $\langle A \rangle_{\psi(t)} = \langle \psi(t)\,|\, A\psi(t)\rangle$ of time-evolved wave functions
\[
\psi(t) = \exp\Bigl(\frac t {\I\eps}H\Bigr)\psi(0)
\]
%for rather general square integrable initial data $\psi(0)\in L^2(\Rb^d)$, 
we observe 
that the time derivative
\begin{equation} \label{wig:comm}
\frac{d}{dt} \langle A \rangle_{\psi(t)} = \bigl\langle \tfrac{1}{\I\eps} [A,H] \bigr\rangle_{\psi(t)}
\end{equation}
involves the {\it commutator}
\[
[A,H] = AH -HA
\] 
of the observable $A$ and the Hamiltonian operator 
$H$. 

On the other hand, along the classical Hamiltonian equations of motion
\begin{align*}
\dot q &= \phantom{+} \nabla_p h(q,p),%\qquad q(0)=q_0,
\\
\dot p &= - \nabla_q h(q,p),%\qquad p(0)=p_0,
\end{align*}
or rewritten in the combined phase space variables $z=(q,p)\in\Rb^{2d}$,
$$
\dot z = J^{-1} \nabla h(z), \qquad \text{where} \quad J= 
\bch\begin{pmatrix} 0 & -\Id_d \\ \Id_d & 0 \end{pmatrix}\ech, \ J^{-1}=-J,
$$
the classical observable $a(z(t))$ changes according to
\begin{equation}\label{wig:poisson}
\frac d{d t}\, a(z) = \nabla a(z) \cdot \dot z = \nabla a(z) \cdot J^{-1} \nabla h(z) =: \{ a, h \}(z),
\end{equation}
which is the {\it Poisson bracket} of the functions $a$ and $h$ at $z$. Like the commutator above, the Poisson bracket is skew-symmetric and satisfies the Jacobi identity. We note that for $h(q,p)=\tfrac12 |p|^2 + V(q)$, we have
$\{ a,h \} =  \nabla_q a \cdot p - \nabla_p a\cdot \nabla V$.

In the following, we prove a result of quantum-classical correspondence,  
in the sense that the scaled commutator of the Weyl operators $A=\op(a)$ and $H=\op(h)$ is approximately the Weyl-quantised Poisson bracket of $a$ and $h$,
\[
\tfrac{1}{\I\eps} [A,H] = \op(\{a,h\}) + \mathcal O(\eps^2).
\]
The arguments of our proof involve Taylor expansion and partial integration, 
up to a final estimate that relies on the Calder\'on--Vaillancourt theorem on 
the $L^2$-continuity of pseudo-differential operators, \bch see for example \cite[Theorem~2.8.1]{Mar02}. \ech

\begin{proposition}[commutator estimate]\label{prop:comm}
Let the phase space function $a:\Rb^{2d}\to\Rb$ that defines the observable $A=\op(a)$ and the potential 
$V:\Rb^d\to\Rb$ both be smooth functions such that their derivatives of order $\ge 3$ are all bounded. 
Then, there exists a constant $C<\infty$ such that for all $\eps>0$ and all Schwartz 
functions $\varphi:\Rb^d\to\C$
\[
\|\bigl(\tfrac{1}{\I\eps} [A,H] - \op(\{a,h\})\bigr)\varphi\| \le C \eps^2 \|\varphi\|.
\]
If $a$ or $V$ is a polynomial of degree at most $2$, then $C=0$. 
\end{proposition}

\begin{proof}
We first compute the commutator with the Laplacian. We obtain by partial integration 
\begin{align*}
\op(a)\Delta_x\varphi(x) &= 
(2\pi\eps)^{-d} \int_{\Rb^{2d}} a(\tfrac{x+y}{2},p) \e^{\I p\cdot(x-y)/\eps}\Delta_y\varphi(y) \,\bch\D(p,y)\ech\\
&=(2\pi\eps)^{-d} \int_{\Rb^{2d}} \Delta_y \left(a(\tfrac{x+y}{2},p) \e^{\I p\cdot(x-y)/\eps}\right)\varphi(y) \,\bch \D(p,y)\ech.
\end{align*}
Now we exchange the variable of differentiation according to
\begin{align*}
&\Delta_y \left(a(\tfrac{x+y}{2},p) \e^{\I p\cdot(x-y)/\eps}\right)\\
&=\Delta_x \left(a(\tfrac{x+y}{2},p) \e^{\I p\cdot(x-y)/\eps}\right) 
- \tfrac{2\I}{\eps}\nabla_x a(\tfrac{x+y}{2},p)\cdot p \, \bch \e^{\I p\cdot(x-y)/\eps}\ech.
\end{align*}
This implies
\[
\op(a)\Delta_x\varphi(x) = \Delta_x\op(a)\varphi(x) - \tfrac{2\I}{\eps}\op( \nabla_q a \cdot p)\varphi(x),
\]
so we obtain the identity
\[
\tfrac{1}{\I\eps}[\op(a),-\tfrac{\eps^2}{2}\Delta_x] = \op(\nabla_q a\cdot p).
\]
Next we compute the commutator with the potential, 
\begin{align}\label{eq:commV}
&[\op(a),V]\varphi(x) \\\nonumber
&= -
(2\pi\eps)^{-d} \int_{\Rb^{2d}} a(\tfrac{x+y}{2},p) \left(V(x) -V(y)\right)\e^{\I p\cdot(x-y)/\eps}\varphi(y) \,
\bch \D(p,y)\ech.
\end{align}
Taylor expansion around the midpoint yields
\[
V(x) - V(y) = \nabla V(\tfrac{x+y}{2})^T (x-y) + \tfrac14 \sum_{|m|=3}  r_m(x,y) (x-y)^m,
\]
where for multi-indices $m=(m_1,\dots,m_d)\in\N_0^d$ with $|m|=\sum_{i=1}^d m_i =3$, we write $(x-y)^m=\prod_{i=1}^d (x_i-y_i)^{m_i}$ and
$\partial^mV =\partial_1^{m_1}\ldots\partial_d^{m_d}V$, and
\[
r_m(x,y) = \int_0^1 \tfrac12 (1-\theta)^{2} 
\left( \partial^m V(\tfrac{x+y}{2}+\theta \tfrac{x-y}{2}) 
- \partial^m V(\tfrac{x+y}{2}+\theta \tfrac{y-x}{2})\right) \bch \D\theta.\ech
\]
 We observe that
\[
(x-y) \e^{\I p\cdot(x-y)/\eps} = \tfrac{\eps}{\I}\,\nabla_p \e^{\I p\cdot(x-y)/\eps}.
\]
Inserting the first summand $\nabla V(\tfrac{x+y}{2})\cdot (x-y)$ back into the integral, we perform a partial integration and obtain
\begin{align*}
& (2\pi\eps)^{-d}  \int_{\Rb^{2d}} a(\tfrac{x+y}{2},p) \,\nabla V(\tfrac{x+y}{2})\cdot
\tfrac{\eps}{\I}\nabla_p\e^{\I p\cdot(x-y)/\eps}\varphi(y) \,\bch \D(p,y)\ech \\
&= \I\eps\, \op(\nabla_p a\cdot \nabla V) \varphi(x).
\end{align*}
Analogously, we obtain for the terms in the second summand
\begin{align*}
&(2\pi\eps)^{-d}  \int_{\Rb^{2d}} a(\tfrac{x+y}{2},p) \,r_m(x,y)\, (\tfrac{\eps}{\I}\partial_p)^m 
\e^{\I p\cdot(x-y)/\eps}\varphi(y) \,\bch \D(p,y)\ech\\
&= (2\pi\eps)^{-d}  \int_{\Rb^{2d}} (-\tfrac{\eps}{\I}\partial_p)^m a(\tfrac{x+y}{2},p) \,r_m(x,y)\, 
\e^{\I p\cdot(x-y)/\eps}\varphi(y) \,\bch \D(p,y).\ech
\end{align*}
To bound these integrals we use a Calder\'on--Vaillancourt theorem, see for example \cite[Theorem~2.8.1]{Mar02}.
These theorems prove boundedness for a large class of pseudo-differential operators that contains  
Weyl-quantised operators provided that their symbols have bounded derivatives: 
Let $b:\Rb^{3d}\to\C$ be a smooth function with bounded derivatives and $\widetilde\op(b)$ the operator with the integral kernel
\[
\kappa_b(x,y) = (2\pi\eps)^{-d}\int_{\Rb^d} b(x,y,p)\ \e^{\I p\cdot(x-y)/\eps} dp.
\]
Then, $\widetilde\op(b)$ is a bounded operator on $L^2(\Rb^d)$ with  
\begin{equation}\label{eq:cv}
\|\widetilde\op(b)\| \le C \sum_{|m|\le M} \|\partial^m b\|_\infty, 
\end{equation}
where the positive constants $C,M>0$ only depend on the dimension $d$. 
In our case, we set
\[
b(x,y,p) = \tfrac14 \sum_{|m|=3}  \partial^m_p a(\tfrac{x+y}{2},p) \,r_m(x,y)
\]
and summarise our calculations as
\[
\tfrac{1}{\I\eps}[\op(a),\op(h)] =  \op(\{a,h\}) + \eps^2\, \widetilde\op(b).
\]
Applying the Calder\'on--Vaillancourt estimate \eqref{eq:cv}, we conclude our proof.
\end{proof}

%\subsection{Quantum vs. classical evolution of observables}
The correspondence between commutators of operators and Poisson brackets of functions 
can be lifted to a similar relation between quantum and classical propagation of observables. 
For the formulation of this result, we consider the classical flow map $\Phi^t:\Rb^{2d}\to\Rb^{2d}$ of the Hamilton function 
$h$, that is, $\Phi^t(z_0)$  equals the solution $z(t)$ of the  Hamiltonian differential equations with initial value $z_0$.
The classical flow $\Phi^t$ thus satisfies
\[
\partial_t \Phi^t = J^{-1} \nabla h\circ\Phi^t,\qquad \Phi^t|_{t=0} = \Id_{2d},
\]
and for functions $a$ that evolve along this flow, we have from \eqref{wig:poisson} (noting energy conservation $h=h\circ\Phi^t$)
\[
\partial_t (a\circ\Phi^t) = \{a\circ\Phi^t,h\}.
\]
%\begin{lemma}[Poisson bracket and classical flow]\label{lem:poisson}
%Let $a:\Rb^{2d}\to\C$ be a smooth function. Then, for all $t\in\Rb$,  
%\[
%\partial_t (a\circ\Phi^t) = \{h,a\circ\Phi^t\}.
%\]
%\end{lemma}
%
%\begin{proof}
%We first calculate the time derivative, 
%\begin{align*}
%\partial_t (a\circ\Phi^t )&= (\nabla a\circ\Phi^t) \cdot\partial_t \Phi^t 
%= (\nabla a\circ\Phi^t) \cdot J^T (\nabla h\circ\Phi^t)\\
%&= (\nabla h\circ\Phi^t) \cdot J (\nabla a\circ\Phi^t).
%\end{align*}
%By conservation of energy $h$ and symplecticity of the Jacobian $D\Phi^t$,
%\begin{align*}
%\{h,a\circ\Phi^t\} &= \{h\circ\Phi^t,a\circ\Phi^t\}
%=\nabla (h\circ\Phi^t)^T J \,\nabla (a\circ\Phi^t)\\
%&=  (\nabla h\circ\Phi^t)^T \, D\Phi^t J (D\Phi^t)^T (\nabla a\circ\Phi^t)\\
%&= (\nabla h\circ\Phi^t) \cdot J\, (\nabla a\circ\Phi^t),
%\end{align*}
%which yields the result.
%\end{proof}
%Lemma~\ref{lem:poisson} in combination with the commutator estimate of 
Proposition~\ref{prop:comm} yields the following key result 
of quantum-classical correspondence for the evolution of expectation values.  

\begin{theorem}[Egorov's theorem]\label{theo:egorov}
We consider classically evolved expectation values. We assume the following:
\begin{enumerate}
%\item[1.] 
%%The momentum derivatives of order greater or equal than one 
%The derivatives of the classical flow $\Phi^t$  with respect to the initial values 
%are bounded, for all orders $m\ge 1$, by $c_m\,\e^{\lambda t}$ for all times $t\in[0,\overline t]$. 
\item[1.] The potential function $V$ is smooth, and its derivatives of order $\ge 2$ are all bounded.
\item[2.] The function $a$ that defines the observable $A=\op(a)$ is smooth and 
bounded together with all its derivatives. 
\end{enumerate}
Then, the error between the classically evolved expectation value and that of the 
 solution $\psi(t)$ of the Schr\"odinger equation with Schwartz class initial data of unit norm satisfies
\[
\left| \langle A \rangle_{\psi(t)} - 
\langle \op(a\circ\Phi^t)\rangle_{\psi(0)} \right| \le c\,t\, \eps^2,
\qquad 0\le t\le \overline t, 
\]
where $c$ depends on the derivative bounds 
of $a$ and $V$ and on $\bar t$ but is independent of  $\psi(0)$, $\eps$ and $t\le \bar t$.
\end{theorem}

\begin{proof}
We first compare quantum and classical evolution,
\begin{align*}
&\e^{\I tH/\eps}\,\op(a)\,\e^{-\I tH/\eps} - \op(a\circ\Phi^t)\\
&= \int_0^t \tfrac{d}{ds} \left( \e^{\I sH/\eps}\,\op(a\circ\Phi^{t-s})\,\e^{-\I sH/\eps}\right) \D s\\
&= \int_0^t \e^{\I sH/\eps}
\left( \tfrac{1}{\I\eps}[\op(a\circ\Phi^{t-s}),H]-\partial_t\,\op(a\circ\Phi^{t-s})\right)\e^{-\I sH/\eps}\D s,
\end{align*}
where we note that
$$
\partial_t\,\op(a\circ\Phi^{t-s}) = \op( \partial_t(a\circ\Phi^{t-s}))= \op (\{a\circ\Phi^{t-s},h\}).
$$
This implies
\begin{align*}
&\langle \psi(t)\,|\, A \psi(t) \rangle - \langle \psi(0)\,|\, \op(a\circ\Phi^t) \,\psi(0)\rangle \\
&= 
\int_0^t \langle \psi(s)\,|\,
\left( \tfrac{1}{\I\eps}[\op(a\circ\Phi^{t-s}),H]-\op (\{a\circ\Phi^{t-s},h\})\right)\psi(s)\rangle \D s.
\end{align*}
Since the Schr\"odinger evolution is norm-conserving, we then obtain
\begin{align*}
&\left|\langle \psi(t)\,|\, A \psi(t) \rangle - \langle \psi(0)\,|\, \op(a\circ\Phi^t) \,\psi(0)\rangle \right|\\
&\le 
\int_0^t \|\left( \tfrac{1}{\I\eps}[\op(a\circ\Phi^{t-s}),H]-\op (\{a\circ\Phi^{t-s},h\})\right)\psi(s)\| \D s.
\end{align*}
The assumptions on $a$ and $V$ ensure that $a\circ\Phi^t$ has bounded derivatives of all orders uniformly on bounded time intervals.
By the commutator estimate of Proposition~\ref{prop:comm}, there exists a constant $c$ such that for $0\le s \le t \le \bar t$
$$%\begin{align*}
\| \left( \tfrac{1}{\I\eps}[\op(a\circ\Phi^{t-s}),H]-\op (\{a\circ\Phi^{t-s},h\})\right)\psi(s)\| 
\le c \,\eps^2.
$$%\end{align*}
Inserting this bound into the above inequality proves the result.
\end{proof}

\begin{remark}
The above argument not only proves an error bound for expectation values but also 
provides an evolution estimate in operator norm. Revisiting 
the proof, we observe, that under the assumptions of Theorem~\ref{theo:egorov},
\begin{equation}\label{wig:egorov-norm}
\left\|\e^{\I tH/\eps}\,\op(a)\,\e^{-\I tH/\eps} - \op(a\circ\Phi^t)\right\| \ \le\  c\, t\, \eps^2,
\qquad 0\le t\le \overline t.
\end{equation}
\end{remark}

\bigskip
\begin{corollary}[norm and energy conservation]
The approximation provided by Egorov's theorem conserves norm and energy. 
\end{corollary}

\begin{proof} Norm and energy are associated with the observables $a=1$ and $a=h$, respectively. 
We have
\[
\frac{d}{dt}\left\langle \op(1\circ\Phi^t)\right\rangle_{\psi(0)} = 0\quad\text{and}\quad
\frac{d}{dt}\left\langle \op(h\circ\Phi^t)\right\rangle_{\psi(0)} = 0.
\]
\end{proof}

\subsection{Wigner functions}\label{subsec:wigner}

Egorov's theorem turns into a computational method for the simulation of expectation values when 
describing the inner products involving Weyl-quantised operators $A=\op(a)$ as phase space averages 
weighted by the wave function's Wigner function. Indeed, we observe that for Schwartz functions $a:\Rb^{2d}\to\Rb$ 
and $\psi\in L^2(\Rb^d)$, the measure-preserving change of variables 
$\Rb^{2d}\to\Rb^{2d}$, $(q,y)\mapsto(q+\tfrac12 y,q-\tfrac12 y)$ 
yields
\begin{align*}
\langle A\rangle_\psi 
&= \int_{\Rb^d} \overline{\psi}(q)\, \op(a)\psi(q) \D q\\
&= (2\pi\eps)^{-d} \int_{\Rb^{3d}} \overline{\psi}(q) \,
a(\tfrac{1}{2}(q+y),p) \,\e^{\I p\cdot (q-y)/\eps} \psi(y)\D(q,p,y)\\
&= \int_{\Rb^{2d}} a(q,p)\, 
(2\pi\eps)^{-d} \int_{\Rb^d} \overline{\psi}(q+\tfrac{1}{2}y) \, \psi(q-\tfrac{1}{2}y)
\, \e^{\I p\cdot y/\eps} \D y \D(q,p).
\end{align*}
This simple calculation has far-reaching consequences, and so we reformulate the result as a combined theorem and definition.

\begin{theorem}[Wigner function]\label{theo:wigner function} Let $A=\op(a)$ be the Weyl quantisation of a Schwartz function $a:\Rb^{2d}\to\Rb$.
Then, the expectation value of the observable $A$ with respect to  a wave function $\psi\in L^2(\Rb^d)$ equals
\begin{equation} \label{wig:av-wig}
\langle A\rangle_\psi = \int_{\Rb^{2d}} a(z)\, \calW_\psi(z)\D z,
\end{equation}
where the {\it Wigner function} $\calW_\psi:\Rb^{2d}\to\Rb$ is given by
\[
\calW_\psi(q,p) = (2\pi\eps)^{-d} \int_{\Rb^d}  \overline{\psi}(q+\tfrac12 y)\psi(q-\tfrac12 y) \,
\e^{\I p\cdot y/\eps} \D y.
\]
\end{theorem}

%This calculation motivates the following definition. 
%
%\begin{definition}[Wigner function]
%For a function $\psi\in L^2(\Rb^d)$ we define its Wigner function $\calW_\psi:\Rb^{2d}\to\Rb$, 
%\[
%\calW_\psi(q,p) = (2\pi\eps)^{-d} \int_{\Rb^d}  \overline{\psi}(q+\tfrac12 y)\psi(q-\tfrac12 y) 
%\e^{\I p\cdot y/\eps}\D y.
%\]
%\end{definition}
%
%With this definition, we may write the previous expectation value as a phase space average weighted by the 
%Wigner function according to the key identity
%\begin{equation} \label{wig:av-wig}
%\langle A\rangle_\psi = \int_{\Rb^{2d}} a(z)\, \calW_\psi(z)\D z.
%\end{equation}

\medskip
The values of Wigner functions are real, but not always positive, since for example 
$\psi(x) = -\psi(-x)$ implies  \bch $\calW_\psi(0) = -(\pi\eps)^{-d} \|\psi\|^2\le 0$\ech.
One can even fully characterize the square-integrable functions with non-negative Wigner function by 
\[
\calW_\psi \ge 0 \quad \text{if and only if}\quad  \psi\ \text{is a complex-valued Gaussian};
\]
see \cn{Hud74} and \cn{SotC83}. In particular, the Wigner function of a complex-valued Gaussian wave packet 
is a real-valued phase space Gaussian as Proposition~\ref{lem:WG} later in this section will prove.

\medskip
Despite its general lack of positivity, the Wigner function has several properties of a simultaneous probability density with respect to positions and momenta. For $\psi\in L^2(\Rb^d)$, we let 
\[
\calF_\eps \psi(p) = (2\pi\eps)^{-d/2} \int_{\Rb^d} \psi(x) \, \e^{-\I x\cdot p/\eps} \D x,\qquad p\in\Rb^d,
\]
denote the scaled Fourier transform of the function $\psi$. 

\begin{lemma}[marginals of the Wigner function] Let $\psi\in L^2(\Rb^d)$. Then, for all $(q,p)\in\Rb^{2d}$, 
\[
|\psi(q)|^2 = \int_{\Rb^d} \calW_\psi(q,p) \D p\quad\text{and}\quad
|\calF_\eps\psi(p)|^2 = \int_{\Rb^d} \calW_\psi(q,p) \D q.
\]
In particular, 
\[
\|\psi\|^2 = \int_{\Rb^{2d}} \calW_\psi(z)\D z.
\]
\end{lemma}

\begin{proof} The formula for the position density follows by Fourier inversion:
with $\sigma_q(y) =   \overline{\psi}(q+\tfrac12 y)\psi(q-\tfrac12 y) $,
\begin{align*}
\int_{\Rb^d} \calW_\psi(q,p) \D p 
&=  (2\pi\eps)^{-d/2}\int_{\Rb^{d}}  (2\pi\eps)^{-d/2}\int_{\Rb^{d}} \sigma_q(y) \,
\e^{\I p\cdot y/\eps} \D y \D p \\
&=  (2\pi\eps)^{-d/2}\int_{\Rb^{d}} (\calF_\eps^{-1} \sigma_q)(p) \D p = \sigma_q(0)=
|\psi(q)|^2.
\end{align*}
The formula for the momentum density then follows from the following lemma, in which the roles of $q$ and $p$ are exchanged by Fourier transformation.
%Using Lemma~\ref{lem:FW}, we also obtain
%\begin{align*}
%\int_{\Rb^d} \calW_\psi(q,p) dq = \int_{\Rb^d} \calW_{\calF_\eps\psi}(p,-q) dq 
% = |\calF_\eps\psi(p)|^2.
%\end{align*}
\end{proof}

%The duality between Wigner transformation and Weyl quantisation is complemented by 
%a simple relation to the Fourier transform. 

\begin{lemma}[Fourier and Wigner transformation]\label{lem:FW}
%For $\psi\in L^2(\Rb^d)$, we denote by 
%\[
%\calF_\eps \psi(p) = (2\pi\eps)^{-d/2} \int_{\Rb^d} \psi(x) \, \e^{-\I x\cdot p/\eps} \D x,\qquad p\in\Rb^d,
%\]
%the scaled Fourier transform of the function $\psi$. 
For $(q,p)\in\Rb^{2d}$,
\[
\calW_{\calF_\eps\psi}(q,p) = \calW_{\psi}(-p,q)\quad\text{and}\quad
\calW_{\calF_\eps^{-1}\psi}(q,p) = \calW_{\psi}(p,-q).
\]
\end{lemma}

\begin{proof} We have 
\begin{align*}
&\calW_{\calF_\eps\psi}(q,p) \\
&= 
(2\pi\eps)^{-2d} \int_{\Rb^{3d}} \overline\psi(x) \psi(\widetilde x) 
\e^{\I x\cdot(q+y/2)/\eps - \I\widetilde x\cdot(q-y/2)/\eps + \I p\cdot y/\eps}\D(x,\widetilde x,y).
\end{align*}
We write
\[
x\cdot(q+\tfrac12y)-\widetilde x\cdot(q-\tfrac12y) +p\cdot y = 
q\cdot(x-\widetilde x) + y\cdot(\tfrac12 x + \tfrac12\widetilde x + p)
\]
such that by Fourier inversion
\begin{align*}
&\calW_{\calF_\eps\psi}(q,p) \\
&= 
(2\pi\eps)^{-2d}\, 2^d \int_{\Rb^{3d}} \overline\psi(x) \psi(2\widetilde x) 
\e^{\I q\cdot(x-2\widetilde x)/\eps + y\cdot(x/2 + \widetilde x + p)}
\D(x,\widetilde x,y)\\
&=
(2\pi\eps)^{-d}\,2^d \int_{\Rb^{d}} \overline\psi(x) \psi(-2p-x) 
\e^{\I q\cdot(2x+2p)/\eps} \D x\\
&=
(2\pi\eps)^{-d} \int_{\Rb^{d}} \overline\psi(-p+\tfrac12 y) \psi(-p-\tfrac12 y) 
\e^{\I q\cdot y/\eps}\D y \ =\ \calW_\psi(-p,q).
\end{align*}
The second formula is proved in the same way.
\end{proof}

Wigner transformation turns the position and the momentum operator into quadratic 
operators as follows. 

\begin{lemma}[position and momentum operator] \label{wig:lem:pos-mom}
For any Schwartz function $\psi:\Rb^d\to\C$ and for any $j=1,\ldots,d$,
\[
\calW_{\widehat q_j\psi} = (q_j^2+\tfrac{\eps^2}{4}\partial_{p_j}^2)\calW_\psi\quad\text{and}\quad
\calW_{\widehat p_j\psi} = (p_j^2+\tfrac{\eps^2}{4}\partial_{q_j}^2)\calW_\psi.
\]
\end{lemma}

We provide an almost effortless proof of these two formulas, even though 
Proposition~\ref{prop:Wladder} contains them as a special case. 

\begin{proof} We calculate 
\begin{align*}
\calW_{\widehat q_j\psi}(q,p) 
&= (2\pi\eps)^{-d} \int_{\Rb^d}  (q_j^2-\tfrac14y_j^2)\overline{\psi}(q+\tfrac12 y)\psi(q-\tfrac12 y) 
\e^{\I p\cdot y/\eps}\D y\\
&= (q_j^2+\tfrac{\eps^2}{4}\partial_{p_j}^2)\calW_\psi(q,p), 
\end{align*}
since
$
(-\I\eps\partial_{p_j})^2 \e^{\I p\cdot y/\eps} = y_j^2 \e^{\I p\cdot y/\eps}.
$
The second formula then follows from the first formula and  Lemma~\ref{lem:FW}.
%Moreover, by Lemma~\ref{lem:FW},
%\begin{align*}
%\calW_{\widehat p_j\psi}(q,p) &= \calW_{\calF_\eps^{-1}p_j\calF_\eps\psi}(q,p)
%= \calW_{p_j\calF_\eps\psi}(p,-q)\\
%&= (p_j^2+\tfrac{\eps^2}{4}\partial_{q_j}^2)\calW_{\calF_\eps\psi}(p,-q)
%= (p_j^2+\tfrac{\eps^2}{4}\partial_{q_j}^2)\calW_\psi(q,p).
%\end{align*}
\end{proof}

The duality relation between the Weyl quantisation of observables and 
the Wigner transformation of wave functions, as given by Theorem~\ref{theo:wigner function},
immediately yields the following version of Egorov's theorem (Theorem~\ref{theo:egorov}).

\begin{corollary}[Egorov's theorem for Wigner functions]\label{cor:egorow}
Under the assumptions of Theorem~\ref{theo:egorov},
the error between the classically evolved Wigner function and the expectation value of the 
Schr\"odinger solution $\psi(t)$ with initial value $\psi(0)$ of unit norm satisfies
\[
\left| \langle A\rangle_{\psi(t)} - \int_{\Rb^{2d}}  a(\Phi^t(z))\,\calW_{\psi(0)}(z) \D z \right| 
\le c\, t\, \eps^2,
\qquad 0<t\le \bar t,
\]
where the constant $c<\infty$ depends only on the derivative bounds 
of $a$ and~$V$ but is independent of $\psi(0)$, $\eps$ and $t$.
\end{corollary}

This result motivates a numerical
method for the computation of expectation values that is particularly inexpensive and second 
order accurate with respect to the semiclassical parameter~$\eps$, provided that the 
initial Wigner function is accessible:
\begin{enumerate}
\item[1.] Choose a set of numerical quadrature points $z_i\in\Rb^{2d}$
and evaluate the initial Wigner function $\calW_{\psi(0)}$ at the points $z_i$.
\item[2.] Transport the points $z_i$ by the classical flow $\Phi^t$ and
evaluate the classical observable $a$ at the points $\Phi^t(z_i)$ of the classical trajectories.
\item[3.] Form the (possibly weighted) sum of products of the quantities from 1.~and 2.~according to the chosen quadrature rule.
\end{enumerate}

%\begin{proof}
%We rewrite the error estimate of Theorem~\ref{theo:egorov} as
%\[
%\left|  \langle A\rangle_{\psi(t)} - 
%\int_{\Rb^{2d}} \calW_{\psi(0)}(z)\, a(\Phi^{t}(z))\, d z \right| 
%\le c\, t\, \eps^2.
%\]
%Since the classical flow is symplectic, we have
%\[
%\int_{\Rb^{2d}} \calW_{\psi(0)}(z) a(\Phi^{t}(z)) d z = 
%\int_{\Rb^{2d}} \left(\calW_{\psi(0)}\circ\Phi^{-t}\right)\!(z)\, a(z)\, d z,
%\]
%which proves the result.
%\end{proof}

\subsection{Husimi functions and spectrograms} 
It is tempting to view the above computational scheme  
as a particle method that propagates samples from the initial Wigner probability distribution along classical trajectories. 
However, due to Hudson's theorem, this point of view is limited to initial data that are Gaussian wave packets, since they are the only functions with positive Wigner function. 

On the other hand, the following analysis shows that 
%the Wigner function fails to be positive 
%just by  terms of the order~$\eps$, and that 
a systematic correction 
is possible that reconciles sampling of positive phase space densities with classical transport. This is important because it enables the use of Monte Carlo methods, which sample from a probability distribution and are able to approximate phase space integrals in high dimensions and under low regularity requirements.

We now reenter the study of the basic properties of Wigner functions
with the  fundamental result of the next lemma, which states that the convolution of two Wigner 
functions is always non-negative:
$$
(\calW_\varphi*\calW_\psi)(z) =  \int_{\Rb^{2d}} \calW_\varphi(z-w) \,\calW_\psi(w) \D w \ge 0.
$$
Such convolutions are known in the literature 
as {\em spectrograms.} 

\begin{lemma}[spectrograms]\label{lem:spectro} For all $\varphi,\psi\in L^2(\Rb^d)$, we have
%\begin{align*}
%\langle \calW_\varphi\mid\calW_\psi \rangle &= (2\pi\eps)^{-d}\left|\langle \varphi|\psi\rangle\right|^2,\\*[1ex]
$$
(\calW_\varphi*\calW_\psi)(z) = (2\pi\eps)^{-d}\left|\langle\varphi_z|\psi\rangle\right|^2,\qquad z=(q,p)\in\Rb^{2d},
$$
%\end{align*}
where
$\varphi_z(x) = \e^{-\I p\cdot (q-x)/\eps} \varphi(q-x)$ for $ x\in\Rb^d$.
Moreover, 
\[
\int_{\Rb^{2d}} (\calW_\varphi*\calW_\psi)(z) \D z = \|\varphi\|^2\, \|\psi\|^2.
\]
For $\varphi$ and $\psi$ of unit norm, $\calW_\varphi*\calW_\psi$ is thus a probability density on $\Rb^{2d}$.
\end{lemma}

\begin{proof} We begin with the noteworthy identity
$$
\langle \calW_\varphi\mid\calW_\psi \rangle = (2\pi\eps)^{-d}\left|\langle \varphi|\psi\rangle\right|^2.
$$
This is obtained by Fourier inversion,
\begin{align*}
\langle \calW_\varphi\mid \calW_\psi \rangle 
&= \int_{\Rb^{2d}} \overline{\calW_\varphi(z)} \calW_\psi(z)\D z \\
&= (2\pi\eps)^{-2d} \int_{\Rb^{4d}} \varphi(q+\tfrac12 y) \overline\varphi(q-\tfrac12 y) \\
&\qquad\times\overline\psi(q+\tfrac12 y') \psi(q-\tfrac12 y') \,\e^{\I p\cdot (y'-y)/\eps}\D(y,y',q,p)\\*[1ex]
&= (2\pi\eps)^{-d} \int_{\Rb^{2d}} \varphi(q+\tfrac12 y) \overline\varphi(q-\tfrac12 y)
\overline\psi(q+\tfrac12 y) \psi(q-\tfrac12 y) \D(y,q)
\end{align*}
and by the measure-preserving change of variables $(y,q)\mapsto (q+\tfrac12y,q-\tfrac12y)$, 
\begin{align*}
\langle \calW_\varphi \mid \calW_\psi \rangle 
&= (2\pi\eps)^{-d} \int_{\Rb^{2d}} \varphi(x) \overline\varphi(x')
\overline\psi(x) \psi(x') \D(x,x')
%\\*[1ex]
%&
= (2\pi\eps)^{-d}\left|\langle \varphi|\psi\rangle\right|^2.
\end{align*}
We now turn to the convolution $\calW_\varphi*\calW_\psi$. We denote $z=(q_z,p_z)$ and $w=(q_w,p_w)\in\Rb^{2d}$, and write
$$
\calW_\varphi(z-w)
= (2\pi\eps)^{-d} \int_{\Rb^d}
 \overline\varphi(q_z-q_w+\tfrac12y) \varphi(q_z-q_w-\tfrac12y) \e^{\I (p_w-p_z)\cdot y/\eps} \D y,
$$
where the integrand equals $\overline\varphi_z(q_w+\tfrac12y) \varphi_z(q_w-\tfrac12y) \e^{\I p_w\cdot y/\eps}$. This yields
$$
\calW_\varphi(z-w) = \overline{\calW_{\varphi_z}(w)},
$$
and hence  the convolution becomes
\begin{align*}
(\calW_\varphi*\calW_\psi)(z) 
&= \int_{\Rb^{2d}} \calW_\varphi(z-w) \calW_\psi(w) \D w
= \int_{\Rb^{2d}} \overline{\calW_{\varphi_z}(w)} \, \calW_\psi(w) \D w
\\
&= \langle \calW_{\varphi_z}\mid\calW_\psi\rangle = (2\pi\eps)^{-d}\left|\langle \varphi_z|\psi\rangle\right|^2.
\end{align*}
To calculate the integral of the spectrogram, we again use Fourier inversion, 
\begin{align*}
&\int_{\Rb^{2d}} (\calW_\varphi*\calW_\psi)(z) \D z = 
(2\pi\eps)^{-d} \int_{\Rb^{2d}} \langle\varphi_z|\psi\rangle \langle \psi|\varphi_z\rangle \D z\\
&=
(2\pi\eps)^{-d} \int_{\Rb^{4d}} \overline\varphi(q-x) \psi(x) \overline\psi(y) \varphi(q-y) \e^{\I p\cdot (y-x)}\, d(x,y,q,p) \\
&= \int_{\Rb^{4d}} \overline\varphi(q-x) \psi(x) \overline\psi(x) \varphi(q-x)\D (x,q)
= \|\varphi\|^2\, \|\psi\|^2.
\end{align*} 
\end{proof}

The most popular spectrogram stems from the convolution with a standard Gaussian function 
of unit width and phase space centre zero. 

\begin{definition}[Husimi function]
For a function $\psi\in L^2(\Rb^d)$ we define its {\em Husimi function} $\calH_\psi:\Rb^{2d}\to[0,\infty)$, 
\[
\calH_\psi(z) = (\calW_{g_0}*\calW_\psi)(z) = (2\pi\eps)^{-d}\, \left|\left\langle g_z|\psi\right\rangle\right|^2,
\]
where 
\[
g_z(x) = (\pi\eps)^{-d/4} \exp(-\tfrac{1}{2\eps}|x-q|^2 + \tfrac\I\eps\, p\cdot(x-q)),\qquad x\in\Rb^d,
\]
is a unit width Gaussian centred in the phase space point $z=(q,p)\in\Rb^{2d}$.
\end{definition}

The Husimi function has the benefit of being a non-negative phase space distribution, $\calH_\psi\ge0$, 
and this is accompanied by a convenient normalisation property. Since the standard Gaussian $g_0$ 
has unit norm, Lemma~\ref{lem:spectro} yields that 
\[
\int_{\Rb^{2d}} \calH_\psi(z) \D z = \|\psi\|^2.
\]
The price to be paid is the loss of the exact duality relation for expectation values with 
Weyl-quantised observables,  as enjoyed by the Wigner function. The following estimates 
quantify how much Husimi averages deviate from the expectation value. 

\begin{theorem}[expectation values]\label{theo:husimi} 
Let $a:\Rb^{2d}\to\Rb$ be a smooth function that is bounded together with all its derivatives, 
and consider its Weyl quantisation $A=\op(a)$. Then, for  all $\psi\in L^2(\Rb^d)$,   
\[
\left|\langle A\rangle_\psi - 
\int_{\Rb^{2d}} (a-\tfrac\eps4\Delta a)(z) \,\calH_\psi(z) \D z\right| \le C\, \eps^2 \|\psi\|^2,
\]
where the constant $C$ depends on derivatives of $a$ of order $\ge 4$. 
\end{theorem}

In particular, this implies
\[
\left|\langle A\rangle_\psi - 
\int_{\Rb^{2d}} a(z) \,\calH_\psi(z) \D z\right| \le c\, \eps \|\psi\|^2
\]
for some constant $c$ that depends on derivatives of $a$ of order $\ge 2$.

\begin{proof} 
We use the symmetry $\calW_{g_0}(-z) = \calW_{g_0}(z)$ to obtain 
\begin{align*}
\int_{\Rb^{2d}} a(z) \calH_\psi(z) \D z 
&= \int_{\Rb^{2d}} a(z) (\calW_{g_0}*\calW_\psi)(z) \D z\\
&= \int_{\Rb^{4d}} a(z)  \calW_{g_0}(z-w)\,\calW_\psi(w)\D (w,z)\\
&= \int_{\Rb^{2d}} (a*\calW_{g_0})(z)\, \calW_\psi(z) \D z.
\end{align*}
We then write the convolution, 
\[
(a*\calW_{g_0})(z) = (\pi\eps)^{-d} \int_{\Rb^{2d}} a(z-w) \,\e^{-|w|^2/\eps} \D w, 
\]
and perform a third-order Taylor expansion, 
\begin{align*}
a(z-w) &= a(z) - \nabla a(z)^T w + \tfrac12 w^T \nabla^2 a(z) w - 
\tfrac{1}{3!} \sum_{|m|=3} w^m \partial^m a(z) \\
&+ \tfrac{1}{3!} \sum_{|m|=4} w^m \int_0^1 (1-\theta)^3 \partial^m a(z-\theta w) \D\theta.
\end{align*}
We observe that the odd moments vanish, 
\[
\int_{\Rb^{2d}} w^m \,\e^{-|w|^2/\eps} \D w = 0,\qquad |m| \text{ odd},
\]
and calculate the second order contribution,  
\begin{align*}
&\tfrac12 (\pi\eps)^{-d}\int_{\Rb^{2d}} w^T \nabla^2 a(z) w \,\e^{-|w|^2/\eps} \D w \\
&= 
\tfrac12 \sum_{k,l=1}^{2d} \nabla^2 a(z)_{kl} \, 
(\pi\eps)^{-d}\int_{\Rb^{2d}} w_k w_l \,\e^{-|w|^2/\eps} \D w\\
&= \tfrac12 \Delta a(z)\ \frac{1}{\sqrt{\pi\eps}} \int_{\Rb} x^2 \,\e^{-x^2/\eps} \D x 
\, =\,  \tfrac\eps4 \Delta a(z).
\end{align*}
Then, 
\[
(a*\calW_{g_0})(z) = a(z) + \tfrac\eps4\Delta a(z) + \eps^2 r(z)
\]
with 
\[
r(z) = \tfrac{1}{3!} \sum_{|m|=4} \pi^{-d}
\int_{\Rb^{2d}} \int_0^1 w^m (1-\theta)^3 \partial^m a(z-\theta w) \e^{-|w|^2} \D\theta \D w.
\]
This implies
\begin{align*}
&\int_{\Rb^{2d}} a(z) \calH_\psi(z)\D z - \int_{\Rb^{2d}} (a+\tfrac{\eps}{4}\Delta a)(z) \calW_\psi(z)\D z
%\\*[1ex]
%&\qquad
= \eps^2 \int_{\Rb^{2d}} r(z) \calW_\psi(z)\D z,
\end{align*}
where, by the Calder\'on--Vaillancourt theorem, see for example \cite[Theorem~2.8.1]{Mar02},  
\[
\left|\int_{\Rb^{2d}} r(z) \calW_\psi(z)\D z\right| = \left|\langle \psi\ |\ \op(r)\psi\rangle\right| 
\le C \|\psi\|^2
\]
with a constant $C>0$ that depends on derivatives of $a$ of order $\ge 4$.  
With $\Delta a$ instead of $a$ we also have
\begin{align*}
&\int_{\Rb^{2d}} \! \Delta a(z) \calH_\psi(z)\D z - \!\int_{\Rb^{2d}} (\Delta a+\tfrac{\eps}{4}\Delta^2 a)(z) \calW_\psi(z)\D z
%\\*[1ex]
%&\qquad
= \eps^2 \!\!\int_{\Rb^{2d}} \widetilde r(z) \calW_\psi(z)\D z,
\end{align*}
where the right-hand side with the remainder $\widetilde r$ is bounded like the corresponding integral with $r$ above. Subtracting $\eps/4$ times the equation for $\Delta a$ from that for $a$ then yields the result.
\end{proof}

The Husimi function has the same first moments as the Wigner function, but otherwise fails to 
reproduce expectation values with an error of order~$\eps$. However, due to its explicit form
one can use the error term as an additive correction to the Husimi function on noting  that
$$
\int_{\Rb^{2d}} (a-\tfrac\eps4\Delta a)(z) \,\calH_\psi(z) \D z =
\int_{\Rb^{2d}}  a(z) \,\Bigl(\calH_\psi - \tfrac\eps4\Delta\calH_\psi\Bigr)(z) \D z.
$$
\begin{definition}[Husimi correction]
For a  function $\psi\in L^2(\Rb^d)$ we define its {\em
Husimi correction} by  
\[
\calH^{\rm c}_\psi = \calH_\psi - \tfrac\eps4\Delta\calH_\psi.
\]
\end{definition}

By its construction, the Husimi correction reproduces expectation values up to an error that is second 
order with respect to $\eps$ and thus perfectly compatible with the approximation error of Egorov's theorem. 

\begin{corollary}[Egorov's theorem for Husimi corrections]\label{cor:husimi}
Under the assumptions of Theorem~\ref{theo:egorov},
the error between the classically evolved Husimi correction  
and the expectation value of the  solution $\psi(t)$ of the Schr\"odinger equation with initial value $\psi(0)$ of unit norm satisfies
\[
\left| \langle A\rangle_{\psi(t)} - \int_{\Rb^{2d}} a\bigl(\Phi^t(z)\bigr)\, \calH^{\rm c}_{\psi(0)}(z) \D z \right| 
\le c\, t\, \eps^2,
\qquad 0<t\le \bar t,
\]
where the constant $c<\infty$ depends only on the constant $\rho$ and the derivative bounds 
on $a$ and $V$ but is independent of $\psi(0)$, $\eps$ and $t\le \bar t$.
\end{corollary}

%\begin{proof}
%We combine Theorem~\ref{theo:egorov} with Proposition~\ref{prop:husimi} to obtain
%\[
%\left| \langle A\rangle_{\psi(t)} - \int_{\Rb^{2d}}  \calH^{\rm c}_{\psi(0)}(z) \,a(\Phi^t(z))\, d z \right| 
%\le c\, t\, \eps^2. 
%\]
%Since the classical flow preserves volume, we may move it from the observable to the Husimi correction 
%without changing the integral value. 
%\end{proof}

Since the Husimi correction is a difference $\calH_\psi^{\rm c} = \calH_\psi-\tfrac\eps4\Delta\calH_\psi$, 
it loses the positivity of the Husimi function. However, one can express it as a linear combination of 
spectrograms and thus write the weighted phase space integral
as a linear combination of integrals with positive weight. 

\begin{theorem}[spectrogram expansion]\label{theo:spec} 
For every $\psi\in L^2(\Rb^d)$, 
the Husimi correction is a linear combination of spectrograms, 
\[
\calH_\psi^{\rm c} = 
(1+\tfrac{d}{2}) \calH_\psi - \tfrac12 \sum_{j=1}^d (\calW_{\varphi_j}*\calW_\psi),
\]
where 
\[
\varphi_j(x) = (\pi\eps)^{-d/4} \sqrt{\frac{2}{\eps}} \,x_j\, \e^{-|x|^2/\eps},\qquad x\in\Rb^d,
\]
is a first order Hermite function. Moreover, 
\[
\int_{\Rb^{2d}} \calH_\psi^{\rm c}(z)\D z = \|\psi\|^2.
\]
\end{theorem}

\begin{proof} Since
\[
\Delta\calH_\psi = \Delta(\calW_{g_0}*\calW_\psi) =\Delta\calW_{g_0}* \calW_\psi,
\]
we just have to calculate the Laplacian of the Wigner function of the standard Gaussian $g_{0}$. 
We have $\calW_{g_0} = (\pi\eps)^{-d} \,\e^{-|z|^2/\eps}$ (by direct computation, or see the next subsection) and hence
\begin{align*}
\Delta \calW_{g_0}(z) &= %(\pi\eps)^{-d}\Delta \e^{-|z|^2/\eps} = 
%(\pi\eps)^{-d}\, \nabla^T \left(-\tfrac2\eps z\, \e^{-|z|^2/\eps}\right) \\
%&= 
(\pi\eps)^{-d} \,\tfrac2\eps\left(-2d + \tfrac{2}{\eps}|z|^2\right) \e^{-|z|^2/\eps}.
\end{align*}
By calculating the Wigner function $\calW_{\varphi_j}$ via Lemma~\ref{wig:lem:pos-mom} (or by invoking the Laguerre connection given in Theorem~\ref{theo:laguerre}),
\begin{align*}
\sum_{j=1}^d \calW_{\varphi_j}(z) 
&= -(\pi\eps)^{-d}\, \sum_{j=1}^d \left(1-\tfrac2\eps|z_j|^2\right) \e^{-|z|^2/\eps}\\ 
&= -(\pi\eps)^{-d} \left(2d-\tfrac2\eps|z|^2\right) \e^{-|z|^2/\eps} + d\,\calW_{g_0}(z).
\end{align*}
Therefore, 
\[
\Delta\calW_{g_0} = \tfrac2\eps\, \sum_{j=1}^d \calW_{\varphi_j} -\tfrac{2d}{\eps}\,\calW_{g_0},
\]
which proves the stated formula for $\calH_\psi^c$.
Since the Gaussian and the Hermite functions are all normalised to one, this formula and
Lemma~\ref{lem:spectro} yield the formula for the integral of $\calH_\psi^c$.
%$$
%\int_{\Rb^{2d}} \calH_\psi^{\rm c}(z)\D z = (1+\tfrac{d}{2}) \|\psi\|^2 - \tfrac12 \,d\,\|\psi\|^2 = \|\psi\|^2.
%$$
%\begin{align*}
%&\int_{\Rb^{2d}} (\calH_\psi-\tfrac\eps4\Delta\calH_\psi)(z) \D z \\
%&= 
%(1+\tfrac{d}{2}) \int_{\Rb^{2d}} \calH_\psi(z)\D z 
%- \tfrac12\sum_{j=1}^d \int_{\Rb^{2d}} (\calW_{\varphi_j}*\calW_\psi)(z) \D z \ = \|\psi\|^2.
%\end{align*}
\end{proof}

The classically propagated Husimi correction shares the norm and energy conservation obeyed by the 
Wigner function. 

\begin{corollary}[norm and energy]
The approximation provided by Egorov's theorem for the Husimi correction given in Corollary~\ref{cor:husimi} 
conserves norm and energy. 
\end{corollary}

\begin{proof}
For the norm, that is the expectation of $1 = \op(1)$, the volume preservation of the 
classical flow together with the norm relation given in Theorem~\ref{theo:spec} yields
\[
%\int_{\Rb^{2d}}  \left(\calH^{\rm c}_{\psi(0)}\circ\Phi^{-t}\right)\!(z) \, d z \, =\,  
\int_{\Rb^{2d}}  \calH^{\rm c}_{\psi(0)}(z) \, d z \, =\, \|\psi(0)\|^2,\qquad t\in\Rb.
\]
For the energy, that is the observable $H=\op(h)$, we have by classical energy conservation that
$$
\int_{\Rb^{2d}}  h(\Phi^t(z)) \,\calH^{\rm c}_{\psi(0)}(z) \, d z  
$$
is independent of $t$.
\end{proof}

\subsection{Wigner functions for Gaussian wave packets}

We now explicitly calculate the only possible positive Wigner functions, 
namely those of Gaussian wave packets. 

\begin{proposition}[Gaussian Wigner function]\label{lem:WG}
We consider a normalized Gaussian wave packet 
\[
u(x) = (\pi\eps)^{-d/4} \det(\Im C)^{1/4}
 \exp(\tfrac{\I}{2\eps}(x-q)^T C(x-q) + \tfrac{\I}{\eps}p^T(x-q))
\]
with phase space centre $z=(q,p)\in\Rb^{2d}$ and complex symmetric width matrix 
$C\in\C^{d\times d}$ such that $\Im C$ is positive definite. 
Then, 
\[
\calW_u(\zeta) = (\pi\eps)^{-d} \exp(-\tfrac1\eps (\zeta-z)^T G(\zeta-z)),\qquad \zeta\in\Rb^{2d},
\]
where 
\[
G = \begin{pmatrix} \Im C + (\Re C)^T (\Im C)^{-1} \Re C & -(\Re C)^T (\Im C)^{-1} \\*[1ex]
-(\Im C)^{-1}\Re C & (\Im C)^{-1}\end{pmatrix}
\]
is a real symmetric, positive definite, symplectic matrix in $\Rb^{2d\times 2d}$.
\end{proposition}

\begin{proof}
For $x,\xi\in\Rb^d$, we denote 
\[
x_q = x-q,\quad \xi_p = \xi-p,
\] 
and calculate
\begin{align*}
&\overline u(x+\tfrac12 y) u(x-\tfrac12 y) = (\pi\eps)^{-d/2} \det(\Im C)^{1/2}\\*[1ex]
&\quad\exp\!\left(\tfrac{\I}{2\eps}(x_q^T(C-\bar C)x_q - x_q^T(C+\bar C)y +\tfrac14 y^T(C-\bar C)y) 
-\tfrac{\I}{\eps}p^T y\right).
\end{align*}
Therefore,
\begin{align*}
\calW_u(x,\xi) &= (2\pi\eps)^{-d} (\pi\eps)^{-d/2} \det(\Im C)^{1/2} \exp(-\tfrac{1}{\eps} x_q^T\Im C x_q) \\
&\quad \int_{\Rb^d} \exp(-\tfrac{1}{4\eps} y^T \Im C y + \tfrac{\I}{\eps} y^T(\xi_p -\Re C x_q))\D y.
\end{align*}
We denote $k = \xi_p -\Re C x_q$ and diagonalize the positive definite matrix 
\[
\Im C = S^T {\rm diag}(\lambda_1,\ldots,\lambda_d) S
\] 
by an orthogonal matrix $S$. 
This allows us to write the Fourier integral as
\begin{align*}
&\int_{\Rb^d} \exp(-\tfrac{1}{4\eps} y^T \Im C y + \tfrac{\I}{\eps} y^T(\xi_p -\Re C x_q))\D y \\
&= \prod_{n=1}^d \int_{\Rb} \exp(-\tfrac{\lambda_n}{4\eps}y_n^2 + \tfrac{\I}{\eps}\, y_n (Sk)_n)\D y_n\\
&= \prod_{n=1}^d \sqrt{\tfrac{4\pi\eps}{\lambda_n}} \exp(-\tfrac{1}{\lambda_n\eps}\,(Sk)_n^2)\\*[1ex]
&= (4\pi\eps)^{d/2} \det(\Im C)^{-1/2} \exp(-\tfrac1\eps k^T (\Im C)^{-1} k).
\end{align*}
We thus obtain
\begin{align*}
\calW_u(x,\xi) &= (\pi\eps)^{-d}\exp(-\tfrac{1}{\eps} x_q^T \Im C x_q) \\
&\qquad
\exp(-\tfrac1\eps (\xi_p -\Re C x_q)^T (\Im C)^{-1} (\xi_p -\Re C x_q)),
\end{align*}
which implies that $\calW_u$ is a real-valued Gaussian with covariance matrix~$G$. 
The proof that $G$ is positive definite and symplectic is given in Lemma~\ref{lem:G}.
\end{proof}

For a better understanding of the special properties of the covariance matrix just obtained we 
continue the analysis of symplectic matrix relations previously started in Lemma~\ref{lem:hag-rel}.
%We recall that we have used the matrix
%\[
% J=\begin{pmatrix} 0 & -I_d \\ I_d & 0
%\end{pmatrix}
%\]
%for defining our symplectic structure on $\Rb^{2d}$.

\begin{lemma}[symplecticity]\label{lem:sympl}
Let $Q,P\in\C^{d\times d}$ be matrices that satisfy Hagedorn's conditions
$Q^TP- P^T Q = 0$ and $Q^*P-P^*Q = \bch 2\I\, \Id_d\ech$.
We define the rectangular matrix 
\[
Z = \begin{pmatrix}Q\\ P\end{pmatrix}\in\C^{2d\times d}.
\]
Then, 
\begin{equation}\label{eq:Zsympl}
Z^T J^T Z = 0\quad\text{and}\quad Z^* J^T Z = 2\I I_d,
\end{equation}
and $ZZ^* = \Re\!(ZZ^*) + \I\, \Im\!(ZZ^*)$, where 
\[
\Im\!(ZZ^*) = J, 
\]
and $\Re\!(ZZ^*)$ is a real symmetric, positive definite, symplectic matrix.
\end{lemma}

\begin{proof}
We recall from Lemma~\ref{lem:hag-rel} that $(\Re\!(Z) , \Im\!(Z))$ is a symplectic matrix. 
Hence, $(Z, \bar Z)$ is an invertible matrix. We observe that
\begin{align*}
\Im\!(ZZ^*)J^T (Z, \bar Z) &= \tfrac{1}{2\I}(ZZ^*-\bar Z Z^T) J^T (Z, \bar Z) \\*[1ex]
&= (Z, \bar Z),
\end{align*}
which implies $\Im\!(ZZ^*) = J$. 
The real part of $ZZ^*$ is also symplectic, since
\begin{align*}
\Re\!(ZZ*)^T J \Re\!(ZZ^*) &= \tfrac14(\bar Z Z^T+ZZ^*)J(ZZ^*+\bar Z Z^T)\\*[1ex]
&=\tfrac14(-2\I ZZ^*+2\I \bar Z Z^T) \\*[1ex]
&= \Im\!(ZZ^*) = J.
\end{align*}
Moreover, for all $z\in\Rb^{2d}$,
\begin{align*}
z^T \Re\!(ZZ^*)z &= \tfrac12 z^T(ZZ^*z + \bar ZZ^T z) \\*[1ex]
&= |Z^*z|^2\ge0.
\end{align*}
If $Z^*z=0$, then $ZZ^*z=0$ and $\Im\!(ZZ^*)z=0$, which implies $z=0$. Hence, the real part of $ZZ^*$ 
is a positive definite matrix.  
\end{proof}

The explicit form of the real and the imaginary part of $ZZ^*$ allows us to express   
the covariance matrix of the Gaussian Wigner function in terms of  Hagedorn's parametrisation.

\begin{lemma}[covariance matrix]\label{lem:G}
Let $C\in\C^{d\times d}$ be a complex symmetric matrix with $\Im(C)>0$ and 
$C=PQ^{-1}$ a factorisation according to Lemma~\ref{lem:hag-rel}. 
Then, the covariance matrix $G$
of the Gaussian Wigner function in Lemma~\ref{lem:WG} can be rewritten as
\begin{align*}
G &= J^T \Re\!(Z Z^*)J = \begin{pmatrix}PP^* & -\Re\!(PQ^*) \\*[1ex] -\Re\!(QP^*) & QQ^*\end{pmatrix}.
\end{align*}
\end{lemma}

\begin{proof}
By Lemma~\ref{lem:hag-rel}, $(\Im C)^{-1} = QQ^*$. Using $Q^*P-P^*Q=2\I I_d$, we simplify the 
matrix blocks of $G$ according to
\begin{align*}
&(\Im C)^{-1}\,\Re C = 
\tfrac12 QQ^* (PQ^{-1} + Q^{-*}P^*) \\
&\quad= \tfrac12 \left(Q(2\I I_d + P^*Q)Q^{-1} + QP^*\right) \\
&\quad= \I I_d + QP^*
\end{align*}
and
\begin{align*}
&(\Re C)^T (\Im C)^{-1} \Re C =\tfrac12 (PQ^{-1} + Q^{-*}P^*)(\I I_d + QP^*)\\
&\quad= \tfrac12(\I PQ^{-1} + PP^* + \I Q^{-*}P^* + Q^{-*}(Q^*P-2\I I_d)P^*)\\
&\quad=\tfrac12(\I PQ^{-1} + 2PP^* - \I Q^{-*}P^*)\\
&\quad=PP^* -\Im C.
\end{align*}
This implies
\begin{align*}
G &= \begin{pmatrix} 
PP^* & \I I_d - PQ^*\\
-\I I_d - QP^* & QQ^*
\end{pmatrix}\\*[1ex]
&= J^T ZZ^* J - iJ = J^T\Re\!(ZZ^*) J,
\end{align*}
where the last equation uses Lemma~\ref{lem:sympl}. 
\end{proof}

\subsection{Wigner functions for Hagedorn's semiclassical wave packets}
We now revisit the raising procedure for Hagedorn's semiclassical wave packets,
\[
\varphi_{k+\langle j\rangle} = \frac{1}{\sqrt{k_j+1}} A_j^\dagger \varphi_k,
\qquad k\in\N_0^d,\qquad j=1,\ldots,d,
\]
and reformulate it in terms of Wigner functions. We rewrite the raising operator 
\[
A^\dagger[q,p,Q,P] = \displaystyle 
\frac \I{\sqrt{2\eps}}\Bigl( P^*(\widehat q - q) -
Q^*(\widehat p - p) \Bigr)
\]
in phase space notation as 
\begin{equation}\label{W:ladder}
A^\dagger[z,Z] =  \frac{\I}{\sqrt{2\eps}} \,Z^*J(\widehat z-z)
\end{equation}
with the phase space centre, the width matrices, and the position and momentum operators denoted by 
\[
z = \begin{pmatrix}q\\p\end{pmatrix},\quad 
Z = \begin{pmatrix}Q\\P\end{pmatrix},\quad
\widehat z = \begin{pmatrix}\widehat q\\ \widehat p\end{pmatrix}.
\] 
We first calculate the action of the raising operator on arbitrary Wigner functions. 

\begin{lemma}[raising of a Wigner function]\label{prop:Wladder} 
Let $A^\dagger = A^\dagger[z,Z]$ be the raising operator given in \eqref{W:ladder} and denote 
$\calR = \calR[z,Z]$, 
\[
\calR = -\frac{\I}{\sqrt{2\eps}} \ Z^T (J(\zeta-z) - \tfrac{\I\eps}{2} \nabla_\zeta).
\]
Then, for all Schwartz functions $\psi:\Rb^d\to\C$ and all $m=1,\ldots,d$,
\[
\calW_{A_m^\dagger\psi} = \calR_m \overline{\calR_m} \, \calW_\psi. 
\]
\end{lemma}

\begin{proof} 
As an auxiliary tool for our calculation, we use the bilinear Wigner--Moyal transform 
of two functions $\psi,\phi:\Rb^d\to\C$, 
\[
\calW_{\psi,\phi}(\zeta) = (2\pi\eps)^{-d} \int_{\Rb^d}  \overline{\psi}(x+\tfrac12 y)\phi(x-\tfrac12 y) 
\e^{\I \xi\cdot y/\eps}\D y
\]
for all $\zeta=(x,\xi)\in\Rb^{2d}$. 
With a slight abuse of vectorial notation, 
\begin{align*}
\calW_{\widehat q\psi,\phi}(\zeta)
&= (2\pi\eps)^{-d} \int_{\Rb^d}  (x+\tfrac12y)\overline{\psi}(x+\tfrac12 y)\phi(q-\tfrac12 y) 
\e^{\I \xi\cdot y/\eps}\D y\\
&= (x + \tfrac{\eps}{2\I}\partial_\xi) \calW_{\psi,\phi}(\zeta).
\end{align*}
Moreover, by the product rule,
\[
-\I\eps\partial_x\, \calW_{\psi,\phi}(\zeta) = -\calW_{\widehat p\psi,\phi}(\zeta) 
+ \calW_{\psi,\widehat p\phi}(\zeta).
\]
By partial integration,
\begin{align*}
\calW_{\widehat p\psi,\phi}(\zeta) &= 
(2\pi\eps)^{-d} \int_{\Rb^d}  2\I\eps\,\partial_y\overline{\psi}(x+\tfrac12 y)\ \phi(x-\tfrac12 y) 
\e^{\I \xi\cdot y/\eps}\D y\\
&= -\calW_{\psi,\widehat p\phi}(\zeta) +2\xi\, \calW_{\psi,\phi}(\zeta).
\end{align*}
Therefore, 
\begin{align*}
\calW_{\widehat q\psi,\phi}(\zeta) &= (x - \tfrac{\I\eps}{2}\partial_\xi) \calW_{\psi,\phi}(\zeta),\\
\calW_{\widehat p\psi,\phi}(\zeta) &= (\xi + \tfrac{\I\eps}{2}\partial_x)\calW_{\psi,\phi}(\zeta),
\end{align*}
which yields
\[
\calW_{\widehat z\psi,\phi}(\zeta) = (\zeta+\tfrac{\I\eps}{2}J\nabla_\zeta) \calW_{\psi,\phi}(\zeta).
\]
Since the Wigner--Moyal transform is anti-linear for the first argument, we then obtain 
\[
\calW_{A^\dagger_m\psi,\phi} = \calR_m\,\calW_{\psi,\phi}.
\]
Next we observe that $\calW_{\psi,\phi} = \overline{\calW_{\phi,\psi}}$,
such that
\[
\calW_{A^\dagger_m\psi} =\calR_m\, \calW_{\psi,A^\dagger_m\psi} =  \calR_m \overline{\calR_m}\, \calW_\psi.
\]
\end{proof}

The raising operator $\calR$ contains the affine map $\zeta\mapsto Z^T J(\zeta-z)$ 
whose gradient inherits the isotropy and the normalisation property of the rectangular matrix $Z$ 
in the following sense.

\begin{lemma}[affine map]\label{lem:affine} 
We consider $z\in\Rb^{2d}$ and $Z\in\C^{2d\times d}$ satisfying the relations~\eqref{eq:Zsympl}.
Define the affine map
\[
\ell: \Rb^{2d}\to\C^{d},\quad \ell(\zeta) = Z^T J(\zeta-z).
\]
Then, for all $m,n=1,\ldots, d$,
\[
(Ze_m)^T\, \nabla_\zeta \ell_n(\zeta) = 0\quad\text{and}\quad
(Ze_m)^*\, \nabla_\zeta \ell_n(\zeta) = 2\I\, \delta_{mn}.
\]
Moreover, 
\[
|\ell(\zeta)|^2 = \zeta^T G\zeta,\qquad \zeta\in\Rb^{2d}.
\]
\end{lemma}

\begin{proof} Using \eqref{eq:Zsympl}, we calculate
\begin{align*}
(Ze_m)^T\nabla_\zeta \ell_n(\zeta) &= e_m^T (Z^T J^T Z) e_n = 0,\\
(Ze_m)^*\nabla_\zeta \ell_n(\zeta) &= e_m^T (Z^* J^T Z) e_n = 2\I \delta_{mn}.
\end{align*}
By Lemma~\ref{lem:G}, the imaginary part of $ZZ^*$ equals $J$. Therefore, by skew-symmetry,
\begin{align*}
|\ell(\zeta)|^2 &= |Z^T J\zeta|^2  = \zeta^T J^T ZZ^* J \zeta \\
&= \zeta^T J^T\Re(ZZ^*)J\zeta = 
\zeta^T G \zeta,\qquad \zeta\in\Rb^{2d}.
\end{align*}
\end{proof}

The above relation between the quadratic form that defines the Wigner function 
of a Gaussian wave packet, and the affine map $\ell(\zeta)$, suggests a factorised reformulation 
of the  Wigner function of the wave packet. 

\begin{lemma}[Gaussian Wigner function]\label{lem:WGfactor}
Let $z\in\Rb^{2d}$, and let the matrix $Z\in\C^{2d\times d}$ satisfy the conditions \eqref{eq:Zsympl}. 
Consider the Gaussian wave packet $\varphi_0 = \varphi_0[z,Z]$. Then, 
\[
\calW_{\varphi_0}(\zeta) = 
(\pi\eps)^{-d} \prod_{n=1}^d \exp(-\tfrac{1}{\eps}|\ell_n(\zeta)|^2),\qquad \zeta\in\Rb^{2d}.
\]
\end{lemma}

\begin{proof}
We combine Proposition~\ref{lem:WG} and Lemma~\ref{lem:affine} to obtain
\begin{align*}
\calW_{\varphi_0}(\zeta) &= (\pi\eps)^{-d} \exp(-\tfrac{1}{\eps}\,\zeta^T G\zeta) 
=  (\pi\eps)^{-d} \exp(-\tfrac{1}{\eps}\,|\ell(\zeta)|^2)\\
&=(\pi\eps)^{-d} \prod_{n=1}^d \exp(-\tfrac{1}{\eps}|\ell_n(\zeta)|^2).
\end{align*}
\end{proof}

Due to the orthogonality property of the affine map,
the factorization of the zeroth-order Wigner function is 
preserved by the raising process in phase space.
This striking phenomenon does not unfold in the ``half-dimensional'' 
construction of Hagedorn's wave packets in configuration space.

\begin{proposition}[factorization in phase space]\label{prop:factor}
Let $z\in\Rb^{2d}$, and let $Z\in\C^{2d\times d}$ satisfy the conditions \eqref{eq:Zsympl}. 
Consider the raising operator $\calR = \calR[z,Z]$ and the affine map $\ell = \ell[z,Z]$. 
Then, for all $m,n = 1,\ldots,d$ and all Schwartz functions $f_n:[0,\infty[\to\Rb$,
\[
\calR_m \overline{\calR_m}\, \Big(\prod_{n=1}^d f_n(|\ell_n(\zeta)|^2)\Big) = 
g_m(|\ell_m(\zeta)|^2) \prod_{n\neq m} f_n(|\ell_n(\zeta)|^2)
\]
with $\zeta\in\Rb^{2d}$, where $g_m:[0,\infty[\to\Rb$ is defined by 
\begin{equation}\label{eq:gm}
g_m(x) = \tfrac{1}{2\eps}\left(
\left(x-\eps\right) f_m(x) 
-2\eps \,x\, f_{m}'(x) 
+ \eps^2 \left( f_m'(x) + x f_m''(x)\right) 
\right). 
\end{equation}
\end{proposition}

\begin{proof}
We calculate the product of the raising operator with itself, 
\begin{align*}
&\calR_m\overline{\calR_m}\\
&= \tfrac{1}{2\eps}
\big( \ell_m- \tfrac{\I\eps}{2} (Ze_m)^T\nabla\big)
\big( \overline{\ell_m} + \tfrac{\I\eps}{2} (Ze_m)^*\nabla\big)\\*[1ex]
&= \tfrac{1}{2\eps}\left( |\ell_m|^2 
+ \tfrac{\I\eps}{2}\ell_m(Ze_m)^*\nabla
- \tfrac{\I\eps}{2}(Ze_m)^T\nabla\overline{\ell_m}
+\tfrac{\eps^2}{4}\,(Ze_m)^T\nabla\  (Ze_m)^*\nabla\right).
\end{align*}
By Lemma~\ref{lem:affine}, we express the above gradient term as
\[
\tfrac{\I\eps}{2}\ell_m(Ze_m)^*\nabla
- \tfrac{\I\eps}{2}\underbrace{(Ze_m)^T\nabla\overline{\ell_m}}_{-2\I} = 
-\eps + \tfrac{\I\eps}{2}\left( \ell_m(Ze_m)^* - \overline{\ell_m}(Ze_m)^T\right)\nabla,
\]
and we observe that all derivatives in $\calR_m\overline{\calR_m}$ are directional. We therefore calculate
\[
(Ze_m)^T \nabla\Big( \prod_{n=1}^d f_n(|\ell_n|^2)\Big) =
\sum_{n=1}^d f_{n}'(|\ell_n|^2)\ (Ze_m)^T\nabla |\ell_n|^2\ \prod_{j\neq n} f_{j}(|\ell_j|^2).
\]
Again by Lemma~\ref{lem:affine},
\begin{align*}
(Ze_m)^T \nabla |\ell_n|^2 &= 
(Ze_m)^T \left( \ell_n\nabla\overline{\ell_n} + \overline{\ell_n}\nabla \ell_n\right) \\*[1ex]
&= -2\I\delta_{mn}\ell_m.
\end{align*}
Hence, the summation collapses to a single summand, and we obtain
\[
(Ze_m)^T \nabla\Big( \prod_{n=1}^d f_n(|\ell_n|^2)\Big) = 
-2\I \,\ell_m\, f_{m}'(|\ell_m|^2) \ \prod_{n\neq m} f_{n}(|\ell_n|^2).
\]
This implies for the first order derivative term
\begin{align*}
&\tfrac{\I\eps}{2}\left( \ell_m(Ze_m)^* - \overline{\ell_m}(Ze_m)^T\right)\nabla 
\Big( \prod_{n=1}^d f_n(|\ell_n|^2)\Big)\\*[1ex]
&= -2\eps \,|\ell_m|^2\, f_{m}'(|\ell_m|^2) \prod_{n\neq m} f_{n}(|\ell_n|^2). 
\end{align*}
For the second order derivative contribution we obtain
\begin{align*}
&\tfrac{\eps^2}{4}\left((Ze_m)^T\nabla\  (Ze_m)^*\nabla\right)\Big( \prod_{n=1}^d f_n(|\ell_n|^2)\Big) \\
& = \tfrac{\eps^2}{4}(Ze_m)^T \nabla \Big( 2\I \,\overline{\ell_m}\, f_{m}'(|\ell_m|^2) \ \prod_{n\neq m} f_{n}(|\ell_n|^2)\Big).
\end{align*}
The first of the three above derivatives satisfies 
\[
(Ze_m)^T\nabla \overline{\ell_m} = -2\I.
\] 
The second one is
\begin{align*}
(Ze_m)^T\nabla f_{m}'(|\ell_m|^2) &= f_m''(|\ell_m|^2) (Ze_m)^T\nabla |\ell_m|^2\\*[1ex]
& = -2\I \ell_m f_m''(|\ell_m|^2),
\end{align*}
while the third derivative simply vanishes
\begin{align*}
(Ze_m)^T\nabla\Big(\prod_{n\neq m} f_{n}(|\ell_n|^2)\Big) &= 
\sum_{n\neq m} f_n'(|\ell_n|^2) \underbrace{(Ze_m)^T\nabla |\ell_n|^2}_{=0} 
\prod_{j\neq n} f_{j}(|\ell_j|^2) \\
&= 0.
\end{align*}
In summary, we obtain for the second derivatives that
\begin{align*}
&\tfrac{\eps^2}{4}\left((Ze_m)^T\nabla\  (Ze_m)^*\nabla\right)\Big( \prod_{n=1}^d f_n(|\ell_n|^2)\Big) \\
&=\eps^2\left( f_m'(|\ell_m|^2) + |\ell_m|^2 f_m''(|\ell_m|^2)\right) \prod_{n\neq m}^d f_n(|\ell_n|^2).
\end{align*}
Now adding the constant, first- and second-order contributions together, we get
\[
\calR_m\overline{\calR_m} \Big( \prod_{n=1}^d f_n(|\ell_n|^2)\Big) = 
g_m(|\ell_m|^2) \prod_{n\neq m}^d f_n(|\ell_n|^2)
\]
with
\[
g_m(x) = \tfrac{1}{2\eps}\left(
\left(x-\eps\right) f_m(x) 
-2\eps \,x\, f_{m}'(x) 
+ \eps^2 \left( f_m'(x) + x f_m''(x)\right) 
\right). 
\]
\end{proof}

The classical Hermite--Laguerre connection proves that the Wigner transform of a Hermite 
function is a Laguerre function; see \cn{Gro46} or \cite[Chapter 1.9]{Fol89}. 
Proposition~\ref{prop:factor} allows us to extend this connection to the Wigner transform of a Hagedorn 
wave packet.

\begin{theorem}[Laguerre connection]\label{theo:laguerre} For all $k\in\N_0^d$, 
the Wigner function of the $k$th Hagedorn wave packet 
$\varphi_k = \varphi_k[q,p,Q,P] = \varphi_k[z,Z]$ satisfies 
\[
\calW_{\varphi_k}(\zeta) = \frac{(-1)^{|k|}}{(\pi\eps)^{d}} \, \, 
\prod_{n=1}^d L_{k_n}(\tfrac2\eps\,|\ell_n(\zeta)|^2)  \, \exp(-\tfrac1\eps |\ell_n(\zeta)|^2),
\]
where $L_{k_n}$ denotes the $k_n$th Laguerre polynomial, and $\ell(\zeta) = \ell[z,Z](\zeta)$ is the affine map 
\[
\ell(\zeta)= Z^T J(\zeta-z), \quad\zeta\in\Rb^{2d}.
\]
\end{theorem}

\begin{proof} 
We perform an inductive proof over $k\in\N_0^d$. 
The base case $k=0$ was proved in Lemma~\ref{lem:WGfactor}. 
For the inductive step we choose $k\in\N_0^d$ and write
\[
\calW_{\varphi_k}(\zeta) = \prod_{n=1}^d f_n(|\ell_n(\zeta)|^2)\quad\text{with}\quad
f_n(x) = \tfrac{(-1)^{k_n}}{\pi\eps}  L_{k_n}(\tfrac2\eps x)  \, \e^{-x/\eps}.
\]
By Proposition~\ref{prop:factor}, we have for all $m=1,\ldots,d$ that
\[
\calR_m\overline{\calR_m}\, \calW_{\varphi_k}(\zeta) = 
g_m(|\ell_m(\zeta)|^2) \prod_{n\neq m} f_n(|\ell_n(\zeta)|^2)
\]
with $g_m$ given in equation \eqref{eq:gm}. To determine the function $g_m$ explicitly, 
we calculate the first and second derivatives of $f_m$,
\begin{align*}
f_m'(x) &= 
\frac{(-1)^{k_m}}{\pi\eps^2}\left( 2L_{k_m}'(\tfrac2\eps x) - L_{k_m}(\tfrac2\eps x)\right) \e^{-x/\eps},\\*[1ex]
f_m''(x) &= 
\frac{(-1)^{k_m}}{\pi\eps^3}\left( 4L_{k_m}''(\tfrac2\eps x) - 4L_{k_m}'(\tfrac2\eps x) + L_{k_m}(\tfrac2\eps x)\right) \e^{-x/\eps}.
\end{align*}
We obtain 
\[
g_m(x) = \frac{(-1)^{k_m}}{\pi\eps} 
p_m(x) \e^{-x/\eps}
\]
with
\[
p_m(x) = \tfrac2\eps x L_{k_m}''(\tfrac2\eps x) + (1-\tfrac4\eps x)L_{k_m}'(\tfrac2\eps x)  
+ (\tfrac2\eps x-1) L_{k_m}(\tfrac2\eps x) .
\]
By Laguerre's equation, 
\[
xL_k'' + (1-x)L_k' + kL_k = 0,
\]  
the above polynomial simplifies to
\[
p_m(x) = (\tfrac2\eps x-1-k_m)L_{k_m}(\tfrac2\eps x) - \tfrac2\eps x L_{k_m}'(\tfrac2\eps x). 
\]
Combining the two recursion formulas 
\begin{align*}
xL_k' &= kL_k - kL_{k-1},\\
(k+1)L_{k+1} &= (2k+1-x)L_k - kL_{k-1}, 
\end{align*}
we obtain
\[
(k+1)L_{k+1} = (k+1-x)L_k  +  xL_k' 
\]
and
\[
p_m(x) = -(k_m+1) L_{k_m}(\tfrac2\eps x).
\]
\end{proof}

\subsection{Aside: Proof of the commutator bound of Lemma~\ref{lem:obs-err-ineq}}\label{sec:proofobs}

As another application of semiclassical commutator estimates we now give the postponed proof 
of Lemma~\ref{lem:obs-err-ineq} for the variational Gaussian wave packets. We begin with the following useful remark.

\begin{remark}
Tightening the assumptions in Proposition~\ref{prop:comm} on the observable $a$ 
and the potential function $V$ such that derivatives of order $\ge 2$ 
are bounded, we can prove a less ambitious but also useful commutator estimate, 
\begin{equation}\label{eq:comm1}
\|\left(\tfrac{1}{\I\eps}[\op(a),V] - \op(-\nabla_p a) \nabla V\right)\varphi\| \le C\eps \|\varphi\|,
\end{equation}
where the constant $C>0$ depends on derivatives of $a$ and $V$ of order greater or equal 
than two. For the proof, we insert in the integral \eqref{eq:commV} 
the right point Taylor expansion
\begin{align*}
&\left( V(x) - V(y)\right)\e^{\I p\cdot(x-y)/\eps} 
= \nabla V(y)^T(x-y) \ \e^{\I p\cdot(x-y)/\eps} + {\mathcal O}(\|x-y\|^2)\\*[1ex]
&\qquad= \tfrac{\eps}{\I}\, \nabla V(y)^T \nabla_p \e^{\I p\cdot(x-y)/\eps} + {\mathcal O}(\|x-y\|^2)   
\end{align*}
and continue the argument with a second-order remainder that, after multiplication 
by $\I/\eps$, results in an upper bound that is first order with respect to $\eps$.
\end{remark}

\begin{proof}{\em (of Lemma~\ref{lem:obs-err-ineq})}\quad  For the variationally evolving  
Gaussian wave packet $u$ we prove the bound
\[
\left| \left\langle \tfrac{1}{\I\eps}[W_{u(t-s)},A(s)]\right\rangle_{u(t-s)}\right| \le c\,\eps,
\qquad 0\le s\le t\le \bar t.
\]
We first analyse the remainder potential 
\[
W_u = V-U_u.
\] 
 By Proposition~\ref{prop:project},
\[
U_u = \langle V\rangle_u +\eps\widetilde\alpha 
+ \langle\nabla V\rangle_u^T (x-q) + \tfrac12(x-q)^T \langle\nabla^2 V\rangle_u(x-q)
\]
with $\widetilde\alpha = -\tfrac14\tr((\Im C)^{-1}\langle \nabla^2 V\rangle_u)$. 
A second order Taylor approximation around the centre point $q$ provides 
\[
V(x) = V(q) + \nabla V(q)^T(x-q) + \tfrac12 (x-q)^T \nabla^2 V(q) (x-q) + \widetilde W_q(x),
\]
where by Lemma~\ref{lem:int-bound} the non-quadratic remainder satisfies
\[
\|\widetilde W_q u\| \le \widetilde c_0\,\eps^{3/2}\quad\text{and}\quad 
\|\nabla\widetilde W_q u\| \le \widetilde c_1\,\eps
\]
for $\eps$-independent constants $\widetilde c_0,\widetilde c_1>0$. 
We obtain
\begin{align*}
W_u &= \left(V(q) - \langle V\rangle_u -\eps\widetilde\alpha\right) 
+ \left( \nabla V(q) - \langle\nabla V\rangle_u\right)^T(x-q) \\
&\qquad + 
\tfrac12(x-q)^T\left( \nabla^2 V(q) - \langle\nabla V^2\rangle_u\right)(x-q) + \widetilde W_q.
\end{align*}
Lemma~\ref{lem:av} then provides constants $c_0,c_1>0$ such that
\[
\|W_u u\|\le c_0\, \eps\quad\text{and}\quad\|\nabla W_u u\| \le c_1\, \eps.
\]
Now we start working on the expectation value. We have
\[
\left\langle \tfrac{1}{\I\eps}[W_{u(t-s)},A(s)]\right\rangle_{u(t-s)} 
=\left\langle \tfrac{1}{\I\eps}[W_{u(t-s)},\op(a\circ\Phi^s)]\right\rangle_{u(t-s)} 
+ r_1(s,t)
\]
with remainder
\[
r_1(s,t) = \left\langle \tfrac{1}{\I\eps}[W_{u(t-s)},(A(s)-\op(a\circ\Phi^s))]\right\rangle_{u(t-s)}. 
\]
By Egorov's theorem and the above estimate on the remainder potential, 
\begin{align*}
|r_1(s,t)| &\le \tfrac{1}{\eps}\, \|W_{u(t-s)} u(t-s)\|\ \|A(s)-\op(a\circ\Phi^s)\|\ \|u(t-s)\|  \\*[1ex]
&\le c_1\, \eps^2. 
\end{align*}
Moreover, by the first order commutator estimate~\eqref{eq:comm1},
\begin{align*}
&\left\langle \tfrac{1}{\I\eps}[W_{u(t-s)},\op(a\circ\Phi^s)]\right\rangle_{u(t-s)} \\*[1ex]
&\quad= \left\langle \op(-\partial_p (a\circ\Phi^s)) u(t-s) \mid \nabla W_{u(t-s)} u(t-s)\right\rangle +\eps\, r_2(s,t)
\end{align*}
with remainder $r_2(s,t)$ bounded independently from $\eps$. We conclude the proof by observing that
\begin{align*}
&\left|\left\langle \op(-\partial_p (a\circ\Phi^s)) u(t-s) \mid \nabla W_{u(t-s)} u(t-s)\right\rangle\right| \\*[1ex]
&\le \|\op(-\partial_p (a\circ\Phi^s))\| \ \|u(t-s)\| \ \|\nabla W_{u(t-s)} u(t-s)\|
\ \le\  c_2\,\eps.
\end{align*}
%Beal's theorem yields that the operator $A(s) = U(s)^* A U(s)$ is the Weyl quantisation $\op(a(s))$ 
%of a smooth real-valued phase space function $a(s)$ that inherits the growth properties of $a$, 
%see \cite[Theorem~11.1]{Zwo12}. In particular, 
%\[
%\| \op(\partial_p a(s)) \| \le \gamma(s), 
%\]  
%where $\gamma(s)>0$ is bounded independently from $\eps$. 
\end{proof}

\subsection{Notes}

The symmetrised mapping from classical phase space functions to operators, usually referred to as Weyl quantisation,  is due to \cn{Wey27}. Our exposition here, which first introduces Weyl operators for polynomials and then lifts the construction to more general functions via exponentials and Fourier transforms,   
is inspired by \cite[Chapter 13]{Hal13}. The proof of the semiclassical commutator estimate, which provides the link between the classical Poisson bracket of functions $\{a,b\}$ and the 
commutator of operators $[\op(a),\op(b)]$, follows the elementary approach to symbolic Weyl calculus 
presented in \cite[Section~4]{Fer14}. 

The evolution of pseudo-differential operators via classical dynamics was first formulated by \cn{Ego69}. 
Refined versions of Egorov's theorem, in particular higher-order estimates with respect to the semiclassical 
parameter~$\eps$ and the validity on the Ehrenfest time scale, are addressed in \cite{BouR02} and also in the 
monograph of \cn[Chapter~11]{Zwo12}. 

The Wigner function was introduced as a quasi-probability distribution on phase space by \cn{Wig32}, when
developing the thermodynamics of quantum-mechanical systems. 
Decades later, the Wigner function was also used as an efficient tool for computational quantum dynamics. In this context, Egorov's theorem has appeared under at least three different names:
first, as the linearised semiclassical 
initial value representation of \cn{Mil74b} and \cn{WanSM98}, which is mostly referred to by its impressive 
acronym LSC-IVR;
second, as the Wigner phase space method of \cn{Hel76b} and \cn{BroH81}; 
third, as the statistical quasiclassical method of \cn{LeeS80}. 
In the mathematical literature, numerical realisations of Egorov's theorem were considered in 
\cite{LasR10}.

The Husimi function or Husimi Q representation of a function was intro\-duced 
by \cn{Hus40}. It is a Wigner function convolved with a Gaussian function of appropriate covariance, 
such that a nonnegative phase space distribution is obtained. 
The classical propagation of the Husimi function, with suitable corrections that render an approximation 
that is second order accurate with respect to $\eps$, was carried out in \cite{KelL13};
see also \cite{GaiL14} for a related approach to the numerical computation of fourth-order corrections 
to Egorov's theorem. 
The novel spectrogram method that combines initial sampling of positive phase space distributions with plain, 
uncorrected classical dynamics has been proposed by \cn{KelLO16}. Higher-order spectrogram 
expansions have recently been analysed in \cite{Kel19}.

The Wigner function of a complex-valued Gaussian function whose covariance matrix is  
complex symmetric with positive definite imaginary part, was explicitly calculated in 
\cite[Chapter 11.2]{Gos11}. The observation that the Wigner function of a 
Hagedorn wave packet generalizes the well-known Hermite--Laguerre connection 
of the classical Hermite functions, is due to \cn{LasT14}. Our proof here using the 
raising and lowering operators seems to be new, but draws on 
ideas developed earlier by \cn{LasST18}, see also \cite{DieKT17}.

%\section{Time integration}
\def\bfpsi{{\pmb\psi}}
\def\bfphi{{\pmb\varphi}}
\def\dt{\tau}

\section{Time integration}
\label{sec:time}

Seemingly off-topic, we begin this section by recapitulating the St\"ormer--Verlet time integration method for the {\it classical} equations of motion, a method that is based on splitting the Hamilton function into kinetic energy and potential energy. This prelude is chosen because the numerical solution of the classical equations of motion and their linearisation is required in various semiclassical approximations encountered in previous sections, and also because even for direct discretisation methods of the semiclassically scaled Schr\"odinger equation, the error analysis requires error bounds for time integrators of the classical equations.

Splitting methods prevail in the time integration of the Schr\"odinger equation. We study the widely used Strang splitting in Section~\ref{subsec:t:strang} and combined with spatial discretisation by Fourier collocation (the split-step Fourier method) in
Section~\ref{subsec:t:ssf}. A more refined splitting method, the symmetric Zassenhaus splitting, is studied in Section~\ref{subsec:t:zass}.

In Section~\ref{subsec:t:gwp} we study a structure-preserving splitting integrator for the equations of motion of variational Gaussians, and in Section~\ref{subsec:t:hagwp} a splitting integrator for Hagedorn's semiclassical wave packets.

\subsection{Symplectic integration of the classical equations of motion and their linearisation: the St\"ormer--Verlet method}
\label{subsec:t:sv}
Various algorithms in previous sections require the time integration of the classical equations of motion
\begin{equation} \label{t:eom}
\dot q = p,\qquad \dot p = -\nabla V(q)
\end{equation}
for $q(t),p(t)\in\Rb^d$
and of their linearisation
\begin{equation} \label{t:lin-eom}
\dot Q = P,\qquad \dot P = -\nabla^2 V(q)Q
\end{equation}
for complex matrices $Q(t), P(t)\in \C^{d\times d}$.

The flow $\Phi^t$ of the Hamiltonian system \eqref{t:eom} is {\it symplectic}, \ie\ for all $z=(q,p)\in\Rb^{2d}$, the Jacobian matrix $D\Phi^t(z)$ satisfies
\begin{equation}\label{t:symp-flow}
D\Phi^t(z)^T J D\Phi^t(z) = J \quad\text{ with }\quad J=
\bch\begin{pmatrix} 0 & -\Id \\ \Id & 0 
\end{pmatrix}\ech;
\end{equation}
see any textbook on classical mechanics or, \eg,
\cn{HaiLW06}, Chapter VI. Therefore, the flow preserves volume in phase space: 
\begin{equation}\label{t:vol}
\det(D\Phi^t(z))=1.
\end{equation}
Volume preservation is important in the approximations of Sections~\ref{sec:csg} and~\ref{sec:wigner}.

A solution of the linearised equations \eqref{t:lin-eom}, which is given as
$$
\begin{pmatrix}  Q(t) \\  P(t)
\end{pmatrix} 
= D\Phi^t(q(0),p(0)) 
\begin{pmatrix}  Q(0) \\  P(0)
\end{pmatrix} ,
$$
therefore satisfies the symplecticity relation
\begin{equation}\label{t:symplecticity-PQ}
Y(t)^T J Y(t) = J \qquad \text{for} \qquad
Y(t) =
\begin{pmatrix} \Re Q(t) & \Im Q(t) \\ \Re P(t) & \Im P(t) 
\end{pmatrix},
\end{equation}
provided that this relation holds for the initial values $(Q(0),P(0))$. This relation is crucial in Hagedorn's parametrisation of Gaussian wave packets (Section~\ref{subsec:hag-gauss}), which is also heavily used in the approximations of Sections~\ref{sec:hagwp} and~\ref{sec:csg}.

General time integration methods do not preserve the properties \eqref{t:symp-flow}--\eqref{t:symplecticity-PQ}, but symplectic integrators do. There is a plethora of such methods of arbitrary order of accuracy; see, \eg~\cn{BlaC16}, \cn{HaiLW06} and \cn{LeiR04}. Here we just describe the {\it St\"ormer--Verlet method} (or leapfrog method), which is a simple, explicit symplectic integrator of second order that enjoys many remarkable properties; see \cn{HaiLW03}. 
It is the standard integrator of classical molecular dynamics. The St\"ormer--Verlet method applied to \eqref{t:eom} can be interpreted as a Strang splitting of the vector field that corresponds to the Hamilton function
$$
H(q,p) = T(p) + V(q), \qquad \text{with}\ \ T(p)=\tfrac12 |p|^2,
$$
that is, the flow $\Phi^\dt(z)$ over a time step $\tau$ is approximated by 
$$
\Phi^\dt \approx \Sigma^\tau =\Phi_V^{\dt/2} \circ \Phi_T^\dt \circ \Phi_V^{\dt/2} ,
$$
where $\Phi_V^t(q,p)=(q,p-t \nabla V(q))$ is the flow of the potential part, and
$\Phi_T^t(q,p)=(q+tp,p)$ is the flow of the kinetic part of the Hamilton function. A step of the St\"ormer--Verlet integrator for \eqref{t:eom}, from the approximation $(q^n,p^n)$ at time $t^n=n\dt$ to the new approximation $(q^{n+1},p^{n+1})=\Sigma^\dt(q^n,p^n)$ at time $t^{n+1}$, then reads
\begin{align}
\nonumber
p^{n+1/2} &= p^n -\tfrac12 {\dt}\, \nabla V(q^n) \\
\label{t:sv}
q^{n+1} &= \bch q^n + \dt\, p^{n+1/2} \ech \\
\nonumber
p^{n+1} &= p^{n+1/2} 
-\tfrac12 {\dt}\, \nabla V(q^{n+1}).
\end{align}
The linearisation of these equations coincides with applying the St\"ormer--Verlet method to the linearised equations of motion \eqref{t:lin-eom},
\begin{align}
\nonumber
P^{n+1/2} &= P^n -\tfrac12 {\dt}\, \nabla^2 V(q^n) Q^n \\
\label{t:lin-sv}
Q^{n+1} &= \bch Q^n + \dt\, P^{n+1/2} \ech \\
\nonumber
P^{n+1} &= P^{n+1/2} 
-\tfrac12 {\dt}\, \nabla^2 V(q^{n+1}) Q^{n+1}.
\end{align}
As a composition of symplectic maps, the St\"ormer--Verlet method is symplectic, \ie\ for all $z=(q,p)\in \Rb^{2d}$,
$$
D\Sigma^\dt(z)^T J D\Sigma^\dt(z) = J,
$$
and hence volume-preserving, $ |\det(D\Sigma^\dt(z))| =1$.
Since differentiation of the numerical flow is equivalent to applying the integrator to the differentiated equations of motion,
it also preserves the symplecticity relation \eqref{t:symplecticity-PQ}: for all $n$,
\begin{equation}\label{t:symplecticity-PQ-n}
(Y^n)^T J Y^n = J \qquad \text{for} \qquad
Y^n =
\begin{pmatrix} \Re Q^n & \Im Q^n \\ \Re P^n & \Im P^n 
\end{pmatrix},
\end{equation}
provided this relation is satisfied for the initial values $(Q^0,P^0)$.

For a smooth potential, the St\"ormer--Verlet method is convergent of second order (see the above references): uniformly for $n$ and $\dt$ with $n\dt\le \bar t$ and for all $z$ in an arbitrary compact set $K$,
$$
|(\Sigma^\dt)^n(z) - \Phi^{n\dt}(z)| \le c\dt^2,
$$
where $c=c(\bar t, K)$. A second-order error bound of the same type also holds true for the derivative,
$$
|D(\Sigma^\dt)^n(z) - D\Phi^{n\dt}(z)| \le c\dt^2,
$$
again because differentiation of the numerical flow is equivalent to applying the integrator to the differentiated equations of motion, and similarly for higher derivatives. Under appropriate conditions on the potential $V$, such as a smooth potential that is bounded together with all its derivatives, the error bounds become uniform for all $z\in\Rb^{2d}$ and not just for $z$ in a compact set.

%%%%%%%%%%%%%%%%%%%%%%%%%%%%%%%%%%%%%%%%%%%%%%%%%%%%%%%%%%%%%%%

\subsection{Strang splitting for the semiclassically scaled 
Schr\"odinger equation}
\label{subsec:t:strang}
We consider the semi-discretisation in time of the Schr\"odinger equation in semiclassical scaling,
\begin{equation}\label{t:tdse}
\I\eps\partial_t \psi = H\psi,\qquad H = -\frac{\eps^2}{2}\Delta_x + V,
\end{equation}
for $t\in\Rb$ and $x\in \Rb^d$ or in a hypercube with periodic boundary conditions, $x\in\T^d$. For the following it is convenient to rewrite the equation more concisely as
\begin{equation}\label{t:eq-XY}
\partial_t \psi = (X+Y)\psi \qquad\text{with }\  X=-\I \eps^{-1} V,\ \ Y= \tfrac12 \I \eps \Delta_x.
\end{equation}
Splitting integrators make use of the fact that the actions of the exponential operators $\exp(tX)$ and $\exp(tY)$ on a wave function are easier to compute (or to approximate) than that of the full solution operator $U(t)=\exp(t(X+Y))$. In fact, 
$\bigl(\exp(tX)\varphi\bigr)(x)= e^{-\I (t/\eps) V(x)}\varphi(x)$ is given by pointwise multiplication in position space, and
$(\mathcal{F} \exp(tY)\varphi\bigr) (\xi) = e^{-\I t \eps |\xi|^2/2} (\mathcal{F}\varphi)(\xi)$ is given by pointwise multiplication in momentum space. 

A widely used splitting method is the Strang splitting  or Mar\v{c}uk splitting, named after \cn{Str68} and \cn{Mar68}, respectively. It approximates the wave function 
$\psi(\dt)=U(\dt)\psi^0$ at a (small) time step $\dt$ by
\begin{equation}\label{t:strang-1}
\psi^1 = S(\dt)   \psi^0 \quad\text{ with } \quad S(\dt)   = \exp(\tfrac12 \dt X) \exp(\dt Y) \exp(\tfrac12 \dt X).
\end{equation}
At further discrete times $t_n=n\dt$ ($n=1,2,\dots$), the exact wave function $\psi(t_n)=U(t_n)\psi^0=U(\dt)^n\psi^0$ is approximated by
\begin{equation}\label{t:strang-n}
\psi^n = S(\dt)  ^n \psi^0.
\end{equation}
There are many more splitting methods: the roles of $X$ and $Y$ can be reversed in the definition of $S(\dt)  $, there are the first-order Lie--Trotter splittings $\psi^1 = \exp(\dt Y) \exp(\dt X)\psi^0$ and $\psi^1 = \exp(\dt X) \exp(\dt Y)\psi^0$, whose concatenation in an alternating way yields the Strang splitting (with twice the step size), and there are various higher-order splittings
$$
\psi^1 = \exp(b_m\dt Y) \exp(a_m\dt X)\dots\exp(b_1\dt Y) \exp(a_1\dt X)\psi^0
$$ 
with suitably chosen coefficients $a_i$ and $b_i$; see, \eg\  \cite{McLQ02,HaiLW06,BlaCM08} and references therein. Here, a method is said to be of order $r$ if, for fixed {\it matrices} $X$ and~$Y$, the error after one time step is $\psi^1-\psi(\dt)=O(\dt^{r+1})$, where the constant implied by the $O$-notation may depend on the norms of the matrices $X$ and~$Y$. This then implies 
$\psi^n-\psi(t_n) = O(\dt^r)$ on bounded time intervals ${t_n\le \bar t}$. The Strang splitting is of order two, as is readily verified by comparison of the Taylor series of the matrix exponentials. 

However, it is by no means obvious that the notion of order based on fixed matrices $X$ and $Y$ has any significance for the situation of the Schr\"odinger equation, which has an unbounded operator $Y$ (for which the exponential series is not defined) and, in the semiclassical scaling, has operators $X$ and $Y$ that depend on the small parameter $\eps$. Second-order convergence of the Strang splitting for the Schr\"odinger equation with a bounded, sufficiently regular potential $V$ and without semiclassical scaling (\ie\  for $\eps=1$), was shown in the $L^2$-norm for $H^2$ initial data by \cn{JahL00}. 

Rigorous error bounds for general splitting methods in the semiclassical scaling $\eps\ll 1$ were proved by
\cn{DesT10}. Here we present their result for the Strang splitting, with a different proof. We require that the initial data are bounded in the following $\eps$-scaled Sobolev norm for $m=2$:
$$
\| \varphi \|_{H^m_\eps} ^2 =
\sum_{|\alpha|\le m} \| \eps^{|\alpha|} \partial^\alpha \varphi \|_{L^2}^2,
$$
where the sum is over all multi-indices $\alpha=(\alpha_1,\dots,\alpha_d)\in \N^d$ with $|\alpha|=\sum_{i=1}^d \alpha_i\le m$,  and 
$\eps^{|\alpha|}\partial^\alpha \varphi=(\eps\partial_1)^{\alpha_1} \dots (\eps\partial_d)^{\alpha_d}\varphi$.

\begin{theorem} [$L^2$-error bound for the Strang splitting] \label{t:thm:strang-schroedinger}
Let the potential $V$ and its partial derivatives up to fourth order be continuous and bounded. Then, the error of the Strang splitting
\eqref{t:strang-1}-\eqref{t:strang-n}
is bounded in the $L^2$-norm by
$$
\| \psi^n - \psi(t_n) \|_{L^2} \le C\, t_n \, \frac{\dt^2}\eps \max_{0\le t \le t_n}\| \psi (t) \|_{H^2_\eps}, \qquad n\ge 0,
$$
where $C$ is independent of $\eps$, $\dt$ and $n$ (but depends on bounds of partial derivatives of $V$ up to order $4$).
\end{theorem}

The $H^m_\eps$-norms are appropriate norms for the semiclassical Schr\"odinger equation, 
as the following regularity result shows.

\begin{lemma} [$\mathbf{\eps}$-uniform wellposedness in $H^m_\eps$] \label{t:lem:Hmeps}
Let $m\ge 0$ be an integer, and assume that the potential $V$ and its partial derivatives up to order $m$ are continuous and bounded. Then, the $H^m_\eps$-norm of any solution $\psi(t)=U(t)\psi^0$ of the semiclassically scaled Schr\"odinger equation \eqref{t:tdse} with initial data $\psi^0\in H^m_\eps$ is bounded by
$$
\| \psi(t) \|_{H^m_\eps} \le (1+ct)^m \, \| \psi^0 \|_{H^m_\eps}, \qquad t\ge 0,
$$
where $c$ is independent of $\eps$ and $t$ and $\psi^0$ (but depends on $m$ and bounds of partial derivatives of $V$ up to order $m$). Moreover, the same bound holds for $\exp(-\I t\eps^{-1}V)\psi^0$, and $\exp(\I t\eps  \Delta)$ is an isometry with respect to the  $H^m_\eps$-norm.
\end{lemma}

\begin{proof}
The proof proceeds by induction on $m$. The result clearly holds for $m=0$. Let $m\ge 1$, and let $\alpha\in\N^d$ be a multi-index with $|\alpha|=m$.
We denote $\psi^{(\alpha)} = \eps^{|\alpha|}\partial^\alpha \psi$, which satisfies the equation
$$
\I\eps \partial_t \psi^{(\alpha)} = H\psi^{(\alpha)} + [ \eps^{|\alpha|}\partial^\alpha, H ] \psi.
$$
Here, the commutator equals
$$
[ \eps^{m}\partial^\alpha, H ] \psi = [ \eps^m\partial^\alpha, V ] \psi = \sum_{|\beta|\le m-1} V_{\alpha,\beta} \,\psi^{(\beta)} \eps^{m-|\beta|},
$$
where the functions $V_{\alpha,\beta}$, which are integer multiples of partial derivatives of $V$, result from the product rule 
$\partial^\alpha (V\psi)= V\partial^\alpha \psi+ \sum_{|\beta|\le m-1} V_{\alpha,\beta} \,\partial^\beta\psi$.
By Lemma~\ref{lem:stability} we then have in the $L^2$-norm
\begin{align*}
\| \psi^{(\alpha)} (t)\| &\le \| \psi^{(\alpha)} (0)\| + \int_0^t \,\Bigl\| \frac1\eps [ \eps^{|\alpha|}\partial^\alpha, H ] \psi(s)\Bigr\| \, \D s
\\
&\le  \| \psi^{(\alpha)} (0)\| + \int_0^t \, \Bigl\|\sum_{|\beta|\le m-1} V_{\alpha,\beta} \,\psi^{(\beta)}(s) \eps^{m-1-|\beta|}\Bigr\| \, \D s.
\end{align*}
Using the induction hypothesis for $\psi^{(\beta)}(s)$ with $|\beta|\le m-1$, we obtain
$$
\| \psi^{(\alpha)} (t)\| \le \| \psi^{(\alpha)} (0)\| + ct(1+ct)^{m-1} \, \| \psi(0) \|_{H^{m-1}_\eps},
$$
and the result for $\psi(t)$ follows. The result for $\exp(-\I t\eps^{-1}V)$ is obtained by the same argument, and the isometry property of
$\exp(\I t\eps  \Delta)$ on $H^m_\eps$ follows immediately by using Fourier transformation and the Plancherel formula.
\end{proof}

We further need the following estimates.

\begin{lemma}[commutator bounds] \label{t:lem:comm-bounds}
The commutators and the iterated commutators of the operators  $ X=-\I \eps^{-1} V$ and $Y= \tfrac12 \I \eps \Delta$ are bounded  by
\begin{align*}
&\| [X,Y]\varphi \|_{L^2} \le \frac{c_1}\eps \| \varphi \|_{H^{1}_\eps} ,
\\
& \| [X, [X,Y]]\varphi \|_{L^2} +  \| [Y, [Y,X]]\varphi \|_{L^2} \le \frac{c_2}\eps \| \varphi \|_{H^2_\eps},
\end{align*}
where $c_1$ and $c_2$ are independent of $\eps$ and $\varphi\in H^2_\eps$, but depend on second- and fourth-order derivatives of $V$, respectively.
\end{lemma}

\begin{proof} The bounds follow from a straightforward direct calculation.
\end{proof}

%\bch
%\begin{remark} In the proof of Proposition~\ref{prop:comm}, we have calculated that
%\[
%\tfrac{1}{\I\eps}[a,-\tfrac{\eps^2}{2}\Delta_x] = \op(\nabla_q a\cdot p)
%\]
%This implies for $a=V$ that
%\[
%[X,Y] = 
%\tfrac{1}{\I\eps}\op(\nabla_q V\cdot p) = 
%\tfrac{1}{\I\eps}\left(\nabla_q V\cdot\hat p + \hat p \cdot\nabla_q V\right).
%\]
%Hence, we obtain the estimate
%\[
%\|[X,Y]\varphi\|_{L^2} \le 
%\tfrac{2}{\eps} \sup_{|\alpha|\le 2} \|\partial^\alpha V\|_\infty\  \|\varphi\|_{H^1_\eps}.
%\]
%If $V$ is subquadratic, then there exists constant $\tilde c<\infty$
%\[
%\|[X,Y]\varphi\|_{L^2} \le 
%\tfrac{\tilde c}{\eps}\  \|\varphi\|_{\Sigma^1_\eps}.
%\]
%For the double commutator, we obtain 
%\[
%[Y,[X,Y]] =
%\tfrac{1}{(\I\eps)^2}[-\tfrac{\eps^2}{2}\Delta_x,\op(\nabla_q V\cdot p)] = 
%-\tfrac{1}{\I\eps}\op(p\cdot\nabla^2 V p)
%\]
%and therefore, there exist $d_2,\tilde d_2<\infty$ with
%\[
%\|[Y,[X,Y]]\varphi\|_{L^2} \le \tfrac{d_2}{\eps} \sup_{|\alpha|\le 4} \|\partial^\alpha V\|_\infty\  
%\|\varphi\|_{H^2_\eps}
%\]
%and
%\[
%\|[Y,[X,Y]]\varphi\|_{L^2} \le \tfrac{\tilde d_2}{\eps} \sup_{|\alpha|\le 4} \|\varphi\|_{\Sigma^2_\eps}
%\]
%Similarly we obtain for the other double commutator
%\[
%[X,[X,Y]] = \tfrac{1}{(\I\eps)^2}[V,\op(\nabla_q V\cdot p)] = \tfrac{1}{\I\eps} \op(\{V,\nabla_q V\cdot p\}) 
%= \tfrac{1}{\I\eps} |\nabla_q V|^2,
%\]
%which implies 
%\[
%\|[X,[X,Y]]\varphi\|_{L^2} \le \tfrac{1}{\eps} \sup_{|\alpha|\le 1} \|\partial^\alpha V\|_\infty^2\  
%\|\varphi\|_{L^2}
%\]
%and the existence of $d<\infty$
%\[
%\|[X,[X,Y]]\varphi\|_{L^2} \le \tfrac{d}{\eps}\|\varphi\|_{\Sigma^2_\eps}
%\]
%\end{remark}
%\ech

\begin{proof} (of Theorem~\ref{t:thm:strang-schroedinger}) 
To avoid swarms of factors $1/2$ swirling around, it is convenient to consider, as a function of $t$ that will later be evaluated at $t=\dt/2$,
$$
S(2t)=  \exp(t X) \exp(2t Y) \exp(t X).
$$
The solution operator $U(2t)=\exp(2t(X+Y))$ satisfies
$$
\frac \D{\D t} U(2t) =  (X+Y) U(2t) + U(2t)(X+Y),
$$
and time differentiation of $S(2t)$ yields
\begin{align}
\nonumber
\frac \D{\D t} S(2t) &=   X S(2t)  + \exp( t X)  Y\exp(t Y) \exp( t X) 
\\ 
\nonumber
&\quad +
\exp( t X) \exp(t Y)  Y \exp( t X) + S(2t)    X
\\
\label{t:evolution-with-defect}
&=  (X+Y) S(2t)  + S(2t)   (X+Y) + R(2t)
\end{align}
with the defect
$$
R(2t) =  [ \exp( t X),Y] \exp(- tX) S(2t) -  S(2t)  \exp(- tX) [  \exp( t X),Y ].
$$
By the variation-of-constants formula, we therefore have
\begin{equation} \label{t:R-eq}
S(\dt) = U(\dt) + E(\dt) \quad\text{ with }\quad E(\dt) = \tfrac12 \int_0^{\dt} U(\tfrac12(\dt-s)) R(s) U(\tfrac12(\dt-s)) \D s.
\end{equation}
The terms before and after $S(2t)$ in the formula of $R(2t)$ equal the following expressions, as can be verified by differentiating the left-hand sides and integrating from $0$ to $t$:
\begin{align}
\nonumber
 [ \exp(  t X),Y] \exp(-  tX) &=   \int_0^t \exp(  sX) [X,Y] \exp(-  sX) \D s
 \\
 \label{t:exp-comm}
  \exp(-  tX) [  \exp(  t X),Y ] &=   \int_0^t \exp(-  sX) [X,Y] \exp(  sX) \D s.
\end{align}
Here we observe the following identity for the integrands, which is again verified by differentiation,
\begin{align}
\nonumber
&\exp(\pm  sX) C \exp(\mp  sX)  
\\
&\qquad = C \pm   \int_0^s \exp(\pm  rX) [X,C] \exp(\mp  rX) \D r.
\label{t:exp-comm-2}
\end{align}
So we obtain
\begin{equation}\label{t:R1}
R(2t) =  t [[X,Y] ,S(2t)]  + R_2(2t)
\end{equation}
with
\begin{align*}
R_2(2t) = &  \int_0^t \int_0^s \exp(  rX) [X,[X,Y]] \exp(-  rX) \D r \D s\, S(2t) \\
&+
S(2t)  \int_0^t \int_0^s \exp(-  rX) [X,[X,Y]] \exp(  rX) \D r \D s.
\end{align*}
With  Lemmas~\ref{t:lem:Hmeps} and~\ref{t:lem:comm-bounds} we obtain the bound
$$
\| R_2(2t) \varphi \|_{L^2} \le t^2 \,  \frac{c_2(\gamma+1)}{2\eps}  \, \| \varphi \|_{H^2_\eps},
$$
where $\gamma$ is a bound of $\exp(tY)$ as an operator on $H^2_\eps$, as provided by Lemma~\ref{t:lem:Hmeps}.
We now turn to the first term in $R(2t)$.
With the commutator $C=[X,Y]$, and observing that the commutator has the derivative-like product rule $[C,AB]=[C,A]B + A[C,B]$, we have
\begin{align}\nonumber
[C,S(2t)] =& \ [C,\exp(  t X)] \exp(2t Y) \exp(  t X) 
\\ 
\label{t:CS-comm}
&
+\exp(  t X) [C,\exp(2t Y)] \exp(  t X) 
\\ 
\nonumber
&
+ \exp(  t X) \exp(2t Y) [C,\exp(  t X)].
\end{align}
From \eqref{t:exp-comm} (with $C$ instead of $Y$) we have
$$
[C,\exp(  t X)]  = -   \int_0^t \exp(  (t-s) X)  [X,C]  \exp(  s X) \D s
$$
and likewise (with $C$ instead of $Y$ and $2Y$ instead of $  X$)
$$
[C,\exp(2t Y)]  = - \int_0^t \exp((t-s) 2Y)  [2Y,C]  \exp(s \,2Y) \D s.
$$
Using these formulae in \eqref{t:CS-comm}, together with the mapping properties of $\exp(tX)$ and $\exp(tY)$ as stated in Lemma~\ref{t:lem:Hmeps} and the commutator bounds of Lemma~\ref{t:lem:comm-bounds}, we obtain
$$
\| [C,S(2t)] \varphi \|_{L^2} \le 2t \,\frac{c_2}\eps \, \|\varphi\|_{H^2_\eps}.
$$
This gives us, with $c=\tfrac14 c_2 (1+(\gamma+1)/2)$,
$$
\| R(2t) \varphi \|_{L^2} \le t^2 \, \frac{8c}\eps \, \|\varphi\|_{H^2_\eps},
$$
and using this bound in the formula \eqref{t:R-eq} for the error operator $E(\dt)=S(\dt)-U(\dt)$ then yields the following local error bound, for one step of the Strang splitting method,
\begin{equation}\label{t:strang-local-error}
\| S(\dt)\varphi - U(\dt)\varphi \|_{L^2} \le c\, \frac{\dt^3}\eps\, \|\varphi\|_{H^2_\eps}.
\end{equation}
The error after $n$ steps can be written as
\begin{align*}\psi^n-\psi(t_n) &= S(\dt)  ^n\psi^0 - U(\dt)^n\psi^0
%\\&
= \sum_{j=0}^{n-1} S(\dt)  ^{n-j-1} \,(S(\dt)  -U(\dt)) \, U(\dt)^j\psi^0 
\\&
= \sum_{j=0}^{n-1} S(\dt)  ^{n-j-1} \,(S(\dt)  -U(\dt)) \psi(t_j).
\end{align*}
Since $S(\dt)$  is a unitary operator,
we finally obtain from \eqref{t:strang-local-error}
$$
\|\psi^n-\psi(t_n)\|_{L^2} \le n \,c\, \frac{\dt^3}\eps\, \max_{0\le j \le n-1}\|\psi(t_j)\|_{H^2_\eps}
\le c \,t_n\, \frac{\dt^2}\eps\, \max_{0\le t \le t_n}\|\psi(t)\|_{H^2_\eps},
$$
which is the stated result.
\end{proof}

\begin{remark}
For potential functions $V$ that are smooth and subquadratic in the sense that 
all partial derivatives of order $\ge 2$ are bounded,
%$\|\partial^k V\|_\infty<\infty$ for all $k\in\N_0^d$ with $|k|\ge 2$, 
$\eps$-uniform 
well-posedness holds as well if the Sobolev norms $\|\cdot\|_{H^m_\eps}$ are 
replaced by stronger norms that additionally control the decay at infinity,
\[
\|\varphi\|_{\Sigma^m_\eps} = \sum_{|\alpha|\le m} 
\left(\| \eps^{|\alpha|} \partial^\alpha \varphi \|_{L^2}^2 + \| x^\alpha \psi \|_{L^2}^2\right);
\]
see \cite[Proposition~A.2]{Car13}. For subquadratic potentials, commutator bounds as in Lemma~\ref{t:lem:comm-bounds} can be shown in the $\Sigma^m_\eps$- instead of $H^m_\eps$-norms. Therefore, the proof of Theorem~\ref{t:thm:strang-schroedinger} yields the same second-order error bounds for the Strang splitting also in the case of a subquadratic potential, but with the $\Sigma^m_\eps$- instead of $H^m_\eps$-norms.
\end{remark}

While the $L^2$-error of Strang splitting is only $O(\dt^2/\eps)$, the following result shows that the error bound for observables can be improved to $O(\dt^2+\eps^2)$.

\begin{theorem} [error in observables for the Strang splitting] \label{t:thm:strang-schroedinger-obs}
Let the potential $V$ be smooth and let its partial derivatives of order $\ge 2$ be bounded. Let the observable $A=\op(a)$ be the Weyl quantisation of a Schwartz function $a:\Rb^{2d}\to\Rb$.
Then, the error of the Strang splitting 
\eqref{t:strang-1}-\eqref{t:strang-n} in the expectation values of  $A$ on the time interval $0\le t \le \bar t$
is bounded  by
$$
| \langle A \rangle_{\psi^n} - \langle A \rangle_{\psi(t_n)} | \le c\, (\dt^2+\eps^2) ,
$$
where $c<\infty$ is independent of $\psi^0$ of unit $L^2$-norm, of $\eps$, $\dt$ and $n$ with $n\tau\le \bar t$ (but depends on $\bar t$).
\end{theorem}

\begin{proof} The proof uses Egorov's theorem (Theorem~\ref{theo:egorov}), the approximate representation of expectation values of observables via the Husimi function (Theorem~\ref{theo:husimi}), and error bounds of the St\"ormer--Verlet method, which is the Strang splitting for the classical  equations of motion.

With the kinetic energy operator $T=-\tfrac12\eps^2\Delta$, the Strang splitting operator reads
$$
S^\dt = \exp(\tfrac{\dt}{2\I\eps} V)\, \exp(\tfrac{\dt}{\I\eps} T)\, \exp(\tfrac{\dt}{2\I\eps} V).
$$
For classical motion under the Hamiltonian 
$$
h(q,p)=\tfrac12|p|^2 + V(q) \equiv h_T(p) + h_V(q)
$$ with corresponding flows
$\Phi^t$, $\Phi^t_T$, $\Phi^t_V$, the Strang splitting is the St\"ormer--Verlet method (see Section~\ref{subsec:t:sv}) with the one-step map
$$
\Sigma^\dt = \Phi^{\dt/2}_V \circ \Phi^\dt_T \circ  \Phi^{\dt/2}_V .
$$
We then write, with $\varphi= \exp(\tfrac{\dt}{\I\eps} T)\, \exp(\tfrac{\dt}{2\I\eps} V)\psi^0$,
$$
\langle A \rangle_{S^\dt\psi^0} = \langle  \exp(-\tfrac{\dt}{2\I\eps} V) A  \exp(\tfrac{\dt}{2\I\eps} V)\rangle_{\varphi}.
$$
By Egorov's theorem (Theorem~\ref{theo:egorov}), we have, uniformly for $\varphi$ of unit norm, that the right-hand side is
$$
\langle  \exp(-\tfrac{\dt}{2\I\eps} V) A  \exp(\tfrac{\dt}{2\I\eps} V)\rangle_{\varphi} = 
\langle \op(a \circ \Phi^{\dt/2}_V) \rangle_\varphi + O(\tau\eps^2).
$$
Since $a\circ \Phi^{\dt/2}_V$ still satisfies the assumptions of Theorem~\ref{theo:egorov}, we obtain further, 
for $\eta= \exp(\tfrac{\dt}{2\I\eps} V)\psi^0$,
\begin{align*}
\langle \op(a \circ \Phi^{\dt/2}_V) \rangle_\varphi 
&= \langle \exp(-\tfrac{\dt}{\I\eps} T)\,\op(a \circ \Phi^{\dt/2}_V) \exp(\tfrac{\dt}{\I\eps} T)\,\rangle_\eta
\\
&= \langle \op(a \circ \Phi^{\dt/2}_V \circ \Phi^\dt_T) \rangle_\varphi + O(\tau\eps^2)
\end{align*}
and yet further
$$
\langle A \rangle_{S^\dt\psi^0}  =  \langle \op(a \circ \Sigma^\dt) \rangle_{\psi^0} + O(\tau\eps^2).
$$
For arbitrary $n$ with $n\dt\le \bar t$, it can be verified that all derivatives of $a\circ(\Sigma^\dt)^n$ are bounded independently of
$n$ and $\dt$ (because differentiating the St\"ormer--Verlet method is the same as applying the method to the differentiated equations of motion), and hence this relation extends to
$$
\langle A \rangle_{\psi^n}  =  \langle \op(a \circ (\Sigma^\dt)^n) \rangle_{\psi^0} + O(n\tau\eps^2).
$$
By Theorem~\ref{theo:husimi} we thus have, with the Husimi function $\calH_{\psi^0}$ of the initial data,
\begin{equation}\label{t:husimi-num}
\langle A \rangle_{\psi^n} = \int_{\Rb^{2d}} \calH_{\psi^0}(z) \, 
\Bigl( a \circ (\Sigma^\dt)^n(z) - \tfrac\eps 4 \Delta (a \circ (\Sigma^\dt)^n)(z) \Bigr) \D z + O(\eps^2)
\end{equation}
together with the corresponding formula for the average over the exact wave function $\psi(t)$ at $t=n\dt$,
\begin{equation}\label{t:husimi-ex}
\langle A \rangle_{\psi(t)} = \int_{\Rb^{2d}} \calH_{\psi^0}(z) \, 
\Bigl( a \circ \Phi^t(z) - \tfrac\eps 4 \Delta (a \circ \Phi^t)(z) \Bigr) \D z + O(\eps^2).
\end{equation}
Under the given assumptions on $a$ and $V$,  the classical St\"ormer--Verlet method satisfies the second-order error bounds
\begin{align*}
&\| a \circ (\Sigma^\dt)^n - a \circ \Phi^t \|_{L^\infty(\Rb^{2d})}  \le C \tau^2,
\\
&\| \Delta (a \circ (\Sigma^\dt)^n) - \Delta(a \circ \Phi^t) \|_{L^\infty(\Rb^{2d})}  \le C \tau^2.
\end{align*}
Since the Husimi function $\calH_{\psi^0}$ is a probability density, using these error bounds in the difference of \eqref{t:husimi-num} and \eqref{t:husimi-ex}
yields the result.
\end{proof}

\begin{remark} Using higher-order versions of the Egorov theorem, the $\eps^2$ error term can be reduced to $\eps^N$ for arbitrary $N\ge 2$. As this refinement becomes very technical, we do not present it here.
\end{remark}

\begin{remark} \cn{GolJP19} give another error bound for the Strang splitting that is robust as $\eps\to 0$. They prove an $O(\tau^2 + \eps^{1/2})$ error bound for the quadratic Monge--Kantorovich or Wasserstein distance between the Husimi functions of the approximate and the exact quantum density operators.
\end{remark}

\subsection{Periodisation and full discretisation: Split-step Fourier method} 
\label{subsec:t:ssf}
For the actual computation, the partial differential equation is truncated to a finite domain, usually by restriction to a sufficiently large interval (or square, cube, hypercube in higher dimensions) and periodisation. This is a reasonable approach  as long as the wave function is well localised, as remains the case in the semiclassical setting up to the Ehrenfest time; see Section~\ref{sec:hagwp}. The $L^2$-error analysis of the Strang splitting extends immediately to the periodised equation, since the same commutator bounds remain valid. 

For a full discretisation, the periodised problem then needs to be discretised in space. This is usually done by Fourier collocation. For notational simplicity, let us consider the one-dimensional situation. Here, the interval, which we rescale to $[-\pi,\pi]$, is discretized with $K=2^L$ equidistant grid points $x_j=j\,2\pi/K$ ($j=-K/2,\dots,K/2-1$), and the wave function $\psi(x_j,t)$ at these points is approximated by values $\psi_j(t)$, which are collected in a vector $\bfpsi(t)\in\C^K$. The Schr\"odinger equation \eqref{t:eq-XY} (on the interval with periodic boundary conditions) is then replaced by a system of ordinary differential equations,
$$
\frac{\D}{\D t} \bfpsi = \mathbf{X} \bfpsi + \mathbf{Y} \bfpsi.
$$
Here, the matrix $\mathbf{X}$ is diagonal with entries $-\I \eps^{-1} V(x_j)$, and \bch$\mathbf{Y} =  \mathbf{F}^{-1} \mathbf{D} \mathbf{F}$\ech, where $\mathbf{F}$ denotes the discrete Fourier transform acting on vectors of dimension $K$, and $\mathbf{D}$ is the diagonal matrix with entries $d_k = \tfrac12 \I \eps k^2$ ($k=-K/2,\dots,K/2-1$).

The Strang splitting in this context reads
$$
\bfpsi^{n+1} = \exp(\tfrac12 \dt  \mathbf{X}) \exp(\dt  \mathbf{Y}) \exp(\tfrac12 \dt  \mathbf{X}) \bfpsi^{n},
$$
or equivalently, for the vector $\widehat \bfpsi = \mathbf{F} \bfpsi$ of Fourier coefficients,
\begin{equation}\label{t:ssf-f}
\widehat \bfpsi^{n+1} = \mathbf{F}\exp(\tfrac12 \dt  \mathbf{X}) \mathbf{F}^{-1} \exp(\dt  \mathbf{D}) \mathbf{F} \exp(\tfrac12 \dt  \mathbf{X}) \mathbf{F}^{-1} \widehat \bfpsi^{n}.
\end{equation}
Algorithmically, this method alternates between pointwise multiplications of vectors and fast Fourier transforms that implement the actions of $\mathbf{F}$ and $\mathbf{F}^{-1}$ on vectors. It is known as the {\em split-step Fourier method} and goes back to \cn{HarT73} and \cn{FleMF76}. \cn{BaoJM02}  consider this method in the semiclassical setting and prove the first rigorous $L^2$-error bounds of the analogous method based on the first-order Lie--Trotter splitting instead of the Strang splitting.

With the vector $\bfpsi^{n}\in\C^K$ we associate its trigonometric interpolation $\psi_K^n\in C_{\rm per}[-\pi,\pi]$. The error bound of Theorem~\ref{t:thm:strang-schroedinger} then extends to the following error bound for the fully discrete method.

\begin{theorem} [$L^2$-error bound for the split-step Fourier method] \label{t:thm:ssf}
Assume that the potential $V$ and its partial derivatives up to $m$th order are continuous and bounded, with $m\ge 4$. For the split-step Fourier method with the starting value $\psi_K^0$ taken as the $K$-point trigonometric interpolation of the given initial value $\psi^0$, the error
is  bounded in the $L^2$-norm by
$$
\| \psi_K^n - \psi(t_n) \|_{L^2} \le c\, \frac{t_n}\eps \, \bigl( \dt^2 +(K\eps)^{-m}\bigr) 
\max_{0\le t \le t_n}\| \psi (t) \|_{H^{m}_\eps}, \qquad n\ge 0,
$$
where $c$ is independent of $\eps$, $\dt$, $n$ and $K$ (but depends on $m$ and on bounds of partial derivatives of $V$).
\end{theorem}

For an accurate approximation, the number $K$ of grid points must be chosen considerably larger than $\eps^{-1}$, \ie\  the grid spacing must be considerably smaller than~$\eps$; this is actually already required for the trigonometric interpolation of initial data that are bounded in $H^m_\eps$. On the other hand, a time step size $\dt=O(\eps)$ still suffices for obtaining $O(\dt)$ accuracy.

\begin{proof} We interpret the split-step Fourier method as the space-continuous Strang splitting method with inexact evaluation of $\exp(tX)\varphi$. To this end, let  $P_K$ denote the space of trigonometric polynomials that are linear combinations of $K$ exponentials $e^{ikx}$ with $k$ ranging from $-K/2$ to $K/2-1$.
We write $I_K f\in P_K$ for the trigonometric interpolation of a periodic function $f$ and we note that $I_K\varphi = \varphi$ for all $\varphi\in P_K$. We further need the simple fact that for a function $\varphi$ with the vector of grid values $\bfphi\in\C^K$, the vector of Fourier coefficients of $I_K\varphi$ is given by $\mathbf{F}\bfphi$.

Starting from $\psi_K^0\in P_K$,  the split-step Fourier method computes $\psi_K^n$ as
$$
\psi_K^n = S_K(\dt)^n \psi_K^0, \quad\text{with}\ \
S_K(\dt) = I_K  \exp(\tfrac12 \dt X) I_K \exp(\dt Y) I_K \exp(\tfrac12 \dt X).
$$
Since for $\varphi\in P_K$
we  also have $\exp(tY)\varphi \in P_K$, the expression for $S_K(\dt)$ simplifies to
$$
S_K(\dt) =  I_K \exp(\tfrac12 \dt X)  \exp(\dt Y)  I_K \exp(\tfrac12 \dt X),
$$
so that studying the error $S_K(\dt) - S(\dt)$ reduces to studying the error in the interpolation of $\exp(\dt X)\varphi$. Now, the trigonometric interpolation error is known to satisfy the following bound; see, \eg\  \cite[p.\,77]{Lub08},
\begin{equation}\label{t:int-err}
\| I_K f - f \|_{L^2} \le C_m K^{-m} \, \| \partial_x^m f \|_{L^2},\qquad m\ge 1.
\end{equation}
With $f=\exp(\dt X)\varphi=e^{-\dt\I\eps^{-1} V} \varphi$ for $\varphi\in P_K$, 
rewritten in the form $\exp(\dt X)\varphi = \varphi + \dt X \int_0^1 \exp(\theta \dt X) \varphi\D\theta$,
this implies
$$
\| I_K (\exp(\dt X) \varphi) -  \exp(\dt X)\varphi \|_{L^2} \le c \,\frac\dt\eps (K\eps)^{-m} \| \varphi \|_{H^m_\eps} ,
$$
where $c$ is independent of $\eps$, but depends on bounds of derivatives of $V$. Since $\exp(\tfrac12\dt Y)$ is an isometry on $L^2$ and $H^m_\eps$, this implies that 
$$
\| S_K(\dt)\psi - S(\dt)\psi \|_{L^2} \le c \frac\dt\eps (K\eps)^{-m} \| \psi \|_{H^m_\eps}.
$$
Together with \eqref{t:strang-local-error} and \eqref{t:int-err}, this shows that the local error of the method is bounded, for all $\psi\in H^m_\eps$, by
$$
\| S_K(\dt)\psi - U(\dt)\psi \|_{L^2} \le c \frac\dt\eps \bigl(\dt^2 + (K\eps)^{-m}\bigr) \| \psi \|_{H^m_\eps}.
$$
Furthermore, \eqref{t:ssf-f} shows that $S_K(\dt)$ is conjugate to a unitary operator on $P_K$. By the same argument for the global error as at the end of the proof of Theorem~\ref{t:thm:strang-schroedinger}, beginning after~\eqref{t:strang-local-error} with $S_K(\dt)$ in place of $S(\dt)$, we therefore obtain the stated result.
\end{proof}

%By the same perturbation argument, also Theorem~\ref{t:thm:strang-schroedinger-obs} extends to the full discretisation. It is clearly seen that the problematic error term comes only from the space discretisation.
%
%\begin{theorem} [Error in observables for the split-step Fourier method] \label{t:thm:strang-schroedinger-obs-full}
%Assume that the potential $V$ is smooth 
%and its partial derivatives of all orders are continuous and bounded.
%Let the observable $A=\op(a)$ be the Weyl quantisation of a Schwartz function $a:\Rb^{2d}\to\Rb$. For the split-step Fourier method with the starting value $\psi_K^0$ taken as the $K$-point trigonometric interpolation of the given initial value $\psi^0$, the error
%in the expectation values of  $A$ on the time interval $0\le t \le \bar t$
%is bounded  by
%$$
%| \langle A \rangle_{\psi_K^n} - \langle A \rangle_{\psi(t_n)} | \le c\, (\dt^2+\eps^2) +
%c\, \frac{t_n}\eps \,(K\eps)^{-m}
%\max_{0\le t \le t_n}\| \psi (t) \|_{H^{m}_\eps}
%$$
%where $c<\infty$ is independent of $\psi^0$ of unit $L^2$-norm, of $\eps$, $\dt$ and $n$ with $n\tau\le \bar t$ (but depends on $\bar t$).
%\end{theorem}

The extension to higher dimensions is trivial in theory for a full tensor grid, but computationally this is not feasible except for very small dimensions. The extension of the split-step Fourier method to sparse grids in higher dimensions was studied by~\cn{Gra07a} 
and \cn{Gra07b} 
for the Schr\"odinger equation without semiclassical scaling ($\eps=1$). 
%A certain difficulty with this approach is that the discrete Fourier transform on sparse grids, which is due to \cn{Hal92}, is no longer a unitary operator. Even worse i
In the semiclassical case $\eps\ll 1$, 
the sparse-grid Fourier interpolation error estimate, which is derived by \cn{Gra07a}, leads to a catastrophic scaling in $\eps$ for functions in $H^m_\eps$ with sufficiently large $m$: denoting by $I_\Gamma$ the interpolation operator on a sparse grid $\Gamma$ in $d$ dimensions that corresponds to a hyperbolic cross for the Fourier coefficients (see, \eg\  \cn{BunG04} for these notions) with a maximum of $K=2^L$ grid points in each co-ordinate direction, the interpolation error bound of \cn{Gra07a} can be rewritten as
$$
\| I_\Gamma f - f \|_{L^2} \le C\, L^{d-1} K^{-m+1} \eps^{-md} \| (\eps\partial_1)^m\dots (\eps\partial_d)^m f \|_{L^2}.
$$
This indicates that the total number of grid points must be chosen larger than $\eps^{-d}$, the same condition that arises also for full tensor grids. In view of these theoretical considerations and confirmed by numerical experiments, sparse grids are not appropriate for the direct discretisation of higher-dimensional semiclassical Schr\"odinger equations.

\cn{SuzSN19} proposed and studied a version of the split-step Fourier method with lattice rules for the full discretisation of high-dimensional Schr\"odinger equations (for $\eps\sim 1$, not in the semiclassical scaling). In contrast to the non-unitary sparse grid Fourier transform, the lattice rule uses a unitary fast Fourier transform, which leads to better norm and energy conservation in numerical experiments. The obtained error bounds are in terms of bounds of mixed derivatives like for the sparse grid method and therefore lead to the same unfavourable scaling in $\eps$ for functions in $H^m_\eps$ with $m\ge 1$.

\subsection{Symmetric Zassenhaus splitting}
\label{subsec:t:zass}
Higher orders of convergence of the time discretisation can be achieved in different ways. As an interesting alternative to splitting methods with several factors $\exp(a_i tX)$ and $\exp(b_i tY)$, the symmetric Zassenhaus splitting proposed by \cn{BadIKS14} approximates
$\psi(t)=\exp(t(X+Y))\psi^0$  by $\psi^n = Z(\dt)^n \psi^0$ at $t=n\dt$, where
\begin{align}
\nonumber
Z(\dt) = &\exp(\tfrac12 \dt X) \exp(\tfrac12\dt Y) \exp(2\dt^3C_3) \exp(\tfrac12\dt Y)\exp(\tfrac12 \dt X)
\\
\label{t:zass}
&\text{with } \ 
C_3=\frac1{48}\,[X,[X,Y]] + \frac1{24}\, [Y,[X,Y]],
\end{align}
or higher-order versions of the type
\begin{align*}
Z_m(\dt) = &\exp(\tfrac12 \dt X) \exp(\tfrac12\dt Y) \exp(\dt^3C_3)\exp(\dt^5C_5)\dots\exp(\dt^{2m-1}C_{2m-1}) \times
\\
&\exp(\dt^{2m-1}C_{2m-1})\dots \exp(\dt^5C_5) \exp(\dt^3C_3)\exp(\tfrac12\dt Y)\exp(\tfrac12 \dt X)
\end{align*}
with skew-hermitian operators $C_{k}$, for odd $k$, which are linear combinations of $(k-1)$-fold commutators of $X$ and $Y$. An elegant algorithm for their construction was given by \cn{ArnCC17}. \cn{BadIKS14} applied such methods  to space discretisations of the semiclassical Schr\"odinger equation, under appropriate scaling relations between the semiclassical parameter $\eps$, the step size $\dt$, and the spatial mesh size, which imply that $\dt^3C_3$ is of moderate or small norm so that the action of its exponential on a vector can be approximated efficiently using a few Lanczos iterations, which only require just as few matrix-vector multiplications. The Lanczos method for the action of the  exponential of a skew-hermitian matrix was first proposed by \cn{ParL86}; see \cite[Section III.2.2]{Lub08} for {\it a priori} and {\it  a posteriori} error analyses.

Here we study the error of the symmetric Zassenhaus splitting \eqref{t:zass} for the semiclassical Schr\"odinger equation without any space discretisation, as we did previously for the Strang splitting. By the same arguments as in the proof  of Theorem~\ref{t:thm:ssf}, the result can be extended to the fully discrete situation with Fourier collocation in space.

\begin{theorem} [$L^2$-error of the symmetric Zassenhaus splitting] \label{t:thm:zass}
\quad\\
Assume that the potential $V$ together with its partial derivatives up to sufficiently high order is bounded. Then, the error of the symmetric Zassenhaus splitting $\psi^n = Z(\dt)^n \psi^0$ with $Z(\dt)$ of
\eqref{t:zass} is bounded in the $L^2$-norm by
$$
\| \psi^n - \psi(t_n) \|_{L^2} \le C\, t_n \, \frac{\dt^4}\eps \max_{0\le t \le t_n}\| \psi (t) \|_{H^4_\eps}, \qquad n\ge 0,
$$
where $C$ is independent of $\eps$, $\dt$ and $n$ (but depends on bounds of partial derivatives of $V$).
\end{theorem}

\begin{proof}
The proof proceeds similarly to the proof of the error bound for the Strang splitting. 
%By the same arguments, we first obtain a preliminary second-order local error bound that shows that the inclusion of the extra exponential $\exp(2\dt^3C_3)$, written for $L=2\dt^3C_3$ as
%$$
%\exp(L) = 1 +  L\int_0^1 \exp(\theta L) \, \D\theta,% \quad\text{ with the entire function } \varphi_1(z)= \frac{e^z-1}z,
%$$
%does no harm: as in \eqref{t:strang-local-error}, we obtain
%$$
%\| Z(\dt)\psi^0 - U(\dt)\psi^0 \|_{L^2} \le \tilde c\, \frac{\dt^3}\eps\, \|\psi^0\|_{H^2_\eps}.
%$$
As in \eqref{t:evolution-with-defect}, we obtain
$$
\frac \D{\D t} Z(2t) =  (X+Y) Z(2t) + Z(2t)(X+Y) + R(2t)
$$
with the defect
\begin{align*}
R(2t) =  &\ [ \exp( t X),Y] \exp(- tX) Z(2t) -  Z(2t)  \exp(- tX) [  \exp( t X),Y ]
\\
&\ + [\exp(tX)\exp(tY),48\,t^2 C_3] \exp(-tY) \exp(-tX) Z(2t) \\
&\qquad - Z(2t)  \exp(- tX) \exp(-tY) [\exp(tY)\exp(tX), 48\, t^2 C_3]
\\
&\ + 48\,t^2 \bigl(C_3 Z(2t) + Z(2t) C_3 \bigr).
\end{align*}
By repeated use of \eqref{t:exp-comm-2} we expand
\begin{align*}
 &[ \exp(  t X),Y] \exp(-  tX) =   \int_0^t \exp(  sX) [X,Y] \exp(-  sX) \D s
 \\
 &\qquad= t[X,Y] + \frac{t^2}{2!}[X,[X,Y]] + \frac{t^3}{3!}[X,[X,[X,Y]]] +  R_4^+(t),
 \end{align*}
 where the remainder term reads
 $$
 R_4^+(t) = \int_0^t \int_0^s \int_0^r \int_0^q \exp(   pX) [X,[X,[X,[X,Y]]]] \exp(-  pX) 
 \D p \,\D q\,\D r \, \D s  .
 $$
 Since the four-fold commutator satisfies a bound like in Lemma~\ref{t:lem:comm-bounds} from $H^4_\eps$ to $L^2$, we have for all $\varphi\in H^4_\eps$,
 $$
 \| R_4^+(t) \varphi \|_{L^2} \le \frac{c_4}\eps \, \frac{t^4}{4!} \, \| \varphi \|_{H^4_\eps}.
 $$
 Similarly,
 $$
\exp(- tX) [  \exp( t X),Y ] = t[X,Y] - \frac{t^2}{2!}[X,[X,Y]] + \frac{t^3}{3!}[X,[X,[X,Y]]] +  R_4^-(t),
$$
where $R_4^-(t)$ satisfies the same bound as $R_4^+(t)$.
The next two terms in $R(2t)$ can be expanded similarly, since by the commutator product rule,
\begin{align*}
&[\exp(tX)\exp(tY), C_3] \exp(-tY) \exp(-tX) 
\\
&= [\exp(tX), C_3] \exp(-tX) + \exp(tX) [\exp(tY), C_3] \exp(-tY) \exp(-tX),
\end{align*}
and analogously for the other commutator term with $C_3$. Collecting the dominant terms, we obtain
\begin{align*}
R(2t) =&\ t [[X,Y],Z(2t)] + \tfrac12 t^2 \bigl( [X,[X,Y]] Z(2t) + Z(2t) [X,[X,Y]] \bigr) 
\\
&\ +48\,t^2 (C_3 Z(2t) + Z(2t) C_3) + R_3(t),
\end{align*}
where for all $\varphi\in H^3_\eps$,
$$
\| R_3(t) \varphi \|_{L^2} \le c \, \frac{t^3}{\eps} \, \| \varphi \|_{H^3_\eps}.
$$
Now, for the commutator $[[X,Y],Z(2t)]$ we use the commutator product rule as in \eqref{t:CS-comm} and expand the resulting terms as above. This yields
$$
[[X,Y],Z(2t)] =  2t \, [[X,Y],X+Y] + R_2(t),
$$
where $R_2(t)$ is bounded like $R_3(t)$ with the exponent $2$ instead of $3$. Using that $\exp(L)= 1 +\int_0^1 L\exp(\theta L)\D\theta$ for $L$ any of $tX$, $tY$, $16 t^3 C_3$,
we further find that
\begin{align*}
&[X,[X,Y]] Z(2t) + Z(2t) [X,[X,Y]] = 2 [X,[X,Y]] + R_1(t), 
\\ 
&C_3 Z(2t) + Z(2t) C_3 = 2 C_3  + \widetilde R_1(t),
\end{align*}
where $R_1(t)$ and $\widetilde R_1(t)$ are bounded like $R_3(t)$ with the exponent $1$ instead of $3$. Altogether, we obtain
\begin{align*}
R(2t) =&\ t^2\Bigl(- [X,[X,Y]] - 2 [Y,[X,Y]] + 48\,  C_3\Bigr) + \widetilde R_3(t),
\end{align*}
where $\widetilde R_3(t)$ is bounded like $R_3(t)$. With the choice \eqref{t:zass} of $C_3$, the term in big brackets vanishes, and so we obtain, for all $\varphi\in H^3_\eps$,
$$
\| R(2t) \varphi \|_{L^2} \le c \, \frac{t^3}{\eps} \, \| \varphi \|_{H^3_\eps}.
$$
This is not yet the $O(t^4/\eps)$ estimate that we need. The improved bound is obtained by taking the expansion one power further and verifying that the term multiplying $t^3$ vanishes. Instead of actually performing this formidable calculation by hand (or by Maple or Mathematica), this can be concluded from the fact that because of the symmetry of the method, the terms corresponding to odd powers in the error expansion of the Zassenhaus method vanish identically in the case of ordinary differential equations $z'=(X+Y)z$ with arbitrary {\it matrices} $X$ and $Y$. Then, these algebraic combinations of higher commutators must also vanish for the {\it operators} $X$ and $Y$ considered here. The fourth-order remainder terms are under control with the expansions used above, since also the three-fold and four-fold commutators are bounded as in Lemma~\ref{t:lem:comm-bounds}. Hence, the previous estimate improves to
$$
\| R(2t) \varphi \|_{L^2} \le c \, \frac{t^4}{\eps} \, \| \varphi \|_{H^4_\eps}.
$$
With this bound, the result follows in the same way as for the Strang splitting.
\end{proof}

\begin{remark} The proof of the error bound in observables of Theorem~\ref{t:thm:strang-schroedinger-obs} can be transferred to yield an $O(\tau^4+\eps^2)$ error bound in observables for the symmetric Zassenhaus splitting. The only additional argument required is an appropriate fourth-order error bound for the symmetric Zassenhaus splitting of the classical equations of motion, which has the one-step map
$$
\Psi^\tau = \Phi_V^{\tau/2} \circ \Phi_T^{\tau/2} \circ \Phi_{C_3}^{2\tau^3} \circ \Phi_T^{\tau/2} \circ  \Phi_V^{\tau/2},
$$
where $\Phi_{C_3}^t$ is the flow of the Hamilton function 
$$
C_3=\frac1{48} \{ V, \{ V, T \} \} + \frac1{24} \{ T, \{ V,T \} \}.
$$
\end{remark}

\begin{remark} For the semiclassical Schr\"odinger equation with {\it time-depend\-ent} potential,
the symmetric Zassenhaus splitting has been appropriately combined with the Magnus expansion of the time-dependent Hamiltonian by
\cn{BadIKS16}.
\end{remark}

%%%%%%%%%%%%%%%%%%%%%%%%%%%%%%%%%%%%%%%%%%%%

\subsection{Variational splitting for Gaussian wave packet dynamics}
\label{subsec:t:gwp}
We now return to the setting of Section~\ref{sec:gwp} and recall the notation ${\mathcal M}$ for the manifold of complex Gaussians, the tangent space ${\mathcal T}_{u}{\mathcal M}$ at $u\in{\mathcal M}$, and the orthogonal projection onto the tangent space, $P_u:L^2(\Rb^d)\to {\mathcal T}_{u}{\mathcal M}$.
The variational Gaussian approximation $u(t)\in{\mathcal M}$ is determined by the projected differential equation
\begin{equation} \label{t:PuHu}
\I\eps \partial_t u = P_u H u, \qquad u(0)=u^0\in{\mathcal M}.
\end{equation}
For the Hamiltonian $H=T + V$ with $T=-\tfrac12\eps^2 \Delta_x$,  Strang splitting of this differential equation results in an algorithm where
a time step from
$u^n\in\calM$ to the new approximation $u^{n+1}\in\calM$ at
time $t_{n+1}=t_n+\dt$ reads as follows.

\medskip\noindent
{\bf Abstract formulation of the variational splitting integrator:}
\begin{itemize}
\item[(i)] {\rm Half-step with $V$:\/}
determine $u^{n}_{+}\in\calM$ as the solution at time 
$\dt/2$ of the equation for $u$,
\begin{equation}
\label{t:vsplit-V}
\I\eps \partial_t u = P_u  V u
 \end{equation}
with initial value $u(0) = u^n\in\calM$. 
\item[(ii)] {\rm Full step with $T$:\/}
determine $u^{n+1}_{-}$ as the solution at time $\dt$ of 
\begin{equation}
\label{t:vsplit-T}
\I\eps \partial_t u = P_u  T u
\end{equation}
with initial value $u(0) = u^{n}_{+} $. 
\item[(iii)] {\rm Half-step with $V\,$:\/}
$u^{n+1}$ is the solution at time $\dt/2$ of 
(\ref{t:vsplit-V}) with 
initial value $u(0) = u^{n+1}_{-}$.
\end{itemize} 
Written in the Hagedorn parameters $(q,p,Q,P,\zeta)$ of \eqref{hag-par}, viz.
$$%\begin{equation}\label{hag-par}
u(x,t) = \exp\!\left( \tfrac\I\eps \left(\tfrac12 (x-q(t))^T P(t)Q(t)^{-1}(x-q(t)) + p(t)^T(x-q(t)) + \zeta(t)\right) \right),
$$%\end{equation}
the differential equations above become the following in view of Proposition~\ref{prop:project} and Theorem~\ref{thm:eom-gauss-PQ}:
\begin{itemize}
\item[(i)] and (iii)\ \ Half-step with $V$:
\begin{align*}
\dot{q} &= 0, \quad\
\dot{p} = -\langle \nabla_x V\rangle_u,\\
\dot{Q} &= 0, \quad\
\dot{P} =  - \langle \nabla_x^2 V\rangle_u Q,\\
\dot{\zeta} &= -\langle V\rangle_u +  \bch
\tfrac{\eps}{4} \,\tr( Q^*\langle\nabla^2_x V\rangle_u Q). \ech
\end{align*}
\item[(ii)] Full step with $T$:
\begin{align*}
\dot{q} &= p, \quad\
\dot{p} = 0,\\
\dot{Q} &= P, \quad\
\dot{P} =  0,\\
\dot{\zeta} &= \tfrac12|p|^2  + \tfrac{\I\eps}{2}\,\tr(PQ^{-1}) .
\end{align*}
\end{itemize}
In (i) and (iii) we note that the averages $\langle V\rangle_u$, $\langle \nabla_x V\rangle_u$ and  $\langle \nabla_x^2 V\rangle_u$
depend only on the parameters $q,Q,\Im \zeta$ of the Gaussian, which remain constant in these substeps. Therefore, these differential equations can be solved explicitly, and we arrive at the Gaussian wave packet integrator proposed and studied by \cn{FaoL06}; 
see also \cn[Section IV.4]{Lub08}.

Starting from the Gaussian $u^n$ with parameters $q^n$, $p^n$,
$Q^n$, $P^n$, $\zeta^n$, 
a time step for \eqref{t:PuHu}
from time $t^n$ to $t^{n+1}=t^n+\dt$ proceeds as follows.

\medskip\noindent
{\bf Practical algorithm of the variational splitting integrator:}
\begin{itemize}
\item[(i)]  With the averages $\langle W \rangle^n=\langle u^n\, |  \,
W  u^n \rangle$ for $W=V,\nabla V, \nabla^2 V$, compute
\begin{eqnarray}
\nonumber
p^{n+1/2} &=& p^n -\tfrac12 {\dt}\,\langle \nabla V\rangle^n 
\\
P^{n+1/2} &=& P^n -\tfrac12 {\dt}\,\langle \nabla^2 V\rangle^n Q^n
\\
\nonumber
\zeta^n_+ &=& \zeta^n - \tfrac12 {\dt}\, \langle V \rangle^n +
\bch
\frac{\dt\,\eps} {8}\, {\rm tr}\, 
\bigl( (Q^n)^*\langle \nabla^2 V\rangle^n Q^n\bigr) . \ech
\end{eqnarray}
\item[(ii)]  Compute
\begin{eqnarray}
\nonumber
q^{n+1} &=& \bch q^n + \dt\, p^{n+1/2} \ech
\\
Q^{n+1} &=& \bch Q^n + \dt\, P^{n+1/2}\ech
\\
\nonumber
\zeta^{n+1}_- &=& \zeta^n_+ +
\tfrac12 {\dt}\, \big| p^{n+1/2} \big|^2 + \bch
\tfrac \I 2 \,\eps\, 
{\rm tr}\bigl( \log (\Id + \dt\,P^{n+1/2}(Q^{n})^{-1})\bigr)\,. \ech
\end{eqnarray}
\item[(iii)]  With the averages over the Gaussian at time $t^{n+1}$,
which are the same as those for the previously computed parameters $q^{n+1}$,
$Q^{n+1}$, $\Im\zeta^{n+1}_-$, compute
\begin{eqnarray}
\nonumber
p^{n+1} &=& p^{n+1/2} 
-\tfrac12 {\dt}\,\langle \nabla V\rangle^{n+1}
\\
P^{n+1} &=& P^{n+1/2} 
-\tfrac12 {\dt}\,\langle \nabla^2 V\rangle^{n+1} Q^{n+1}
\\
\nonumber
\zeta^{n+1} &=& \zeta^{n+1}_- \bch
- \tfrac12 {\dt}\, \langle V \rangle^{n+1} +
\frac{\dt\,\eps} {8}\, {\rm tr}\, 
\bigl( (Q^{n+1})^*\langle \nabla^2 V\rangle^{n+1} Q^{n+1}\bigr) .
\ech
\end{eqnarray}
\end{itemize}

\noindent
{\bf Classical limit.}\/
It is instructive to see how this algorithm behaves in the classical limit $\eps\to 0$. In this limit, the averages of $V,\nabla V,\nabla^2V$ tend to point evaluations at the centre of the Gaussian, and hence the equations for position~$q$ and momentum~$p$ become the St\"ormer--Verlet method \eqref{t:sv} for the classical equations of motion.
%\begin{align}
%\nonumber
%p^{n+1/2} &= p^n -\tfrac12 {\dt}\, \nabla V(q^n) \\
%\label{t:sv}
%q^{n+1} &= q^n - \dt\, p^{n+1/2} \\
%\nonumber
%p^{n+1} &= p^{n+1/2} 
%-\tfrac12 {\dt}\, \nabla V(q^{n+1}).
%\end{align}
%This is the leapfrog or St\"ormer--Verlet method for the classical equations of motion $\dot p=-\nabla V(q),\ \dot q =p$.  This method is the standard integrator for classical molecular dynamics. We refer to \cn{HaiLW03} for a review of the many remarkable properties of this numerical method.

\medskip
\noindent
{\bf Conservation properties.}
%\begin{proposition}[Norm conservation] \label{t:prop:norm-cons}
The variational splitting integrator conserves the {\it norm} of the Gaussian wavepacket:
$$
\| u^{n+1} \|_{L^2} = \| u^{n} \|_{L^2}.
$$
%\end{proposition}
%\begin{proof} 
This property is best seen from the abstract formulation: since each of the substeps \eqref{t:vsplit-V}--\eqref{t:vsplit-T} conserves the norm by the argument of the proof of Proposition~\ref{prop:cons}, their composition also conserves the norm.
%\end{proof}

{\it Energy} is not conserved exactly by the variational splitting integrator, but  \cn{FaoL06} show that the  integrator conserves the norm of the Gaussian wavepacket up to $O(\dt^2)$ for exponentially long times $t\le e^{c/\dt}$, provided the eigenvalues of the positive definite width matrix $(Q^n(Q^n)^*)^{-1}$ are bounded from below by a positive multiple of $\eps$ for such times and the positions $q^n$ remain in a compact subset of the domain of analyticity of the potential. The proof of this long-time near-conservation result relies on the symplecticity of the method (in the form of preservation of a Poisson structure) and on results from a backward error analysis as given in \cn[Chapter~IX]{HaiLW06}.

{\it Linear and angular momentum} are conserved exactly for a potential that is invariant under translations or rotations, respectively. This is shown using the abstract formulation of the integrator and the same arguments as in the proof of Proposition~\ref{prop:ang-mom} for the substeps of the abstract algorithm.

The {\it symplecticity relation} \eqref{PQ-symp} or equivalently \eqref{hag-rel} of the matrices $Q$ and $P$ is conserved exactly by the integrator. This follows from the differential equations solved by $Q$ and $P$ in the substeps with the same argument as in the proof of Theorem~\ref{thm:eom-gauss-PQ}.

\medskip
\noindent
{\bf Error bounds.} The method is a standard Strang splitting of the differential equations for the parameters $\Pi=(q,p,Q,P,\zeta)$ that are given by Theorem~\ref{thm:eom-gauss-PQ}. Since $\eps$ is a regular perturbation parameter in these differential equations, the standard error bounds of Strang splitting for ordinary differential equations as given, for example, by \cn{McLQ02} or 
\cn[Section~III.5.3]{HaiLW06}, show that the error in the parameters is $O(\dt^2)$ over bounded time intervals: for sufficiently small $\dt\le \dt_0$ (with $\dt_0$ independent of $\eps$),
\begin{equation} \label{t:gwp-par-err}
\| \Pi^n - \Pi(t_n) \| \le c \dt^2 \qquad\text{for }\ t_n \le \bar t,
\end{equation}
where $c$ is independent of $\eps$, $n$ and $\dt$ with $t_n=n\dt\le \bar t$, but depends exponentially on $\bar t$. Since the parameters appear divided by $\eps$ in the exponent of the Gaussian wave packet, its approximation properties are less obvious. 
%The $L^2$-error bound of the following result is given in \cn{FaoL06}.

\begin{theorem}[error of the Gaussian wave packet integrator] 
\label{t:thm:gwp-err}
\hfill \\
%Let the potential $V$ be smooth with polynomially bounded derivatives.
(a) The $L^2$-error of the wave packet $u^n$ is bounded by
\begin{equation} \label{t:gwp-u-err}
\| u^n - u(t_n) \|_{L^2} \le C \,\frac{\dt^2}{\eps} \qquad\text{for }\ t_n \le \bar t.
\end{equation}
%where $C$ is independent of $\eps$, $n$ and $\dt$ with $t_n=n\dt\le \bar t$, but depends exponentially on $\bar t$.
(b) %A better approximation is, however, obtained for averages of observables. 
Let the observable $A$ be an $\eps$-independent  
polynomial of the position and momentum operators $\widehat q$ and $\widehat p$ 
(with $\widehat q \varphi(x)=x\varphi(x)$ and $\widehat p \varphi(x)=-\I\eps \nabla_x\varphi(x)$), and consider averages $\langle A \rangle_u =\langle u\, |\, Au \rangle$. Then, the error of the average of $A$ over $u^n$ is bounded by
\begin{equation} \label{t:gwp-obs-err}
|\langle A \rangle_{u^n} - \langle A \rangle_{u(t_n)}| \le C \dt^2 \qquad\text{for }\ t_n \le \bar t.
\end{equation}
In both (a) and (b), $C$ is independent of $\eps$, $n$ and $\dt$ with $t_n=n\dt\le \bar t$, but depends exponentially on $\bar t$.
\end{theorem}

\begin{proof} (a) The $L^2$-error bound follows immediately from the error bound \eqref{t:gwp-par-err} and the Lipschitz continuity of the
exponential function on the imaginary axis.

(b) Since a normalised Gaussian $u$ with parameters $\Pi=(q,p,Q,P)$ (here we can ignore $\zeta$) satisfies the relation
$
\widehat p u = -\I\eps \nabla_x u = (PQ^{-1}(x-q) + p) u,
$
it follows that 
$$
\langle A \rangle_u = \int_{\Rb^d} a_\Pi(x) \, |u(x)|^2 \D x
$$
with an $\eps$-independent polynomial $a_\Pi$ whose coefficients depend smoothly on the Gaussian parameters $\Pi$. We substitute $y=(x-q)/\sqrt{\eps}$ and note that $|u(x)|^2 = \eps^{-d/2} |\gamma(Q,y)|^2$ with an $\eps$-independent Gaussian \bch $\gamma(Q,y) = \pi^{-d/4} \det(Q)^{-1/2} \exp(-\frac12 y^T (QQ^*)^{-1} y)$\ech of unit $L^2$-norm, using \eqref{QQ} with \eqref{CPQ}.
So we have
$$
\langle A \rangle_u = \int_{\Rb^d} a_\Pi(q+\sqrt\eps y) \, |\gamma(Q,y)|^2 \D y.
$$
The error bound \eqref{t:gwp-obs-err} now follows by 
taking this formula for $u^n$ and $u(t_n)$ and using the error bound \eqref{t:gwp-par-err} for the parameters.
\end{proof}

We remark that the assumption on the observable $A$  to be polynomial in $\widehat q$ and $\widehat p$ is not essential. The proof shows immediately that the result remains valid when $A$ is a polynomial in $\widehat p$ and in finitely many functions $\alpha_i(\widehat q)$ where the functions $\alpha_i$ and sufficiently many of its derivatives have at most polynomial growth. Moreover, by repeating the argument after taking Fourier transforms, the roles of $\widehat q$ and $\widehat p$ can be reversed. We do not strive for utmost generality here, as this would obscure the basically simple argument given in the proof above.

%%%%%%%%%%%%%%%%%%%%%%%%%%%%%%%%%%%%%%%%%

\subsection{Time integration of Hagedorn's semiclassical wave packets}
\label{subsec:t:hagwp}
We return to the setting of Section~\ref{sec:hagwp}. The  wave function $\psi(x,t)$ is approximated by a linear combination of time-varying Hagedorn functions,
\begin{equation}\label{t:hag-series}
\psi_\calK(x,t)=e^{\I S(t)/\eps} 
\sum_{k\in\calK} c_k(t)\, \varphi_k^\eps[q(t),p(t),Q(t),P(t)](x)
\end{equation}
over a finite multi-index set $\calK$, where
$(q(t),p(t),Q(t),P(t))$  is a solution to the classical equations
\begin{equation}\label{t:classical-eom}
\dot q=p, \ \ \dot p=-\nabla V(q) \quad\text{and}\quad \dot Q=P,\ \ \dot P=-\nabla^2V(q)Q,
\end{equation}
and the classical action $S(t)$ satisfies
\begin{equation}\label{t:S}
\dot S = \tfrac12|p|^2-V(q).
\end{equation}
The coefficients $c_k(t)$ in this approximation are determined by a Galerkin condition that yields the system of differential equations
\eqref{V:hag-coeff-ode}, viz.
\begin{equation}\label{t:hag-coeff-ode}
\I\eps \dot c(t) = G(t)c(t) \quad\text{ with }\quad
G(t)=\bigl( \langle \varphi_\ell(\cdot,t)\,|\,W_{q(t)}\varphi_k(\cdot,t)
\rangle\bigr)_{\ell,k\in\calK}\,,
\end{equation}
where $\varphi_k(x,t)= \varphi_k^\eps[q(t),p(t),Q(t),P(t)](x)$
and $W_q$ is the non-quadratic remainder in the Taylor expansion of the potential $V$ at $q$. 
We recall from Lemma~\ref{V:lem:F-bound} that $\| G(t) \| = O(\eps^{3/2})$, so  the coefficients change slowly: 
$\| \dot c(t) \| = O(\eps^{1/2})$.

\medskip\noindent
{\bf Time integration of the position and momentum parameters.}
The approximate solution of the differential equations \eqref{t:classical-eom} is best done by a numerical integrator that preserves quadratic invariants (a {\it symplectic} integrator), for example a  splitting method based on splitting the kinetic and potential parts, or a Gauss--Runge--Kutta method; see, \eg\ ~\cn{BlaC16}, \cn{HaiLW06} and \cn{LeiR04}. These integrators not only give favourable behaviour for the Hamiltonian equations of motion for $(q,p)$, but also preserve the symplecticity relation~\eqref{hag-rel} of the matrices $Q$ and $P$ along the numerical solution. \cn{FaoGL09} use the St\"ormer--Verlet method \eqref{t:sv} in this context, but it is pointed out by \cn{GraH14} that a higher-order method for the classical equations is favourable. Using a method of order $r$ with step size~$\dt$, we then have for the approximations at time $t_n=n\dt\le \bar t$ the error bound (with $C$ depending on~$\bar t$)
\begin{equation} \label{t:hag-qp-err}
\| (q^n,p^n,Q^n,P^n,S^n) - (q(t_n),p(t_n),Q(t_n),P(t_n),S(t_n)) \| \le C \dt^r.
\end{equation}
This yields an $O(\dt^r/\eps)$  $L^2$-error in the corresponding Hagedorn functions,
\begin{equation} \label{t:hag-phi-err}
\| \varphi_k^\eps[q^n,p^n,Q^n,P^n]-\varphi_k^\eps[q(t_n),p(t_n),Q(t_n),P(t_n)] \|_{L^2} \le C_k\frac{\dt^r}\eps,
\end{equation}
and similarly for the phase factor,
\begin{equation} \label{t:hag-S-err}
|e^{\I S^n/\eps} - e^{\I S(t_n)/\eps}| \le C\frac{\dt^r}\eps.
\end{equation}

\medskip\noindent
{\bf Time integration of the coefficients.}
Let us denote the approximate Galerkin matrix by
$$
G^n = \bigl( \langle \varphi_\ell^n\,|\,W_{q^n}\varphi_k^n
\rangle\bigr)_{\ell,k\in\calK}\,,
$$
where $\varphi_k^n(x)=\varphi_k^\eps[q^n,p^n,Q^n,P^n](x)$.
We integrate  \eqref{t:hag-coeff-ode} numerically with the approximate Galerkin matrix, using the exponential midpoint rule with the double step size $2\tau$,
\begin{equation} \label{t:c-mpr}
c^{2(n+1)} = \exp\Bigl( -2\tau \frac\I\eps G^{2n+1} \Bigr) c^{2n},
\end{equation}
which preserves the Euclidean norm of the coefficient vectors.

Computationally, this can be done efficiently using Lanczos iterations; see \cn[Section III.2.2]{Lub08} and references therein.  This requires only a few matrix-vector products which, moreover, can be computed with high accuracy by a fast matrix-free algorithm for which the Galerkin matrix $G^{2n+1}$ itself need not be assembled; see \cn{FaoGL09} and \cn{Bru15}. 
This matrix-free algorithm uses the recurrence relations of the Hagedorn functions. Its computational cost is proportional to the cardinality of the multi-index set $\calK$ --- and not to its square, that is the number of entries in the Galerkin matrix, which is prohibitively large for  problems in higher dimensions $d$. In an alternative approach due to \cn{HagL17}, the Galerkin matrix is assembled after a transformation that converts the multidimensional Galerkin integrals to a product of one-dimensional integrals.

\medskip\noindent
{\bf Approximation to the wave function.}
We thus obtain, for even $n$, an approximation to the wave function at time $t_n$ that is given by
\begin{equation} \label{t:psiKn}
\psi_{\calK}^n(x) = e^{\I S^n/\eps} 
\sum_{k\in\calK} c_k^n\, \varphi_k^\eps[q^n,p^n,Q^n,P^n](x).
\end{equation}
Since the Hagedorn functions are orthonormal and the Euclidean norm of the coefficient vector is preserved,  the $L^2$-norm of the approximate wave function is also preserved in time: $\| \psi_{\calK}^n \|_{L^2}=1$ for all $n$.

%\medskip\noindent
%{\bf Error bound.}
We now bound the error of this time integration method. Here we also recall the error bound for $\psi_\calK(t) - \psi(t)$ given by Theorem~\ref{V:cor:hag-err}. The following $L^2$-error bound (without the factor $\eps^{1/2}$ in the second term) is due to  \cn{GraH14}.

\begin{theorem}[error of the Hagedorn wave packet integrator] \label{t:thm:hag-err}
\bch We consider the time discretisation of order $r$ with stepsize $\tau$ given in 
\eqref{t:hag-qp-err} for the classical equations and \eqref{t:c-mpr} for the coefficients.  
If the potential $V$ is smooth with polynomially bounded derivatives, then the following error bounds hold true:\ech

(a) The $L^2$-error
is bounded by
\begin{equation} \label{t:hag-Ltwo-err}
\| \psi_\calK^n - \psi_\calK(t_n) \|_{L^2} \le C_1 \frac{\tau^r}\eps + C_2\, \tau^2 \eps^{1/2}.
\end{equation}
%where $C_1$ and $C_2$ are independent of $\eps$, $\dt$ and $n$ with $t_n=n\dt\le \bar t$ (but depend on $\calK$).

(b) %A better approximation is, however, obtained for averages of observables. 
Let the observable $A$ be an $\eps$-independent  
polynomial of the position and momentum operators $\widehat q$ and $\widehat p$. 
%(with $\widehat q \varphi(x)=x\varphi(x)$ and $\widehat p \varphi(x)=-\I\eps \nabla_x\varphi(x)$).
%, and consider averages $\langle A \rangle_u =\langle u\, |\, Au \rangle$. 
Then, the error of the average of $A$ over $\psi_\calK^n$ is bounded by
\begin{equation} \label{t:hag-obs-err}
|\langle A \rangle_{\psi_\calK^n} - \langle A \rangle_{\psi_\calK(t_n)}| \le 
C_1 {\tau^r} + C_2\, \tau^2 \bch \eps^{1/2}\ech
 \qquad\text{for }\ t_n \le \bar t.
\end{equation}
In both (a) and (b), $C_1$ and $C_2$ are independent of $\eps$, $\dt$ and $n$ with $t_n=n\dt\le \bar t$ (but depend on $\calK$ and $\bar t$).
\end{theorem}

The proof uses the following lemmas.

\begin{lemma}[scaling of the Galerkin matrix]
\label{t:lem:G}  
For the Galerkin matrix we have
$$
G(t) = \eps^{3/2} \widetilde G(t),
$$
where all derivatives of $\widetilde G$ are bounded independently of $\eps$.
\end{lemma}

\begin{proof} We transform variables $y=(x-q)/\sqrt{\eps}$ and note from \eqref{V:phi-0} and \eqref{V:rec} that
$$
\varphi_k^\eps[q,p,Q,P](x)  = \eps^{-d/4} \varphi_k^1[0,p/\sqrt{\eps},Q,P](y).
$$
We therefore obtain (omitting the omnipresent argument $t$)
\begin{align*}
G_{jk}  &= \int_{\Rb^d} \overline{\varphi_j^\eps[q,p,Q,P](x)}\, W_q(x) \, \varphi_k^\eps[q,p,Q,P](x) \D x
\\
&= \int_{\Rb^d} \overline{\varphi_j^1[0,p/\sqrt{\eps},Q,P](y)}\, W_q(q+\sqrt{\eps}y) \, \varphi_k^1[0,p/\sqrt{\eps},Q,P](y) \D y.
\end{align*}
The exponential factors containing $p$ cancel in this expression, and so this simplifies to
$$
G_{jk} = \int_{\Rb^d} \overline{\varphi_j^1[0,0,Q,P](y)}\, W_q(q+\sqrt{\eps}y) \, \varphi_k^1[0,0,Q,P](y) \D y.
$$
Since $W_q(x)$ is the non-quadratic remainder term of the Taylor expansion of $V$ at $q$, we have by the integral remainder formula
$$
W_q(q+\sqrt{\eps}y) = \eps^{3/2} \int_0^1 \tfrac12 (1-\theta)^2 \, V'''(q+\theta \sqrt\eps y)(y,y,y) \, \D \theta.
$$
Inserting this formula into the above integral yields the result.
\end{proof}

\begin{lemma}[error in the Galerkin matrix]
\label{t:lem:G-err}
With the error bounds \eqref{t:hag-qp-err}, the error in the Galerkin matrix is bounded by
$$
\| G^n - G(t_n) \| \le C\, \tau^r \eps^{3/2},
$$
where $C$ is independent of  $\eps$, $\dt$ and $n$ with $t_n=n\dt\le \bar t$ (but depends on the multi-index set $\calK$).
\end{lemma}

\begin{proof} This follows immediately from the formulas for $G_{jk}^n$ and $G_{jk}(t_n)$ obtained in the previous proof.
\end{proof}

\begin{lemma}[error of the coefficients] \label{t:lem:c-err}
With the error bounds \eqref{t:hag-qp-err} for $r\ge 2$, the error in the coefficients is bounded by
$$
\| c^n - c(t_n) \| \le C\, \tau^2 \eps^{1/2},
$$
where $C$ is independent of  $\eps$, $\dt$ and $n$ with $t_n=n\dt\le \bar t$.
\end{lemma}

\begin{proof} 
We begin by studying the error of auxiliary coefficient vectors $a^n$ that are defined like $c^n$ but using the Galerkin matrix $G(t)$ with exact parameters $(q(t),p(t),Q(t),P(t))$:
$$
a^{2(n+1)} = \exp\Bigl( -2\tau \frac\I\eps G(t_{2n+1}) \Bigr) a^{2n}.
$$
Over a time step, for $t_{2n} \le t \le t_{2n+2}$, we consider the differential equation for the coefficients $c(t)$ of the time-continuous Galerkin approximation,
$$
\I\eps \dot c(t) = G(t) c(t) 
$$
and the differential equation with fixed matrix $G(t_{2n+1})$,
$$
\I\eps \dot a(t) = G(t_{2n+1}) a(t)
$$
with the same starting values $a(t_{2n})=c(t_{2n})$. Writing the differential equation for $c(t)$
as
$$
\I\eps \dot c(t) = G(t_{2n+1}) c(t)  + \bigl( G(t) - G(t_{2n+1}) \bigr) c(t),
$$
subtracting the two differential equations and applying the variation-of-constants formula yields
\begin{align*}
&a(t_{2(n+1)}) - c(t_{2(n+1)})
\\
&= \int_{t_{2n}}^{t_{2(n+1)}} \exp\Bigl( -(t_{2(n+1)}-t)\,\frac\I\eps\, G(t_{2n+1}) \Bigr) 
\frac\I\eps \bigl( G(t) - G(t_{2n+1}) \bigr) c(t) \D t.
\end{align*}
Here we note that by Lemma~\ref{t:lem:G},
\begin{align*}
 &\exp\Bigl( -(t_{2(n+1)}-t)\,\frac\I\eps\, G(t_{2n+1}) \Bigr) = \bch \Id + O(\eps^{1/2}\dt),\ech
 \\
 &\frac1\eps \bigl( G(t) - G(t_{2n+1}) \bigr) = \eps^{1/2}  \bigl( \widetilde G(t) - \widetilde G(t_{2n+1}) \bigr)=O(\eps^{1/2}\dt),
\end{align*}
so that 
\begin{align*}
a(t_{2(n+1)}) - c(t_{2(n+1)}) &= \I \eps^{1/2} \Bigl(  \int_{t_{2n}}^{t_{2(n+1)}} \widetilde G(t) \D t - 2\dt \widetilde G(t_{2n+1}) \Bigr)
+ O(\eps\tau^3) 
\\
&= O(\eps^{1/2}\tau^3)
\end{align*}
by the error bound for the midpoint rule. From this local error bound for the exponential midpoint rule with the hermitian matrix $G(t)$ we conclude the global error bound
$$
a^{2n} - c(t_{2n}) = O(\eps^{1/2}t_n \tau^2).
$$
To estimate the difference between $c^{2n}$ from \eqref{t:c-mpr} and $a^{2n}$, we write the error equation
\begin{align*}
c^{2n+2} - a^{2n+2} &= \exp \Bigl( -2\tau\,\frac\I\eps\, G^{2n+1} \Bigr) \bigl( c^{2n} - a^{2n} \bigr) 
\\
&\quad 
 + \Bigl( \exp \Bigl( -2\tau\,\frac\I\eps\, G(t_{2n+1})\Bigr) - \exp \Bigl( -2\tau\,\frac\I\eps\, G^{2n+1} \Bigr) \Bigr)c^{2n}.
\end{align*}
The last term is of size $O(\eps^{1/2} \tau^{r+1})$ by Lemma~\ref{t:lem:G-err}, and so we obtain
$$
c^{2n} - a^{2n} = O(\eps^{1/2}t_n \tau^r).
$$
Together with the above estimate for $a^{2n} - c(t_{2n})$, this yields the result.
\end{proof}

\begin{proof} (of Theorem~\ref{t:thm:hag-err} )
(a) The $L^2$-error bound follows by combining the error bounds \eqref{t:hag-phi-err}--\eqref{t:hag-S-err} and 
Lemma~\ref{t:lem:c-err} in \eqref{t:psiKn}.

(b) In view of \eqref{phik-pol}--\eqref{phik-pol-rec}, we can write the Hagedorn wave packet with coefficients $c(t)=\bigl(c_k(t)\bigr)$ and parameters $\Pi(t)=(q(t),p(t),Q(t),P(t))$ and $S(t)$ in the form
$$
\psi_\calK(x,t) = \e^{\I S(t)/\eps} \, f\Bigl(c(t),Q(t),\frac{x-q(t)}{\sqrt\eps}\Bigr)\varphi_0^\eps[\Pi(t)](x),
$$
where $f(\cdot,\cdot,\cdot)$ is a smooth function with derivatives bounded independently of $\eps$ and polynomial in the third argument. As in the proof of part (b) of Theorem~\ref{t:thm:gwp-err},
this yields
$$
\langle A \rangle_{\psi_\calK(t)} = \int_{\Rb^d} g(c(t),\Pi(t), y) \, |\gamma(Q(t),y)|^2 \D y,
$$
with a smooth function $g$ that has derivatives bounded independently of $\eps$ and has polynomial growth in the third argument, and with an $\eps$-independent Gaussian $\gamma(Q(t),y)= \mu(Q(t)) \exp(-\frac1 2 y^T (Q(t)Q(t)^{*})^{-1}y)$ of unit $L^2$-norm.
 The error bound \eqref{t:hag-obs-err} now follows by 
comparing this formula and the analogous formula for the numerical approximation $\psi_\calK^n(x)$ (with coefficients $c^n$ and parameters $\Pi^n$ in place of $c(t)$ and $\Pi(t)$, respectively) and using the error bound \eqref{t:hag-qp-err} for the parameters and the error bound of Lemma~\ref{t:lem:c-err} for the coefficients.
\end{proof}

As in Theorem~\ref{t:thm:gwp-err}, the error bound of (b) actually holds for much larger classes of observables $A$.

We mention that  high-order splitting integrators for Hagedorn wave packets have recently been constructed by \cn{BlaG19}.

%\section{High-dimensional quadrature}
\section{High-dimensional quadrature}
\label{sec:quad}

The various numerical approaches we have described require the computation of high-dimensional integrals
$$
\int_{\Rb^n} f(x) \D x,
$$
where $n=d, 2d$ or even $4d$ with the dimension $d$ of the Schr\"odinger equation.  In particular we have the following, listed with increasing difficulty and generality.
\begin{itemize}
\item $f$ is a real Gaussian of width $O(\sqrt\eps)$  times an $\eps$-independent smooth function, in the case of variational Gaussian wave packets (Section~\ref{sec:gwp}).
\item $f$ is a real Gaussian of width $O(\sqrt\eps)$ times a low-degree polynomial of $x/\sqrt\eps$ times an $\eps$-independent smooth function, in the case of Hagedorn wave packets (Section~\ref{sec:hagwp}).
\item $f$ is a  probability density multiplied with a smooth function, in the case of computing expectation values of observables by Wigner/Husimi transform techniques (Section~\ref{sec:wigner}).
%$f$ is the Wigner function (or the Husimi function --- a probability density) of the initial wave function multiplied with a phase-space observable function that is transported in time by classical dynamics (Section~\ref{sec:wigner})
\item $f$ is a probability density multiplied by a highly oscillatory function (with a wave length proportional to $\eps$), in the case of  Herman--Kluk and Gaussian beam approximations (Section~\ref{sec:csg}).
\end{itemize}

The first two cases, and also the third case with a Gaussian probability density, can be treated efficiently by sparse-grid Gauss--Hermite quadrature for moderate dimensions. 
Very high dimensions and more general probability densities are the realm of various types of Monte Carlo methods and quasi-Monte Carlo methods.
The high-dimensional highly oscillatory case is problematic for all of these methods, with a critical dependence of the required number of quadrature points on $\eps$ and/or the dimension.

In this section we briefly discuss these three approaches to high-dimensional quadrature. While there is ample literature available on Monte Carlo and quasi-Monte Carlo methods, this is not the case for Gauss--Hermite quadrature and its sparse-grid version, and so we give more details for the latter. 

Acta Numerica articles related to this section, which contain a wealth of results and many relevant references, are \cite{BunG04,Caf98,DicKS13,Gil15,BouS18}.

\subsection{Sparse-grid Gauss--Hermite quadrature}

{\bf One-dimensional Gauss--Hermite quadrature.}
We begin with one-dimensional Gauss--Hermite quadrature, which is Gaussian quadrature for the weight function $e^{-x^2}$ on the real line; see, \eg\ \cn{Gau97}. Here, the quadrature nodes $x_1,\dots,x_m$ are chosen as the zeros of the $m$th Hermite polynomial $H_m$. With the corresponding weights $\omega_i$, the quadrature formula
$$
 \sum_{i=1}^m \omega_i \, \phi(x_i) \approx \int_{-\infty}^\infty \e^{-x^2}\, \phi(x)\D x 
$$
is exact whenever $\phi$ is  a polynomial of degree less than $2m$. Written alternatively for $f(x)=\e^{-x^2}\, \phi(x)$ and with $w_i=\omega_i \,\e^{x_i^2}$, the quadrature formula
$$
\sum_{i=1}^m w_i \, f(x_i) \approx \int_{-\infty}^\infty f(x)\D x 
$$
is exact whenever $f(x)$ is $\e^{-x^2}$ times a polynomial of degree less than $2m$. The nodes $x_i$ and weights $w_i$ clearly depend on $m$, but we do not indicate this obvious dependence in the notation.

We prove the following error bound and refer to  \cn{MasM94}  for a different error bound.

\begin{theorem}[error bound for Gauss--Hermite quadrature]
\label{thm:gauss-hermite}
For the multiplication and differentiation operator $A= \frac1{\sqrt 2}(x + d/dx)$ on $L^2(\Rb)$, assume that $f(x)= e^{-x^2/2} g(x)$ with $g\in D(A^r)$ for some integer $r$ with $3\le r\le 2m$. Then, the error of Gauss--Hermite quadrature is bounded by
$$
\Bigl| \sum_{i=1}^m w_i \, f(x_i) - \int_{-\infty}^\infty f(x)\D x  \Bigr| \le \frac{Cm}{\sqrt{2m(2m-1)\dots(2m-r+1)}}  \, \| A^r g \|_{L^2},
$$
where $C$ is independent of $m$, $r$ and $f$.
\end{theorem}

We note that for fixed $r$ and large $m$, this yields an $O(m^{-(r/2-1)})$ error bound.

\begin{proof}
We expand $g$ as a  Hermite--Fourier series
$$
g= \sum_{k\ge 0} c_k \, \varphi_k, \qquad\text{with }\quad c_k = \langle \varphi_k\mid g \rangle,
$$
where the functions $\varphi_k$ ($k\ge 0$) are the $L^2$-orthonormal basis of Hermite functions; see, \eg\ \cn{Tha00}. Up to a normalisation factor, $\varphi_k(x)$ equals $H_k(x) \, \e^{-x^2/2}$.
%, which are the eigenfunctions of the harmonic oscillator Hamiltonian $\tfrac12(\widehat p\,^2 + \widehat q\,^2)$.
The zero-order Hermite function is a normalized Gaussian,
$$
\varphi_0(x)=\pi^{-1/4}\, \e^{-x^2/2} ,
$$
and the further Hermite functions are constructed recursively via Dirac's raising operator $A^\dagger= \frac1{\sqrt 2}(x - d/dx)$:
$$
\varphi_{k} = \frac1{\sqrt{k}} A^\dagger \varphi_{k-1}, \qquad k\ge 1.
$$
For $k\ge r$, the coefficient $c_k$ therefore equals
\begin{align*}
c_k = \langle \varphi_k\mid g \rangle &= \frac1{\sqrt{k(k-1)\dots(k-r+1)}}\, \langle (A^\dagger)^r\varphi_{k-r}\mid g \rangle
\\
&= \frac1{\sqrt{k(k-1)\dots(k-r+1)}}\, \langle \varphi_{k-r} \mid A^r g \rangle,
\end{align*}
and hence we obtain with the Cauchy--Schwarz inequality 
\begin{equation}\label{quad:ck}
|c_k| \le \frac1{\sqrt{k(k-1)\dots(k-r+1)}}\, \| A^r g \|, \qquad k\ge r.
\end{equation}
We let $g_{2m}$ denote the truncated Hermite expansion
$$
g_{2m} = \sum_{k < 2m} c_k \, \varphi_k.
$$
Since the quadrature formula integrates $f_{2m}(x):=\e^{-x^2/2} g_{2m}(x)$ exactly, the quadrature error equals
\begin{align*}
&E_m := \sum_{i=1}^m w_i \, f(x_i)  - \int_{-\infty}^\infty f(x)\D x 
\\
&= \sum_{i=1}^m w_i \, \bigl( f(x_i) - f_{2m}(x_i) \bigr) - \int_{-\infty}^\infty \bigl(f(x)-f_{2m}(x)\bigr) \D x 
\\
&= \sum_{k\ge 2m} c_k \Bigl(   \sum_{i=1}^m w_i \, \e^{-x_i^2/2} \varphi_k(x_i) - \int_{-\infty}^\infty \e^{-x^2/2} \varphi_k(x) \D x \Bigr)
\\
&=  \sum_{k\ge 2m} c_k \sum_{i=1}^m w_i \, \e^{-x_i^2/2} \varphi_k(x_i),
\end{align*}
where the last equality uses that $\varphi_k$ is orthogonal to $\e^{-x^2/2} = \pi^{1/4} \varphi_0$. To proceed further, we need to use the following facts:
\begin{enumerate}
\item[(i)] $|\varphi_k(x)|\le 1$ for all real $x$ and all $k\ge 0$;
\item[(ii)] $w_i>0$ for all $i=1,\dots,m$ (and all $m$);
\item[(iii)] there exists $K<\infty$ such that for all $m$,
$
\ \ \sum_{i=1}^m w_i \, \e^{-x_i^2/2} \le K.
$
\end{enumerate}
Property (i) follows from the pointwise bound of Hermite polynomials in \cite[p.\,787]{AbrS64}. We remark that the bound (i) is not sharp but it is sufficient for our purpose.
Property (ii) holds generally for all Gaussian quadrature rules; see \cn{Gau97}. Property (iii) follows from a result by
\cn{Usp28}, which in particular yields that $\sum_{i=1}^m w_i \, \e^{-x_i^2/2}$ converges to $\int_{-\infty}^\infty  \e^{-x^2/2} \D x$ as $m\to\infty$. With (i)--(iii), we estimate
\begin{align*}
|E_m| \le  \sum_{k\ge 2m} |c_k| \sum_{i=1}^m w_i \, \e^{-x_i^2/2} |\varphi_k(x_i)| \le K \sum_{k\ge 2m} |c_k|.
\end{align*}
The result now follows with the bound \eqref{quad:ck}.
\end{proof}

\medskip\noindent
{\bf Sparse-grid Gauss--Hermite quadrature.}
Sparse-grid quadrature was introduced by \cn{Smo63}; see also \cn{Zen91} and \cn{GerG98} for further developments. Here we describe and study sparse-grid quadrature
when it is based on
one-dimensional Gauss--Hermite quadrature in every coordinate direction.
For $\ell=0,1,2,\dots$, let $x_i^\ell$ denote the zeros of 
the Hermite polynomial of degree $2^\ell$, and let $w_i^\ell$
be the corresponding weights, so that we have
the one-dimensional $2^\ell$-point Gauss--Hermite quadrature
formula
$$
 Q_\ell f =
\sum_{i=1}^{2^\ell} w_i^\ell\, f(x_i^\ell)\approx
\int_{-\infty}^\infty  f(x)\, dx 
\,.
$$ 
We introduce the difference formulas between successive levels,
$$
\Delta_\ell f = Q_\ell f - Q_{\ell-1} f\,,
$$
and for the lowest level we set $\Delta_0 f = Q_0 f$. We clearly have
$$
Q_L f = \sum_{\ell=0}^L \Delta_\ell f .
$$
The full tensor quadrature approximation at level $L$ 
to a $d$-dimensional
integral $\int_{\Rb^d} f(x_1,\dots, x_d)\D x_1\dots \!\D x_d$
reads
\begin{equation*}
Q_L \otimes \ldots \otimes Q_L f = 
\sum_{i_1=1}^{2^L}\ldots \sum_{i_d=1}^{2^L}
w_{i_1}^L\ldots w_{i_d}^L \, 
f(x_{i_1}^L,\dots, x_{i_d}^L)\,.
\end{equation*}
This uses $(2^L)^d$ grid points at which $f$ is evaluated. It can be rewritten as
\begin{equation}\label{III:Q-full}
Q_L \otimes \ldots \otimes Q_L f = 
\sum_{\ell_1=0}^{L} \ldots \sum_{\ell_d=0}^{L} 
\Delta_{\ell_1}\otimes\ldots\otimes \Delta_{\ell_d}f.
\end{equation} 
The number of function evaluations is substantially reduced in {\it Smolyak
quadrature}, which neglects all contributions from the
difference terms with $\ell_1+\ldots+\ell_d>L$ and thus arrives
at the quadrature formula
\begin{equation}\label{III:smolyak}
S_L^d f =  \sum_{\ell_1+\ldots+\ell_d\le L} \! 
\Delta_{\ell_1}\otimes\ldots\otimes \Delta_{\ell_d}f
\approx \int_{\Rb^d} f(x_1,\ldots, x_d)\D x_1\ldots \!\D x_d\,.
\end{equation}
%\begin{figure}[!ht]
%  \centering
%  \input{figures/sparseGrid.tex}
%  \caption{Gauss--Hermite sparse grid ($L = 5$, $d=2$).}
%  \label{III:fig:sparse-hermite}
%\end{figure}
%\noindent
Here, $f$ is evaluated only at the points of the 
{\it sparse grid}
$$
\{(x_{i_1}^{\ell_1},\dots, x_{i_d}^{\ell_d}):\,
\ell_1+\ldots+\ell_d\le L\} \,,
$$ 
which, for $K=2^L$, has less than
$K\, (\log K)^{d-1}$  quadrature nodes, instead of $K^{d}$ for the full tensor grid. We prove the following error bound. 

\begin{theorem}[error bound for sparse-grid Gauss--Hermite quadrature]
\label{thm:sparse-grid-gauss-hermite}
For the multiplication and differentiation operators $A_j= \frac1{\sqrt 2}(x_j + \partial_j)$ ($j=1,\dots,d)$ on $L^2(\Rb^d)$, assume that $f(x)= e^{-|x|^2/2} g(x)$ with $g\in D(A_1^r\dots A_d^r)$ for some integer $r\ge 3$. Then, the error of sparse-grid Gauss--Hermite quadrature is bounded by
$$
|S_L^d f - \int_{\Rb^d} f(x)\D x  \Bigr| \le   C^d \,L^{d-1} \,\bigl(2^L\bigr)^{-(r/2-1)}\, \max_{r_1\le r,\ldots,r_d\le r}
%{(r_1,\dots,r_d):\,r_j\le r}
\| A_1^{r_1} \dots A_d^{r_d} \,g \|_{L^2},
$$
where $C>1$ is independent of $L$, $d$ and $f$ (but depends on $r$).
\end{theorem}

This error bound shows that for sufficiently smooth and rapidly decaying functions $f$ the sparse-grid Gauss--Hermite quadrature still  yields an approximation whose error is proportional to a potentially large power of the inverse of the number $N$ of function evaluations, provided that  $N\gg (\log N)^{2d}$:%, i.e. $d<c\log N/\log\log N$ with a sufficiently small constant $c$
$$
|S_L^d f - \int_{\Rb^d} f(x)\D x  \Bigr| \le (\widetilde C \log N)^{d r/2} N^{-(r/2-1)} \, \max_{r_j\le r}\| A_1^{r_1} \dots A_d^{r_d} \,g \|_{L^2}.
$$
In contrast, for the full tensorized quadrature, where $N=(2^L)^d$, we have only 
$$
|Q_L \otimes \ldots \otimes Q_L f- \int_{\Rb^d} f(x)\D x  \Bigr| \le C^d N^{-(r/2-1)/d} \max_{r_1+\ldots+r_d=r} \| A_1^{r_1} \dots A_d^{r_d} \,g \|_{L^2}.
$$
\begin{proof} (of Theorem~\ref{thm:sparse-grid-gauss-hermite}) We consider the Hermite--Fourier expansion of $g$,
$$
g= \sum_{k=(k_1,\dots,k_d) \atop k_j\ge 0} c_k \, \varphi_{k_1}\!\otimes\dots \otimes\varphi_{k_d} \qquad\text{with }\quad c_k = \langle \varphi_{k_1}\otimes\ldots\otimes \varphi_{k_d} \mid g \rangle.
$$
We write
$$
\Delta_{\ell_1}\otimes\ldots\otimes \Delta_{\ell_d}f = \textstyle\bigl( (Q_{\ell_1} -\int_{\Rb}) - (Q_{\ell_1-1} -\int_{\Rb}) \bigr) \otimes \ldots \otimes
\bigl( (Q_{\ell_d} -\int_{\Rb}) - (Q_{\ell_d-1} -\int_{\Rb}) \bigr)f,
$$
and as in the proof of Theorem~\ref{thm:gauss-hermite}, we find that
$$
|\Delta_{\ell_1}\otimes\ldots\otimes \Delta_{\ell_d}f| \le (2K)^d \sum_{k=(k_1,\dots,k_d) \atop k_j\ge 2^{\ell_j}} |c_k|.
$$
Like in \eqref{quad:ck}, we obtain for $r_j\le k_j$
$$
|c_k| \le \prod_{j=1}^d \frac 1 {\sqrt{k_j (k_j-1)\ldots(k_j-r_j+1)}} \| A_1^{r_1} \dots A_d^{r_d} \,g \|,
$$
which yields, for some $C_r$ depending only on $r$,
$$
|c_k| \le C_r^d \,K_r(g)\,\prod_{j=1}^d k_j^{-r/2}   \quad\text{with}\quad K_r(g) =  \max_{r_j\le r}\| A_1^{r_1} \dots A_d^{r_d} \,g \|.
$$
Combining these bounds and using that $\sum_{k_j\ge 2^{\ell_j}} k_j^{-r/2} \le \widehat C_r \bigl(2^{\ell_j}\bigr)^{-(r/2-1)}$, we obtain the bound
\begin{align*}
\delta &:=\sum_{(\ell_1,\ldots,\ell_d) \atop \ell_j\le L,\, \ell_1+\ldots+\ell_d>L} |\Delta_{\ell_1}\otimes\ldots\otimes \Delta_{\ell_d}f| 
%\le \sum_{\ell_j\le L,\,\ell_1+\ldots+\ell_d>L} (2N)^d \sum_{k=(k_1,\dots,k_d) \atop k_j\ge 2^{\ell_j}} |c_k|
%\\
%& \le \sum_{\ell_j\le L,\,\ell_1+\ldots+\ell_d>L} (2N)^d \sum_{k=(k_1,\dots,k_d) \atop k_j\ge 2^{\ell_j}} C_r^d \,K_r(g)\,\prod_{j=1}^d k_j^{-r/2}
%\\
%& \le (2NC_r)^d \, K_r(g) \,  \sum_{\ell_j\le L,\,\ell_1+\ldots+\ell_d>L} \ \prod_{j=1}^d \bigl(2^{\ell_j}\bigr)^{-(r/2-1)} 
\\
& \,\le \widetilde C_r^d \, K_r(g) \, \sum_{\Lambda>L} \ \sum_{\ell_j\le L,\,\ell_1+\ldots+\ell_d=\Lambda } \bigl(2^\Lambda\bigr)^{-(r/2-1)} .
\end{align*}
As the cardinality of the set $\{(\ell_1,\dots,\ell_d)\in\N^d \mid \ell_j\le L,\,\ell_1+\ldots+\ell_d=\Lambda \}$ is bounded by $\Lambda^{d-1}/(d-1)!$, 
a rough bound  is given by 
$$
\delta \le \widetilde C_r^d \, K_r(g) \,
\frac{(dL)^{d-1}}{(d-1)!} \, \bigl(2^L\bigr)^{-(r/2-1)} .
$$ 
Using Stirling's formula, we finally obtain the bound, with $C$ independent of $L$, $d$ and $f$,
$$
\delta
\le C^d \, L^{d-1}\, \bigl(2^L\bigr)^{-(r/2-1)} \,K_r(g).
$$
This shows that the sparse-grid quadrature and the full tensor quadrature differ by no more than the right-hand side. In view of the above error bound for the full tensor quadrature, which is proved as in the one-dimensional case, the result follows.
%\end{align*}
%which yields the stated error bound.
\end{proof}

We note that for a Gaussian integral with the inverse complex width matrix $QQ^*$,
$$
 (\pi\eps)^{-d/2} (\det Q)^{-1} \int_{\Rb^d} \exp\Bigl(-\frac1\eps\,(x-q)^T (QQ^*)^{-1} (x-q) \Bigr) \phi(x) \D x,
$$
as occurs in Sections~\ref{sec:gwp} and~\ref{sec:hagwp}, and also in Section~\ref{sec:wigner} with a Gaussian or Hagedorn wave packet as initial data, one would first change variables $y=R(x-q)/\sqrt\eps \in \Rb^d$, where $R=Q^{-1}$ if $Q$ is real, and  is otherwise obtained from a $QR$ decomposition
$$
W \begin{pmatrix} \Re Q^{-1} \\ \Im Q^{-1} \end{pmatrix} = \begin{pmatrix}R \\ 0 \end{pmatrix}
$$
with an orthogonal matrix $W\in\Rb^{2d\times 2d}$, so that for the real vector $v=(x-q)/\sqrt\eps$ we have
$$
v^T (QQ^*)^{-1}  v = v^T R^T R v  = |y|^2.
$$
 Then, (sparse-grid) Gauss--Hermite quadrature is applied to the transformed integral
$$
 \pi^{-d/2} \int_{\Rb^d} \e^{-|y|^2} \, \phi(q+\sqrt\eps R^{-1}y) \D y.
$$

\subsection{Quasi-Monte Carlo methods}

In this and the following subsection we consider the problem of computing an integral
$$
I= \int_{\Rb^d} f(x)\D\mu(x)
$$
of a (complex-valued) integrand $f$ over a probability measure $\mu$.

Quasi-Monte Carlo quadrature is an equi-weighted quadrature on well-chosen deterministic quadrature points. 
\bch For any $x\in\Rb^d$, we denote by $(-\infty,x] \, := \{ y \in \Rb^d\mid y \le x \}$ the rectangular interval
with componentwise inequality, and by $\chi_{(-\infty,x]}$ the associated characteristic function. 
Let $x_1,\ldots,x_N\in\Rb^{d}$ and let
$$
D_N(x_1,\ldots,x_N;x) = \frac{1}{N} \sum_{n=1}^N \chi_{(-\infty, x]}(x_n) -\mu((-\infty, x]),\qquad x\in\Rb^{d},
$$
denote the discrepancy function of the probability measure $\mu$ that quantifies the deviation of the empirical distribution for the interval $(-\infty,x]$.\ech This function can also be viewed as the first-order Peano kernel of the equi-weighted quadrature with nodes $x_n$.

If the measure $\mu$ is the product of one-dimensional probability measures so that the inverses of the one-dimensional cumulative distribution functions are accessible, then the well-established low-discrepancy sets for the uniform measure on the unit cube $[0,1]^{d}$ allow to construct points $x_1,\ldots,x_N\in\Rb^{d}$ with 
\begin{equation}\label{qmc-d}
\sup_{x\in\Rb^{d}} |D_N(x_1,\ldots,x_N;x)| = O\bigl((\log N)^{d-1} / N\bigr);
\end{equation}
see \cite[Theorem~4]{AisD15}. However, the practically important question of how to obtain optimal
low-discrepancy sets for  more general probability measures seems to be open. The following result clarifies why the discrepancy function is crucial for equi-weighted quadrature. It gives a Koksma--Hlawka-type result as presented in \cn{LasS17}; see~also \cite{AisD15,DicKS13}, and the original papers by \cn{Kok42} and \cn{Hla61}.

\begin{theorem}[quasi-Monte Carlo error]\label{lem:Koksma} Let $f$ be a Schwartz function on $\Rb^d$ and $\mu$ a probability distribution on $\Rb^{d}$. Then, for all $x_1,\ldots,x_N\in\Rb^{d}$,
$$
\frac{1}{N} \sum_{n=1}^N f(x_n) - \int_{\Rb^{d}} f(x) \D\mu(x) = (-1)^d\int_{\Rb^{d}}\partial_1\ldots\partial_{d} f(x)\,D_N(x_1,\ldots,x_N;x)\D x.
$$
%where $\partial^{1:2d} = \partial_1\partial_2\cdots\partial_{2d}$ denotes the mixed partial derivative through all dimensions.
\end{theorem}
\begin{proof}
For any $x\in\Rb^d$ we have
\begin{align*}
f(x) &= -\int_{x_1}^\infty \partial_1 f(y_1,x_2,\ldots,x_n) \D y_1 \\
&= (-1)^d \int_{x_1}^\infty \cdots \int_{x_d}^\infty \partial_1\ldots\partial_d f(y_1,\ldots,y_d) \D y_d \cdots \D y_1\\ 
&= (-1)^d \int_{[x,\infty)} \partial_1\ldots\partial_d f(y) \D y.
\end{align*}
This implies for the arithmetic mean
\begin{align*}
\frac1N \sum_{n=1}^N f(x_n) &= 
\frac{(-1)^d}{N} \sum_{n=1}^N \int_{\Rb^d} \chi_{[x_n,\infty)}(y) \;\partial_1\ldots\partial_d f(y) \D y\\
&= (-1)^d \int_{\Rb^d} \partial_1\ldots\partial_d f(y) \frac1N\sum_{n=1}^N \chi_{(-\infty,y]}(x_n) \D y
\end{align*}
and for the integral
$$
\int_{\Rb^d} f(x) \D \mu(x) =(-1)^d \int_{\Rb^d} \partial_1\ldots\partial_d f(y) \,\mu((-\infty,y]) \D y,
$$
where the last equation also uses Fubini's theorem. 
\end{proof}

The discrepancy bound \eqref{qmc-d} does not look more favourable than the error bounds for sparse-grid methods, but it was a significant achievement in recent years to develop versions of quasi-Monte Carlo methods in weighted Sobolev spaces, which allow for dimension-independent error bounds for appropriate integrands that conform to the weighting in that there is a varying degree of importance between the variables; see \cn{DicKS13}. It is not clear at present if this weighting technique can be put to good use for integrals arising in our context.

\subsection{Monte Carlo methods}
To compute the integral
$$
I= \int_{\Rb^d} f(x)\D\mu(x)
$$
of a (complex-valued) integrand $f$ over a probability measure $\mu$, the basic Monte Carlo method takes $N$ independent samples $x_1,\dots,x_N \in\Rb^d$ of the probability distribution $\mu$ and approximates the integral $I$ by the arithmetic mean
$$
I_N = \frac1N \sum_{n=1}^N f(x_n).
$$
The following simple, yet basic result shows an $O(N^{-1/2})$ error behaviour irrespective of the dimension and of differentiability properties of the integrand.

\begin{theorem}[Monte Carlo error]\label{thm:mc} The expected value of the squared error is given by
$$
\mathbb{E}\bigl( | I_N - I |^2 \bigr) = \frac{\mathbb{V}(f)} N
$$
\vskip 2mm
\noindent
with the variance $\mathbb{V}(f) = \int_{\Rb^d} |f(x)|^2 \D\mu(x) - |I|^2.$
\end{theorem}

\begin{proof} We observe that
$$
\mathbb{E}(I_N) = \frac 1N \sum_{n=1}^N \mathbb{E}(f) =I.
$$
Since the samples are independent and identically distributed, we obtain
$$
\mathbb{E}\bigl( | I_N - I |^2 \bigr) = \mathbb{V}(I_N) = \frac 1 {N^2}  \sum_{n=1}^N \mathbb{V}(f) = \frac{\mathbb{V}(f)}N.
$$
Moreover,
\begin{align*}
\mathbb{V}(f) &= \mathbb{E}\bigl( | f - I |^2 \bigr) = \int_{\Rb^d} |f(x)-I|^2 \D\mu(x)
\\
&=  \int_{\Rb^d} |f(x)|^2 \D\mu(x) -   \int_{\Rb^d} 2\,\Re\!(\overline f(x)\, I) \D\mu(x) + |I|^2
\\
&=  \int_{\Rb^d} |f(x)|^2 \D\mu(x) - 2 |I|^2 + |I|^2 =  \int_{\Rb^d} |f(x)|^2 \D\mu(x) -  |I|^2.
\end{align*}
This proves the result.
\end{proof}

The error bound indicates a difficulty with highly oscillatory integrands of approximate wave length $\eps$ integrated against a probability measure with smooth probability density. 
Consider  the prototypical example of a Gaussian  distribution $\mu$ of width 
$\sqrt\eps$, that is, $\mu(x) = (2\pi\eps)^{-d/2} \,\e^{-|x|^2/2\eps}$, and $f_\xi(x)=(2\pi\eps)^{-d/2} \,\e^{\I \xi\cdot x/\eps}$ for $\xi\in\Rb^d$.
The exact integral is
\[
I(\xi)=(2\pi\eps)^{-d} \int_{\Rb^d} \bch \e^{\I \xi\cdot x/\eps}\ech \e^{-|x|^2/\eps} \D x = (2\pi\eps)^{-d/2} \,\e^{-|\xi|^2/(2\eps)},
\]
which is scaled such that $\int_{\Rb^d} I(\xi) \D \xi =1$. However, for all $\xi\in\Rb^d$,
$$
 \int_{\Rb^d} |f_\xi(x)|^2 \D\mu(x) = (2\pi\eps)^{-d}.
$$
%
%
%$$
%  \int_{\Rb^d} |f(x)|^2 \D\mu(x) \Big/ \Bigl| \int_{\Rb^d} f(x)\D\mu(x) \Bigr|^2 
%  \textcolor{blue}{= \e^{|p|^2/(2\eps)}}  \sim \eps^{-d},
%$$
%\textcolor{blue}{since
%\[
%(2\pi\eps)^{-d/2} \int_{\Rb^d} \e^{\I p\cdot x/\eps} \e^{-|x|^2/\eps} \D x = (2\pi\eps)^{-d/2} \e^{-|p|^2/(2\eps)}.
%\]
%}
The error bound also 
motivates modifications that aim to reduce the variance, such as multi-level Monte Carlo methods; see \cn{Gil15}.

Except for special probability measures $\mu$, in particular Gaussians, it is not known {\it a priori} how to draw independent identically distributed (i.i.d.) samples, as is required in the simple Monte Carlo method described above. This difficulty is addressed by Markov chain Monte Carlo methods which replace the i.i.d. variables with variables of a Markov chain that has  $\mu$ as an invariant distribution. This is based on the Metropolis--Hastings acceptance/rejection algorithm \cite{MetRRTT53,Has70}. We refer to \cn{BouS18} for a concise review of basic concepts of Monte Carlo methods and for an analysis of  Hamiltonian (or hybrid) Monte Carlo methods, which are an important subclass of Markov chain Monte Carlo methods.

\subsection{Quadrature for the continuous Gaussian superpositions}
In Section~\ref{alg:gauss_super} we considered particle methods for the thawed and frozen Gaussian superpositions. These numerical algorithms pose problems of high-dimensional oscillatory quadrature 
when it comes to evaluating the defining integrals 
\[
\calI_{\rm th}(t)\psi_0(x) = (2\pi\eps)^{-d} \int_{\Rb^{2d}} \langle g_{z}|\psi_0\rangle \,  
\e^{\I S(t,z)/\eps} \, g[C(t,z)]_{\Phi^t(z)}(x) \D z
\]
and
\[
\calI_\natural(t)\psi_0(x) = (2\pi\eps)^{-d} \int_{\Rb^{2d}} \langle g_{z}|\psi_0\rangle \, a_\natural(t,z) \, 
\e^{\I S(t,z)/\eps} \, g_{\Phi^t(z)}(x) \D z.
\]
Following \cn{LasS17}, 
we briefly apply the previous results on quasi-Monte Carlo quadrature (Theorem~\ref{lem:Koksma}) and 
plain Monte Carlo quadrature (Theorem~\ref{thm:mc}) for the case that 
the initial data are a Gaussian centred in some point $z_0\in\Rb^{2d}$, that is, $\psi_0 = g_{z_0}$. Then, 
the initial wave packet transform satisfies
\[
\langle g_{z}|\psi_0\rangle = \exp(-\tfrac{1}{4\eps}|z-z_0|^2 + \tfrac{\I}{2\eps}(p+p_0)^T(q-q_0)).
\]
We may write 
\[
(2\pi\eps)^{-d} \langle g_{z}|\psi_0\rangle = r_0(z) \mu_0(z)
\]
with 
\[
r_0(z) = 2^d \exp(\tfrac{\I}{2\eps}(p+p_0)^T(q-q_0))
\]
and
\[
\mu_0(z) = (4\pi\eps)^{-d} \exp(-\tfrac{1}{4\eps}|z-z_0|^2),
\]
where $\mu_0$ defines a probability density on phase space associated with the initial data. Correspondingly, the integrals can be expressed as
\[
\calI_{\rm th}(t)\psi_0(x) = %(2\pi\eps)^{-d} 
\int_{\Rb^{2d}} r_0(z)\, 
\e^{\I S(t,z)/\eps} \, g[C(t,z)]_{\Phi^t(z)}(x) \D\mu_0(z)
\]
and
\[
\calI_\natural(t)\psi_0(x) = %(2\pi\eps)^{-d} 
\int_{\Rb^{2d}} r_0(z) \, a_\natural(t,z) \, 
\e^{\I S(t,z)/\eps} \, g_{\Phi^t(z)}(x) \D\mu_0(z).
\]
The integrands 
\[
f_{\rm th}(t,z) = r_0(z) \, \e^{\I S(t,z)/\eps} \, g[C(t,z)]_{\Phi^t(z)}
\]
and
\[
f_\natural(t,z) = r_0(z) \, a_\natural(t,z) \, \e^{\I S(t,z)/\eps} \, g_{\Phi^t(z)}
\]
are time-dependent functions on phase space $\Rb^{2d}$ with values in the space 
of Schwartz functions $\calS(\Rb^d)$. In particular, the 
exponential of the action integral creates large mixed derivatives, in the sense that
\[
\|\partial_1\cdots\partial_{2d} f_j(t,z)\|\sim\eps^{-2d}
\]
for $j={\rm th}$ and $j=\natural$, where the norm belongs to $L^2(\Rb^d)$. 
Therefore, both integrals are difficult candidates for quasi-Monte Carlo quadrature. 
In contrast, the variance of the integrands, 
\[
\mathbb{V}(f_j(t)) = \int_{\Rb^{2d}} \|f_j(t,z)\|^2 \D\mu_0(z) - \|\calI_j(t)\psi_0\|^2, 
\]
grows exponentially with the dimension, but does not heavily depend on the 
oscillation frequency $\eps$. Indeed, 
\[
\mathbb{V}(f_{\rm th}(t)) \; =\;  4^d - \|\calI_{\rm th}(t)\psi_0\|^2 \ \sim\  4^d
\]
and 
\[
\mathbb{V}(f_\natural(t)) \; =\;  
4^d \int_{\Rb^{2d}} |a_\natural(t,z)|^2 \D\mu_0(z)- \|\calI_\natural(t)\psi_0\|^2 \ \sim\  4^d.
\]
Therefore, Monte--Carlo quadrature seems to 
be a better choice for both continuous Gaussian superpositions, working well at least as long as the 
dimension $d$ is moderately large.

%\section{Several potential energy surfaces}

%\section{Further topics}
\section{Further topics}\label{sec:further}

\subsection{Systems of semiclassical Schr\"odinger equations}
The Born--Oppenheimer approximation for the time-dependent molecular Schr\"odinger equation
\[
\I\eps \partial_t\Psi = -\frac{\eps^2}{2}\Delta_x\Psi + H_e(x)\Psi
\]
relies on the presence of gaps in the spectrum of the electronic Hamiltonian~$H_e(x)$. 
For most polyatomic molecules such spectral gaps exist only locally in $x\in\Rb^{3N}$, since different electronic 
eigenvalues may come rather close to each other or even coalesce for certain nuclear 
configurations $x$. 

To illustrate the importance of spectral gaps for the validity of the Born--Oppenheimer approximation, 
let us focus on the case of \bch two\ech electronic eigenvalues $E_1(x)$ and 
$E_2(x)$ with corresponding normalized eigenfunctions $\Phi_1(x,\cdot)$ and $\Phi_2(x,\cdot)$, that is, 
\begin{equation}\label{eq:eig}
H_e(x)\Phi_k(x,\cdot) = E_k(x)\Phi_k(x,\cdot),\qquad k=1,2.
\end{equation}
We seek the Galerkin approximation of the molecular wave function within the subspace
\[
\calV = \left\{u\in L^2_{xy}:\ u(x,y) = \psi_1(x) \Phi_1(x,y) + \psi_2(x)\Phi_2(x,y),\ \psi_j\in L^2_x\right\}.
\]
The associated $2\times 2$ Schr\"odinger system for the motion of the nuclear wave function is then given by
\[
i\eps\, \frac{\partial\psi}{\partial t} = -\frac{\eps^2}{2}\Delta_x\psi + V\psi+ \eps B_1\psi + \eps^2 B_2\psi \for 
\psi=\begin{pmatrix} \psi_1 \\ \psi_2 \end{pmatrix},
\]
where the potential matrix \bch $V=V_{NN}\Id_2 + {\rm diag}(E_1,E_2)$\ech is diagonal, while the matrix operators
\[
B_j =\begin{pmatrix} B_j^1 & C_j \\*[1ex] 
C_j^* & B_j^2 \end{pmatrix},\qquad j=1,2,
\]
carry the diagonal contributions 
\bch
\[
B_1^{k} = \Im\langle \nabla_x\Phi_k\mid \Phi_k\rangle_{L^2_y} \cdot\hat p\quad \text{and}\quad 
B_2^k = \tfrac12 \|\nabla_x\Phi_k\|^2_{L^2_y}.
\]
\ech
The first and second order off-diagonal 
operators $C_1$ and $C_2$ are responsible for the non-adiabatic coupling between the eigenspaces. 

We examine the first of the two coupling operators 
\[
C_1 = -\I \langle\Phi_1(x)\mid\nabla_x\Phi_2(x)\rangle_{L^2_y} \cdot\hat p
\]  
in more detail. Differentiating the eigenvalue equation \eqref{eq:eig} for the second eigenvalue with respect to $x$ and taking the inner product with the first eigenfunction,  we obtain
\begin{align*}
&\langle \Phi_1(x)\mid\nabla_x H_e(x)\Phi_2(x)\rangle_{L^2_y} + 
\langle \Phi_1(x)\mid H_e(x)\nabla_x \Phi_2(x)\rangle_{L^2_y}\\*[1ex]
&= \langle \Phi_1(x)\mid \nabla_x E_2(x)\Phi_2(x)\rangle_{L^2_y} + 
\langle \Phi_1(x)\mid E_2(x)\nabla_x \Phi_2(x)\rangle_{L^2_y}.
\end{align*}
For different eigenvalues, the eigenfunctions are orthogonal, such that
\[
\langle \Phi_1(x)\mid \nabla_x E_2(x)\Phi_2(x)\rangle_{L^2_y} = 0,
\]
and we obtain the fraction
\begin{equation}\label{eq:coup}
\langle \Phi_1(x)\mid\nabla_x\Phi_2(x)\rangle_{L^2_y} \ =\  \frac{\langle\Phi_1(x)\mid \nabla H_e(x)\Phi_2(x)\rangle_{L^2_y}}{E_2(x)-E_1(x)},
\end{equation}
which carries the difference between the two electronic eigenvalues in the denominator.

\subsubsection{A first rule of thumb for adiabatic decoupling}
As a first rule of thumb, we might approximate the first order coupling operator as
\[
|C_1\psi(x)| \ \approx\ |E_2(x)-E_1(x)|^{-1}\, |(\hat p\psi)(x)|. 
\]
From this rule we expect that the coupling diverges for configurations of the nuclei $x_*$ 
where $E_1(x_*) = E_2(x_*)$. In such a situation, the singularity prevents that the $2\times 2$ system 
can approximately be decoupled into two scalar equations, one governed by the potential energy surface $E_1(x)$, 
the other by $E_2(x)$. Then, one has \bch to consider\ech the dynamics of the two-level system altogether. 

\subsubsection{More on adiabatic decoupling} When analysing the difference of the dynamics of the coupled $2\times 2$ system on the one hand and of the fully decoupled one with diagonal Hamiltonian 
\[
-\frac{\eps^2}{2}\Delta_x+ 
\begin{pmatrix} E_1 & 0\\ 0 & E_2\end{pmatrix}
\]
on the other, one encounters the crucial inner product $\langle \Phi_1(x)\mid\nabla_x\Phi_2(x)\rangle_{L^2_y}$ that defines the coupling operator $C_1$ in the following way. One works with the two eigenprojectors related to the electronic subspaces, 
\[
P_k(x,\cdot)\eta = \langle \Phi_k(x,\cdot)\mid\eta\rangle_{L^2_y}\, \Phi_k(x,\cdot),\quad\eta\in L^2_y,
\quad k=1,2,
\]
and has to control (due to the presence of the Laplacian $-\eps^2\Delta_x$) first and second order derivatives of these projectors. We focus on the first derivatives and calculate
\[
\nabla_x P_2(x)\eta = \langle \nabla_x\Phi_2(x)\mid\eta\rangle_{L^2_y}\, \Phi_2(x) + 
\langle \Phi_2(x)\mid\eta\rangle_{L^2_y}\, \nabla_x\Phi_2(x).
\]
Hence, the $(1,2)$-component of the derivative can be written as
\[
P_1(x) \nabla_x P_2(x) P_2(x)\eta = 
\langle \Phi_2(x)\mid\eta\rangle_{L^2_y}\, \langle \Phi_1(x)\mid\nabla_x\Phi_2(x)\rangle_{L^2_y}\,\Phi_1(x).
\]
Being off-diagonal with respect to the electronic eigenspaces, the above operator can be represented as a commutator with the electronic Hamiltonian, 
namely
\begin{equation}\label{eq:comm12}
P_1(x) \nabla_x P_2(x) P_2(x) = [H_e(x),F_{12}(x)]
\end{equation}
with
\[
F_{12}(x) = \frac{1}{E_1(x)-E_2(x)}\  P_1(x) \nabla_x P_2(x) P_2(x).
\]
The presence of such a commutator allows an integration by parts with respect to time when estimating the approximation error. This produces an additional power of the semiclassical parameter $\eps$, which 
allows us to treat the coupling operator $\eps B_1$ effectively as a second order perturbation $\eps^2 B_{1,\natural}$.  
If $B_{1,\natural}$ is uniformly bounded in $\eps$, as is the case in the presence of a  
uniform spectral gap, then one arrives at the space adiabatic error estimate 
of Theorem~\ref{II:thm:ad}. However, if the electronic eigenvalues are not uniformly 
separated, then the operator
\begin{align*}
&F_{12}(x)\eta = \\
&-\frac{1}{(E_2(x)-E_1(x))^2}\, 
\langle \Phi_2(x)\mid\eta\rangle_{L^2_y} \,
\langle\Phi_1(x)\mid \nabla H_e(x)\Phi_2(x)\rangle_{L^2_y}\,
\Phi_1(x)
\end{align*}
contributes significantly to the error constant $c$ of Theorem~\ref{II:thm:ad}, 
thus providing quantitative information on non-adiabatic transitions between the eigenspaces. 

\subsubsection{Avoided and non-avoided crossings} If there are nuclear configurations $x_*$ 
for which the eigenvalue difference falls below a threshold that is approximately smaller than $\sqrt\eps$,
\[
|E_2(x_*) - E_1(x_*)| \ \lessapprox\  \sqrt\eps, 
\]
then one has to expect that the coupling $\eps F_{12}(x_*)\eta$ is either singular or approximately of the order one with respect to $\eps$, and one obtains leading order non-adiabatic transitions between the eigenspaces. 
These transitions, which are ubiquituous for polyatomic molecules and explain spectacular chemical reactions 
as for example the first step of vision (the cis-trans isomerization of retinal in rhodopsin),
have been discussed for the propagation of Hagedorn wave packets through actual crossings by \cn{Hag94} and through avoided crossings in \cite{Hag98a,HagJ99a,BouGH12}. With respect to general initial data, the propagation of expectation values through conical intersections has been carried out in \cite{FerG02,LasT05,FerL08}, while 
avoided crossings have been considered in \cite{FerL17}. These Egorov-type theorems for 
Schr\"odinger systems motivate a class of particle methods based on classical trajectories that switch 
the potential function, whenever they reach a zone of small eigenvalue gap. They have successfully been 
applied to various model systems for the ultrafast conversion of pyrazine up to full dimension \cite{LasS08,XieSDSD19}, 
a twelve-dimensional model of the ammonia cation \cite{BelDLT15}, and a two-dimensional model 
for the hydrogen-detachment in phenol \cite{XieD17}. The switching mechanism of these rigorously analysed  
methods differs from the one of the highly popular surface hopping algorithm of the fewest switches \cite{Tul90}. 
%As not to be expected by its name, the fewest switches algorithm relies on a hopping mechanism evoked 
%at {\em every} time step of the discretization. Moreover, it uses an auxiliary ordinary differential equation
%that is oscillatory with respect to the semiclassical parameter $\eps$, while incorporating the non-adiabatic 
%coupling vectors $\langle\Phi_1(x)\mid\nabla_x\Phi_2(x)\rangle_{L^2_y}$ that become 
%very large in magnitude if not singular in the regions of interest. 
It seems that so far no rigorous derivation 
of the fewest switches approach has been achieved.

\subsubsection{Codimension one crossings}
There is one exceptional class of eigenvalue crossings that do not generate leading order transitions between the eigenspaces, because the numerator of the fraction in \eqref{eq:coup} 
compensates the vanishing of the eigenvalue gap. 
We illustrate this phenomenon for the simplified situation where the electronic Hamiltonian is replaced by a real symmetric $2\times 2$ matrix with eigenvalues $E_1(x)$ and 
$E_2(x)$. We assume that the trace-free part of this matrix is a 
scalar multiple of a matrix whose eigenvalues are uniformly separated. That is, 
there exist smooth scalar functions $\tau, \gamma:\Rb^{3N}\to\C$
and a smooth matrix-valued function $G:\Rb^{3N}\to\Rb^{2\times 2}$, whose 
eigenvalues $\lambda_1(x)$ and $\lambda_2(x)$ are uniformly separated from each other, 
such that
\[
H_e(x) = \tau(x)\Id_{2} + \gamma(x) G(x).
\]
Then, the coincidence of the eigenvalues $E_1(x)$ and $E_2(x)$ can be characterized as
\[
E_1(x_*) = E_2(x_*) \quad\text{if and only if}\quad \gamma(x_*) = 0.
\]
The submanifolds of $\Rb^{3N}$ that satisfy such a condition have generically codimension one. 
Therefore, these crossings are usually referred to as codimension one crossings. 
We observe that the matrices $H_e(x)$ and $G(x)$ share the same eigenfunctions. Thus, we may 
express the non-adiabatic coupling vector as 
\[
\langle \Phi_1(x)\mid\nabla_x\Phi_2(x)\rangle_{\C^2} \ =\  \frac{\langle\Phi_1(x)\mid \nabla G(x)\Phi_2(x)\rangle_{\C^2}}{\lambda_2(x)-\lambda_1(x)}.
\]
This expression is always finite even for nuclear configurations $x_*$ where the eigenvalues $E_1(x)$ and 
$E_2(x)$ coalesce. We therefore obtain a commutator representation of the form \eqref{eq:comm12} with 
a coupling operator
\begin{align*}
&F_{12}(x)\eta = \\
&-\frac{1}{\gamma(x) (\lambda_2(x)-\lambda_1(x))^2}\, 
\langle \Phi_2(x)\mid\eta\rangle_{L^2_y} \,
\langle\Phi_1(x)\mid \nabla G(x)\Phi_2(x)\rangle_{L^2_y}\,
\Phi_1(x)
\end{align*}
that is less singular than for the other crossing cases, since it diverges linearly and not quadratically 
with respect to $|E_2(x)-E_1(x)|^{-1}$. Codimension one crossings have been considered by 
\cn[Chapter~5]{Hag94} and more recently in \cite{LuZ18}.

\subsection{WKB approximation}

WKB approximations, named after work by Wentzel, Kramers \& Brillouin in 1926, are well documented in the literature; see, \eg\
 \cn{Car08}, \cn{JinMS11}  and references therein. We will therefore describe them only very briefly here. An approximation to the wave function $\psi(x,t)$ of the semiclassically scaled Schr\"odinger equation is sought for in the form
$$
\psi(x,t) \approx  a(x,t)\, \e^{\I S(x,t)/\eps}
$$
where $S$ is a real-valued function, $a$ may (or may not) take complex values, and the derivatives of $a$ and $S$ are bounded independently of $\eps$. Typically, $S$ is chosen independently of $\eps$, and $a$ is sought for in the form of a truncated expansion in powers of $\eps$. The initial data are assumed in this form. A short calculation shows that the defect of such an approximation is of order $O(\eps)$ if $S$ satisfies the
{\it eikonal equation}
$$
\partial_t S + \tfrac12 |\nabla S |^2 + V = 0
$$
and $a$ satisfies the transport equation
$$
\partial_t a + \nabla S \cdot \nabla a + \tfrac12 a \Delta S = 0.
$$
The eikonal equation is the Hamilton--Jacobi equation for the Hamilton function
$H(q,p) = \tfrac12 |p|^2 + V(q)$.
Hamilton--Jacobi theory (around 1840) tells us that 
$$
\nabla S(q(t,z_0),t) =  p(t,z_0),
$$
where $\Phi^t(z_0) = (q(t,z_0),p(t,z_0))$ is the flow of the classical equations of motion $\dot q=p, \ \dot p = - \nabla V(q)$ that correspond to the Hamilton function $H$, for initial data $z_0=(q_0,p_0)$ with $p_0=\nabla S(q_0,0)$;
see, e.g., \cn{HaiLW06}, Theorem VI.5.6, from a classical mechanics perspective or
\cn{Car08}, Section 1.3.1, from a partial differential equations perspective. Then, $S(q(t,z_0))$ is obtained from
\begin{align*}
S(q(t,z_0),t) &- S(q_0,0) = \int_0^t \Bigl(\nabla S(q(s,z_0),s)\cdot \dot q (s,z_0) + \partial_t S(q(s,z_0),s)\Bigr)\D s 
\\
&= \int_0^t \Bigl( |p(s,z_0)|^2 - \bigl(\tfrac12 |\nabla S (q(s,z_0),s)|^2 + V(q(s,z_0))\bigr) \Bigr)\D s
\\
&= \int_0^t \Bigl( \tfrac12 |p(s,z_0)|^2 - V(q(s,z_0)) \Bigr) \D s,
\end{align*}
which is the classical action integral along the  trajectory starting from $z_0=(q_0,p_0)$ with $p_0=\nabla S(q_0,0)$.
This motivates using a particle method for the approximate solution of the eikonal equation,
solving the classical equations of motion for many initial positions $q_0$. Also the transport equation for $a$ can then be 
numerically solved by the particle method, since
$$
\frac\D{\D t} a(q(t,z_0),t) = -\tfrac12 a(q(t,z_0),t)\, \Delta S(q(t,z_0),t) 
$$
and since $\Delta S(q(t,z_0),t)$ can be computed by differentiating $\nabla S(q(t,z_0),t) =  p(t,z_0)$ with respect to $q_0$ and using the chain rule, which yields
$$
\Delta S(q(t,z_0),t) =  \text{trace} \biggl(\Bigl( \frac{\partial p}{\partial z_0} \frac{\partial z_0}{\partial q_0} \Bigr)\Bigl( \frac{\partial q}{\partial z_0} \frac{\partial z_0}{\partial q_0} \Bigr) ^{-1} \biggr),
$$
as long as the matrix inverse on the right-hand side exists.
The derivatives of $p$ and $q$ with respect to the initial data are computed by numerically solving the linearised classical equations of motions.
We remark that the pure particle method can be refined to a semi-Lagrangian scheme.

The main difficulty with the WKB approximation is that the eikonal equation develops singularites in finite time, known as {\it caustics}, 
which arise at points where $x=q(t,z_0)$ for more than one initial datum $z_0=(q_0, \nabla S(q_0,0))$.
This limits the applicability of WKB approximations to short times, as opposed to other approximations considered in this review.

Caustics can be mitigated to quasi-caustics by introducing an asymptotic\-ally vanishing viscosity term in the eikonal equation;
see \cn{BesCM13}. However, estimates beyond the time of appearance of caustics of the eikonal equation are not uniform in $\eps$ with 
this modified WKB approach. Based on this reformulation, a time-splitting method for the integration of the semiclassically scaled linear Schr\"odinger equation is proposed and studied by \cn{ChaLM19}.

\subsection{Nonlinear Schr\"odinger equations in the semiclassical regime}

Nonlinear Schr\"odinger equations in semiclassical scaling, such as
$$
\I\eps \,\partial_t \psi(x,t) = -\frac{\eps^2}2 \Delta_x \psi(x,t) + V(x)\psi(x,t) + \eps^\alpha f(|\psi(x,t)|^2) \psi(x,t), 
$$
have been analysed using WKB techniques; see \cn{Car08} as the authoritative reference. \bch\cn{CarF11} studied the propagation of wave packets for the nonlinear Schr\"odinger equation in the semi-classical regime up to the Ehrenfest time.\ech Numerical studies for various
nonlinear Schr\"odinger equations in the semiclassical regime include those by \cn{BaoJM03} and \cn{Kle08}; see also \cn{JinMS11} and further references therein.
A numerical approach based on the caustics-mitigating WKB reformulation is developed by \cn{BesCM13}. Time-splitting for the above
semiclassical nonlinear Schr\"odinger equation is studied by \cn{Car13} and \cn{CarG17} using WKB analysis, and adaptive splitting methods are developed by \cn{AuzKKT16}. Splitting methods for
nonlinear Schr\"odinger-type systems in the context of coupled Ehrenfest dynamics were considered by 
\cn{JinSZ17} and \cn{FanJS18}.

\bigskip
%\newpage
\bibliography{acta_bib}
\end{document}